\documentclass[11pt,a4paper,reqno]{amsart}

\usepackage{amssymb,amsmath}
\usepackage{amsthm}
\usepackage{latexsym}
\usepackage[colorlinks=true,linktocpage=true,pagebackref=false, citecolor=black,linkcolor=black]{hyperref}
\usepackage{graphics}
\usepackage{euscript}
\usepackage{mathrsfs}
\usepackage[dvips]{curves}
\usepackage{tikz}
\usetikzlibrary{arrows,decorations.pathmorphing,backgrounds,positioning,fit,matrix}
\usepackage{comment}
\usepackage{multirow}
\usepackage{xcolor}
\usepackage[all]{xy}
\usepackage{caption,subcaption}
\usepackage{pgfplots}

\pgfplotsset{soldot/.style={color=black,only marks,mark=*}} \pgfplotsset{holdot/.style={color=black,fill=white,only marks,mark=*}}

\newtheorem{thm}{Theorem}[section]
\newtheorem{mthm}[thm]{Main Theorem}
\newtheorem{prop}[thm]{Proposition}
\newtheorem{cor}[thm]{Corollary}
\newtheorem{lem}[thm]{Lemma}

\newtheorem{prob}[thm]{Problem}
\newtheorem{conj}[thm]{Conjecture}

\theoremstyle{definition}

\newtheorem{define}[thm]{Definition}

\theoremstyle{remark}
\newtheorem{remark}[thm]{Remark}

\newtheorem{example}[thm]{Example}
\newtheorem{examples}[thm]{Examples}

\numberwithin{equation}{section}

\renewcommand{\theparagraph}{\thesubsection.\arabic{paragraph}}
 \newcommand{\J}{{\mathcal J}}

 \newcommand{\R}{{\mathbb R}}
 \newcommand{\C}{{\mathbb C}}

\newcommand{\sph}{{\mathbb S}} 
 \newcommand{\PP}{{\mathbb P}}

\newcommand{\Kk}{{\EuScript K}}

\newcommand{\Ss}{{\EuScript S}}

\newcommand{\Qq}{{\EuScript Q}}

\newcommand{\Aa}{{\EuScript A}}
\newcommand\Rr{{\EuScript R}}
\newcommand{\Bb}{{\EuScript B}}
\newcommand{\Dd}{{\EuScript D}}

\newcommand{\Tt}{{\EuScript T}}

\newcommand{\Cc}{{\EuScript C}}

\newcommand{\Ee}{{\EuScript E}}
\newcommand{\Gg}{{\EuScript G}}
\newcommand{\tildebaja}{{\raise.17ex\hbox{$\scriptstyle\sim$}}}

\newcommand{\Int}{\operatorname{Int}}
\newcommand{\im}{\operatorname{Im}}
\newcommand{\cl}{\operatorname{Cl}}
\newcommand{\dist}{\operatorname{dist}}
\newcommand{\id}{\operatorname{id}}

\newcommand{\zar}{\operatorname{zar}}

\newcommand{\rk}{\operatorname{rk}}

\newcommand{\Reg}{\operatorname{Reg}}
\newcommand{\Sing}{\operatorname{Sing}}

\newcommand{\x}{{\tt x}} \newcommand{\y}{{\tt y}}
\newcommand{\z}{{\tt z}} \renewcommand{\t}{{\tt t}}

\newcommand{\gtm}{{\mathfrak m}}

 \newcommand{\gtM}{{\mathfrak M}}

\newcommand{\ol}{\overline}
\newcommand{\qq}[1]{\langle{#1}\rangle}
\newcommand{\QQ}[1]{\Big\langle{#1}\Big\rangle}
\newcommand{\veps}{\varepsilon}

\begin{document}
\title[On Nash images of Euclidean spaces]{On Nash images of Euclidean spaces}

\author{Jos\'e F. Fernando}
\address{Departamento de \'Algebra, Facultad de Ciencias Matem\'aticas, Universidad Complutense de Madrid, 28040 MADRID (SPAIN)}
\email{josefer@mat.ucm.es}
\thanks{Author supported by Spanish GAAR MTM2014-55565-P, Spanish MTM2017-82105-P and Grupos UCM 910444. This article has been finished during a one year research stay of the author in the Dipartimento di Matematica of the Universit\`a di Pisa. The author would like to thank the department for the invitation and the very pleasant working conditions. The $1$-year research stay of the author was partially supported by MECD grant PRX14/00016.}

\dedicatory{Dedicated to Prof. Masahiro Shiota, in memoriam}

\date{06/04/2018}

\begin{abstract}
In this work we characterize the subsets of $\R^n$ that are images of Nash maps $f:\R^m\to\R^n$. We prove Shiota's conjecture and show that \em a subset $\Ss\subset\R^n$ is the image of a Nash map $f:\R^m\to\R^n$ if and only if $\Ss$ is semialgebraic, pure dimensional of dimension $d\leq m$ and there exists an analytic path $\alpha:[0,1]\to\Ss$ whose image meets all the connected components of the set of regular points of $\Ss$\em. Two remarkable consequences are the following: (1) pure dimensional irreducible semialgebraic sets of dimension $d$ with arc-symmetric closure are Nash images of $\R^d$; and (2) semialgebraic sets are projections of irreducible algebraic sets whose connected components are Nash diffeomorphic to Euclidean spaces. 
\end{abstract}

\maketitle
\tableofcontents

\section{Introduction}\enlargethispage{4mm}

Although it is usually said that the first work in Real Geometry is due to Harnack \cite{ha}, who obtained an upper bound for the number of connected components of a non-singular real algebraic curve in terms of its genus, modern Real Algebraic Geometry was born with Tarski's article \cite{ta}, where it is proved that the image of a semialgebraic set under a polynomial map is a semialgebraic set. A map $f:=(f_1,\ldots,f_n):\R^m\to\R^n$ is \em polynomial \em if its components $f_k\in\R[{\tt x}]:=\R[{\tt x}_1,\ldots,{\tt x}_m]$ are polynomials. Analogously, $f$ is \em regular \em if its components can be represented as quotients $f_k=\frac{g_k}{h_k}$ of two polynomials $g_k,h_k\in\R[{\tt x}]$ such that $h_k$ never vanishes on $\R^m$. A subset $\Ss\subset\R^n$ is \em semialgebraic \em when it has a description by a finite boolean combination of polynomial equalities and inequalities, which we will call a \em semialgebraic \em description. Unless stated otherwise, the topology employed in the article is the Euclidean one.

We are interested in studying what might be called the `inverse problem' to Tarski's result. In the 1990 \em Oberwolfach reelle algebraische Geometrie \em week \cite{g} Gamboa proposed: 

\begin{prob}\label{prob0}
To characterize the (semialgebraic) subsets of $\R^n$ that are either polynomial or regular images of $\R^m$. 
\end{prob}

During the last decade we have attempted to understand better polynomial and regular images of $\R^m$. Our main objectives have been the following: 
\begin{itemize}
\item To find obstructions to be either polynomial or regular images. 
\item To prove (constructively) that large families of semialgebraic sets with piecewise linear boundary (convex polyhedra, their interiors, complements and the interiors of their complements) are either polynomial or regular images of Euclidean spaces. 
\end{itemize}

In \cite{fg1,fg2} we presented the first step to approach Problem \ref{prob0}. In \cite{fe1} appears a complete solution to Problem \ref{prob0} for the $1$-dimensional case, whereas in \cite{fgu1,fgu3,fu1,fu2,fu4,fu5,u1,u2} we approached constructive results concerning the representation as either polynomial or regular images of the semialgebraic sets with piecewise linear boundary commented above. A survey concerning this topic, which provides the reader a global idea of the state of the art, can be found in \cite{fgu4}. Articles \cite{fgu2,fu3} are of different nature. In them we find new obstructions for a semialgebraic subset of $\R^n$ to be either a polynomial or a regular image of $\R^m$. In the first one we found some properties concerning the difference ${\rm Cl}(\Ss)\setminus\Ss$ for a polynomial image $\Ss$ of $\R^m$ whereas in the second it is shown that \em the set of points at infinite of $\Ss$ is a connected set\em.

The rigidity of polynomial and regular maps makes really difficult to approach Problem \ref{prob0} in its full generality. Taking into account the flexibility of Nash maps, Gamboa and Shiota discussed in 1990 the possibility of approaching the following variant of Problem \ref{prob0}. 

\begin{prob}\label{prob1}
To characterize the (semialgebraic) subsets of $\R^n$ that are Nash images of $\R^m$. 
\end{prob}

A \em Nash function on an open semialgebraic set \em $U\subset\R^n$ is a semialgebraic smooth function on $U$. Recall that a (non-necessarily continuous) map $f:\Ss\to\Tt$ is \em semialgebraic \em if its graph is a semialgebraic set (in particular we assume that both $\Ss$ and $\Tt$ are semialgebraic sets). Given a semialgebraic set $\Ss\subset\R^n$, a \em Nash function on $\Ss$ \em is the restriction to $\Ss$ of a Nash function on an open semialgebraic neighborhood $U\subset\R^n$ of $\Ss$. 

In 1990 Shiota outlined to Gamboa and the rest of the Real Geometry team at Madrid a vague schedule that sustains the following conjecture (wrongly announced in \cite{g,fg1} as proved by Shiota) in order to provide a satisfactory answer to Problem \ref{prob1}. 

\begin{conj}[Shiota]\label{main0}
Let $\Ss\subset\R^n$ be a semialgebraic set of dimension $d$. Then $\Ss$ is a Nash image of $\R^d$ if and only if $\Ss$ is pure dimensional and there exists an analytic path $\alpha:[0,1]\to\Ss$ whose image meets all connected components of the set of regular points of $\Ss$.
\end{conj}
The set of regular points of a semialgebraic set $\Ss\subset\R^n$ is defined as follows. Let $X$ be the Zariski closure of $\Ss$ in $\R^n$ and let $\widetilde{X}$ be the complexification of $X$, that is, the smallest complex algebraic subset of $\C^n$ that contains $X$. Define $\Reg(X):=X\setminus\Sing(\widetilde{X})$ and let $\Reg(\Ss)$ be the interior of $\Ss\setminus\Sing(\widetilde{X})$ in $\Reg(X)$. We will explain this in more detail in \ref{s2a}.

In 2004 we met again with Shiota and discussed about possible ways to attack his conjecture. It was not clear how to follow certain parts of his 1990 schedule and we have performed strong variations and substantially simplified the architecture of the approach. However, that fruitful meeting was the starting point for the present work and some related ones \cite{bfr,fgr}. The latter include useful tools for this article concerning: 
\begin{itemize}
\item[(1)] Extension of Nash functions on a Nash manifold $H$ with boundary to a Nash manifold $M$ of its same dimension that contains $H$ as a closed subset \cite{fgr}.
\item[(2)] Approximation results on a Nash manifold relative to a Nash subset with monomial singularities \cite{bfr}.
\item[(3)] Equivalence of Nash classification and ${\mathcal C}^2$ semialgebraic classification for Nash manifolds with boundary \cite{bfr}.
\end{itemize}

Recall that an \em (affine) Nash manifold with or without boundary \em is a pure dimensional semialgebraic subset $M$ of some affine space $\R^m$ that is a smooth submanifold with or without boundary of an open subset of $\R^m$. As all the Nash manifolds with or without boundary appearing in this work are affine, we will assume this property when referring to Nash manifolds with or without boundary. In addition when we refer to a Nash manifold with boundary, we assume that this boundary is smooth and in fact a Nash submanifold. The zero set of a Nash function on a Nash manifold $M$ is called a \em Nash subset of $M$\em. 

\subsection{Main results}

The main result of this work is Theorem \ref{main} that includes a positive solution to Shiota's Conjecture. Its statement requires some preliminary definitions. Let $\alpha:[0,1]\to\R^n$ be a continuous semialgebraic path. Let $A\subset(0,1)$ be the smallest (finite) subset of $(0,1)$ such that the restriction $\alpha|_{(0,1)\setminus A}$ is a Nash map. Denote ${\eta}(\alpha):=\alpha(A)$.

A semialgebraic set $\Ss\subset\R^n$ is \em well-welded \em if it is pure dimensional and for each pair of points $x,y\in\Ss$ there exists a continuous semialgebraic path $\alpha:[0,1]\to\Ss$ such that $\alpha(0)=x$, $\alpha(1)=y$ and ${\eta}(\alpha)\subset\Reg(\Ss)$.

\begin{mthm}[Characterization of Nash images]\label{main}
Let $\Ss\subset\R^n$ be a semialgebraic set of dimension $d$. The following assertions are equivalent:
\begin{itemize}
\item[(i)] $\Ss$ is a Nash image of $\R^d$.
\item[(ii)] $\Ss$ is a Nash image of $\R^m$ for some $m\geq d$.
\item[(iii)] $\Ss$ is connected by Nash paths.
\item[(iv)] $\Ss$ is connected by analytic paths.
\item[(v)] $\Ss$ is pure dimensional and there exists a Nash path $\alpha:[0,1]\to\Ss$ whose image meets all the connected components of the set of regular points of $\Ss$.
\item[(vi)] $\Ss$ is pure dimensional and there exists an analytic path $\alpha:[0,1]\to\Ss$ whose image meets all the connected components of the set of regular points of $\Ss$.
\item[(vii)] $\Ss$ is well-welded.
\end{itemize}
\end{mthm}

The implications (i) $\Longrightarrow$ (ii) $\Longrightarrow$ (iii) $\Longrightarrow$ (iv) and (i) $\Longrightarrow$ (ii) $\Longrightarrow$ (v) $\Longrightarrow$ (vi) are straightforward. Only the proof of the non-completely trivial implication (ii) $\Longrightarrow$ `$\Ss$ is pure dimensional' requires a comment and it is shown in Corollary \ref{imnm}. We will show in Section \ref{s7} that a semialgebraic set $\Ss$ satisfying either condition (iii), (iv), (v) or (vi) is well-welded. Finally, we will prove in Section \ref{s8} that (vii) $\Longrightarrow$ (i). An important milestone to prove Theorem \ref{main} is the following result, that will be approached in Section \ref{s6}.

\begin{thm}\label{mstone}
Let $H\subset\R^n$ be a connected $d$-dimensional Nash manifold with boundary. Then $H$ is a Nash image of $\R^d$.
\end{thm}

In Section \ref{s3} we treat separately the $1$-dimensional case and we characterize $1$-dimensional Nash images of Euclidean spaces in terms of their irreducibility. The ring ${\mathcal N}(\Ss)$ of Nash functions on a semialgebraic set $\Ss\subset\R^n$ is a noetherian ring \cite[Thm.2.9]{fg3} and we say that $\Ss$ is \em irreducible \em if and only if ${\mathcal N}(\Ss)$ is an integral domain \cite{fg3}.

\begin{prop}[The $1$-dimensional case]\label{curves}
Let $\Ss\subset\R^n$ be a $1$-dimensional semialgebraic set. Then $\Ss$ is a Nash image of some $\R^m$ if and only if $\Ss$ is irreducible. In addition, if such is the case $\Ss$ is a Nash image of $\R$.
\end{prop}

Compare Proposition \ref{curves} with the more restrictive characterization results for polynomial and regular images of Euclidean spaces \cite{fe1}.

\subsection{Two consequences}
We present next two remarkable consequences of Theorem \ref{main}. 

\subsubsection{Representation of arc-symmetric semialgebraic sets} Arc-symmetric semialgebraic sets were introduced by Kurdyka in \cite{k} and subsequently studied by many authors. Recall that a semialgebraic set $\Ss\subset\R^n$ is \em arc-symmetric \em if for each analytic arc $\gamma:(-1,1)\to\R^n$ with $\gamma((-1,0))\subset\Ss$ it holds that $\gamma((-1,1))\subset\Ss$. In particular arc-symmetric semialgebraic sets are closed subsets of $\R^n$. An arc-symmetric semialgebraic set $\Ss\subset\R^n$ is \em irreducible \em if it cannot be written as the union of two proper arc-symmetric semialgebraic subsets \cite[\S2]{k}. Equivalently, $\Ss$ is irreducible if and only if the ring ${\mathcal N}(\Ss)$ is an integral domain. It follows from Theorem \ref{main} and \cite[Cor.2.8]{k} that a pure dimensional irreducible arc-symmetric semialgebraic set is a Nash image of $\R^d$ where $d:=\dim(\Ss)$. In addition, it holds:

\begin{cor}\label{consq1}
Let $\Ss\subset\R^n$ be a pure dimensional irreducible semialgebraic set of dimension $d$ whose closure $\cl(\Ss)$ is arc-symmetric. Then $\Ss$ is a Nash image of $\R^d$. 
\end{cor}

\subsubsection{Elimination of inequalities} Tarski-Seidenberg principle on elimination of quantifiers can be restated geometrically by saying that the projection of a semialgebraic set is again semialgebraic. A converse problem, to find an algebraic set in $\R^{n+k}$ whose projection is a given semialgebraic subset of $\R^n$, is known as the \em problem of eliminating inequalities\em. Motzkin proved in \cite{m2} that this problem always has a solution for $k=1$. However, his solution is rather complicated and is generally a reducible algebraic set. In another direction Andradas--Gamboa proved in \cite{ag1,ag2} that if $\Ss\subset\R^n$ is a closed semialgebraic set whose Zariski closure is irreducible, then $\Ss$ is the projection of an irreducible algebraic set in some $\R^{n+k}$. In \cite{pe} Pecker gives some improvements on both results: for the first by finding a construction of an algebraic set in $\R^{n+1}$ that projects onto the given semialgebraic subset of $\R^n$, far simpler than the original construction of Motzkin; for the second by proving that if $\Ss$ is a locally closed semialgebraic subset of $\R^n$ with an interior point, then $\Ss$ is the projection of an irreducible algebraic subset of $\R^{n+1}$.

In this article we prove the following result that looks for a non-singular algebraic set with the simplest possible topology that projects onto a semialgebraic set.

\begin{cor}\label{consq2}
Let $\Ss\subset\R^n$ be a semialgebraic set of dimension $d$. We have:
\begin{itemize}
\item[(i)] If $\Ss$ is Nash path-connected, it is the projection of an irreducible non-singular algebraic set $X\subset\R^{n+k}$ (for some $k\geq0$) whose connected components are Nash diffeomorphic to $\R^d$. In addition: 
\begin{itemize}
\item[(1)] Each connected component of $X$ projects onto $\Ss$.
\item[(2)] Given any two of the connected components of $X$ there exists an automorphism of $X$ that swaps them. 
\end{itemize}
\item[(ii)] In general $\Ss$ is the projection of an algebraic set $X\subset\R^{n+k}$ (for some $k\geq0$) that is Nash diffeomorphic to a finite pairwise disjoint union of affine subspaces of $\R^{d+1}$.
\end{itemize}
\end{cor}
Even for dimension $1$, it is not possible to impose the connectedness of $X$ (see Example \ref{nint} and Remark \ref{nint2}).

\subsection{Structure of the article}

The article is organized as follows. In Section \ref{s2} we present some basic notions and notations used in this paper as well as some preliminary results. The reader can start directly in Section 3 and resort to the Preliminaries only when needed. In Section \ref{s3} we afford the $1$-dimensional case. Its presentation is short and evidences strong differences with the polynomial and regular cases \cite{fe1}. In Sections \ref{s4} and \ref{s5} we analyze with care the main properties of Nash collars, Nash doubles and the \em drilling blow-up\em, which is an adaptation to the Nash setting of the \em oriented blow-up \em \cite{hi2,ds} of a real analytic space with center a closed real analytic subspace. We refer the reader to \cite[\S5]{hpv} for a presentation of the oriented blow-up of a real analytic manifold $M$ with center a closed real analytic submanifold $N$ whose vanishing ideal inside $M$ is finitely generated (this happens for instance if $N$ is compact). In our case, we take advantage of the noetherianity of the ring of Nash functions on a Nash manifold to develop the drilling blow-up and to obtain stronger global properties than in the general real analytic case. The previous tools (Nash collars, Nash doubles and drilling blow-ups) are the key to build boundaries on Nash manifolds in Proposition \ref{doubleivc} and to modify the topology of a Nash manifold in order to prove Theorem \ref{mstone} in Section \ref{s6}. We are convinced that the strategy followed to prove Theorem \ref{mstone} will have further applications. In Section \ref{s7} we study the main properties of well-welded semialgebraic sets and we show that a semialgebraic set satisfying any of the assertions (iii), (iv), (v) or (vi) in the statement of Theorem \ref{main} is well-welded. Next, in Section \ref{s8} we prove Theorem \ref{main}. In Section \ref{s9} we introduce the concept of \em Nash path-connected components \em of a semialgebraic set and we prove that each semialgebraic set can be (uniquely) written as the (finite) union of its Nash path-components. Finally, in Section \ref{s10} we prove Corollaries \ref{consq1} and \ref{consq2}. The article ends with two Appendices. In the first one we present a miscellanea of ${\mathcal C}^2$ semialgebraic homeomorphisms between intervals that are used in the article and in the second we recall certain results concerning strict transforms of analytic and Nash paths under finite chains of blow-ups.

\subsection*{Acknowledgements}

The author is strongly indebted to Prof. Shiota for some key discussions and suggestions during the initial preparation of some parts of this article. The author thanks Prof. Gamboa and Dr. Ueno for their helpful comments when writing and correcting this article. The author is also very grateful to Dr. Benito for informal discussions about resolution of singularities in order to improve the redaction and simplify some of the results in Sections \ref{s7} and \ref{s8}. The mathematics of this article were finished during a one year research stay at the Dipartimento di Matematica of the Universit\`a di Pisa. The author strongly appreciates the effort of the members of the Departamento de \'Algebra of the Universidad Complutense de Madrid for giving him the opportunity to enjoy a sabbatical year (entirely dedicated to his research activity) during the academic course 2014-2015. The author is also deeply indebted with the anonymous referee who has helped him with his suggestions to improve the final presentation of this work and to clarify some misleadings in the redaction of the proof of implication (vii) $\Longrightarrow$ (i) of Theorem \ref{main}.

\section{Preliminaries on semialgebraic sets and Nash manifolds}\label{s2}

In this section we introduce many concepts and notation needed in the article. We establish the following conventions: $M\subset\R^m$ and $N\subset\R^n$ are Nash manifolds. Nash subsets of a Nash manifold or algebraic subsets of $\R^n$ are denoted with $X, Y$ and $Z$. The semialgebraic sets are denoted with $\Ss,\Tt,\Rr,\ldots$ On the other hand, $H\subset\R^m$ is a Nash manifold of dimension $d$ with (smooth) boundary, $\partial H$ is its boundary and $\Int(H)=H\setminus\partial H$ is its interior. In addition, ${\mathcal C}^r$ semialgebraic and Nash functions on a semialgebraic set are denoted with $f,g,h,\ldots$

Recall some general properties of semialgebraic sets. Semialgebraic sets are closed under Boolean combinations and by quantifier elimination they are also closed under projections. Any set defined by a first order formula in the language of ordered fields is a semialgebraic set \cite[p.28, 29]{bcr}. Thus, the basic topological constructions as closures, interiors or boundaries of semialgebraic sets are again semialgebraic. Also images and preimages of semialgebraic sets by semialgebraic maps are again semialgebraic. The \em dimension \em $\dim(\Ss)$ of a semialgebraic set $\Ss$ is the dimension of its Zariski closure \cite[\S2.8]{bcr}. The \em local dimension $\dim(\Ss_x)$ of $\Ss$ at a point $x\in\cl(\Ss)$ \em is the dimension $\dim(U)$ of a small enough open semialgebraic neighborhood $U\subset\cl(\Ss)$ of $x$. The dimension of $\Ss$ coincides with the maximum of these local dimensions. For any fixed $k$ the set of points $x\in\Ss$ such that $\dim(\Ss_x)=k$ is a semialgebraic subset of $\Ss$. 

\subsection{Set of regular points of a semialgebraic set}\label{s2a}

Let $Z\subset\C^n$ be a complex algebraic set and let $I_\C(Z)$ be the ideal of all polynomials $F\in\C[\x]$ such that $F(z)=0$ for each $z\in Z$. A point $z\in Z$ is \em regular \em if the localization of the polynomial ring $\C[\x]/I_\C(Z)$ at the maximal ideal $\gtM_z$ associated to $z$ is a regular local ring. In this complex setting the Jacobian criterion and Hilbert's Nullstellensatz imply that $z\in Z$ is regular if and only if there exists an open neighborhood $U\subset\C^n$ of $z$ such that $U\cap Z$ is an analytic manifold. We denote $\Reg(Z)$ the set of regular points of $Z$ and it is an open dense subset of $Z$. If $Z$ is irreducible, it is pure dimensional and $\Reg(Z)$ is a connected analytic manifold. In case $Z$ is not irreducible, then the connected components of $\Reg(Z)$ are finitely many analytic manifolds (possibly of different dimensions). We denote $\Sing(Z):=Z\setminus\Reg(Z)$ the \em set of singular points of $Z$\em.

Let $X\subset\R^n$ be a (real) algebraic set and let $I_\R(X)$ be the ideal of all polynomials $f\in\R[\x]$ such that $f(x)=0$ for each $x\in X$. A point $x\in X$ is \em regular \em if the localization of $\R[\x]/I_\R(X)$ at the maximal ideal $\gtm_x$ associated to $x$ is a regular local ring \cite[\S3.3]{bcr}. Let $\widetilde{X}\subset\C^n$ be the complex algebraic set that is the zero set of the extended ideal $I_\R(X)\C[\x]$. We call $\widetilde{X}$ the \em complexification of $X$\em. The ideal $I_\C(\widetilde{X})$ coincides with the tensorized ideal $I_\R(X)\otimes_\R\C$, so $\widetilde{X}$ is the smallest complex algebraic subset of $\C^n$ that contains $X$ and
$$
\C[\x]/I_\C(\widetilde{X})\cong(\R[\x]/I_\R(X))\otimes_\R\C.
$$
The localization $(\R[\x]/I_\R(X))_{\gtm_x}$ is a regular local ring if and only so is its complexification 
$$
(\R[\x]/I_\R(X))_{\gtm_x}\otimes_\R\C\cong(\C[\x]/I_\C(\widetilde{X}))_{\gtM_x}.
$$
Thus, the \em set of regular points \em of $X$ is $\Reg(X)=\Reg(\widetilde{X})\cap X$ and its \em set of singular points \em is $\Sing(X):=X\setminus\Reg(X)=\Sing(\widetilde{X})\cap X$. The connected components of the open semialgebraic subset $\Reg(X)$ of $X$ is a finite union of Nash manifolds (possibly of different dimensions). 

Let $\Ss\subset\R^n$ be a semialgebraic set of dimension $d$. The Zariski closure $\ol{\Ss}^{\zar}$ of $\Ss$ in $\R^n$ is the smallest algebraic subset of $\R^n$ that contains $\Ss$. We define 
$$
\Reg(\Ss):=\Int_{\Reg(\ol{\Ss}^{\zar})}(\Ss\setminus\Sing(\ol{\Ss}^{\zar}))\quad\text{and}\quad\Sing(\Ss):=\Ss\setminus\Reg(\Ss).
$$
The connected components of the open subset $\Reg(\Ss)$ of $\ol{\Ss}^{\zar}$ is a finite union of Nash manifolds (possibly of different dimensions) and $\Sing(\Ss)$ is a semialgebraic set of dimension $<d$, which is closed in $\Ss$. The set $\Reg_k(\Ss)$ of points of dimension $k$ of $\Reg(\Ss)$ is either the empty-set or a Nash manifold of dimension $k$ for each $k=0,1,\ldots,d$. If $\Ss$ is pure dimensional, $\Reg(\Ss)$ is a dense subset of $\Ss$. A point $x\in\Ss$ is \em smooth \em if there exists an open neighborhood $U\subset\R^n$ of $x$ such that $U\cap\Ss$ is a Nash manifold. It holds that each regular point is a smooth point, but the converse is not always true even if $\Ss=X$ is a real algebraic set, as it shows the following example.

\begin{example}\label{snr}\enlargethispage{1mm}
Consider the algebraic set $X:=\{(x^2+zy^2)x-y^4=0\}\subset\R^3$. The set of regular points of $X$ is $X\setminus\{x=0,y=0\}$, whereas the set of smooth points of $X$ is $X\setminus\{x=0,y=0,z\leq0\}$ (see Figure \ref{fig1}). To prove that the points of the open half-line $\{x=0,y=0,z<0\}$ are non-\makebox[\textwidth][s]{smooth we proceed by contradiction. Pick a point $p:=(0,0,-a^2)\in\{x=0,y=0,z<0\}$ and} 

\vspace*{-3mm}
\begin{figure}[ht]
\noindent\begin{minipage}{0.525 \textwidth}
assume that it is smooth. As the line $\{x=0,y=0\}\subset X$, the vector $(0,0,1)$ would be tangent to $X$ at $p$, so the plane $z=-a^2$ would be transversal to $X$ at $p$. Thus, the intersection $X\cap\{z=-a^2\}$ should be a curve that is smooth at $p$, but this is a contradiction because such curve $\{(x^2-(ay)^2)x-y^4=0,z=-a^2\}$ has three tangent lines at $p$, which are those lines of equations $\{x-ay=0\}$, $\{x+ay=0\}$ and $\{x=0\}$ inside the plane $\{z=-a^2\}$. The origin cannot be a smooth point of $X$ because the set of smooth points of $X$ is an open subset of $X$. Consequently, the set of non-smooth points of $X$ contains the closed half-line $\{x=0,y=0,z\leq0\}$. To finish we prove that the points of the open half-line $\{x=0,y=0,z>0\}$ are smooth. To that end, observe that the map $\varphi:\{(t,s)\in\R^2:\ t>0\}\to\R^3, (s,t)\mapsto((s^2+t^2)s^2,(s^2+t^2)s,t^2)$ is a Nash embedding whose image is $X\cap\{z>0\}$.
\end{minipage}\hfill 
\begin{minipage}{0.475\textwidth}
\begin{center}
\includegraphics[width=0.8\textwidth]{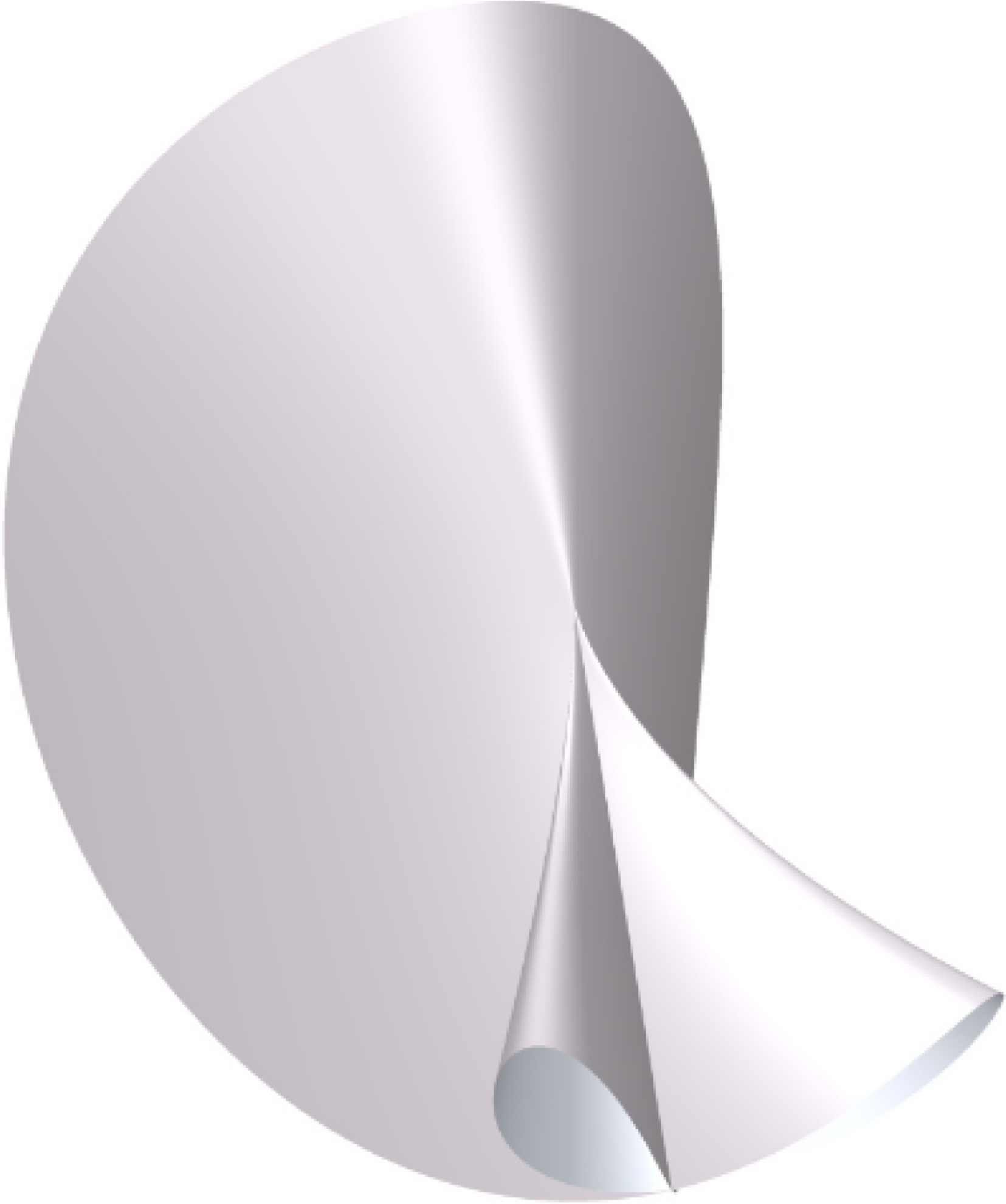}
\caption{$X:\ (x^2+zy^2)x-y^4=0$\label{fig1}}
\end{center}
\end{minipage}
\end{figure}
\end{example}

\vspace*{-4mm}
The set ${\tt Sth}(\Ss)$ of smooth points of a semialgebraic set $\Ss\subset\R^n$ is by \cite{st} a semialgebraic subset of $\R^n$ (and consequently a Nash submanifold of $\R^n$), which contains $\Reg(\Ss)$ (maybe as a proper subset as it happens in Example \ref{snr}), and it is open in $\Ss$. The set ${\tt Sth}_k(\Ss)$ of points of dimension $k$ of ${\tt Sth}(\Ss)$ is either the empty-set or a Nash manifold of dimension $k$ for each $k=0,1,\ldots,d$. Denote ${\tt NSth}(\Ss):=\Ss\setminus{\tt Sth}(\Ss)$ the set of non-smooth points of $\Ss$. If $X$ is an algebraic set, $\Sing(X)$ is always an algebraic subset of $X$ whereas ${\tt NSth}(X)$ is in general only a semialgebraic subset of $X$, see Example \ref{snr}.

\begin{remark}\label{regsmooth}
Let $\Ss\subset\R^n$ be a pure dimensional semialgebraic set such that $\ol{\Ss}^{\rm zar}\subset\R^n$ is a non-singular real algebraic set. Then $\Reg(\Ss)=\Int_{\ol{\Ss}^{\zar}}(\Ss)={\tt Sth}(\Ss)$ and $\Sing(\Ss)={\tt NSth}(\Ss)$.

As $\ol{\Ss}^{\rm zar}$ is a non-singular real algebraic set, $\Reg(\ol{\Ss}^{\zar})=\ol{\Ss}^{\rm zar}$ and $\Sing(\ol{\Ss}^{\rm zar})=\varnothing$. Thus, as $\Ss$ is pure dimensional, $\Int_{\ol{\Ss}^{\zar}}(\Ss)={\tt Sth}(\Ss)$ and 
\begin{align*}
\Reg(\Ss)&=\Int_{\Reg(\ol{\Ss}^{\zar})}(\Ss\setminus\Sing(\ol{\Ss}^{\zar}))=\Int_{\ol{\Ss}^{\zar}}(\Ss)={\tt Sth}(\Ss),\\\Sing(\Ss)&=\Ss\setminus\Reg(\Ss)=\Ss\setminus{\tt Sth}(\Ss)={\tt NSth}(\Ss).
\end{align*}
\end{remark}

Using the definition of smooth point one proves readily the following result.

\begin{lem}\label{sth}
Let $M\subset\R^m$ and $N\subset\R^n$ be Nash manifolds and let $f:M\to N$ be a Nash diffeomorphism. Let $\Ss\subset M$ be a semialgebraic set. Then $f({\tt Sth}(\Ss))={\tt Sth}(f(\Ss))$ and $f({\tt NSth}(\Ss))={\tt NSth}(f(\Ss))$.
\end{lem}

The previous result is no longer true if we consider the set $\Reg(\Ss)$ of regular points of $\Ss$ instead of ${\tt Sth}(\Ss)$. 

\begin{example}\label{atention}
Consider the semialgebraic set $\Ss:=\{x^2+y^2-y^3=0\}\subset\R^2$ and the Nash diffeomorphism 
$$
f:\R^2\to\R^2,\ (x,y)\mapsto(x(x^2+1),x^2+1+y).
$$
Denote $\Tt:=\{(0,-1)\}\cup\{y=0\}$ and observe that $f(\Tt)=\Ss$, $\Sing(\Tt)=\varnothing$ and $\Sing(\Ss)=\{(0,0)\}$. Consequently, $f(\Reg(\Tt))=f(\Tt)=\Ss\neq\Reg(\Ss)$.
\end{example}

The previous example shows that the sets $\Reg(\Ss)$ and $\Sing(\Ss)$ depend on how $\Ss$ is immersed in an affine space, whereas Lemma \ref{sth} points out that the sets ${\tt Sth}(\Ss)$ and ${\tt NSth}(\Ss)$ are subsets of $\Ss$ of intrinsic nature. As we will use in the sequel resolution of singularities, we will need to employ the sets $\Reg(\Ss)$ and $\Sing(\Ss)$ instead of the sets ${\tt Sth}(\Ss)$ and ${\tt NSth}(\Ss)$ in a large part of the article. However, in the proof of implication (vii) $\Longrightarrow$ (i) of Theorem \ref{main} we will take advantage of the intrinsic nature of the sets ${\tt Sth}(\Ss)$ and ${\tt NSth}(\Ss)$. This justifies the introduction of both pairs of concepts: `regular/singular' points and `smooth/non-smooth' points. In Lemma \ref{retoque} we show that under certain conditions we may embedded $\Ss$ in some affine space in order to have $\Reg(\Ss)={\tt Sth}(\Ss)$ and $\Sing(\Ss)={\tt NSth}(\Ss)$.

\subsection{Desingularization of algebraic sets}
Let $X\subset Y\subset\R^n$ be algebraic sets such that $Y$ is non-singular. Recall that $X$ is a \em normal-crossings divisor of $Y$ \em if for each point $x\in Y$ there exists a regular system of parameters $x_1,\ldots,x_d$ for $Y$ at $x$ such that $X$ is given on an open Zariski neighborhood of $x$ in $Y$ by the equation $x_1\cdots x_k=0$ for some $k\leq d$. In particular, the irreducible components of $X$ are non-singular and have codimension $1$ in $Y$. 

A rational map $f:=(f_1,\ldots,f_n):Z\to\R^n$ on an algebraic set $Z\subset\R^m$ is \em regular \em if its components are quotients of polynomials $f_k:=\frac{g_k}{h_k}$ such that $Z\cap\{h_k=0\}=\varnothing$.

Hironaka's desingularization results \cite{hi} are powerful tools that we will use fruitfully in Sections \ref{s7} and \ref{s8}. We recall here the two results we need.

\begin{thm}[Desingularization]\label{hi1}
Let $X\subset\R^n$ be an algebraic set. Then there exist a non-singular algebraic set $X'\subset\R^m$ and a proper regular map $f:X'\to X$ such that
$$
f|_{X'\setminus f^{-1}(\Sing(X))}:X'\setminus f^{-1}(\Sing(X))\rightarrow X\setminus\Sing X
$$
is a diffeomorphism whose inverse map is also regular.
\end{thm}

\begin{remark}
If $X$ is pure dimensional, $X\setminus\Sing X$ is dense in $X$. As $f$ is proper, it is surjective.
\end{remark}

\begin{thm}\label{hi2}
Let $X\subset Y\subset\R^n$ be algebraic sets such that $Y$ is non-singular. Then there exists a non-singular algebraic set $Y'\subset\R^m$ and a proper surjective regular map $g:Y'\to Y$ such that $g^{-1}(X)$ is a normal-crossings divisor of $Y'$ and the restriction
$$
g|_{Y'\setminus g^{-1}(X)}:Y'\setminus g^{-1}(X)\to Y\setminus X
$$
is a diffeomorphism whose inverse map is also regular.
\end{thm}

\subsection{Nash manifolds and Nash normal-crossings divisors}\setcounter{paragraph}{0}

The open semialgebraic subsets of a Nash manifold $M\subset\R^m$ are a base of the topology and therefore Nash functions define a sheaf that we denote with ${\mathcal N}$. In particular, the sheaf ${\mathcal N}$ induces a notion of Nash function $f:U\to \R$ over an arbitrary open subset $U$ of $M$ possibly not semialgebraic. In case $U$ is an open subset of $\R^m$, \em Nash \em means that $f$ is smooth and there exists a nonzero polynomial $P(\x,\t)\in\R[\x,\t]=\R[\x_1,\ldots,\x_m,\t]$ such that $P(x,f(x))=0$ for all $x\in U$. If $U$ is semialgebraic this definition is equivalent to the one in the Introduction \cite[Prop.8.1.8]{bcr}. We apparently have two definitions of Nash function $f:M\to \R$. Both notions are equal because \em every Nash manifold has a Nash tubular neighborhood\em. 

\paragraph{}\label{fact:retract}\cite[Cor.8.9.5]{bcr} \em Let $M\subset\R^m$ be a Nash manifold. Then there exists an open semialgebraic neighborhood $U$ of $M$ in $\R^m$ and a Nash retraction $\rho:U\to M$\em. 

Recall the existence of finite atlas for $M$ with domains Nash diffeomorphic to $\R^{\dim(M)}$:

\paragraph{}\label{cartaspm} \cite[Lem.2.2]{fgr} \em A Nash manifold $M\subset\R^m$ of dimension $d$ admits a finite open (semialgebraic) covering $M=\bigcup_{i=1}^rM_i$ by Nash manifolds $M_i\subset M$ each of them Nash diffeomorphic to $\R^d$\em. 

Let $N\subset M$ be a closed Nash submanifold of dimension $e$. By \cite[Cor.II.5.4]{sh} $N$ is a non-singular Nash subset of $M$ and by \cite[Thm.1.4]{bfr} we have:

\paragraph{} \label{fnashnc2}
\em The Nash submanifold $N$ can be covered by finitely many open semialgebraic subsets $U$ of $M$ equipped with Nash diffeomorphisms $u:=(u_1,\dots,u_d):U\to \R^d$ such that $U\cap N=\{u_1=0,\ldots,u_{d-e}=0\}$\em.

It is also possible to construct tubular neighborhoods of $N$ inside $M$. A \em Nash (vector) bundle over $M$ \em is a (vector) bundle $(\mathscr{E},\theta,M)$ such that $\mathscr{E}$ is an (affine) Nash manifold and the projection $\theta:\mathscr{E}\to M$ is a Nash map. Examples of Nash bundles are: the trivial bundle of $M$, the tangent bundle of $M$, the normal bundle of $M$, etc.

\paragraph{}\label{nmtub} \cite[Lem.II.6.2]{sh} \em There exists a Nash subbundle $(\mathscr{E},\theta,N)$ of the trivial Nash bundle $(N\times\R^m,\eta,N)$, a (strictly) positive Nash function $\delta$ on $N$ and a Nash diffeomorphism $\varphi$ from a semialgebraic neighborhood $V$ of $N$ in $M$ onto 
$$
\mathscr{E}_\delta:=\{(x,y)\in\mathscr{E}:\ \|y\|<\delta(x)\}
$$
such that $\varphi|_{N}=(\id_N,0)$\em. The tuple $(V,\varphi,\mathscr{E},\theta,N,\delta)$ is \em a Nash tubular neighborhood of $N$ in $M$ \em and the composition $\theta\circ\varphi:V\to N$ is a Nash retraction.

As an application of \ref{fnashnc2} it follows a counterpart of \ref{cartaspm} for Nash manifolds with boundary:

\paragraph{}\label{cartaspmh} \em Let $H\subset\R^m$ be a $d$-dimensional Nash manifold with boundary\em. By \cite[Thm.1.11]{fgr} \em $H$ is a closed subset of a Nash manifold $M\subset\R^m$ of dimension $d$ in such a way that:
\begin{itemize}
\item[(i)] $\partial H$ is a closed Nash submanifold of $M$ \em and so a Nash non-singular subset of $M$.\em
\item[(ii)] $M$ can be covered with finitely many open semialgebraic subsets $U$ equipped with Nash diffeomorphisms $(u_1,\dots,u_d):U\to \R^d$ such that 
$$
\begin{cases}
\text{$U\subset H$ or $U\cap H=\varnothing$}&\text{if $U$ does not meet $\partial H$,}\\ 
\text{$U\cap H=\{u_1\ge0\}$}&\text{if $U$ meets $\partial H$.} 
\end{cases}
$$
\end{itemize}\em

A \em Nash normal-crossings divisor of \em $M$ is a Nash subset $X\subset M$ whose irreducible components are non-singular Nash hypersurfaces $X_1,\dots,X_p$ of $M$ in general position. This means that at every point $x\in X_{i_1},\dots,X_{i_r}$, $x\notin X_i$ for $i\ne i_1,\ldots,i_r$, the tangent hyperplanes $T_xX_{i_1},\dots,T_xX_{i_r}$ are linearly independent in the tangent space $T_xM$. In \cite[Thm.1.6]{fgr} we prove the following:

\paragraph{}\label{fnashnc}
\em Let $X$ be a Nash normal-crossings divisor of $M$. Then $X$ can be covered by finitely many open semialgebraic subsets $U$ of $M$ equipped with Nash diffeomorphisms $(u_1,\dots,u_d):U\to \R^d$ such that $U\cap X=\{u_1\cdots u_r=0\}$, where $r$ depends on $U$\em.

Both Nash submanifolds and Nash normal-crossings divisors are particular cases of \em coherent \em Nash subsets $X$ of $M$ (see \cite[\S2.B, Lem.5.1]{bfr}). We can say by \cite[\S2.B]{bfr} and \cite[Prop.8.6.9]{bcr} that a Nash set $X\subset M$ is \em coherent \em if the ${\mathcal N}$-sheaf of ideals $\J_x=I(X_x)$ for $x\in M$ is of \em finite type\em, that is, for every $x\in M$ there exists an open neighborhood $U$ and a surjective morphism ${\mathcal N}^s|_U\to\J|_U$. By \cite[Eq.(2.2)]{bfr} it holds that if $X\subset M$ is a coherent Nash subset, then for any semialgebraic open subset $U$ of $M$ we have
\begin{equation}\label{eq:ix}
I(X){\mathcal N}(U)=I(X\cap U)
\end{equation}
where $I(X\cap U):=\{f\in{\mathcal N}(U): f|_{X\cap U}=0\}$ and $I(X):=I(X\cap M)$. Consequently, if $x\in X$, the ideal $I(X_x)=I(X){\mathcal N}(M_x)$. Conversely, we have \cite[(I.6.5)]{sh}:

\paragraph{}\label{division} \em If $f_1,\ldots,f_r\in{\mathcal N}(M)$ generate $I(X_x)$ for all $x\in X$, then $f_1,\ldots,f_r$ generate also $I(X)$\em. In particular, if $f\in{\mathcal N}(M)$ satisfies $f_x\in I(X_x)$ for all $x\in X$, then $f\in I(X)$.

\subsection{Approximation of differentiable semialgebraic maps by Nash maps}\setcounter{paragraph}{0}

Let $M\subset\R^m$ be a Nash manifold of dimension $d$. Denote the set of all continuous semialgebraic functions on $M$ with ${\mathcal S}^0(M)$. For every integer $r\ge 1$ we denote the set of all semialgebraic functions $f:M\to\R$ that are differentiable of class $r$ with ${\mathcal S}^r(M)$. We equip ${\mathcal S}^r(M)$ with the \em ${\mathcal S}^r$ semialgebraic Whitney topology \em \cite[\S II.1, p.79--80]{sh}. If $r\geq1$, let $\xi_1,\dots,\xi_s$ be semialgebraic ${\mathcal S}^{r-1}$ tangent fields on $M$ that span the tangent bundle of $M$. For every strictly positive continuous semialgebraic function $\veps:M\to\R$ we denote the set of all functions $g\in{\mathcal S}^r(M)$ such that
$$
\begin{cases}
|g|<\veps&\text{if $r=0$,}\\
|g|<\veps \quad\text{and}\quad|\xi_{i_1}\cdots\xi_{i_\ell}(g)|<\veps\, 
 \text{ for }\, 1\le i_1,\dots,i_\ell\le s, 1\le\ell\le r&\text{if $r\geq1$} 
\end{cases}
$$
with ${\mathcal U}_\epsilon$. These sets ${\mathcal U}_\epsilon$ form a basis of neighborhoods of the zero function for a topology in ${\mathcal S}^r(M)$ that does not depend on the choice of the tangent fields if $r\geq1$. The first important result is that the inclusion ${\mathcal N}(M)\subset {\mathcal S}^r(M)$ is dense.

\paragraph{}\cite[Thm.II.4.1]{sh}\label{shiota}
\em Every semialgebraic ${\mathcal S}^r$ function on $M$ can be approximated in the ${\mathcal S}^r$ topology by Nash functions\em.

Let $N\subset\R^n$ be a Nash manifold. A semialgebraic map $f:=(f_1,\dots,f_n):M\to N\subset\R^n$ is ${\mathcal S}^r$ if each component $f_k:M\to\R$ is ${\mathcal S}^r$. We denote the set of all ${\mathcal S}^r$ maps $M\to N$ with ${\mathcal S}^r(M,N)$. We consider in ${\mathcal S}^r(M,N)$ the subspace topology given by the canonical inclusion in the following product space endowed with the product topology \cite[Rmk.II.1.3]{sh}:
$$
{\mathcal S}^r(M,N)\subset{\mathcal S}^r(M,\R^n)={\mathcal S}^r(M,\R)\times\overset{n}{\cdots}\times{\mathcal S}^r(M,\R):f\mapsto(f_1,\dots,f_n).
$$
Roughly speaking, $g$ is close to $f$ when its components $g_k$ are close to the components $f_k$ of $f$. The previous topologies can be extended to the set ${\mathcal S}^r(\Ss,\Tt)$ of ${\mathcal S}^r$ semialgebraic maps between two semialgebraic sets $\Ss$ and $\Tt$ if $\Ss\subset M$ is closed \cite[\S2.D-E]{bfr}. This allows for instance to provide a topology on the set of ${\mathcal S}^r$ maps between two Nash manifolds with boundary.

A map $h:\Ss\to \Tt$ is an \em ${\mathcal S}^r$ diffeomorphism \em if it is a bijection and both $h$ and $h^{-1}$ are ${\mathcal S}^r$ maps. Diffeomorphisms between Nash manifolds behave well with respect to approximation if $r\geq1$.

\paragraph{}\label{diffo}\cite[Lem.II.1.7]{sh} \em Let $h:M\to N$ be an ${\mathcal S}^r$ diffeomorphism of Nash manifolds. If an ${\mathcal S}^r$ map $g:M\to N$ is ${\mathcal S}^r$ close enough to $h$, then $g$ is also an ${\mathcal S}^r$ diffeomorphism, and $g^{-1}$ is ${\mathcal S}^r$ close to $h^{-1}$\em.

From this and the existence of Nash tubular neighborhoods (\ref{fact:retract}) we deduce that for all $r\ge 1$ \em every ${\mathcal S}^r$ diffeomorphism $f:M\to N$ can be approximated by Nash diffeomorphisms, \em hence ${\mathcal S}^1$ and Nash classifications coincide for Nash manifolds. In the case of Nash manifolds with boundary we proved in \cite[Rmk.9.6]{bfr} that ${\mathcal S}^2$ and Nash classifications coincide.

\paragraph{}\label{s2boun} \em Two Nash manifolds with boundary that are ${\mathcal S}^2$ diffeomorphic are Nash diffeomorphic\em. 

\paragraph{}\label{s2bounb} In addition, \em if $f:H_1\to H_2$ is an ${\mathcal S}^2$ diffeomorphism between two Nash manifolds with boundary and $f|_{\partial H_1}:\partial H_1\to\partial H_2$ is a Nash diffeomorphism, then there exists a Nash diffeomorphism $g:H_1\to H_2$ close to $f$ such that $g|_{\partial H_1}=f|_{\partial H_1}$\em.

We will deal in addition with ${\mathcal S}^0$ homeomorphisms between Nash manifolds. In this case it is not possible to approximate them by Nash diffeomorphisms, but the following result allows us to approximate them by Nash surjective maps. 

\begin{lem}\label{surapprox}
Let $M\subset\R^m$ and $N\subset\R^n$ be Nash manifolds and let $f:M\to N$ be a semialgebraic homeomorphism. Let $g:M\to N$ be a semialgebraic map close to $f$ in the ${\mathcal S}^0$ topology. Then $g$ is surjective.
\end{lem}
\begin{proof}
We adapt the proof of \cite[Thm.2.1.8]{h}. By \cite[Rmk.II.1.15]{sh} the map 
$$
f^{-1}_*:{\mathcal S}^0(M,N)\to{\mathcal S}^0(M,M),\ h\mapsto f^{-1}\circ h
$$ 
is continuous. Thus, if $g$ is close to $f$, then $f^{-1}\circ g$ is close to $\id_M$. If we prove that each $g'\in{\mathcal S}^0(M,M)$ close to the identity is surjective, then each $g$ close to $f$ will satisfy that $f^{-1}\circ g$ is surjective, so $g$ is surjective.

By \cite[Thm.VI.2.1]{sh} there exist a compact affine non-singular algebraic set $X$, a non-singular algebraic subset $Y$ of $X$ that either has codimension $1$ if $M$ is non compact or is empty if $M$ is compact and a union $M'$ of some connected components of $X\setminus Y$ such that $M$ is Nash diffeomorphic to $M'$ and $\cl(M')$ is a compact Nash manifold with boundary $Y$. Let $\veps$ be a non-negative Nash equation of $Y$ in $X$ and let $h\in{\mathcal S}^0(M',M')$ be such that $\|h-\id_{M'}\|<\veps|_{M'}$. We claim: \em $h$ extends continuously to $X\setminus M'$ as the identity map $\id_{X\setminus M'}$\em. It is enough to check: \em if $\{z_k\}_{k\geq1}\subset M'$ tends to $y\in Y$, then $\{h(z_k)\}_k$ tends to $y$\em.

Indeed, $\|h(z_k)-z_k\|<\veps(z_k)$ for each $k\geq1$. As $\{z_k\}_k$ tends to $y\in Y$, we have that $\{\veps(z_k)\}_k$ tends to $0$, so $\{h(z_k)\}_k$ tends to $y$, as claimed.

Consider next the continuous semialgebraic map
$$
H:X\to X,\ x\mapsto\begin{cases}
h(x)&\text{if $x\in M'$},\\
x&\text{if $x\in X\setminus M'$}.
\end{cases}
$$
It satisfies 
$$
\|H(x)-\id_X(x)\|\begin{cases}
<\veps(x)&\text{if $x\in M'$},\\
=0&\text{if $x\in X\setminus M'$}.
\end{cases}
$$
As $X$ is compact, we know by \cite[Thm.2.1.8]{h} that there exists $\delta>0$ such that if $\xi:X\to X$ is a continuous map and $\|\xi-\id_X\|<\delta$, then $\xi$ is surjective. Consequently, if $h\in{\mathcal S}^0(M',M')$ satisfies $\|h-\id_{M'}\|<\min\{\veps,\delta\}$, then $H:X\to X$ is surjective. Let us check that also $h\in{\mathcal S}^0(M',M')$ is surjective. Pick a point $y\in M'$. Then there exists $x\in X$, such that $H(x)=y$. If $x\in X\setminus M'$, then $H(x)=x\in X\setminus M'$, which is a contradiction. So $x\in M'$ and $h(x)=H(x)=y$, that is, $h$ is surjective, as required.
\end{proof}

\subsection{Modification of analytic arcs by Nash arcs}

The handling of well-welded semialgebraic sets in Section \ref{s7} requires the modification of analytic arcs by Nash arcs that avoid certain algebraic sets. To that end we will use the following result:

\begin{lem}\label{mod}
Let $M\subset\R^m$ be a connected Nash manifold and let $Y\subset\R^m$ be an algebraic set. Let $M_1,M_2$ be open semialgebraic subsets of $M$ and let $\alpha:(-1,1)\to M_1\cup M_2\cup\{0\}$ be an analytic arc such that $\alpha(0)=0$, $\alpha((0,1))\subset M_1$ and $\alpha((-1,0))\subset M_2$. Assume that $M\not\subset Y$. Then for every integer $\nu\geq 1$ there exist $\veps>0$ and a Nash arc $\beta:(-\veps,\veps)\to M_1\cup M_2\cup\{0\}$ such that: 
\begin{itemize}
\item[(i)] $\beta(0)=0$, $\alpha(\t)-\beta(\t)\in(\t)^{\nu}\R\{\t\}^m$, 
\item[(ii)] $\beta((-\veps,\veps))\cap Y\subset\{0\}$,
\item[(iii)] $\beta((0,\veps))\subset M_1$ and $\beta((-\veps,0))\subset M_2$.
\end{itemize}
\end{lem}
\begin{proof}
For simplicity we can assume $M_1\cap M_2=\varnothing$, $0\not\in M_1\cup M_2$ and $0\in Y$. Let $V\subset M$ be an open semialgebraic neighborhood of the origin equipped with a Nash diffeomorphism $\varphi:V\to\R^d$ such that $\varphi(0)=0$. Shrinking the domain of $\alpha$, we may assume $\im(\alpha)\subset V$. Denote $\widehat{\alpha}:=\varphi\circ\alpha:(-\delta,\delta)\to\R^d$ where $\delta>0$ is small enough.

Shrinking $M_i$, $V$ and the domain of $\widehat{\alpha}$, we may assume that $0\not\in\varphi(M_i\cap V)$ and $\varphi(M_i\cap V)=\{g_{1i}>0,\ldots,g_{\ell i}>0\}$ for some polynomials $g_{ji}\in\R[\x]$. There exists $s\geq\nu$ large enough such that if $\gamma\in\R\{\t\}^d$ and $\gamma-\widehat{\alpha}\in(\t)^s\R\{\t\}^d$, we have $(g_{j1}\circ\gamma)(t)>0$ and $(g_{j2}\circ\gamma)(-t)>0$ for $t>0$ small enough and $j=1,\ldots,\ell$. Let $\gamma\in\R[\t]^d$ be a polynomial tuple such that $\gamma-\widehat{\alpha}\in(\t)^s\R\{\t\}^d$.

As $Y\cap V$ has dimension $<d$, also the algebraic set $Y':=\ol{\varphi(Y\cap V)}^{\zar}$ has dimension $<d$. Let $h\in\R[\x]$ be a polynomial equation of $Y'$. Consider the surjective polynomial map
$$
\R\times\R^d\to\R^d,\ (t,y)\mapsto\gamma(t)+t^{s+1}y.
$$ 
Let $y_0\in\R^d$ be such that the univariate polynomial $h(\gamma(\t)+\t^{s+1}y_0)\in\R[\t]$ is not identically zero. Let $\veps>0$ be such that $\gamma_0(\t):=\gamma(\t)+\t^{s+1}y_0\in\R[\t]^d$ satisfies
$$
g_{j1}(\gamma_0(t))>0,\quad g_{j2}(\gamma_0(-t))>0,\quad h(\gamma_0(t))\neq0\quad\text{and}\quad h(\gamma_0(-t))\neq0
$$ 
for $0<t<\veps$. The Nash arc $\beta:=\varphi^{-1}\circ\gamma_0:(-\veps,\veps)\to M_1\cup M_2\cup\{0\}$ satisfies the required properties.
\end{proof}

\section{A light start-up: The $1$-dimensional case}\label{s3}

In this short section we prove Proposition \ref{curves} and present some enlightening examples. Nash images of Euclidean spaces contained in the real line are its intervals and all of them are Nash images of $\R$. To be convinced of this fact it is enough to have a look at the following examples.

\begin{examples}\label{dilatacion}
(i) The interval $(0,1)$ is Nash diffeomorphic to $\R$. Consider the Nash diffeomorphism (together with its inverse):
$$
f:\R\to (0,1),\ t\mapsto\frac{t}{2\sqrt{1+t^2}}+\frac{1}{2}
\quad\text{and}\quad
f^{-1}:(0,1)\to\R,\ t\mapsto\frac{2t-1}{2\sqrt{t(1-t)}}.
$$
In addition $f((0,+\infty))=(\frac{1}{2},1)$ and $f([0,+\infty))=[\frac{1}{2},1)$.

(ii) The interval $[0,1)$ is the image of the Nash map
$$
h_1:\R\to\R,\ t\mapsto\frac{t^2}{t^2+1},
$$
whereas $[0,1]$ is the image of the Nash map
$$
h_2:\R\to\R,\ t\mapsto\frac{t}{t^2+1}+\frac{1}{2}.
$$
\end{examples}

Nash images of Euclidean spaces contained in a circumference are its connected subsets and all of them are Nash images of $\R$.

\begin{examples}\label{circle}
(i) The circumference $\sph^1:\ x^2+y^2=1$ is a Nash image of $\R$. Consider the inverse of the stereographic projection from the point $(1,0)$, which is the map
$$
f:\R\to\sph^1\setminus\{(1,0)\},\ t\mapsto\Big(\frac{1-t^2}{1+t^2},\frac{2t}{1+t^2}\Big).
$$
Next, we identify $\R^2$ with $\C$ and the coordinates $(x,y)$ with $x+\sqrt{-1}y$. Consider the map 
$$
g:\C\to\C,\ z:=x+\sqrt{-1}y\mapsto z^2=(x^2-y^2)+\sqrt{-1}(2xy).
$$
The image of $\R$ under $g\circ f$ is $\sph^1$. 

(ii) Any connected proper subset $\Ss$ of $\sph^1$ is a Nash image of $\R$ because it is Nash diffeomorphic to either $(0,1)$, $[0,1)$ or $[0,1]$ and these are Nash images of $\R$.
\end{examples}

We are ready to prove Proposition \ref{curves}.

\begin{proof}[Proof of Proposition \em \ref{curves}]
Assume $\Ss$ is irreducible. Let $X$ be the Zariski closure of $\Ss$ in $\R^n$ and let $\widetilde{X}$ be its complexification. Let $(\widetilde{Y},\pi)$ be the normalization of $\widetilde{X}$ and let $\widehat{\sigma}$ be the involution of $\widetilde{Y}$ induced by the involution $\sigma$ of $\widetilde{X}$ that arises from the restriction to $\widetilde{X}$ of the complex conjugation in $\C^n$. We may assume that $\widetilde{Y}\subset\C^m$ and that $\widehat{\sigma}$ is the restriction to $\widetilde{Y}$ of the complex conjugation of $\C^m$. By \cite[Thm.3.15]{fg3} and since $\Ss$ is irreducible, $\pi^{-1}(\Ss)$ has a $1$-dimensional connected component $\Tt$ such that $\pi(\Tt)=\Ss$. As $X$ has dimension $1$, it is a coherent analytic set, so $\Tt\subset Y:=\widetilde{Y}\cap\R^m$. As $\widetilde{Y}$ is a normal-curve, $Y$ is a non-singular real algebraic curve. We claim: \em the connected components of $Y$ are Nash diffeomorphic either to $\sph^1$ or to the real line $\R$\em.

By \cite[Thm.VI.2.1]{sh} there exist a compact affine non-singular real algebraic curve $Z$, a finite set $F$ which is empty if $Y$ is compact and a union $Y'$ of some connected components of $Z\setminus F$ such that $Y$ is Nash diffeomorphic to $Y'$ and $\cl(Y')$ is a compact Nash curve with boundary $F$. As $Z$ is a compact affine non-singular real algebraic curve, its connected components are diffeomorphic to $\sph^1$, so by \cite[Thm.VI.2.2]{sh} the connected components of $Z$ are in fact Nash diffeomorphic to $\sph^1$. Now, each connected component of $Y$ is Nash diffeomorphic to an open connected subset of $\sph^1$, so it is Nash diffeomorphic either to $\sph^1$ or to the real line $\R$, as claimed.

Consequently, $\Tt$ is Nash diffeomorphic to a connected subset of either $\sph^1$ or $\R$. By Examples \ref{dilatacion} and \ref{circle} the semialgebraic set $\Tt$ is a Nash image of $\R$, so also $\Ss$ is a Nash image of $\R$. The converse is straightforward.
\end{proof}

\section{Building boundaries on Nash manifolds}\label{s4}

The purpose of this section is to develop a tool to build boundaries on a non-compact Nash manifold. More precisely we will prove the following.

\begin{prop}\label{doubleivc}
Let $H\subset\R^m$ be a Nash manifold with boundary. Then there exists a surjective Nash map $f:\Int(H)\to H$ that has local representations of the type $(x_1,\ldots,x_d)\mapsto(x_1^2,x_2,\ldots,x_d)$ at each point of $f^{-1}(\partial H)$.
\end{prop}

In order to ease the understanding of the strategy followed to prove Proposition \ref{doubleivc} we refer the reader to Figure \ref{fig2} (b). This proof requires the use of Nash collars and Nash doubles of a Nash manifold with boundary $H$ (see Figure \ref{fig2} (a)). These constructions for $H$ compact are a common tool in Nash Geometry \cite[\S VI]{sh} but as far as we know there is no explicit reference to them in the literature when $H$ is non-compact. In \ref{scollar} we afford the construction of Nash collars when $H$ is non necessarily compact. In \ref{double} we endow the (smooth) double of $H$ with a Nash manifold structure. The resulting Nash manifold ${\rm D}(H)$ is called the \em Nash double of $H$\em. Its construction requires a Nash equation of $\partial H$ that is strictly positive on $\Int(H)$ and has rank $1$ at the points of $\partial H$. 

\subsection{Nash collars}\label{scollar}\setcounter{paragraph}{0}

Let $H\subset\R^m$ be a Nash manifold with boundary $\partial H$. A \em Nash collar \em of $\partial H$ is an open semialgebraic neighborhood $W\subset H$ of $\partial H$ equipped with a Nash diffeomorphism $\psi:W\to\partial H\times[0,1)$ such that $\psi(x)=(x,0)$ for all $x\in\partial H$. We recall next how Nash collars are constructed. For the smooth case see \cite[Thm.I.5.9]{mu}. 

\begin{lem}\label{neighcollar}
Let $M\subset\R^m$ be a Nash manifold of dimension $d$ and let $N\subset M$ be a Nash submanifold of dimension $d-1$. Let $U\subset M$ be an open semialgebraic neighborhood of $N$ and $\rho:U\to N$ a Nash retraction. Let $h$ be a Nash function on $U$ such that $\{h=0\}=N$ and $d_xh:T_xM\to\R$ is surjective for all $x\in N$. Consider the Nash map $\varphi:=(\rho,h):U\to N\times\R$. Then there exist an open semialgebraic neighborhood $V\subset U$ of $N$ and a strictly positive Nash function $\veps$ on $N$ such that $\varphi(V)=\{(x,t)\in N\times\R:\ |t|<\veps(x)\}$ and $\varphi|_V:V\to\varphi(V)$ is a Nash diffeomorphism.
\end{lem}
\begin{proof}
We show first: \em The derivative $d_x\varphi=(d_x\rho,d_xh):T_xM\to T_xN\times\R$ is an isomorphism for all $x\in N$\em. As $\dim(T_xM)=\dim(T_xN\times\R)$, it is enough to show: \em $d_x\varphi$ is surjective\em. 

As $\varphi|_{N}=(\id_{N},0)$, we have $d_x\varphi|_{T_xN}=(\id_{T_xN},0)$, so $T_xN\times\{0\}\subset\im(d_x\varphi)$. In addition $d_xh:T_xM\to\R$ is surjective, so there exists $v\in T_xM$ such that $d_xh(v)=1$. Thus, $d_x\varphi(v)=(d_x\rho(v),1)$ and $d_x\varphi$ is surjective.

Let $U':=\{x\in U:\ d_x\varphi\ \text{is an isomorphism}\}$, which is an open semialgebraic neighborhood of $N$ in $U$. Thus, $\varphi|_{U'}:U'\to N\times\R$ is an open map and $\varphi(U')$ is an open semialgebraic neighborhood of $N\times\{0\}$ in $N\times\R$. As $\varphi|_{U'}:U'\to\varphi(U')$ is a local homeomorphism and $\varphi|_{N}=(\id_{N},0)$ is a homeomorphism (onto its image), there exist by \cite[Lem.9.2]{bfr} open semialgebraic neighborhoods $U''\subset U'$ of $N$ and $W\subset N\times\R$ of $N\times\{0\}$ such that $\varphi|_{U''}:U''\to W$ is a semialgebraic homeomorphism.

Consider the strictly positive semialgebraic map 
$$
\delta:N\to(0,+\infty),\ x\mapsto\dist((x,0),(N\times\R)\setminus W).
$$ 
By \ref{shiota} there exists a strictly positive Nash function $\veps$ on $N$ such that $\frac{1}{2}\delta<\veps<\delta$. Consider the open semialgebraic neighborhood $W':=\{(x,t)\in N\times\R:\ |t|<\veps(x)\}\subset W$ of $N\times\R$ and define $V:=(\varphi|_{U''})^{-1}(W')$. The restriction $\varphi|_V:V\to W'$ is a Nash diffeomorphism, as required.
\end{proof}

\begin{lem}\label{eqnash}
Let $H\subset\R^m$ be a $d$-dimensional Nash manifold with boundary $\partial H$. Let $M\subset\R^m$ be a Nash manifold of dimension $d$ that contains $H$ as a closed subset. Then there exist an open semialgebraic neighborhood $M'\subset M$ of $H$ and a Nash equation $h$ of $\partial H$ in $M'$ such that $\Int(H)=\{h>0\}$ and $d_xh:T_xH\to\R$ is surjective for all $x\in\partial H$.
\end{lem}
\begin{proof}
We may assume $\partial H\neq\varnothing$. The proof is conducted in several steps:

\paragraph{} We construct first an ${\mathcal S}^2$ semialgebraic function $h_0$ on $M$ such that $\partial H\subset\{h_0=0\}$ and $d_xh_0(v)>0$ for each $x\in\partial H$ and each non-zero vector $v\in T_xM$ pointing `inside $M$'.

By \ref{cartaspmh} we can cover $\partial H$ with finitely many open semialgebraic subsets $U_i$ of $M$ that are equipped with Nash diffeomorphisms $u_i:=(u_{i1},\ldots,u_{id}):U_i\to\R^d$ such that $U_i\cap H=\{u_{i1}\geq 0\}$ for $i=1,\ldots,r$. Let $\{\theta_i\}_{i=1}^{r+1}$ be an ${\mathcal S}^2$ partition of unity subordinated to the finite covering $\{U_i\}_{i=1}^r\cup\{M\setminus\partial H\}$ of $M$ and consider the ${\mathcal S}^2$ function $h_0:=\sum_{i=1}^r\theta_iu_{i1}$. It holds $\partial H\subset\{h_0=0\}$. 

Fix $x\in\partial H$ and let $v\in T_xM$ be a non-zero vector pointing `inside $M$', that is, $d_xu_{i1}(v)>0$ if $x\in U_i$. We have
$$
d_xh_0=\sum_{x\in U_i}u_{i1}(x)d_x\theta_i+\sum_{x\in U_i}\theta_i(x)d_xu_{i1}=\sum_{x\in U_i}\theta_i(x)d_xu_{i1}\ \leadsto\ d_xh_0(v)=\sum_{x\in U_i}\theta_i(x)d_xu_{i1}(v)>0
$$
because $\sum_{x\in U_i}\theta_i(x)=1$, $\theta_i(x)\geq0$ and $d_xu_{i1}(v)>0$ if $x\in U_i$. 

\paragraph{}\label{wh1} By \cite[Prop.8.2]{bfr} there exists a Nash function $h_1$ on $M$ close to $h_0$ in the ${\mathcal S}^2$ topology such that $\partial H\subset\{h_1=0\}$ and $d_xh_1(v)>0$ for each $x\in\partial H$ and each non-zero vector $v\in T_xM$ pointing `inside $M$'. We claim: \em there exists an open semialgebraic neighborhood $W\subset M$ of $\partial H$ such that $\{h_1>0\}\cap W=\Int(H)\cap W$ and $\{h_1=0\}\cap W=\partial H$\em.

Pick a point $x\in\partial H$ and assume $x\in U_1$. As $h_1$ vanishes identically at $\partial H$, we may write $h_1|_{U_1}=u_{11}a_1$ where $a_1$ is a Nash function on $U_1$. Pick $y\in\partial H\cap U_1$ and observe that $d_yh_1=a_1(y)d_yu_{11}$. Let $v\in T_yM$ be a non-zero vector pointing `inside $M$'. As $d_yu_{11}(v)>0$ and $d_yh_1(v)>0$, we deduce $a_1(y)>0$. Define $W_1:=\{a_1>0\}\subset U_1$ and notice that $\partial H\cap U_1\subset W_1$, $\{h_1>0\}\cap W_1=\Int(H)\cap W_1$ and $\{h_1=0\}\cap W_1=\partial H\cap W_1$. Construct analogously $W_2,\dots, W_r$ and observe that $W:=\bigcup_{i=1}^rW_i$ satisfies the required properties.

\paragraph{} Next, we construct $h$. If $W=M$, it is enough to set $h:=h_1$. Suppose $W\neq M$. Let $\veps_0$ be a (continuous) semialgebraic function whose value is $1$ on $\partial H$ and $-1$ on $M\setminus W$. Let $\veps$ be a Nash approximation of $\veps_0$ such that $|\veps-\veps_0|<\frac{1}{2}$. Then
$$
\veps(x)\begin{cases}
>\frac{1}{2}&\text{if $x\in\partial H$,}\\
<\frac{-1}{2}&\text{if $x\in M\setminus W$.}
\end{cases}
$$
Thus, $\{\veps>0\}\subset W$ is an open semialgebraic neighborhood of $\partial H$. By \cite[Prop.II.5.3]{sh} $\partial H$ is a Nash subset of $M$. Let $f$ be a Nash equation of $\partial H$ in $M$. Substituting $f$ by $\frac{f^2}{\veps^2+f^2}$ we may assume that $f$ is non-negative and $f(x)=1$ if $\veps(x)=0$. 

Consider the (continuous) semialgebraic function on $M$ given by
$$
\delta(x):=
\begin{cases}
1&\text{if $\veps(x)>0$,}\\
\frac{1}{f^2(x)}&\text{if $\veps(x)\leq0$}.
\end{cases}
$$
Let $g$ be a Nash function on $M$ such that $\delta<g^2$. Consider the Nash function 
$$
h:=h_1+f^2g^2(h_1^2+1)
$$ 
and let us prove that it satisfies the required conditions. 

\paragraph{}We claim: \em $h$ is positive on $\Int(H)$\em. 

Let $x\in \Int(H)$. If $h_1(x)>0$, then $h(x)>0$. If $h_1(x)\leq0$, then $\veps(x)\leq0$ and
\begin{multline*}
h(x)=h_1(x)+g^2(x)f^2(x)(h_1^2(x)+1)\\
>h_1(x)+\frac{1}{f^2(x)}f^2(x)(h_1^2(x)+1)=h_1^2(x)+h_1(x)+1>0.
\end{multline*}

\paragraph{} It holds: \em $\{h=0\}\cap W=\partial H$, $\{h>0\}\cap W=\Int(H)\cap W$ and $d_xh:T_xH\to\R$ is surjective for all $x\in\partial H$\em.

Recall that $W=\bigcup_{i=1}^rW_i$ and fix $i=1$. We have seen in \ref{wh1} that there exists a Nash function $a_1$ on $U_1$ such that $h_1|_{U_1}=u_{11}a_1$ and $\partial H\cap U_1\subset W_1:=\{a_1>0\}$. As $f$ vanishes identically at $\partial H$, we deduce that $f|_{U_1}=u_{11}b_1$ where $b_1$ is a Nash function on $U_1$. Consequently, 
\begin{multline*}
h|_{U_1}=u_{11}a_1+g^2|_{U_1}u_{11}^2b_1^2(u_{11}^2a_1^2+1)=u_{11}(a_1+g^2|_{U_1}u_{11}b_1^2(u_{11}^2a_1^2+1))\\
\text{and}\quad d_xh=a_1(x)d_xu_{11}=d_xh_1\quad\forall\,x\in\partial H\cap U_1. 
\end{multline*} 
Define $W_1':=\{a_1+g^2|_{U_1}u_{11}b_1^2(u_{11}^2a_1^2+1)>0\}\cap W_1$, which is an open semialgebraic subset of $M$. We have $\partial H\cap U_1\subset W_1'$ and $\{h>0\}\cap W_1'=\Int(H)\cap W_1'$. Construct analogously $W_2',\ldots,W_r'$ and observe that $W':=\bigcup_{i=1}^rW_i'\subset M$ is an open neighborhood of $\partial H$ that satisfies $\{h=0\}\cap W'=\partial H$, $\{h>0\}\cap W'=\Int(H)\cap W'$ and $d_xh:T_xH\to\R$ is surjective for all $x\in\partial H$. 

Consequently, $M':=\Int(H)\cup W'$ satisfies the required conditions.
\end{proof}

\begin{lem}\label{collar}
Let $H\subset\R^m$ be a $d$-dimensional Nash manifold with boundary $\partial H$. Then every semialgebraic neighborhood $U\subset H$ of $\partial H$ contains a Nash collar $(W,\psi)$ of $\partial H$ in $H$.
\end{lem}
\begin{proof}
By \ref{cartaspmh} there exists a Nash manifold $M\subset\R^m$ of dimension $d$ that contains $H$ as a closed subset and such that $\partial H$ is a Nash submanifold of $M$ of dimension $d-1$. Let $U'\subset M$ be an open semialgebraic neighborhood of $\partial H$ such that $U'\cap H=U$.

As $\partial H$ is a Nash manifold without boundary, there exists a Nash tubular neighborhood $A$ of $\partial H$ in $\R^m$ together with a Nash retraction $\rho:A\to\partial H$. Substituting $U'$ by $U'\cap A$, we may assume $U'\subset A$. By Lemma \ref{eqnash} we can shrink $M$ to have a Nash equation $h$ of $\partial H$ in $M$ such that $\{h>0\}\cap M=\Int(H)$ and $d_xh:T_xH\to\R$ is surjective for all $x\in\partial H$. Denote 
$$
\varphi:=(\rho|_{U'},h|_{U'}):U'\to\partial H\times\R. 
$$
By Lemma \ref{neighcollar} there exists an open semialgebraic neighborhood $V\subset U'$ of $\partial H$ and a strictly positive Nash function $\veps$ on $\partial H$ such that $\varphi(V)=\{(y,t)\in\partial H\times\R:\ |t|<\veps(y)\}$ and $\varphi|_V:V\to\varphi(V)$ is a Nash diffeomorphism. Consider the Nash diffeomorphism
$$
\psi:V\to\partial H\times(-1,1),\ x\mapsto\Big(\rho(x),\frac{h(x)}{\veps(\rho(x))}\Big)
$$
and $W:=\psi^{-1}(\partial H\times[0,1))\subset U'\cap H$. The restriction $\psi|_W:W\to\partial H\times[0,1)$ provides a collar of $\partial H$ such that $W\subset U$, as required.
\end{proof}

\subsection{Nash doubles}\label{double}\setcounter{paragraph}{0}
Let $H\subset\R^m$ be a $d$-dimensional Nash manifold with boundary $\partial H$. Let $h$ be a Nash equation of $\partial H$ such that $\Int(H)=\{h>0\}$ and $d_xh:T_xH\to\R$ is surjective for all $x\in\partial H$ (see Lemma \ref{eqnash}). We have:

\paragraph{}\label{doublei} \em ${\rm D}(H):=\{(x,t)\in H\times\R:\ t^2-h(x)=0\}$ is a Nash manifold of dimension $d$ that contains $\partial H\times\{0\}$ as the Nash subset $\{t=0\}$\em.
\begin{proof}
The semialgebraic set ${\rm D}(H)$ satisfies 
$$
D(H)\setminus(\partial H\times\{0\})=\{(x,t)\in\Int(H)\times\R:\ t=\pm\sqrt{h(x)}\}, 
$$
which is the union of two disjoint graphs of Nash functions on the Nash manifold $\Int(H)$. Consequently, it is enough to show: \em for each $(x_0,0)\in\partial H\times\{0\}$ there exists an open semialgebraic neighborhood $W\subset {\rm D}(H)$ of $(x_0,0)$ such that $W$ is a Nash manifold\em. 

Let $M'\subset\R^m$ be a Nash manifold of dimension $d$ that contains $H$ as a closed subset. Pick a point $x_0\in\partial H$ and let $U\subset M'$ be an open semialgebraic neighborhood of $x_0$ equipped with a Nash diffeomorphism $u:=(u_1,\ldots,u_d):U\to(-1,1)\times\R^{d-1}$ such that $u(x_0)=0$ and $U\cap H=\{u_1\geq0\}$. We may assume $u_1=h|_U$. Observe that $V:=(U\cap H)\times\R$ is an open semialgebraic subset of $H\times\R$ and 
$$
W:={\rm D}(H)\cap V=\{(x,\pm\sqrt{h(x)}):\ x\in U\cap H\}
$$
is an open semialgebraic neighborhood of $(x_0,0)$ in ${\rm D}(H)$. Consider the Nash map 
$$
u':=(u_1',\ldots,u_d'):W\to(-1,1)\times\R^{d-1},\ (x,t)\mapsto(t,u_2(x),\ldots,u_d(x))
$$
and let us check that it is a Nash diffeomorphism. As 
$$
(u_2',\ldots,u_d')(W)=(u_2,\ldots,u_d)(U\cap H)=\R^{d-1},
$$
we have $u'(W)=(-1,1)\times\R^{d-1}$, so $u'$ is surjective. 

Pick $(x_1,t_1),(x_2,t_2)\in W$ such that $u'(x_1,t_1)=u'(x_2,t_2)$. Then $t_1=t_2$, so 
$$
u_1(x_1)=h(x_1)=t_1^2=t_2^2=h(x_2)=u_1(x_2)
$$
and $u(x_1)=u(x_2)$. As $u$ is injective, we have $x_1=x_2$, so $(x_1,t_1)=(x_2,t_2)$. Thus, $u'$ is injective. Denote $u^{-1}:=\phi:=(\phi_1,\ldots,\phi_m)$. The inverse of $u'$ is the Nash map
$$
\zeta:(-1,1)\times\R^{d-1}\to W,\ (t,y'):=(t,y_2,\ldots,y_d)\mapsto(\phi(t^2,y'),t).
$$
The differential of $\zeta$ at a point $(t,y')\in(-1,1)\times\R^{d-1}$ is
$$
\left(\begin{array}{cccc}
2t\frac{\partial\phi_1}{\partial y_1}(t^2,y')&\frac{\partial\phi_1}{\partial y_2}(t^2,y')&\cdots&\frac{\partial\phi_1}{\partial y_d}(t^2,y')\\
\vdots&\vdots&\ddots&\vdots\\
2t\frac{\partial\phi_m}{\partial y_1}(t^2,y')&\frac{\partial\phi_m}{\partial y_2}(t^2,y')&\cdots&\frac{\partial\phi_m}{\partial y_d}(t^2,y')\\
1&0&\cdots&0
\end{array}\right)
$$
and it has rank $d$. Consequently, $u'$ is a Nash diffeomorphism as required.
\end{proof}

\paragraph{}\label{doubleii} Consider the surjective Nash map $\pi:{\rm D}(H)\to H,\ (x,t)\to x$ and write $\epsilon=\pm$. We have: 
\begin{itemize}\em
\item[(i)] The restriction $\pi|_{{\rm D}(H)\cap\{\epsilon t>0\}}:{\rm D}(H)\cap\{\epsilon t>0\}\to\Int(H)$ is a Nash diffeomorphism. 
\item[(ii)] $\pi(x,0)=x$ for all $(x,0)\in\partial H\times\{0\}={\rm D}(H)\cap\{t=0\}$. 
\item[(iii)] $\pi$ has local representations $(y_1,\ldots,y_d)\mapsto(y_1^2,y_2,\ldots,y_d)$ at each point of ${\rm D}(H)\cap\{t=0\}$.
\end{itemize}
\begin{proof}
(i) The intersection ${\rm D}(H)\cap\{\epsilon t>0\}$ is the graph of the strictly positive Nash function $\epsilon \sqrt{h}$ on $\Int(H)$. Consequently $\pi|_{{\rm D}(H)\cap\{\epsilon t>0\}}:{\rm D}(H)\cap\{\epsilon t>0\}\to\Int(H)$ is Nash diffeomorphism. 

Statement (ii) is evident.

(iii) Consider the Nash diffeomorphism $u'$ constructed in the proof of \ref{doublei} and its inverse $\zeta$. We have
$$
(-1,1)\times\R^{d-1}\overset{\zeta}{\to} W\overset{\pi}{\to} H\cap U\overset{u}{\to}(-1,1)\times\R^{d-1},\ (t,y')\mapsto(\phi(t^2,y'),t)\mapsto\phi(t^2,y')\mapsto (t^2,y'),
$$
as required.
\end{proof}

\paragraph{}\label{doubleiii} \em $H$ is Nash diffeomorphic to $H_\epsilon:={\rm D}(H)\cap\{\epsilon t\geq0\}$ for $\epsilon=\pm$ and ${\rm D}(H)$ is a Nash manifold structure for the double of $H$\em. In addition, \em the Nash map $\tau:{\rm D}(H)\to {\rm D}(H),\ (x,t)\mapsto(x,-t)$ is an involution such that $\tau(H_+)=H_-$ and whose set of fixed points is $\partial H\times\{0\}$\em.
\begin{proof}
The proof is conducted in several steps:

\noindent{\em Step \em 1}. To prove that \em $H_\epsilon$ and $H$ are Nash diffeomorphic \em we need to construct first suitable neighborhoods $U\subset H$ of $\partial H$ and $V\subset D(H)$ of $\partial H\times\{0\}$.

Let $U\subset H$ be an open semialgebraic neighborhood of $\partial H$ equipped with a Nash retraction $\rho:U\to\partial H$. Define $\varphi:=(\rho,h):U\to\partial H\times\R$. By Lemma \ref{neighcollar} we may assume there exists a strictly positive Nash function $\veps$ on $\partial H$ such that
$$
\varphi(U)=\{(y,s)\in\partial H\times\R:\ 0\leq s<\veps(y)\}
$$
and $\varphi:U\to\varphi(U)$ is a Nash diffeomorphism. 

Let $V:=\pi^{-1}(U)$ and $V':=\{(y,t)\in\partial H\times\R:\ |t|<\sqrt{\veps(y)}\}$. We claim: \em The Nash map 
\begin{equation}\label{retractm}
\psi:V\to\partial H\times\R,\ (x,t)\mapsto(\rho(x),t)
\end{equation} 
is a Nash diffeomorphism onto its image $V'$\em.

(1) \em $\psi$ is injective\em. If $(x_1,t_1),(x_2,t_2)\in V$ satisfy $\psi(x_1,t_1)=\psi(x_2,t_2)$, then $\rho(x_1)=\rho(x_2)$ and $h(x_1)=t_1^2=t_2^2=h(x_2)$. Consequently, $\varphi(x_1)=\varphi(x_2)$, so $x_1=x_2$. Thus, $(x_1,t_1)=(x_2,t_2)$.

(2) $\psi(V)=V'$. Pick $(x,t)\in V$. Then $x\in U$ and $\varphi(x)=(\rho(x),h(x))\in\varphi(U)$, so $t^2=h(x)<\veps(\rho(x))$ and $\psi(x,t)\in V'$. Conversely, let $(y,t)\in V'$. As $(y,t^2)\in\varphi(U)$, there exists $x\in U$ such that $\varphi(x):=(\rho(x),h(x))=(y,t^2)$. Then $(x,t)\in V$ and $\psi(x,t):=(\rho(x),t)=(y,t)$.

(3) \em The derivative $d_z\psi:T_z{\rm D}(H)\to T_{\rho(z)}\partial H\times\R$ is an isomorphism for each $z\in V$\em. Write $z:=(x,t)$ and notice that $T_z{\rm D}(H)=\{(v,r)\in T_xH\times\R:\ d_xh(v)-2tr=0\}$ and $d_z\psi(v,r)=(d_x\rho(v),r)$. If $t\neq0$,
$$
d_z\psi(v,r)=\Big(d_x\rho(v),\frac{1}{2t}d_xh(v)\Big).
$$
As $d_x\varphi=(d_x\rho,d_xh)$ is an isomorphism, also $d_z\psi$ is an isomorphism. If $t=0$, that is, $z=(x,0)\in\partial H\times\R$, then $T_z{\rm D}(H)=\{(v,r)\in T_xH\times\R:\ d_xh(v)=0\}=T_x\partial H\times\R$ and $d_z\psi(v,r)=(v,r)$ because $\rho|_{\partial H}=\id_{\partial H}$. Consequently, $d_z\psi$ is an isomorphism also in this case.

\noindent{\em Step \em 2}. Define 
\begin{align*}
H^{\bullet}&:=H\setminus\varphi^{-1}\Big(\Big\{(y,s)\in\partial H\times\R:\ 0\leq|s|<\frac{\veps(y)}{4}\Big\}\Big),\\
H_\epsilon^{\bullet}&:=H_\epsilon\setminus\psi^{-1}\Big(\Big\{(y,s)\in\partial H\times\R:\ 0\leq|s|<\frac{\sqrt{\veps(y)}}{2}\Big\}\Big).
\end{align*}
and let us show: \em The restriction $\varpi_\epsilon:=\pi|_{H_\epsilon^{\bullet}}:H_\epsilon^{\bullet}\to H^{\bullet}$ is a Nash diffeomorphism\em.

Indeed, $\varpi_\epsilon$ is clearly injective. Let $x\in H^{\bullet}$. As $x\in\Int(H)$, we have $h(x)>0$ and write $t:=\epsilon\sqrt{h(x)}$. It holds that $(x,t)\in{\rm D}(H)$ and $\pi(x,t)=x$. We want to check that $(x,t)\in H_\epsilon^{\bullet}$. If $x\not\in U$, then $(x,t)\in H_{\epsilon}\setminus V \subset H_\epsilon^{\bullet}$. If $x\in U$, then $\psi(x,t)=(\rho(x),t)$. As $x\in H^{\bullet}$, it holds $\frac{\veps(\rho(x))}{4}\leq h(x)$, so $\frac{\sqrt{\veps(\rho(x))}}{2}\leq\sqrt{h(x)}=\epsilon t$. Consequently, $(x,t)\in H_\epsilon^{\bullet}$ and $\varpi_\epsilon$ is surjective. In addition, by \ref{doubleii}(i) $d_z\varpi_\epsilon=d_z\pi$ is an isomorphism for each $z\in H_\epsilon^{\bullet}\subset{\rm D}(H)\cap\{\epsilon t>0\}$. Consequently, $\varpi_\epsilon$ is a Nash diffeomorphism.

\noindent{\em Step \em 3}. \em $H^{\bullet}$ is Nash diffeomorphic to $H$ \em and \em $H_\epsilon^{\bullet}$ is Nash diffeomorphic to $H_\epsilon$\em. By \ref{s2boun} it is enough to show that the pairs of objects above are ${\mathcal S}^2$ diffeomorphic. 

As $\tau$ is an involution such that $\tau(H_+)=H_-$, we assume $\epsilon=+$. Denote 
\begin{align*}
U'&:=\{(y,s)\in\partial H\times\R:\ 0\leq s<\veps(y)\},\\ 
W'&:=\{(y,t)\in\partial H\times\R:\ 0\leq t<\sqrt{\veps(y)}\},\\
W''&:=\{(y,t)\in\partial H\times\R:\ 0\leq -t<\sqrt{\veps(y)}.
\end{align*} 
Observe that $U=\varphi^{-1}(U')$, $V=\psi^{-1}(W'\cup W'')$ and set $W:=\psi^{-1}(W')$. Consider the Nash diffeomorphisms
\begin{align*}
\Lambda:U'\to\partial H\times[0,1),&\ (y,s)\mapsto\Big(y,\frac{s}{\veps(y)}\Big),\\
\Delta:W'\to\partial H\times[0,1),&\ (y,t)\mapsto\Big(y,\frac{t}{\sqrt{\veps(y)}}\Big).
\end{align*}

Observe that $(\Lambda\circ\varphi)(H^{\bullet}\cap U)=\partial H\times[\frac{1}{4},1)$ and $(\Delta\circ\psi)(H_+^{\bullet}\cap W)=\partial H\times[\frac{1}{2},1)$.

Let $f_1:[\frac{1}{4},1)\to[0,1)$ and $f_2:[\frac{1}{2},1)\to[0,1)$ be ${\mathcal S}^2$ diffeomorphisms such that $f_1|_{[\frac{3}{4},1)}=\id_{[\frac{3}{4},1)}$ and $f_2|_{[\frac{3}{4},1)}=\id_{[\frac{3}{4},1)}$ (see Examples \ref{s2diffeo}(i) and (ii) in Appendix \ref{A}). Consider the ${\mathcal S}^2$ diffeomorphisms 
\begin{align*}
&F_1:\partial H\times[\tfrac{1}{4},1)\to\partial H\times[0,1),\ (y,t)\mapsto(y,f_1(t)),\\
&F_2:\partial H\times[\tfrac{1}{2},1)\to\partial H\times[0,1),\ (y,t)\mapsto(y,f_2(t))
\end{align*} 
It holds $F_i|_{\partial H\times[\frac{3}{4},1)}=\id_{\partial H\times[\frac{3}{4},1)}$ for $i=1,2$. Denote again $\varphi$ and $\psi$ the respective restrictions of these Nash maps to $U$ and $W$. Define
$$
\begin{array}{ll}
g:H^{\bullet}\to H,& x\mapsto\begin{cases}
x&\text{if $x\in H^{\bullet}\setminus U$},\\
(\Lambda\circ\varphi)^{-1}(F_1((\Lambda\circ\varphi)(x)))&\text{if $x\in H^{\bullet}\cap U$,}
\end{cases}
\end{array}
$$
$$
\begin{array}{ll}
g_+:H_+^{\bullet}\to H_+,& z\mapsto\begin{cases}
z&\text{if $z\in H_+^{\bullet}\setminus W$},\\
(\Delta\circ\psi)^{-1}(F_2((\Delta\circ\psi)(z)))&\text{if $z\in H_+^{\bullet}\cap W$.}
\end{cases}
\end{array}
$$
Both $g$ and $g_+$ are ${\mathcal S}^2$ diffeomorphisms.

\noindent{\em Step \em 4}. Consequently, both $H_+$ and $H_-=\tau(H_+)$ are Nash diffeomorphic to $H$ and have as common boundary $\partial H\times\{0\}$. In addition $\tau|_{\partial H\times\{0\}}=\id_{\partial H\times\{0\}}$, so ${\rm D}(H)$ is by \cite[\S I.6]{mu} the double of $H$.
\end{proof}

\paragraph{}\label{doubleiv} \em There exists an open semialgebraic neighborhood $V$ of $\partial H\times\{0\}$ in ${\rm D}(H)$ such that $M_\epsilon:=H_\epsilon\cup V$ is a Nash manifold Nash diffeomorphic to $\Int(H)$ that contains $H_\epsilon$ as a closed subset. In particular $\pi(M_\epsilon)=H$\em.
\begin{proof}
Define $V':=\{(y,s)\in\partial H\times\R:\ |s|<\sqrt{\veps(y)}\}$ and $V:=\psi^{-1}(V')$ where $\psi$ is the Nash map defined in \eqref{retractm}. Consider the Nash diffeomorphism
$$
\Theta:V'\to\partial H\times(-1,1),\ (y,s)\to\Big(y,\frac{s}{\sqrt{\veps(y)}}\Big).
$$
By \ref{doubleiii} it is enough to show that $M_\epsilon:=H_\epsilon\cup V$ and $\Int(H_\epsilon)$ are Nash diffeomorphic. As $\tau$ is an involution such that $\tau(H_-)=H_+$, we assume $\epsilon=-$. By \ref{s2boun} it is enough to show that the objects above are ${\mathcal S}^2$ diffeomorphic.

Let $f_3:(-1,1)\to(-1,0)$ be an ${\mathcal S}^2$ diffeomorphism such that $f_3|_{(-1,-\frac{1}{2}]}=\id_{(-1,-\frac{1}{2}]}$ (see Example \ref{s2diffeo}(iii)). Consider the ${\mathcal S}^2$ diffeomorphism 
$$
F_3:\partial H\times(-1,1)\to\partial H\times(-1,0),\ (y,t)\mapsto(y,f_3(t))
$$
and define
$$
h:M_-\to\Int(H_-),\ z\mapsto\begin{cases}
z&\text{if $z\in M_-\setminus V$},\\
(\Theta\circ\psi)^{-1}(F_3((\Theta\circ\psi)(z)))&\text{if $x\in V$,}
\end{cases}
$$
which is an ${\mathcal S}^2$ diffeomorphism.

Observe $H_-\subset M_-$ and $H=\pi(H_-)\subset\pi(M_-)\subset\pi(D(H))\subset H$, as required.
\end{proof}

\begin{figure}[!ht]
\begin{center}
\begin{tikzpicture}[scale=0.75]

\draw[line width=1pt,rotate=-90,dashed] (-3.5,0) parabola bend (-5.5,4) (-7.5,0);
\draw[rotate=-90,draw=none,fill=gray!100] (-4.5,6) parabola bend (-4.5,6) (-6.5,2)--
(-7.5,0) parabola bend (-5.5,4) (-6.5,3)--(-5.5,5);
\draw[rotate=-90,draw=none,fill=gray!20] (-2.5,2) parabola bend (-4.5,6) (-6,3.75) parabola bend (-5.5,4) (-3.5,0);

\draw[line width=0.5pt,rotate=-90,dashed] (-2.5,2) parabola bend (-4.5,6) (-6.5,2);
\draw[line width=0.5pt,dashed](2,2.5)--(0,3.5);
\draw[line width=1.5pt,draw](6,4.5)--(4,5.5);
\draw[line width=0.5pt,dashed](2,6.5)--(0,7.5);
\draw[line width=0.5pt,dashed](5,5.5)--(3,6.5);

\draw[draw=none,fill=gray!20] (2,0.5) -- (6,0.5) -- (4,1.5) -- (0,1.5) -- (2,0.5);
\draw[line width=0.5pt,dashed](2,0.5)--(6,0.5);
\draw[line width=1.5pt,draw](6,0.5)--(4,1.5);
\draw[line width=0.5pt,dashed](4,1.5)--(0,1.5);
\draw[line width=0.5pt,dashed](0,1.5)--(2,0.5);

\draw[<->, line width=0.5pt] (1.5,5) -- (6.5,5);
\draw[<->, line width=0.5pt] (3.5,5.75) -- (6.5,4.25);
\draw[<->, line width=0.5pt] (5,3) -- (5,7);

\draw[->, line width=1pt] (5.5,3) -- (5.5,1.5);

\draw (2,5.25) node{\small$x_1$};
\draw (6.85,4.25) node{\small$x_d$};
\draw (5.25,6.75) node{\small$t$};
\draw (0.6,6.85) node{\small$H_+$};
\draw (0.6,5.5) node{\small${\rm D}(H)$};
\draw (0.6,4) node{\small$H_-$};
\draw (6.1,5.25) node{\small$\partial H\times\{0\}$};
\draw (2.25,1) node{\small$H$};
\draw (6,1) node{\small$\partial H$};
\draw (5.75,2.25) node{\small$\pi$};

\draw[draw=none,fill=gray!20] (9,0.5) -- (13,0.5) -- (11,1.5) -- (7,1.5) -- (9,0.5);
\draw[line width=0.5pt,dashed](9,0.5)--(13,0.5);
\draw[line width=0.5pt,dashed](11,1.5)--(7,1.5);
\draw[line width=0.5pt,dashed](7,1.5)--(9,0.5);

\draw (9.25,1) node{\small$\Int(H)$};

\draw[fill=gray!100,draw=none] (13,0.5) .. controls (10,1) and (10,1) .. (11,1.5) -- (13,0.5);
\draw[line width=1.5pt] (13,0.5) .. controls (10,1) and (10,1) .. (11,1.5);
\draw[line width=1pt,dashed](13,0.5)--(11,1.5);

\draw[line width=1pt,rotate=-90,dashed] (-3.5,13) parabola bend (-5.5,17) (-7.5,13);
\draw[rotate=-90,draw=none,fill=gray!100] (-4.5,19) parabola bend (-4.5,19) (-6.5,15)--
(-7.5,13) parabola bend (-5.5,17) (-6.5,16)--(-5.5,18);
\draw[rotate=-90,draw=none,fill=gray!20] (-2.5,15) parabola bend (-4.5,19) (-6,16.75) parabola bend (-5.5,17) (-3.5,13);

\draw[line width=0.5pt,rotate=-90,dashed] (-2.5,15) parabola bend (-4.5,19) (-6.5,15);

\draw[line width=1pt,draw=none,fill=gray!100] (19,4.5) .. controls (17.5,5.75) and (17.5,5.75) .. (17,5.5) -- (19,4.5);
\draw[line width=1pt,dashed] (19,4.5) .. controls (17.5,5.75) and (17.5,5.75) .. (17,5.5);

\draw[line width=0.5pt,dashed](15,2.5)--(13,3.5);
\draw[line width=1.5pt,draw](19,4.5)--(17,5.5);
\draw[line width=0.5pt,dashed](15,6.5)--(13,7.5);
\draw[line width=0.5pt,dashed](18,5.5)--(16,6.5);

\draw[draw=none,fill=gray!20] (15,0.5) -- (19,0.5) -- (17,1.5) -- (13,1.5) -- (15,0.5);
\draw[line width=0.5pt,dashed](15,0.5)--(19,0.5);
\draw[line width=1.5pt,draw](19,0.5)--(17,1.5);
\draw[line width=0.5pt,dashed](17,1.5)--(13,1.5);
\draw[line width=0.5pt,dashed](13,1.5)--(15,0.5);

\draw[line width=1pt,draw=none,fill=gray!100] (19,0.5) .. controls (16.5,1) and (16.5,1) .. (17,1.5) -- (19,0.5);
\draw[line width=1pt,dashed] (19,0.5) .. controls (16.5,1) and (16.5,1) .. (17,1.5);
\draw[line width=1pt,dashed] (19,0.5)--(17,1.5);
\draw[line width=1.5pt,draw](19,0.5)--(17,1.5);

\draw (16.25,3.75) node{\small$M_-$};

\draw[->, line width=1pt] (12.5,1) -- (13.5,1);
\draw[->, line width=1pt] (11.25,1.75) -- (13.75,2.75);
\draw[->, line width=1pt] (18.5,3) -- (18.5,1.5);
\draw[->, line width=0.5pt] (9.85,2.2) -- (10.35,1.15);

\draw (13.6,6.85) node{\small$H_+$};
\draw (13.6,5.5) node{\small${\rm D}(H)$};
\draw (13.6,4) node{\small$H_-$};
\draw (19.1,4.25) node{\small$\partial H\times\{0\}$};
\draw (15.25,1) node{\small$H$};
\draw (19,1) node{\small$\partial H$};
\draw (18.75,2.25) node{\small$\pi$};
\draw (12.85,1.35) node{\small$f$};
\draw (12.3,2.5) node{\small$g$};
\draw (12.3,1.8) node{\small$\cong$};
\draw (10,2.5) node{\small$f^{-1}(\partial H)$};


\end{tikzpicture}
\end{center}
\caption{(a) Nash double of $H$ and (b) Surjective Nash map $f:\Int(H)\to H$.\label{fig2}}
\end{figure}
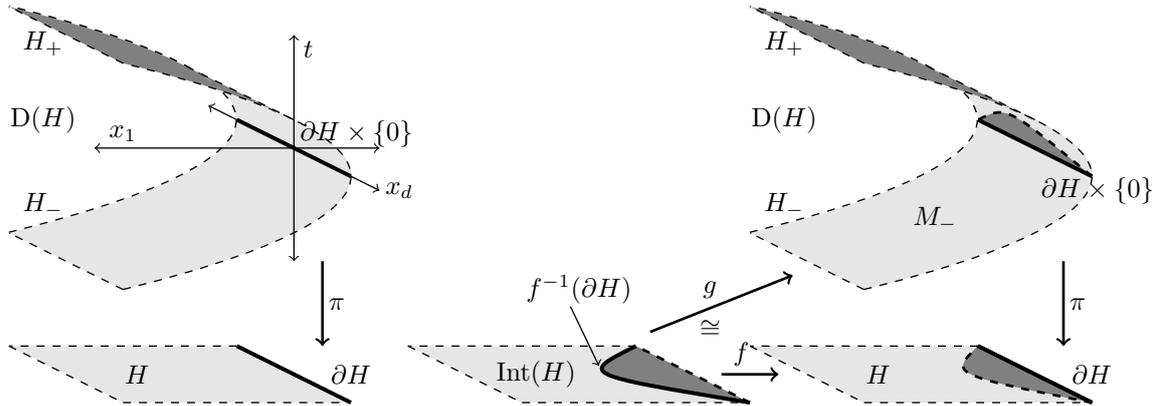

\subsection{Proof of Proposition \ref{doubleivc}}
Let ${\rm D}(H)$ be the Nash double of $H$. By \ref{doubleiv} there exist 
\begin{itemize}
\item a Nash manifold $M_-\subset{\rm D}(H)$ such that $H_-\subset M_-$ is a closed subset and $\pi(M_-)=H$,
\item a Nash diffeomorphism $g:\Int(H)\to M_-$.
\end{itemize}
Let us check that $f:=\pi\circ g$ satisfy the required conditions. First $f(\Int(H))=\pi(M_-)=H$. Next, pick a point $y_0\in\partial H$ and let $x_0\in f^{-1}(y_0)$. We have to prove that $f$ has a local representation $(x_1,\ldots,x_d)\mapsto(x_1^2,x_2,\ldots,x_d)$ at $x_0$. It holds $z_0:=g(x_0)=(y_0,0)$. As $g$ is a Nash diffeomorphism, it is enough to prove that $\pi$ has a local representation $(x_1,\ldots,x_d)\mapsto(x_1^2,x_2,\ldots,x_d)$ at $z_0$. But this follows by \ref{doubleii}(iii), as required.
\qed

\section{A main ingredient: The drilling blow-up}\label{s5}

In this section we construct the \em drilling blow-up \em of a Nash manifold $M$ with center a closed Nash submanifold $N$. We refer the reader to \cite{hi2,ds} for the oriented blow-up of a real analytic space with center a closed subspace, which is the counterpart of the tool we present in this section for the real analytic setting. In \cite[\S5]{hpv} appears a presentation of the oriented blow-up in the analytic case closer to the drilling blow-up we provide here. The authors consider there the case of the oriented blow-up of a real analytic manifold $M$ with center a closed real analytic submanifold $N$ whose vanishing ideal inside $M$ is finitely generated (this happens for instance if $N$ is compact). In \cite[\S3]{fe2} we present a similar construction in the semialgebraic setting, which is used to `appropriately embed' semialgebraic sets in Euclidean spaces. The following result, whose proof makes use of the drilling blow-up and which will be a key tool for our purposes, will allow to erase a closed Nash submanifold from a Nash manifold (see Figure \ref{fig5}).

\begin{lem}[Erasing a closed Nash submanifold]\label{eraseresidual}
Let $M$ be a Nash manifold and let $N\subset M$ be a closed Nash submanifold. Then there exists a surjective Nash map $h:M\setminus N\to M$.
\end{lem} 
 
\subsection{Local structure of the drilling blow-up}\label{bigstep}\setcounter{paragraph}{0}
Let $M\subset\R^m$ be a Nash manifold of dimension $d$ and let $N\subset M$ be a closed Nash submanifold of dimension $e$. Assume that there exists a Nash diffeomorphism $u:=(u_1,\ldots,u_d):M\to\R^d$ such that $N=\{u_{e+1}=0,\ldots,u_d=0\}$. Denote $\psi:=u^{-1}:\R^d\equiv\R^e\times\R^{d-e}\to M$. Let $\zeta_{e+1},\ldots,\zeta_d:\R^d\to\R^k$ be Nash maps such that the vectors $\zeta_{e+1}(y,0),\ldots,\zeta_d(y,0)$ are linearly independent for each $y\in\R^e$. Write $z\in\R^{d-e}$ as $z:=(z_{e+1},\ldots,z_d)$. Consider the Nash maps 
\begin{align*}
&\varphi:\R^d\equiv\R^e\times\R^{d-e}\to\R^k,\ (y,z)\mapsto\zeta_{e+1}(y,z)z_{e+1}+\ldots+\zeta_d(y,z)z_d,\\
&\phi:\R^e\times\R\times\sph^{d-e-1}\to\R^k,\ (y,\rho,w)\mapsto\zeta_{e+1}(y,\rho w)w_{e+1}+\cdots+\zeta_d(y,\rho w)w_d
\end{align*} 
and assume that $\varphi(y,z)=0$ if and only if $z=0$. Consider the projections 
\begin{align*}
&\theta_1:\R^d\equiv\R^e\times\R^{d-e}\to\R^e,\ (y,z)\mapsto y,\\
&\theta_2:\R^d\equiv\R^e\times\R^{d-e}\to\R^{d-e},\ (y,z)\mapsto z. 
\end{align*}
Define
$$
\Phi:\R^e\times\R\times\sph^{d-e-1}\to M\times\sph^{k-1},\ (y,\rho,w)\mapsto\Big(\psi(y,\rho w),\frac{\phi(y,\rho,w)}{\|\phi(y,\rho,w)\|}\Big).
$$

\paragraph{} \em $\Phi$ is a well-defined Nash map\em.
\begin{proof}
Observe first that $\varphi(y,\rho w)=\rho\phi(y,\rho,w)$ for all $(y,\rho,w)\in\R^e\times\R\times\sph^{d-e-1}$. If $\rho\neq0$, the product $\rho w\neq0$, so $\rho\phi(y,\rho,w)=\varphi(y,\rho w)\neq0$ and $\phi(y,\rho,w)\neq0$. If $\rho=0$, then 
$$
\phi(y,0,w)=\zeta_{e+1}(y,0)w_{e+1}+\cdots+\zeta_d(y,0)w_d.
$$
As $\zeta_{e+1}(y,0),\ldots,\zeta_d(y,0)$ are linearly independent and $\|w\|=1$, we conclude $\phi(y,0,w)\neq0$. Consequently, $\Phi$ is a well-defined Nash map, as required.
\end{proof}

\paragraph{}\label{bigstepa2}\em Fix $\epsilon=\pm$ and denote 
$$
I_\epsilon:=\begin{cases}
[0,+\infty)&\text{ if $\epsilon=+$,}\\
(-\infty,0]&\text{ if $\epsilon=-$.}
\end{cases}
$$ 
The closure $\widetilde{M}_\epsilon$ in $M\times\sph^{k-1}$ of the set
$$
\Gamma_\epsilon:=\left\{\Big(\psi(y,z),\epsilon\frac{\varphi(y,z)}{\|\varphi(y,z)\|}\Big)\in M\times\sph^{k-1}:\,z\neq 0\right\}
$$ 
is a Nash manifold with boundary such that:
\begin{itemize}
\item[(i)] $\widetilde{M}_\epsilon\subset\im(\Phi)$.
\item[(ii)] The restriction of $\Phi$ to $\R^e\times I_\epsilon\times\sph^{d-e-1}$ induces a Nash diffeomorphism between $\R^e\times I_\epsilon\times\sph^{d-e-1}$ and $\widetilde{M}_\epsilon$\em. Consequently, \em $\partial\widetilde{M_\epsilon}=\Phi(\R^e\times\{0\}\times\sph^{d-e-1})$ and $\Gamma_\epsilon=\Int(\widetilde{M_\epsilon})=\Phi(\R^e\times(I_\epsilon\setminus\{0\})\times\sph^{d-e-1})$\em.
\end{itemize}\em
\begin{proof}
(i) Recall $\varphi(y,\rho w)=\rho\phi(y,\rho,w)$ for all $(y,\rho,w)\in \R^e\times\R\times \sph^{d-e-1}$. As ${\rm sign}(\rho)=\frac{\rho}{|\rho|}$ for $\rho\neq 0$, we have
\begin{equation}\label{sign}
\frac{\phi(y,\rho,w)}{\|\phi(y,\rho,w)\|}={\rm sign}(\rho)\frac{\varphi(y,\rho w)}{\|\varphi(y,\rho w)\|}.
\end{equation}

Consequently, $\Gamma_+\sqcup\Gamma_-\subset\im(\Phi)$. Let us check: $\widetilde{M}_\epsilon\setminus\Gamma_\epsilon\subset\im(\Phi)$. 

Pick a point $(a,b)\in\widetilde{M}_\epsilon\setminus\Gamma_\epsilon$. By the Nash curve selection lemma \cite[Prop.8.1.13]{bcr} there exists a Nash arc $\gamma:(-1,1)\to M\times \sph^{k-1}$ such that $\gamma(0)=(a,b)$ and $\gamma((0,1))\subset\Gamma_\epsilon$. Let $\pi:M\times\sph^{k-1}\to M$ be the projection onto the first factor and let $(\alpha,\beta):(-1,1)\to\R^e\times\R^{d-e}\equiv\R^d$ be a Nash arc such that $\psi(\alpha,\beta)=\pi\circ\gamma$. We write 
$$
\gamma|_{(0,1)}=\Big(\psi,\epsilon\frac{\varphi}{\|\varphi\|}\Big)\circ(\alpha,\beta)|_{(0,1)}.
$$ 
Note that $\beta(0)=0$ because otherwise $(a,b)\in\Gamma_\epsilon$. In addition, $a=\psi(\alpha(0),0)$. Observe that 
\begin{equation}\label{limitb}
b=\lim_{t\to 0^+}\epsilon\frac{\varphi(\alpha(t),\beta(t))}{\|\varphi(\alpha(t),\beta(t))\|}.
\end{equation}
As $\gamma((0,1))\subset\Gamma_\epsilon$, the Nash arc $\beta$ is not identically $0$ and we may assume $\beta(t)=0$ if and only if $t=0$. Let $\xi\in\R[[\t]]_{\rm alg}$ be a Nash series such that $\|\beta(t)\|=\xi(t)$ for $t>0$ small. As
$$
\text{order}(\xi)=\min_{e+1\leq i\leq d}\{\text{order}(\beta_i)\},
$$
the quotient $\frac{\beta}{\|\beta\|}|_{(0,1)}$ extends to a Nash arc $\eta:(-1,1)\to\sph^{d-e-1}$. We have $\beta=(\epsilon\xi)(\epsilon\eta)$. By \eqref{sign} and \eqref{limitb}
$$
b=\lim_{t\to 0^+}\frac{\phi(\alpha(t),\epsilon\xi(t),\epsilon\eta(t))}{\|\phi(\alpha(t),\epsilon\xi(t),\epsilon\eta(t))\|}
$$
Consequently, as $\xi(0)=0$,
$$
b=\frac{\phi(\alpha(0),0,\epsilon\eta(0))}{\|\phi(\alpha(0),0,\epsilon\eta(0))\|}\quad\text{and}\quad(a,b)=\Phi(\alpha(0),0,\epsilon\eta(0))\in\im(\Phi).
$$

(ii) The proof of this statement is conducted in several steps:

\noindent{\em Step \em 1}. Denote $R:=\Phi(\R^e\times\{0\}\times \sph^{d-e-1})$ and consider the open semialgebraic subset of $\im(\Phi)$
\begin{equation}\label{u}
U:=\{(a,b)\in \im(\Phi):\,{\rm rk}((\zeta_{e+1}\circ u)(a),\ldots,(\zeta_d\circ u)(a))=d-e\}.
\end{equation}
As $\zeta_{e+1}(y,0),\dots,\zeta_d(y,0)$ are linearly independent for each $y\in\R^e$ we have $R\subset U$. In addition,
$$
\Phi^{-1}(U)=\{(y,\rho,w)\in \R^e\times\R\times\sph^{d-e-1}:\ {\rm rk}(\zeta_{e+1}(y,\rho w),\ldots,\zeta_d(y,\rho w))=d-e\}.
$$
We claim: \em $\Phi|_{\Phi^{-1}(U)}:\Phi^{-1}(U)\to U$ is a Nash diffeomorphism\em. In particular, \em $U$ is a Nash manifold that contains $R$\em.

Let $(a,b)\in U$ and let $(y,\rho,w)\in\R^e\times\R\times\sph^{d-e-1}$ be such that $\Phi(y,\rho,w)=(a,b)$, that is, $a=\psi(y,\rho w)$ and $b=\frac{\phi(y,\rho,w)}{\|\phi(y,\rho,w)\|}$. Consequently, $(y,\rho w)=u(a)=(\theta_1(u(a)),\theta_2(u(a)))$. Consider the system of linear equations
\begin{equation}\label{system}
{\tt v}_{e+1}(\zeta_{e+1}\circ u)(a)+\cdots+{\tt v}_d(\zeta_d\circ u)(a)=b.
\end{equation}
As $(a,b)\in U$ and $b\in\sph^{k-1}$, the system \eqref{system} has a unique non-zero solution $v\in\R^{d-e}$. If we solve \eqref{system} using Cramer's rule, one sees that the map $U\to\sph^{d-e-1},\ (a,b)\mapsto w:=(w_{e+1},\dots,w_d):=\frac{v}{\|v\|}$ is Nash. As $\|b\|=1$,
\begin{equation*}
\begin{split}
b&=v_{e+1}(\zeta_{e+1}\circ u)(a)+\cdots+v_d(\zeta_d\circ u)(a)=\frac{v_{e+1}(\zeta_{e+1}\circ u)(a)+\cdots+v_d(\zeta_d\circ u)(a)}{\|v_{e+1}(\zeta_{e+1}\circ u)(a)+\cdots+v_d(\zeta_d\circ u)(a)\|}\\
&=\frac{\frac{v_{e+1}}{\|v\|}(\zeta_{e+1}\circ u)(a)+\cdots+\frac{v_d}{\|v\|}(\zeta_d\circ u)(a)}{\|\frac{v_{e+1}}{\|v\|}(\zeta_{e+1}\circ u)(a)+\cdots+\frac{v_d}{\|v\|}(\zeta_d\circ u)(a)\|}=\frac{w_{e+1}(\zeta_{e+1}\circ u)(a)+\cdots+w_d(\zeta_d\circ u)(a)}{\|w_{e+1}(\zeta_{e+1}\circ u)(a)+\cdots+w_d(\zeta_d\circ u)(a)\|}.
\end{split}
\end{equation*}

As $\rho w=\theta_2(u(a))$, the vectors $\theta_2(u(a))$ and $w$ are linearly dependent. As $w$ is a unitary vector, $\theta_2(u(a))=\qq{\theta_2(u(a)),w}w$. It holds 
\begin{multline*}
\Phi(\theta_1(u(a)),\qq{\theta_2(u(a)),w},w)\\
=\Big(\psi(\theta_1(u(a)),\theta_2(u(a))),\frac{w_{e+1}(\zeta_{e+1}\circ u)(a)+\cdots+w_d(\zeta_d\circ u)(a)}{\|w_{e+1}(\zeta_{e+1}\circ u)(a)+\cdots+w_d(\zeta_d\circ u)(a)\|}\Big)=(a,b)
\end{multline*}
Write $w(a,b):=w$. The Nash map
\begin{equation}\label{psi0}
\Psi_0:U\to\R^e\times\R\times\sph^{d-e-1},\ (a,b)\mapsto(\theta_1(u(a)),\qq{\theta_2(u(a)),w(a,b)},w(a,b))
\end{equation} 
satisfies: \em $\im(\Psi_0)\subset\Phi^{-1}(U)$ and $\Phi\circ\Psi_0=\id_U$\em. Let us check in addition $\Psi_0\circ\Phi|_{\Phi^{-1}(U)}=\id_{\Phi^{-1}(U)}$ to conclude \em $\Phi|_{\Phi^{-1}(U)}$ is a Nash diffeomorphism.\em

Let $(y,\rho,w)\in\Phi^{-1}(U)$. We have
$$
(\Psi_0\circ\Phi)(y,\rho,w)=\Psi_0\Big(\psi(y,\rho w),\frac{\phi(y,\rho,w)}{\|\phi(y,\rho,w)\|}\Big)=\Big(y,\rho\QQ{w,\frac{v}{\|v\|}},\frac{v}{\|v\|}\Big)
$$
where $v\in\R^{d-e}$ is the unique solution of the system
$$
{\tt v}_{e+1}\zeta_{e+1}(y,\rho w)+\cdots+{\tt v}_d\zeta_d(y,\rho w)=\frac{w_{e+1}\zeta_{e+1}(y,\rho w)+\cdots+w_d\zeta_d(y,\rho w))}{\|w_{e+1}\zeta_{e+1}(y,\rho w)+\cdots+w_d\zeta_d(y,\rho w))\|}.
$$
Consequently, 
$$
v=\frac{w}{\|w_{e+1}\zeta_{e+1}(y,\rho w)+\cdots+w_d\zeta_d(y,\rho w))\|}
$$ 
and $\frac{v}{\|v\|}=w$ (recall that $\|w\|=1$). Thus, $(\Psi_0\circ\Phi)(y,\rho,w)=(y,\rho,w)$, as claimed. 

\noindent{\em Step \em 2}. Next, let us check: \em $\Phi^{-1}(\Gamma_\epsilon)=\R^e\times(I_\epsilon\setminus\{0\})\times\sph^{d-e-1}$ and the restriction 
$$
\Phi|:\R^e\times(I_\epsilon\setminus\{0\})\times\sph^{d-e-1}\to\Gamma_\epsilon
$$ 
is a Nash diffeomorphism\em.

Define 
$$
\Psi_\epsilon:\Gamma_\epsilon\to\R^e\times(I_\epsilon\setminus\{0\})\times\sph^{d-e-1}, (a,b)\mapsto\Big(\theta_1(u(a)),\epsilon\|\theta_2(u(a))\|,\frac{\theta_2(u(a))}{\epsilon\|\theta_2(u(a))\|}\Big).
$$
The previous map is well-defined and Nash because $\theta_2(u(a))\neq0$ for all $(a,b)\in\Gamma_\epsilon$. It follows that $\Phi|_{\R^e\times(I_\epsilon\setminus\{0\})\times\sph^{d-e-1}}$ is a Nash diffeomorphism whose inverse is $\Psi_\epsilon$.

\noindent{\em Step \em 3}. We deduce from Steps 1 \& 2 that the restriction $\Phi|:\R^e\times I_\epsilon\times\sph^{d-e-1}\to\widetilde{M}_\epsilon$ is a Nash diffeomorphism. As $\R^e\times I_\epsilon\times\sph^{d-e-1}$ is a Nash manifold with boundary, $\widetilde{M}_\epsilon$ is also a Nash manifold with boundary. In addition $\partial\widetilde{M_\epsilon}=R$ and $\Int(\widetilde{M_\epsilon})=\Gamma_\epsilon$, as required.
\end{proof}

\paragraph{}\label{bigstepa3}\em Denote $R:=\partial\widetilde{M}_+=\partial\widetilde{M}_-$ and $\widehat{M}:=\widetilde{M}_+\cup\widetilde{M}_-=\Gamma_+\sqcup R\sqcup\Gamma_-$. Then $\Phi$ induces a Nash diffeomorphism between $\R^e\times\R\times\sph^{d-e-1}$ and $\widehat{M}$. In addition, the Nash map $\sigma: M\times\sph^{k-1}\to M\times\sph^{k-1}, (a,b)\to(a,-b)$ induces a Nash involution on $\widehat{M}$ without fixed points such that $\sigma(\widetilde{M}_+)=\widetilde{M}_-$ and $\Phi(y,-\rho,-w)=(\sigma\circ\Phi)(y,\rho,w)$ for each $(y,\rho,w)\in\R^e\times\R\times\sph^{d-e-1}$\em. We call $\widehat{M}$ the \em twisted Nash double of $\widetilde{M}_+$\em. 
\begin{proof}
By \ref{bigstepa2} it holds 
\begin{equation}\label{doubleme}
\Phi(\R^e\times\R\times\sph^{d-e-1})=\Gamma_-\sqcup R\sqcup\Gamma_+=\widehat{M}.
\end{equation}
In addition, it follows from the proof of \ref{bigstepa2}(ii) (Steps 1 \& 2) that $\Phi$ is injective and a local Nash diffeomorphism (for the points of $\R^e\times\{0\}\times\sph^{d-e-1}$ use the map $\Psi_0$ introduced in \eqref{psi0}). Consequently, $\Phi$ is a Nash diffeomorphism onto its image $\widehat{M}$ and the latter is a Nash manifold.

The second part of the statement is straightforward.
\end{proof}

\paragraph{}\em\label{bigstepa4} Consider the projection $\pi:M\times\sph^{k-1}\to M$ onto the first factor and denote $\pi_\epsilon:=\pi|_{\widetilde{M_\epsilon}}$. Then
\begin{itemize}
\item[(i)] $\pi_\epsilon$ is proper, $\pi_\epsilon(\widetilde{M_\epsilon})=M$ and $R=\pi^{-1}_\epsilon(N)$.
\item[(ii)] The restriction $\pi_\epsilon|_{\Gamma_\epsilon}:\Gamma_\epsilon\to M\setminus N$ is a Nash diffeomorphism.
\item[(iii)] For each $q\in N$ it holds $\pi_\epsilon^{-1}(q)=\{q\}\times\sph_q^{d-e-1}$ where $\sph_q^{d-e-1}$ is the sphere of dimension $d-e-1$ obtained when intersecting the sphere $\sph^{k-1}$ with the linear subspace $L_q$ generated by $(\zeta_{e+1}\circ u)(q),\ldots,(\zeta_d\circ u)(q)$.
\end{itemize}\em
\begin{proof}
(i) Let $K\subset M$ be a compact set. Then $\pi^{-1}(K)=K\times\sph^{k-1}$, which is a compact set. As $\widetilde{M}_\epsilon=\cl(\Gamma_\epsilon)$ is a closed subset of $M\times\sph^{k-1}$, the intersection $\pi^{-1}(K)\cap\widetilde{M}_\epsilon$ is a compact set, so $\pi_\epsilon$ is proper. Thus, $\pi(\widetilde{M}_\epsilon)=\cl(\pi(\Gamma_\epsilon))=\cl(M\setminus N)=M$. In addition $\pi^{-1}_\epsilon(N)=\widetilde{M}_\epsilon\setminus\Gamma_\epsilon=R$.

(ii) We have the commutative diagram
$$
\xymatrix{
\R^d\setminus\{z=0\}\ar@{<->}[d]_{\id}\ar[rr]^\Theta&&\Gamma_\epsilon\ar[d]^{\pi|_{\Gamma_\epsilon}}&(y,z)\ar@{|->}[d]\ar@{|->}[rr]&&\Big(\psi(y,z),\epsilon\frac{\varphi(y,z)}{\|\varphi(y,z)\|}\Big)\ar@{|->}[d]\\
\R^d\setminus\{z=0\}\ar[rr]^{\psi|}_\cong&&M\setminus N&(y,z)\ar@{|->}[rr]&&\psi(y,z)}
$$ 
As $\Theta$ is a Nash diffeomorphism, we conclude that $\pi|_{\Gamma_\epsilon}$ is also a Nash diffeomorphism.

(iii) Let $q\in N$. We have 
$$
\pi_\epsilon^{-1}(q)=\pi^{-1}(q)\cap R=\pi^{-1}(q)\cap\Phi(\R^e\times\{0\}\times\sph^{d-e-1})=\{q\}\times\Delta_u(\sph^{d-e-1}),
$$
where 
$$
\Delta_u:\sph^{d-e-1}\to\sph^{k-1},\ w:=(w_{e+1},\ldots,w_d)\mapsto\frac{(\zeta_{e+1}\circ u)(q)w_{e+1}+\cdots+(\zeta_d\circ u)(q)w_d}{\|(\zeta_{e+1}\circ u)(q)w_{e+1}+\cdots+(\zeta_d\circ u)(q)w_d\|}.
$$
The map $\Delta_u$ is a Nash diffeomorphism between $\sph^{d-e-1}$ and the sphere obtained when intersecting $\sph^{k-1}$ with the linear subspace $L_q$ generated by $(\zeta_{e+1}\circ u)(q),\ldots,(\zeta_d\circ u)(q)$.
\end{proof}

\paragraph{}\label{bigstepa5}
Denote $\widehat{\pi}:=\pi|_{\widehat{M}}$ and consider the commutative diagram (see also Figures \ref{fig3} and \ref{fig4}).
\begin{equation}\label{diagpol}
\begin{gathered}
\xymatrix{
\R^e\times \R\times\sph^{d-e-1}\ar[d]_{u\circ\widehat{\pi}\circ\Phi}\ar[rr]^(0.6){\Phi}_(0.6){\cong}&&\widehat{M}\ar[d]^{\widehat{\pi}}&(y,\rho,w)\ar@{|->}[d]\ar@{|->}[r]&\Big(\psi(y,\rho w),\frac{\phi(y,\rho,w)}{\|\phi(y,\rho,w)\|}\Big)\ar@{|->}[d]\\
\R^d&&\ar[ll]_{u}^{\cong}M&(y,\rho w)&\psi(y,\rho w)\ar@{|->}[l]
}
\end{gathered}
\end{equation}

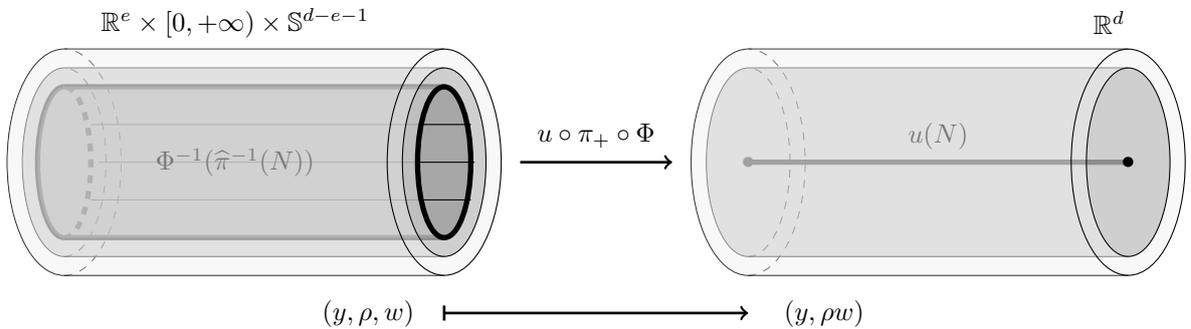
\begin{figure}[!ht]
\begin{center}
\begin{tikzpicture}[scale=1]

\draw (0.75,4) arc (90:270:0.75cm and 1.5cm);
\draw[dashed] (0.75,1) arc (270:450:0.75cm and 1.5cm);

\draw (1.1,2) -- (6.1,2);
\draw (1.2,2.5) -- (6.15,2.5);
\draw (1.1,3) -- (6.1,3);

\draw(0.75,3.75) arc (90:270:0.55cm and 1.25cm);
\draw[dashed] (0.75,1.25) arc (270:450:0.55cm and 1.25cm);

\draw[line width=2pt] (0.75,3.5) arc (90:270:0.35cm and 1cm);
\draw[line width=2pt, dashed] (0.75,1.5) arc (270:450:0.35cm and 1cm);

\draw (0.75,1) -- (5.75,1);
\draw (0.75,4) -- (5.75,4);

\draw(0.75,1.25) -- (5.75,1.25);
\draw(0.75,3.75) -- (5.75,3.75);

\draw[line width=2pt] (0.75,1.5) -- (5.75,1.5);
\draw[line width=2pt] (0.75,3.5) -- (5.75,3.5);

\draw[fill=gray!100,opacity=0.5,draw=none] (0.75,3.5) arc (90:270:0.35cm and 1cm) -- (5.75,1.5) arc (270:90:0.35cm and 1cm) -- (5.75,3.5);
\draw[fill=gray!70,opacity=0.5,draw=none] (0.75,3.75) arc (90:270:0.55cm and 1.25cm) -- (5.75,1.25) arc (270:90:0.55cm and 1.25cm) -- (5.75,3.75);

\draw (3,2.5) node{\small$\Phi^{-1}(\widehat{\pi}^{-1}(N))$};

\draw[fill=gray!10,opacity=0.5,draw=none] (0.75,4) arc (90:270:0.75cm and 1.5cm) -- (5.75,1) arc (270:90:0.75cm and 1.5cm) -- (5.75,4);

\draw[fill=gray!10,opacity=0.5,draw=none] (5.75,2.5) ellipse (0.75cm and 1.5cm);
\draw[fill=gray!70,opacity=0.5,draw=none] (5.75,2.5) ellipse (0.55cm and 1.25cm);
\draw[fill=gray!100,opacity=0.5,draw=none] (5.75,2.5) ellipse (0.35cm and 1cm);

\draw (5.45,2) -- (6.1,2);
\draw (5.40,2.5) -- (6.15,2.5);
\draw (5.45,3) -- (6.1,3);

\draw (5.75,2.5) ellipse (0.75cm and 1.5cm);
\draw (5.75,2.5) ellipse (0.55cm and 1.25cm);
\draw[line width=2pt] (5.75,2.5) ellipse (0.35cm and 1cm);

\draw (9.75,4) arc (90:270:0.75cm and 1.5cm);
\draw[dashed] (9.75,1) arc (270:450:0.75cm and 1.5cm);

\draw (9.75,3.75) arc (90:270:0.55cm and 1.25cm);
\draw[dashed] (9.75,1.25) arc (270:450:0.55cm and 1.25cm);

\draw (9.75,1) -- (14.75,1);
\draw (9.75,4) -- (14.75,4);

\draw (9.75,1.25) -- (14.75,1.25);
\draw (9.75,3.75) -- (14.75,3.75);

\draw (9.75,2.5) node{\small$\bullet$};
\draw[line width=2pt] (9.75,2.5) -- (14.75,2.5);

\draw[fill=gray!70,opacity=0.5,draw=none] (9.75,3.75) arc (90:270:0.55cm and 1.25cm) -- (14.75,1.25) arc (270:90:0.55cm and 1.25cm) -- (14.75,3.75);

\draw (12.25,2.85) node{\small$u(N)$};
\draw[fill=gray!10,opacity=0.5,draw=none] (9.75,4) arc (90:270:0.75cm and 1.5cm) -- (14.75,1) arc (270:90:0.75cm and 1.5cm) -- (14.75,4);

\draw[fill=gray!10,opacity=0.5,draw=none] (14.75,2.5) ellipse (0.75cm and 1.5cm);
\draw[fill=gray!70,opacity=0.5,draw=none] (14.75,2.5) ellipse (0.55cm and 1.25cm);

\draw (14.75,2.5) ellipse (0.75cm and 1.5cm);
\draw (14.75,2.5) ellipse (0.55cm and 1.25cm);

\draw[->, line width=1pt] (6.75,2.5) -- (8.75,2.5);
\draw[->, line width=1pt] (5.75,0.5) -- (9.75,0.5);
\draw[line width=1pt] (5.75,0.4) -- (5.75,0.6);
\draw (4.75,0.5) node{\small$(y,\rho,w)$};
\draw (10.75,0.5) node{\small$(y,\rho w)$};
\draw (14.5,4.35) node{\small$\R^d$};
\draw (3,4.35) node{\small$\R^e\times[0,+\infty)\times\sph^{d-e-1}$};
\draw (7.75,2.85) node{\small$u\circ\pi_+\circ\Phi$};
\draw (14.75,2.5) node{\small$\bullet$};

\end{tikzpicture}
\end{center}
\caption{Local structure of the drilling blow-up $\widetilde{M}_+$ of $M$ of center $N$.\label{fig3}}
\end{figure}

As a consequence, we have: \em The Nash maps $\pi_\epsilon$ and $\widehat{\pi}$ have local representations 
\begin{equation}\label{lr}
x:=(x_1,\ldots,x_d)\mapsto (x_1,\ldots,x_e,x_{e+1},x_{e+1}x_{e+2},\ldots,x_{e+1}x_d)
\end{equation}
in an open neighborhood of each point $p\in R$. In addition, $d\pi_p(T_p \widehat{M})\not\subset T_{\pi(p)}N$\em. 
\begin{proof}
After a change of coordinates in $\R^e\times\R\times\sph^{d-e-1}$, we may assume that $p\in R$ is the image of the point $(0,0,(1,0,\ldots,0))$. Consider the local parametrization around $(0,0,(1,0,\ldots,0))$ of the set $\R^e\times\R\times\sph^{d-e-1}$ given by 
$$
\begin{array}{cccl}
\eta: &\R^e\times\R\times\Bb&\to&\R^e\times\R\times\sph^{d-e-1},\\
&(y,\rho,v:=(v_{e+2},\ldots,v_d))&\mapsto& (y,\rho,(\sqrt{1-\|v\|^2},v))
\end{array}
$$
where $\Bb$ is the open ball of center the origin and radius $1$ in $\R^{d-e-1}$. It holds 
$$
\begin{array}{cccl}
u\circ\widehat{\pi}\circ\Phi\circ\eta:&\R^e\times\R\times\Bb&\to&\R^d,\\ 
&(y,\rho,v)&\mapsto& (y,\rho\sqrt{1-\|v\|^2},\rho v).
\end{array}
$$
Consider the Nash diffeomorphism
$$
f:\R^e\times\R\times\Bb\to\R^d, (y,\rho,v)\to \Big(y,\rho\sqrt{1-\|v\|^2},\frac{v}{\sqrt{1-\|v\|^2}}\Big),
$$
whose inverse is
$$
f^{-1}:\R^d\to\R^e\times\R\times\Bb,\ (y,\rho',v')\mapsto\Big(y,\rho'\sqrt{1+\|v'\|^2},\frac{v'}{\sqrt{1+\|v'\|^2}}\Big).
$$
The Nash map
$$
\pi':=u\circ\widehat{\pi}\circ\Phi\circ\eta\circ f^{-1}:\R^d\to\R^d, (y,\rho',v')\mapsto (y,\rho',\rho'v').
$$
represents $\widehat{\pi}$ locally around $p$ and the restriction
$$
\pi'_\epsilon:=\pi'|_{\{\epsilon\rho'\geq 0\}}:\{\epsilon\rho\geq 0\}\to\R^d, (y,\rho',v')\mapsto (y,\rho',\rho'v').
$$
represents $\pi_\epsilon$ locally around $p$.

To prove that $d\pi_p(T_p \widehat{M})\not\subset T_{\pi(p)}N$, it is enough to show $
d\pi'_0(\R^d)\not\subset\{z=0\}=u(N)$. 
 
If $\vec{\tt e}_{e+1}:=(0,\ldots,0,\overset{(e+1)}{1},0,\ldots,0)$, we have $d\pi'_0({\tt e}_{e+1})={\tt e}_{e+1}\not\in\{z=0\}$, as required.
\end{proof}

\subsection{Global structure of the drilling blow-up}\label{bbu}\setcounter{paragraph}{0} 
We construct next the \em drilling blow-up of a Nash manifold with center a closed Nash submanifold\em. We refer the reader to Figure \ref{fig4} in order to get a global idea of the involved strategy. Let $M\subset\R^m$ be a Nash manifold of dimension $d$ and let $N\subset M$ be a closed Nash submanifold of dimension $e$. Let $f_1,\ldots,f_k\in{\mathcal N}(M)$ be a finite system of generators of the ideal $I(N)$ of Nash functions on $M$ vanishing identically on $N$. Consider the Nash map
\begin{equation}\label{XX}
F:M\setminus N\to\sph^{k-1}, x\mapsto\frac{(f_1(x),\ldots,f_k(x))}{\|(f_1(x),\ldots,f_k(x))\|}.
\end{equation}
We have:

\paragraph{}\label{bigstepa2g}\em Fix $\epsilon=\pm$. The closure $\widetilde{M}_\epsilon$ in $M\times\sph^{k-1}$ of the graph 
$$
\Gamma_\epsilon:=\{(x,\epsilon F(x))\in M\times\sph^{k-1}:\ x\in M\setminus N\}
$$ 
is a Nash manifold with boundary. Denote $R:=\partial\widetilde{M}_+=\partial\widetilde{M}_-$ and $\widehat{M}:=\widetilde{M}_+\cup\widetilde{M}_-=\Gamma_+\sqcup R\sqcup\Gamma_-$. In addition, $\widehat{M}$ is a Nash manifold and the Nash map $\sigma: M\times\sph^{k-1}\to M\times\sph^{k-1}, (a,b)\to(a,-b)$ induces a Nash involution on $\widehat{M}$ without fixed points that maps $\widetilde{M}_+$ onto $\widetilde{M}_-$\em. We call $\widehat{M}$ the \em twisted Nash double of $\widetilde{M}_+$\em. 

\begin{proof}
The Nash manifold $M$ can be covered by \ref{fnashnc2} (applied to $N$) and by \ref{cartaspm} (applied to $M\setminus N$) with finitely many open semialgebraic subsets $U_i$ equipped with Nash diffeomorphisms $u_i:=(u_{i,1},\ldots,u_{i,d}):U_i\to\R^d$ such that either $U_i\cap N=\varnothing$ or $U_i\cap N=\{u_{i,e+1}=0,\ldots,u_{i,d}=0\}$. Denote $\psi_i:=u_i^{-1}$.

The collection $\{U_i\times\sph^{k-1}\}_i$ is a finite open semialgebraic covering of $M\times\sph^{k-1}$. To prove that $\widetilde{M_\epsilon}$ is a Nash manifold with boundary, we will show that $\widetilde{M_\epsilon}\cap(U_i\times\sph^{k-1})$ is a Nash manifold with (possibly empty) boundary for $i=1,\ldots,r$. Analogously, to check that $\widehat{M}$ is a Nash manifold, we will show that each intersection $\widehat{M}\cap(U_i\times\sph^{k-1})$ is a Nash manifold.

\noindent{\em Case \em 1.} If $N\cap U_i=\varnothing$, the intersection $\widetilde{M_\epsilon}\cap(U_i\times\sph^{k-1})$ is the graph of the Nash map $\epsilon F|_{U_i}$, so it is a Nash manifold. Observe that $(\widetilde{M_+}\cap\widetilde{M_-})\cap(U_i\times\sph^{k-1})=\varnothing$, so 
$$
\widehat{M}\cap(U_i\times\sph^{k-1})=(\widetilde{M}_+\cap(U_i\times\sph^{k-1}))\sqcup(\widetilde{M}_-\cap(U_i\times\sph^{k-1}))
$$
is also a Nash manifold.

\noindent{\em Case \em 2.} If $N\cap U_i\neq\varnothing$, the intersection $\widetilde{M_\epsilon}\cap(U_i\times\sph^{k-1})$ is the closure in $U_i\times\sph^{k-1}$ of the set 
$$
\Gamma_{i,\epsilon}:=\Gamma_\epsilon\cap(U_i\times\sph^{k-1})=\{(\psi_i(y,z),\epsilon F(\psi_i(y,z)))\in U_i\times\sph^{k-1}:\, z\neq 0\}
$$
and $\widehat{M}\cap (U_i\times\sph^{k-1})=(\widetilde{M}_+\cap (U_i\times\sph^{k-1}))\cup(\widetilde{M}_-\cap (U_i\times\sph^{k-1}))$.

By \eqref{eq:ix} it holds
$$
I(N\cap U_i)=I(N){\mathcal N}(U_i)=(f_1|_{U_i},\ldots,f_k|_{U_i}){\mathcal N}(U_i).
$$
As $\psi_i$ is a Nash diffeomorphism, $f_1\circ\psi_i,\ldots,f_k\circ\psi_i$ generate the ideal $I(\{z=0\})$ of Nash functions on $\R^d$ vanishing identically on $\psi_i^{-1}(N\cap U_i)=\{z=0\}$. The ideal $I(\{z=0\})$ is also generated by $z_{e+1},\ldots,z_d$. Thus, there exist Nash functions $\zeta_{j,\ell}\in{\mathcal N}(\R^d)$ such that
$$
f_\ell\circ\psi_i=\zeta_{e+1,\ell}(y,z)z_{e+1}+\cdots+\zeta_{d,\ell}(y,z)z_d.
$$
Notice that 
\begin{equation}\label{jc}
\frac{\partial}{\partial z_j}(f_\ell\circ\psi_i)(y,0)=\zeta_{j,\ell}(y,0)
\end{equation}
for each $y\in\R^e$. Write $\zeta_j:=(\zeta_{j,1},\ldots,\zeta_{j,k}):\R^d\to\R^k$. We have
$$
((f_1\circ\psi_i)(y,z),\ldots,(f_k\circ\psi_i)(y,z))=\zeta_{e+1}(y,z)z_{e+1}+\cdots+\zeta_d(y,z)z_d.
$$
As $f_1\circ\psi_i,\ldots,f_k\circ\psi_i$ generate the ideal $I(\{z=0\})$, we deduce:
\begin{itemize}
\item $\zeta_{e+1}(y,z)z_{e+1}+\cdots+\zeta_d(y,z)z_d=0$ if and only if $z=0$.
\item The vectors $\zeta_{e+1}(y,0),\ldots,\zeta_d(y,0)$ are linearly independent for all $y\in\R^e$ (by \eqref{jc} and the Jacobian criterion). 
\end{itemize}
By \ref{bigstepa2} and \ref{bigstepa3} we deduce that $\widetilde{M_\epsilon}\cap(U_i\times\sph^{k-1})$ is a Nash manifold with boundary and $\widehat{M}\cap(U_i\times\sph^{k-1})$ is a Nash manifold.

The rest of the statement follows from \ref{bigstepa3}. 
\end{proof}

\paragraph{}\label{bigstepb2} We keep the notations already introduced in \ref{bigstepa2g}. \em Consider the projection $\pi:M\times\sph^{k-1}\to M$ onto the first factor. Denote $\pi_\epsilon:=\pi|_{\widetilde{M_\epsilon}}$ and $\widehat{\pi}:=\pi|_{\widehat{M}}$. We have:
\begin{itemize}
\item[(i)] $\pi_\epsilon$ is proper, $\pi_\epsilon(\widetilde{M_\epsilon})=M$ and $R=\pi^{-1}_\epsilon(N)$.
\item[(ii)] The restriction $\pi_\epsilon|_{\Gamma_\epsilon}:\Gamma_\epsilon\to M\setminus N$ is a Nash diffeomorphism.
\item[(iii)] Consider the Nash map $f:=(f_1,\ldots,f_k):M\to\R^k$ (whose coordinates generate $I(N)$). Fix $q\in N$ and let $E_q$ be any complementary linear subspace of $T_qN$ in $T_qM$. Then $\pi_\epsilon^{-1}(q)=\{q\}\times\sph_q^{d-e-1}$, where $\sph_q^{d-e-1}$ denotes the sphere of dimension $d-e-1$ obtained when intersecting $\sph^{k-1}$ with the $(d-e)$-dimensional linear subspace $d_qf(E_q)$. 
 \item[(iv)] The Nash maps $\pi_\epsilon$ and $\widehat{\pi}$ have local representations of the type 
$$
x:=(x_1,\ldots,x_d)\mapsto (x_1,\ldots,x_e,x_{e+1},x_{e+1}x_{e+2},\ldots,x_{e+1}x_d)
$$
in an open neighborhood of each point $p\in R$. In addition, $d\pi_p(T_p \widehat{M})\not\subset T_{\pi(p)}N$.
\end{itemize}\em
\begin{proof}
This statement follows straightforwardly from \ref{bigstepa4} and \ref{bigstepa5}. To prove (iii) we use in addition \eqref{jc} and the equality $\ker(d_qf)=T_qN$.
\end{proof}

\paragraph{}\label{bigstepb3}\em The pairs $(\widetilde{M_\epsilon},\pi_\epsilon)$ and $(\widehat{M},\widehat{\pi})$ do not depend on the generators $f_1,\ldots,f_k$ of $I(N)$ up to Nash diffeomorphisms compatible with the respective projections. Moreover, such Nash diffeomorphisms are unique\em. 
\begin{proof}
Let $f_{k+1}:=g_1f_1+\cdots+g_kf_k\in I(N)$ for some $g_j\in{\mathcal N}(M)$. Let $(\widetilde{M_\epsilon}',\pi_\epsilon')$ and $(\widehat{M}',\widehat{\pi}')$ be the pairs associated to the system of generators $f_1,\ldots,f_k,f_{k+1}$ of $I(N)$. Let us construct a Nash diffeomorphism $\Theta:\widehat{M}\to\widehat{M}'$ such that the following diagram commutes.
\begin{equation}\label{diagunic}
\begin{gathered}
\xymatrix{
\widehat{M}\ar_{\widehat{\pi}}[dr]\ar^{\Theta}[rr]&&\widehat{M}'\ar^{\widehat{\pi}'}[dl]\\
&M&
}
\end{gathered}
\end{equation}
Denote $R':=(\widehat{\pi}')^{-1}(N)$. Consider the Nash maps 
$$
g:=(g_1,\ldots,g_k):M\to\R^k,\quad f:=(f_1,\ldots,f_k):M\to\R^k,\quad\text{and}\quad f':=(f,f_{k+1}):M\to\R^{k+1}
$$
and let $F:M\setminus N\to\sph^{k-1}$ be the Nash map introduced in \eqref{XX}.

If $(a,b)\in\widetilde{M}_\epsilon\setminus R$, then $b=\epsilon F(a)=\epsilon\frac{f(a)}{\|f(a)\|}$ and 
\begin{equation}\label{unicden}
\Big(a,\epsilon\frac{f'(a)}{\|f'(a)\|}\Big)=\Big(a,\frac{(\epsilon F(a),\qq{g(a),\epsilon F(a)})}{\sqrt{1+\qq{g(a),\epsilon F(a)}^2}}\Big)=\Big(a,\frac{(b,\qq{g(a),b)}}{\sqrt{1+\qq{g(a),b}^2}}\Big)\in\widetilde{M}_\epsilon'\setminus R'.
\end{equation}
Consider the Nash map
$$
\Theta:\widehat{M}\to\widehat{M}',\ (a,b)\mapsto\Big(a,\frac{(b,\qq{g(a),b)}}{\sqrt{1+\qq{g(a),b}^2}}\Big)
$$
and let us check that it is the Nash diffeomorphism we are looking for.

We claim: \em If $(a,c:=(c_1,\ldots,c_k,c_{k+1}))\in\widehat{M}'$, then $c':=(c_1,\ldots,c_k)\neq0$ and $c=(c',\qq{g(a),c'})$\em. 

We distinguish two cases: 

\noindent{\em Case \em 1.} If $p\not\in N$, there exist $\epsilon=\pm$ such that 
$$
c=\epsilon\frac{f'(a)}{\|f'(a)\|}=\epsilon \frac{(f(a),\qq{g(a),f(a)})}{\|f'(a)\|}=(c',\qq{g(a),c'}).
$$
\noindent{\em Case \em 2.} If $p\in N$, there exists by the Nash curve selection lemma a Nash arc $\gamma:(-1,1)\to \widehat{M}'$ such that $\gamma(0)=(a,c)$ and $\gamma((0,1))\subset\widetilde{M}_\epsilon'\setminus R'$. It holds 
\begin{equation*}
\begin{split}
c&=\epsilon\lim_{t\to0}\frac{(f'\circ\widehat{\pi}'\circ\gamma)(t)}{\|(f'\circ\widehat{\pi}'\circ\gamma)(t)\|}\\
&=\epsilon\lim_{t\to0}\frac{((f\circ\widehat{\pi}'\circ\gamma)(t),\qq{(g\circ\widehat{\pi}'\circ\gamma)(t),(f\circ\widehat{\pi}'\circ\gamma)(t)})}{\|(f'\circ\widehat{\pi}'\circ\gamma)(t)\|}=(c',\qq{g(a),c'}).
\end{split}
\end{equation*}
As in both cases $c\in\sph^k$, we have $c'\neq0$, as claimed.

The Nash map
$$
\Psi:\widehat{M}'\to\widehat{M}, (a,c)\mapsto\Big(a,\frac{c'}{\|c'\|}\Big).
$$
is the inverse of $\Theta$, so $\Theta$ is a Nash diffeomorphism. In addition, $\widehat{\pi}=\widehat{\pi}'\circ\Theta$ and $\Theta(\widetilde{M}_\epsilon)=\widetilde{M}_\epsilon'$ for $\epsilon=\pm$, as required. The unicity of $\Theta$ follows from \eqref{diagunic} and \eqref{unicden}.
\end{proof}

\begin{define}
The pair $(\widetilde{M}_+,\pi_+)$ will be called the \em drilling blow-up of the Nash manifold $M$ with center the closed Nash submanifold $N\subset M$\em. 
\end{define}
\begin{remark}
We will see in \ref{dbme} that the twisted Nash double $\widehat{M}$ of $\widetilde{M}_+$ together with $\widehat{\pi}:\widehat{M}\to M$ relates the drilling blow-up of $M$ with center $N$ with the classical blow-up of $M$ with center $N$.
\end{remark}

We are ready to prove Lemma \ref{eraseresidual}. We will make a strong use of the global structure of the drilling blow-up of a Nash manifold $M$ with center a closed Nash submanifold $N$.

\begin{proof}[Proof of Lemma \em \ref{eraseresidual}]
Let $(\widetilde{M}_+,\pi)$ be the drilling blow-up of $M$ with center $N$ and let $R:=\pi_+^{-1}(N)$, see Figure \ref{fig5}. The Nash map $\pi_+:\widetilde{M}_+\to M$ is surjective. By Proposition \ref{doubleivc} there exists a surjective Nash map 
$$
f:\Int(\widetilde{M}_+)=\widetilde{M}_+\setminus R\to\widetilde{M}_+.
$$
By \ref{bigstepb2} the restriction $\pi|_{\widetilde{M}_+\setminus R}:\widetilde{M}_+\setminus R\to M\setminus N$ is a Nash diffeomorphism. Consequently,
$$
h:=\pi_+\circ f\circ(\pi_+|_{\widetilde{M}_+\setminus R})^{-1}:M\setminus N\to\widetilde{M}_+\setminus\pi^{-1}(N)\to M
$$
is a surjective Nash map, as required.
\end{proof}

\subsection{Alternative presentation of the drilling blow-up}\setcounter{paragraph}{0}
Our purpose next is to extend the construction in diagram \eqref{diagpol} to an open semialgebraic neighborhood of the center $N$ of the drilling blow-up of $M$. Let $M\subset\R^m$ be a Nash manifold of dimension $d$ and let $N\subset M$ be a closed Nash submanifold of dimension $e$. Let $(\widehat{M},\widehat{\pi})$ be the twisted Nash double of the drilling blow-up $(\widetilde{M}_+,\pi_+)$ of $M$ with center $N$. In the next lemma we will use notations already introduced in \ref{nmtub}. We refer the reader to Figure \ref{fig4} to appreciate its importance.

\begin{lem}\label{neighfeten}
Let $U_1\subset M$ and $U_2\subset\widehat{M}$ be respective open semialgebraic neighborhoods of $N$ and $R:=\widehat{\pi}^{-1}(N)$. Then there exist Nash tubular neighborhoods $(V_1,\varphi_1,\mathscr{E},\theta_1,N,\delta)$ of $N$ in $U_1$ and $(V_2,\varphi_2,\mathscr{F},\theta_2,R,\delta\circ\widehat{\pi}|_R)$ of $R$ in $U_2$ and a Nash embedding $\psi:=(\psi_1,\psi_2):R\to N\times\sph^{m-1},\ (x,v)\mapsto(x,\psi_2(x,v))$ such that $\widehat{\pi}(V_2)=V_1$,
$$
\varphi_1\circ\widehat{\pi}\circ \varphi_2^{-1}:\mathscr{F}_{\delta\circ\widehat{\pi}|_R}\to\mathscr{E}_{\delta},\ (x,v,t)\mapsto(x,t\psi_2(x,v))
$$
and $V_2$ is the Nash double of a collar of $R$ in $\widetilde{M}_+$.
\end{lem}
\begin{proof}
By \ref{nmtub} there exists a Nash tubular neighborhood $(V,\varphi,\mathscr{E},\theta,N,\delta)$ of $N$ in $U_1$. Recall that $(\mathscr{E},\theta,N)$ is a Nash subbundle of the trivial bundle $(N\times\R^m,\theta',N)$. Write $\varphi:=(\phi_1,\phi_2):V\to N\times\R^m$ and observe that $\phi_1|_N=\id_N$ and $\phi_2|_N=0$. 

\paragraph{} For each $x\in N$ we have $d_x\varphi=(d_x\phi_1,d_x\phi_2):T_xM\to T_xN\times\R^m$ and $d_x\phi_1|_{T_xN}=(\id_{T_xN},0)$. Let $H_x$ be a complementary linear subspace of $T_xN$ in $T_xM$. As $d_x\varphi$ is an isomorphism, $d_x\phi_1|_{T_xN}=\id_{T_xN}$ and $d_x\phi_1(H_x)\subset T_xN$, we have
\begin{multline}\label{te}
T_{(x,0)}\mathscr{E}=d_x\varphi(T_xM)=d_x\varphi(T_xN\oplus H_x)=(T_xN\times\{0\})\oplus(d_x\phi_1,d_x\phi_2)(H_x)\\
=(T_xN\times\{0\})\oplus(\{0\}\times d_x\phi_2(H_x))=T_xN\times d_x\phi_2(H_x),
\end{multline}
so $\dim(d_x\phi_2(H_x))=d-e$. By the Jacobian criterion the tuple $\phi_2:=(\phi_{21},\ldots,\phi_{2m})$ generates $I(N_x)$ for each $x\in N$, that is, $(\phi_2){\mathcal N}_x:=(\phi_{21},\ldots,\phi_{2m}){\mathcal N}_x=I(N_x)$ for each $x\in N$. Write $I(N):=\{f\in{\mathcal N}(M):\ f|_N=0\}$ and $I_V(N):=\{f\in{\mathcal N}(V):\ f|_N=0\}$. By \eqref{eq:ix} and \ref{division}
$$
(\phi_2){\mathcal N}(V):=(\phi_{21},\ldots,\phi_{2m}){\mathcal N}(V)=I_V(N)=I(N){\mathcal N}(V). 
$$
Let $(\widehat{V},\widehat{\pi}^*)$ be the twisted Nash double of the drilling blow-up of $V$ with center $N$ that arises from a finite system of generators of $I(N)$ restricted to $V$ and let $(\widehat{V}',\widehat{\pi}')$ be the twisted Nash double of the drilling blow-up of $V$ centered at $N$ that arises from $\phi_2$. It holds $\widehat{V}=\widehat{\pi}^{-1}(V)$ and $\widehat{\pi}^*=\widehat{\pi}|_{\widehat{V}}$ if $(\widehat{M},\widehat{\pi})$ is the twisted Nash double of the drilling blow-up of $M$ with center $N$. By \ref{bigstepb3} there exists a Nash diffeomorphism $\Theta:\widehat{V}\to\widehat{V}'$ that makes the following diagram commutative.
$$
\xymatrix{
\widehat{V}\ar_{\widehat{\pi}^*}[dr]\ar^{\Theta}[rr]&&\widehat{V}'\ar^{\widehat{\pi}'}[dl]\\
&V&
}
$$

\paragraph{} Denote the coordinates in $\R^m$ with $\y:=(\y_1,\ldots,\y_m)$. The restriction to $\mathscr{E}_\delta$ of the tuple $(\y_1,\ldots,\y_m)$ generates the ideal $I(N\times\{0\})$ of ${\mathcal N}(\mathscr{E}_\delta)$. Let $(\widehat{\mathscr{E}}_\delta,\widehat{\Pi})$ be the twisted Nash double of the drilling blow-up $(\widetilde{\mathscr{E}}_{\delta,+},\Pi_+)$ of $\mathscr{E}_\delta$ with center $N\times\{0\}$ that arises from the tuple $(\y_1,\ldots,\y_m)$. It holds
$$
\widehat{\mathscr{E}}_\delta:=\cl(\{(x,y,\pm\tfrac{y}{\|y\|})\in\mathscr{E}_\delta\times\sph^{m-1}:\ y\neq0\}).
$$
By \ref{bigstepb2}(iii) $\widehat{\Pi}^{-1}(x,0)=\{(x,0)\}\times\sph_{(x,0)}$ where $\sph_{(x,0)}$ is the intersection of $\sph^{m-1}$ with the linear subspace $\theta^{-1}(x)$ of $\R^m$. Consequently,
\begin{align}
\widehat{\mathscr{E}}_\delta\setminus\widehat{\Pi}^{-1}(N\times\{0\})&=\{(x,y,\pm\tfrac{y}{\|y\|})\in\mathscr{E}_\delta\times\sph^{m-1}:\ y\neq0\},\label{parte1}\\
\widehat{\Pi}^{-1}(N\times\{0\})&=\{(x,0,\pm\tfrac{y}{\|y\|})\in\mathscr{E}_\delta\times\sph^{m-1}:\ (x,y)\in \mathscr{E}_\delta,\ y\neq0\}.\label{parte2}
\end{align}

\paragraph{} It holds $\widehat{V}'\subset V\times\sph^{m-1}$ and $\widehat{\mathscr{E}}_\delta\subset\mathscr{E}_\delta\times\sph^{m-1}$. We claim: \em The image of the Nash map
$$
\Lambda:\widehat{V}'\to\mathscr{E}_\delta\times\sph^{m-1},\ (a,b)\mapsto(\varphi(a),b)
$$
is $\widehat{\mathscr{E}}_\delta$. The image of the Nash map
$$
\Delta:\widehat{\mathscr{E}}_\delta\to V\times\sph^{m-1},\ (x,y,w)\mapsto(\varphi^{-1}(x,y),w).
$$
is $\widehat{V}'$\em. Thus, \em $\Delta:\widehat{\mathscr{E}}_\delta\to{\widehat V}'$ and $\Lambda:{\widehat V}'\to\widehat{\mathscr{E}}_\delta$ are mutually inverse Nash diffeomorphisms.
\em

Pick $(a,b)\in\widehat{V}'\setminus(\widehat{\pi}')^{-1}(N)$. We have 
$$
(a,b)=\Big(a,\pm\frac{\phi_2(a)}{\|\phi_2(a)\|}\Big).
$$
Consequently, 
$$
\Lambda(a,b)=\Big(\varphi(a),\pm\frac{\phi_2(a)}{\|\phi_2(a)\|}\Big)=\Big(\phi_1(a),\phi_2(a),\pm\frac{\phi_2(a)}{\|\phi_2(a)\|}\Big)\in\widehat{\mathscr{E}}_\delta.
$$ 
By continuity $\Lambda(\widehat{V}')\subset\widehat{\mathscr{E}}_\delta$.

Pick now $(x,y,w)\in\widehat{\mathscr{E}}_\delta\setminus\widehat{\Pi}^{-1}(N\times\{0\})$. We have $w=\pm\frac{y}{\|y\|}=\pm\frac{\phi_2(\varphi^{-1}(x,y))}{\|\phi_2(\varphi^{-1}(x,y))\|}$ and 
$$
\Delta(x,y,w)=\Big(\varphi^{-1}(x,y),\pm\frac{\phi_2(\varphi^{-1}(x,y))}{\|\phi_2(\varphi^{-1}(x,y))\|}\Big)\in\widehat{V}'.
$$
By continuity $\Delta(\widehat{\mathscr{E}}_\delta)\subset\widehat{V}'$.

\paragraph{} The maps in the rows of the following commutative diagram are Nash diffeomorphisms:
\begin{equation}\label{diagrama}
\begin{gathered}
\xymatrix{\widehat{V}\ar[rr]^{\Theta}\ar[d]_{\widehat{\pi}^*}&&\widehat{V}'\ar[rr]^{\Lambda}\ar[d]_{\widehat{\pi}'}&&\widehat{\mathscr{E}}_\delta\ar[d]_{\widehat{\Pi}}\\
V\ar[rr]^{\id_V}&&V\ar[rr]^{\varphi}&&\mathscr{E}_\delta
}
\end{gathered}
\end{equation}

\paragraph{} It holds: \em The Nash maps
\begin{align*}
&\rho:\widehat{\mathscr{E}}_\delta\to\widehat{\Pi}^{-1}(N\times\{0\}),\ (x,y,w)\mapsto(x,0,w),\\
&\varrho:\mathscr{E}_\delta\to N\times\{0\},\ (x,y)\mapsto(x,0)
\end{align*}
are Nash retractions such that $\widehat{\Pi}\circ\rho=\varrho\circ\widehat{\Pi}$\em.

\paragraph{} Define 
$$
h:\widehat{\mathscr{E}}_\delta\to\R,\ (x,y,w)\mapsto\begin{cases}
+\|y\|&\text{if $(x,y,w)\in\widetilde{\mathscr{E}}_{\delta,+}$,}\\
-\|y\|&\text{if $(x,y,w)\in\widetilde{\mathscr{E}}_{\delta,-}$.}
\end{cases}
$$
We claim: \em The semialgebraic function $h$ is Nash and $d_qh:T_q\widehat{\mathscr{E}}_\delta\to\R$ is surjective for all $q\in\widehat{\Pi}^{-1}(N\times\{0\})$\em. In addition, \em $|h(x,y,w)|=\|y\|$ for all $(x,y,w)\in\widehat{\mathscr{E}}_\delta$, so $\{h=0\}=\widehat{\Pi}^{-1}(N\times\{0\})$\em.

Pick a point $q\in\widehat{\Pi}^{-1}(N\times\{0\})$. By \ref{bigstepb2}(iv) there exist semialgebraic neighborhoods $A_1\subset\widehat{\mathscr{E}}_\delta$ of $q$ and $A_2\subset\mathscr{E}_\delta$ of $\widehat{\Pi}(q)$ and Nash diffeomorphisms 
\begin{align*}
&u:=(u_1,\ldots,u_d):A_1\to\R^d,\\
&v:=(v_1,\ldots,v_d):A_2\to\R^d 
\end{align*}
such that $u(q)=0$, $v(\widehat{\Pi}(q))=0$, $v((N\times\{0\})\cap A_2)=\{v_{e+1}=0,\ldots,v_d=0\}$ and 
\begin{multline*}
\widehat{\Pi}_0:=v\circ\widehat{\Pi}\circ u^{-1}:\R^d\to\R^d,\\ 
(x_1,\ldots,x_d)\mapsto(z_1,\ldots,z_d):=(x_1,\ldots,x_e,x_{e+1},x_{e+2}x_{e+1},\ldots,x_dx_{e+1}).
\end{multline*}
As ${\tt y}_1,\dots,{\tt y}_m$ generate the ideal $I(N\times\{0\})$ of ${\mathcal N}(\Ee_{\delta})$, their restrictions to $(N\times\{0\})\cap A_2$ generate by \eqref{eq:ix} the ideal $I((N\times\{0\})\cap A_2)$ of ${\mathcal N}(A_2)$. We have $\y_i(v^{-1}(z_1,\ldots,z_e,0,\ldots,0))=0$ for $(z_1,\ldots,z_e)\in\R^e\times\{0\}$ and $i=1,\ldots,m$ and
$$
z_{e+1}=(\y_1\circ v^{-1})\xi_1+\cdots+(\y_m\circ v^{-1})\xi_m
$$
for some $\xi_1,\ldots,\xi_m\in{\mathcal N}(\R^d)$. By Schwarz's inequality
$$
z_{e+1}^2\leq(\xi_1^2+\cdots+\xi_m^2)((\y_1\circ v^{-1})^2+\cdots+(\y_m\circ v^{-1})^2).
$$
Composing with $\widehat{\Pi}_0$, we have
$$
x_{e+1}^2\leq((\xi_1\circ\widehat{\Pi}_0)^2+\cdots+(\xi_m\circ\widehat{\Pi}_0)^2)((\y_1\circ v^{-1}\circ\widehat{\Pi}_0)^2+\cdots+(\y_m\circ v^{-1}\circ\widehat{\Pi}_0)^2).
$$
Comparing orders at the origin we deduce that the Nash series $(\y_1\circ v^{-1}\circ\widehat{\Pi}_0)^2+\cdots+(\y_m\circ v^{-1}\circ\widehat{\Pi}_0)^2$ has order $2$ at the origin. As $\y_i(v^{-1}(z_1,\ldots,z_e,0,\ldots,0))=0$ for each $(z_1,\ldots,z_e)\in\R^e\times\{0\}$,
$$
\y_i\circ v^{-1}\circ\widehat{\Pi}_0=x_{e+1}\gamma_i
$$
where $\gamma_i$ is a Nash function on $\R^d$. Thus,
$$
(\y_1\circ v^{-1}\circ\widehat{\Pi}_0)^2+\cdots+(\y_m\circ v^{-1}\circ\widehat{\Pi}_0)^2=x_{e+1}^2(\gamma_1^2+\cdots+\gamma_m^2).
$$
As $(\y_1\circ v^{-1}\circ\widehat{\Pi}_0)^2+\cdots+(\y_m\circ v^{-1}\circ\widehat{\Pi}_0)^2$ is a Nash series of order $2$ at the origin, $\gamma_1^2+\cdots+\gamma_m^2$ is a unit at the origin. Consequently, $h\circ u^{-1}=x_{e+1}\sqrt{\gamma_1^2+\cdots+\gamma_m^2}$ on an open semialgebraic neighborhood of the origin. In particular,
$$
\frac{\partial(h\circ u^{-1})}{\partial x_{e+1}}(0)=\sqrt{\gamma_1^2+\cdots+\gamma_m^2}(0)>0.
$$
Thus, $h$ is a Nash function and $d_qh:T_q\widehat{M}\to\R$ is surjective for all $q\in\widehat{\Pi}^{-1}(N\times\{0\})$.

\paragraph{} Consider the trivial Nash bundle $\mathscr{G}:=\widehat{\Pi}^{-1}(N\times\{0\})\times\R$ over $\widehat{\Pi}^{-1}(N\times\{0\})$ and let $\theta_2':\mathscr{G}\to\widehat{\Pi}^{-1}(N\times\{0\})$ be the projection onto the first factor. Define
$$
g:=(\rho,h):\widehat{\mathscr{E}}_\delta\to\mathscr{G},\ (x,y,w)\mapsto\begin{cases}
(x,0,w,+\|y\|)&\text{if $(x,y,w)\in\widetilde{\mathscr{E}}_{\delta,+}$,}\\
(x,0,w,-\|y\|)&\text{if $(x,y,w)\in\widetilde{\mathscr{E}}_{\delta,-}$.}
\end{cases}
$$
By Lemma \ref{neighcollar} there exists an open semialgebraic neighborhood $W\subset\widehat{\mathscr{E}}_\delta$ of $\widehat{\Pi}^{-1}(N\times\{0\})$ such that $g(W)\subset\mathscr{G}$ is an open semialgebraic neighborhood of $\widehat{\Pi}^{-1}(N\times\{0\})\times\{0\}$ and the restriction map $g|_W:W\to g(W)$ is a Nash diffeomorphism. 

As $\widehat{\Pi}$ is proper, $\widehat{\Pi}(\widehat{\mathscr{E}}_\delta\setminus W)$ is a closed subset of $\mathscr{E}_\delta$ that does not meet $N\times\{0\}$. As the restriction $\widehat{\Pi}|_{\widehat{\mathscr{E}}_\delta\setminus\widehat{\Pi}^{-1}(N\times\{0\})}$ is a Nash diffeomorphism, $\widehat{\Pi}(\widehat{\mathscr{E}}_\delta\setminus W)=\mathscr{E}_\delta\setminus\widehat{\Pi}(W)$. Consequently, $W':=\widehat{\Pi}(W)$ is an open semialgebraic neighborhood of $N\times\{0\}$ in $\mathscr{E}_\delta$.

\paragraph{} We may assume taking a smaller strictly positive $\delta\in{\mathcal N}(N)$ that $\mathscr{E}_\delta=W'$ and $W=\widehat{\Pi}^{-1}(\mathscr{E}_\delta)=\widehat{\mathscr{E}}_\delta$. Define $\delta':=\delta\circ\theta|_{N\times\{0\}}\in{\mathcal N}(N\times\{0\})$. We claim: 
$$
g(\widehat{\mathscr{E}}_\delta)=\mathscr{G}_{\delta'\circ\widehat{\Pi}}:=\{(x,0,w,t)\in\mathscr{G}:\ |t|<(\delta'\circ\widehat{\Pi})(x,0,w)\}. 
$$
Consequently, \em $\widehat{\mathscr{E}}_\delta$ is a Nash tubular neighborhood of $\widehat{\Pi}^{-1}(N\times\{0\})$ and $g:\widehat{\mathscr{E}}_\delta\to\mathscr{G}_{\delta'\circ\widehat{\Pi}}$ is a Nash diffeomorphism\em.

\clearpage
We have 
\begin{equation*}
\begin{split}
\widehat{\mathscr{E}}_\delta&=\widehat{\Pi}^{-1}(\mathscr{E}_\delta)=\{(x,y,w)\in\widehat{\mathscr{E}}_\delta:\ \|y\|<\delta(x)\}\\
&=\{(x,y,w)\in\widehat{\mathscr{E}}_\delta:\ |h(x,y,w)|<(\delta'\circ\widehat{\Pi})(x,0,w)\}\\
&=g^{-1}(\{(x,0,w,t)\in\mathscr{G}:\ |t|<(\delta'\circ\widehat{\Pi})(x,0,w)\})=g^{-1}(\mathscr{G}_{\delta'\circ\widehat{\Pi}}),
\end{split}
\end{equation*}
as claimed.

\paragraph{} The composition $\widehat{\Pi}\circ g^{-1}$ is a Nash map that satisfies
\begin{equation}
\widehat{\Pi}\circ g^{-1}:g(\widehat{\mathscr{E}}_\delta)\to\mathscr{E}_\delta,\ (x,0,w,t)\mapsto(x,tw).
\end{equation}

\paragraph{}\label{5c10} Consider the trivial Nash vector bundle $\mathscr{F}:=R\times\R$ over $R:=\widehat{\pi}^{-1}(N)$ and let $\theta_2:R\times\R\to R$ be the projection onto the first factor. By \eqref{parte2} and \eqref{diagrama} the Nash map
$$
\psi_0:=(\Lambda\circ\Theta)|_R:R\subset\widehat{V}\to\widehat{\Pi}^{-1}(N\times\{0\})\subset\widehat{\mathscr{E}}_\delta
$$
is a Nash diffeomorphism and there exists a Nash map $\psi_2:R\to\sph^{m-1}$ such that $\psi_0(x,v)=(x,0,\psi_2(x,v))$ for each $(x,v)\in R$. In addition, $\psi_0$ induces a Nash isomorphism of Nash vector bundles 
$$
\Psi:\mathscr{F}\to\mathscr{G}, (x,v,t)\mapsto(x,0,\psi_2(x,v),t)
$$
such that the following diagram is commutative.
$$
\xymatrix{
\mathscr{F}\ar[rr]^{\Psi}_{\cong}\ar[d]_{\theta_2}&&\mathscr{G}\ar[d]_{\theta_2'}\\
R\ar[rr]^(0.4){\psi_0}_(0.4){\cong}&&\widehat{\Pi}^{-1}(N\times\{0\})
}
$$
By \eqref{diagrama} $\widehat{\Pi}\circ\psi_0=\varphi\circ\widehat{\pi}|_R$, so $\delta'\circ\widehat{\Pi}\circ\psi_0=\delta\circ\theta\circ\varphi\circ\widehat{\pi}|_R=\delta\circ\widehat{\pi}|_R$ because $\theta\circ\varphi|_N=\id_N$. Consequently, $\Psi(\mathscr{F}_{\delta\circ\widehat{\pi}|_R})=\mathscr{G}_{\delta'\circ\widehat{\Pi}}$.

\paragraph{} Define $\varphi_1:=\varphi$ and $\varphi_2:=\Psi^{-1}\circ g\circ\Lambda\circ\Theta$. We have the following commutative diagram:
$$
\xymatrix{
\widehat{\pi}^{-1}(V)\ar@{=}[r]\ar[d]_{\widehat{\pi}|_{\widehat{V}}}
\ar@/^2.5pc/[rrrrr]^{\varphi_2}_{\cong}
&\widehat{V}\ar[r]^{\Theta}_{\cong}\ar[d]_{\widehat{\pi}^*}&\widehat{V}'\ar[r]^{\Lambda}_{\cong}\ar[d]_{\widehat{\pi}'}&\widehat{\mathscr{E}}_\delta\ar[d]_{\widehat{\Pi}}\ar[r]^(0.35){g}_(0.35){\cong}&\mathscr{G}_{\delta'\circ\widehat{\Pi}}&\mathscr{F}_{\delta\circ\widehat{\pi}|_R}\ar[l]_{\Psi}^{\cong}\ar[lld]^{\quad\qquad\widehat{\Pi}\circ g^{-1}\circ\Psi=\varphi_1\circ\widehat{\pi}\circ\varphi_2^{-1}}&(x,v,t)\ar@{|->}[d]^{\varphi_1\circ\widehat{\pi}\circ\varphi_2^{-1}}\\
V\ar@{=}[r]^{\id_V}\ar@/_1.5pc/[rrr]_{\varphi_1}^{\cong}&V\ar@{=}[r]^{\id_V}&V\ar[r]^(0.4){\varphi}_(0.4){\cong}&\mathscr{E}_\delta&&&(x,t\psi_2(x,v))
}
$$
Take $V_1:=V$, $\theta_1:=\theta$, $V_2:=\widehat{\pi}^{-1}(V)$. The Nash tubular neighborhoods $(V_1,\varphi_1,\mathscr{E},\theta_1,N,\delta)$ of $N$ in $U_1$ and $(V_2,\varphi_2,\mathscr{F},\theta_2,R,\delta\circ\widehat{\pi}|_R)$ of $R$ in $U_2$ and the Nash embedding 
$$
\psi:R\to N\times\sph^{m-1},\ (x,v)\mapsto(x,\psi_2(x,v))
$$
satisfy the required conditions.
\end{proof}

The following result justifies the first part of the name of the drilling blow-up $(\widetilde{M}_+,\pi_+)$, see also Figures \ref{fig4} and \ref{fig5}.

\begin{cor}[Alternative description of the drilling blow-up]\label{altdb}
Let $M\subset \R^m$ be a Nash manifold, let $N\subset M$ be a closed Nash submanifold and let $U$ be an open semialgebraic neighborhood of $N$ in $M$. Then there exist a Nash tubular neighborhood $(V,\varphi,\mathscr{E},\theta,N,\delta)$ of $N$ in $U$ such that $M\setminus V$ is a Nash manifold with boundary $\partial V$ and a Nash diffeomorphism $g:M\setminus V\to\widetilde{M}_+$.
\end{cor}
\begin{proof}
By Lemma \ref{neighfeten} there exist Nash tubular neighborhoods $(V_1,\varphi_1,\mathscr{E},\theta_1,N,\delta_1)$ of $N$ in $U$ and $(V_2,\varphi_2,\mathscr{F},\theta_2,R,\veps_1:=\delta_1\circ\widehat{\pi}|_R)$ of $R:=\widehat{\pi}^{-1}(N)$ in $\widehat{M}$ and a Nash embedding $\psi:R\to N\times\sph^{m-1},\ (x,v)\mapsto(x,\psi_2(x,v))$ such that
$$
\varphi_1\circ\widehat{\pi}\circ \varphi_2^{-1}:\mathscr{F}_{\veps_1}\to\mathscr{E}_{\delta_1},\ (x,v,t)\mapsto(x,t\psi_2(x,v))
$$
and $V_2\subset\widehat{M}$ is the Nash double of a collar of $R$ in $\widetilde{M}_+$. By \ref{5c10} we may write $\mathscr{F}:=R\times\R$ and $\theta_2:R\times\R\to R$ is the projection onto the first factor. Recall that $\pi_+|_{\widetilde{M}_+\setminus R}:\widetilde{M}_+\setminus R\to M\setminus N$ is a Nash diffeomorphism. Define $\delta:=\frac{\delta_1}{4}$, $V:=\varphi_1^{-1}(\mathscr{E}_{\delta})$, $\varphi:=\varphi_1|_V$ and $\theta:=\theta_1$. We claim: \em $(V,\varphi,\mathscr{E},\theta,N,\delta)$ is the Nash tubular neighborhood we sought\em.

Notice that $\pi_+^{-1}(V)=\varphi_2^{-1}(\mathscr{F}_{\veps})\cap\widetilde{M}_+$ where $\veps:=\delta\circ\widehat{\pi}|_R$. Define $M^{\bullet}:=M\setminus V$ and $\widetilde{M}_+^{\bullet}:=\pi_+^{-1}(M^{\bullet})=\widetilde{M}_+\setminus\varphi_2^{-1}(\mathscr{F}_{\veps})$. It holds $\pi_+|_{\widetilde{M}_+^{\bullet}}:\widetilde{M}_+^{\bullet}\to M^{\bullet}$ is a Nash diffeomorphism.

Denote $W:=\{(z,t)\in\mathscr{F}:\ 0\leq t<\veps_1(z)\}=\varphi_2(\varphi_2^{-1}(\mathscr{F}_{\veps_1})\cap\widetilde{M}_+)$ and consider the Nash diffeomorphism
$$
\Lambda:W\to R\times[0,1),\ (z,t)\to\Big(z,\frac{t}{\veps_1(z)}\Big).
$$
As $\frac{\veps}{\veps_1}=\frac{1}{4}$, we have $(\Lambda\circ\varphi_2)(\widetilde{M}_+^{\bullet}\cap\varphi_2^{-1}(W))=R\times[\frac{1}{4},1)$. Let $f_1:[\frac{1}{4},1)\to[0,1)$ be an ${\mathcal S}^2$ diffeomorphism such that $f_1|_{[\frac{3}{4},1)}=\id_{[\frac{3}{4},1)}$ (see Example \ref{s2diffeo}(i)). Consider the ${\mathcal S}^2$ diffeomorphism 
$$
F_1:R\times[\tfrac{1}{4},1)\to R\times[0,1),\ (z,t)\mapsto(z,f_1(t)).
$$
It holds $F_1|_{R\times[\frac{3}{4},1)}=\id_{R\times[\frac{3}{4},1)}$. Denote again $\varphi_2$ the restriction of this Nash map to $\varphi_2^{-1}(W)$ and define
$$
\Phi:\widetilde{M}_+^{\bullet}\to\widetilde{M}_+,\ x\mapsto\begin{cases}
x&\text{if $x\in\widetilde{M}_+^{\bullet}\setminus\varphi_2^{-1}(W)$},\\
(\Lambda\circ\varphi_2)^{-1}(F_1((\Lambda\circ\varphi_2)(x)))&\text{if $x\in\widetilde{M}_+^{\bullet}\cap\varphi_2^{-1}(W)$,}
\end{cases}
$$
which is an ${\mathcal S}^2$ diffeomorphisms. The restriction $\Phi|_{\pi_+^{-1}(\partial V)}:\pi_+^{-1}(\partial V)=\partial\widetilde{M}_+^{\bullet}\to\pi_+^{-1}(N)=\partial\widetilde{M}_+$ is a Nash diffeomorphism. By \ref{s2bounb} there exists a Nash diffeomorphism $\Phi':\widetilde{M}_+^{\bullet}\to\widetilde{M}_+$ such that $\Phi'|_{\partial\widetilde{M}_+^{\bullet}}=\Phi|_{\partial\widetilde{M}_+^{\bullet}}$. 

The composition $g:=\Phi'\circ(\pi_+|_{\widetilde{M}_+^{\bullet}}^{-1}):M^{\bullet}\to\widetilde{M}_+$ is a Nash diffeomorphism, so $M^{\bullet}$ is a Nash manifold with boundary $\partial V=g^{-1}(\partial\widetilde{M}_+)$, as required.
\end{proof}

\begin{figure}[!ht]
\begin{center}
\begin{tikzpicture}[scale=0.75]

\draw[fill=gray!20,opacity=0.5,draw=none] (2.75,11.25) ellipse (0.75cm and 0.25cm);
\draw (2.75,11.25) ellipse (0.75cm and 0.25cm);

\draw[fill=gray!60,opacity=0.5,draw=none] (2,9.5) arc (180:360:0.75cm and 0.25cm) -- (3.5,11.25) arc (0:-180:0.75cm and 0.25cm) -- (2,9.5);

\draw[fill=gray!100,opacity=0.75,draw=none] (2,8.75) arc (180:360:0.75cm and 0.25cm) -- (3.5,9.5) arc (0:-180:0.75cm and 0.25cm) -- (2,8.75);

\draw[fill=gray!20,opacity=0.5,draw=none] (2,6.25) arc (180:360:0.75cm and 0.25cm) -- (3.5,8.75) arc (0:-180:0.75cm and 0.25cm) -- (2,6.25);

\draw[line width=0.75pt] (2,9.5) arc (180:360:0.75cm and 0.25cm);
\draw[line width=0.75pt, dashed] (3.5,9.5) arc (0:180:0.75cm and 0.25cm);

\draw[line width=1.5pt] (2,8.75) arc (180:360:0.75cm and 0.25cm);
\draw[line width=1.5pt,dashed] (3.5,8.75) arc (0:180:0.75cm and 0.25cm);

\draw (2,6.25) arc (180:360:0.75cm and 0.25cm);
\draw[dashed] (3.5,6.25) arc (0:180:0.75cm and 0.25cm);

\draw (2,6.25) -- (2,11.25);
\draw (3.5,6.25) -- (3.5,11.25);

\draw[fill=gray!60,opacity=0.5,draw=none] (1.5,1.5) -- (4,1.5) -- (4,4) -- (1.5,4) -- (1.5,1.5);
\draw (1.5,1.5) -- (4,1.5) -- (4,4) -- (1.5,4) -- (1.5,1.5);

\draw[fill=gray!100,opacity=0.75,draw=none] (2.75,2.75) ellipse (0.75cm and 0.75cm);
\draw (2.75,2.75) ellipse (0.75cm and 0.75cm);

\draw[fill=gray!20,opacity=0.5,draw=none] (8.5,11) ellipse (2.5cm and 0.5cm);

\draw (8.5,11) ellipse (2.5cm and 0.5cm);

\draw[fill=gray!10,opacity=0.4,draw=none,rotate=-90] (-6.5,6) arc (270:450:0.5cm and 2.5cm) parabola bend (-8.75,9.5) (-8.75,9.5) arc (90:-90:0.25cm and 1cm) parabola bend (-8.75,7.5) (-6.64,6.1);

\draw[fill=gray!20,opacity=0.5,draw=none,rotate=-90] (-7.5,7.025) arc (270:450:0.3cm and 1.48cm) parabola bend (-8.75,9.5) (-8.75,9.5) arc (90:-90:0.25cm and 1cm) parabola bend (-8.75,7.5) (-7.5,7.025);

\draw[fill=gray!60,opacity=0.3,draw=none,rotate=-90] (-8.75,7.5) arc (270:450:0.25cm and 1cm) parabola bend (-8.75,9.5) (-10.86,10.9) -- (-11,11) arc (90:-90:0.5cm and 2.5cm) -- (-10.86,6.1) parabola bend (-8.75,7.5) (-8.75,7.5);

\draw[fill=gray!60,opacity=0.5,draw=none,rotate=-90] (-8.75,7.5) arc (270:450:0.25cm and 1cm) parabola bend (-8.75,9.5) (-10,9.975) arc (90:-90:0.3cm and 1.48cm) parabola bend (-8.75,7.5) (-8.75,7.5);

\draw[fill=black!60,opacity=0.4,draw=none,rotate=-90] (-8.75,7.5) arc (270:450:0.25cm and 1cm) parabola bend (-8.75,9.5) (-9.375,9.61) arc (90:-90:0.25cm and 1.11cm) parabola bend (-8.75,7.5) (-8.75,7.5);

\draw[fill=gray!100,opacity=0.75,draw=none,rotate=-90] (-8.75,7.5) arc (270:450:0.25cm and 1cm) parabola bend (-8.75,9.5) (-9,9.51) arc (90:-90:0.25cm and 1.01cm) -- (-9,7.49) parabola bend (-8.75,7.5) (-8.75,7.5);

\draw[line width=0.75pt] (7.49,9) arc (180:360:1.01cm and 0.25cm);
\draw[line width=0.75pt,dashed] (9.51,9) arc (0:180:1.01cm and 0.25cm);

\draw (7.39,9.375) arc (180:360:1.11cm and 0.25cm);
\draw[dashed] (9.61,9.375) arc (0:180:1.1cm and 0.25cm);

\draw (7.025,7.5) arc (180:360:1.48cm and 0.3cm);
\draw[dashed] (9.975,7.5) arc (0:180:1.48cm and 0.3cm);

\draw (7.025,10) arc (180:360:1.48cm and 0.3cm);
\draw[dashed] (9.975,10) arc (0:180:1.48cm and 0.3cm);

\draw (6,6.5) arc (180:360:2.5cm and 0.5cm);
\draw[dashed] (11,6.5) arc (0:180:2.5cm and 0.5cm);

\draw[line width=1.5pt] (7.5,8.75) arc (180:360:1cm and 0.25cm);
\draw[line width=1.5pt,dashed] (9.5,8.75) arc (0:180:1cm and 0.25cm);

\draw[line width=0.5pt,rotate=-90] (-6.64,6.1) parabola bend (-8.75,7.5) (-10.86,6.1);
\draw[line width=0.5pt,rotate=-90] (-6.64,10.9) parabola bend (-8.75,9.5) (-10.86,10.9);

\draw[fill=gray!20,opacity=0.5,draw=none] (13.5,10) ellipse (1cm and 0.25cm);
\draw (13.5,10) ellipse (1cm and 0.25cm);

\draw[fill=gray!20,opacity=0.5,draw=none] (12.5,7.5) arc (180:360:1cm and 0.25cm) -- (14.5,8.75) arc (0:-180:1cm and 0.25cm) -- (12.5,7.5);
\draw[fill=gray!20,opacity=0.5,draw=none] (12.5,8.75) arc (180:360:1cm and 0.25cm) -- (14.5,10) arc (0:-180:1cm and 0.25cm) -- (12.5,8.75);
\draw[fill=black!60,opacity=0.4,draw=none] (12.5,8.75) arc (180:360:1cm and 0.25cm) -- (14.5,9.375) arc (0:-180:1cm and 0.25cm) -- (12.5,8.75);

\draw (12.5,9.375) arc (180:360:1cm and 0.25cm);
\draw[dashed] (14.5,9.375) arc (0:180:1cm and 0.25cm);
\draw[line width=1.5pt] (12.5,8.75) arc (180:360:1cm and 0.25cm);
\draw[line width=1.5pt,dashed] (14.5,8.75) arc (0:180:1cm and 0.25cm);
\draw (12.5,7.5) arc (180:360:1cm and 0.25cm);
\draw[dashed] (14.5,7.5) arc (0:180:1cm and 0.25cm);

\draw (12.5,7.5) -- (12.5,10);
\draw (14.5,7.5) -- (14.5,10);

\draw[fill=gray!100,opacity=0.4,draw=none] (18.5,8.75) ellipse (2cm and 1cm);
\draw (18.5,8.75) ellipse (2cm and 1cm);
\draw[fill=gray!100,opacity=0.4,draw=none] (18.5,8.75) ellipse (1.5cm and 0.75cm);
\draw(18.5,8.75) ellipse (1.5cm and 0.75cm);
\draw[fill=gray!100,opacity=0.75,draw=none] (18.5,8.75) ellipse (1cm and 0.5cm);
\draw[line width=1.5pt] (18.5,8.75) ellipse (1cm and 0.5cm);
\draw (18.5,8.75) node{\large$\bullet$}; 

\draw[fill=gray!60,opacity=0.3,draw=none] (6.5,4) .. controls (7,3.5) and (8,3.5) .. (9,3.75) .. controls (10,4) and (10.5,4.25) .. (11.5,4) .. controls (10.4,3) and (10.3,2) .. (10.5,0.75) .. controls (9.5,0.5) and (9,0.5) .. (8,0.75) .. controls (7.5,0.9) and (6.5,1.1) .. (6,1) .. controls (6,2.5) and (6.2,3.5) .. (6.5,4);
\draw (6.5,4) .. controls (7,3.5) and (8,3.5) .. (9,3.75) .. controls (10,4) and (10.5,4.25) .. (11.5,4) .. controls (10.4,3) and (10.3,2) .. (10.5,0.75) .. controls (9.5,0.5) and (9,0.5) .. (8,0.75) .. controls (7.5,0.9) and (6.5,1.1) .. (6,1) .. controls (6,2.5) and (6.2,3.5) .. (6.5,4);

\draw[fill=gray!60,opacity=0.5,draw=none] (7.25,3) .. controls (7,3) and (6.9,2) .. (7,1.5) .. controls (7.5,1.5) and (8.5,1.4) .. (9.5,1.1) .. controls (9.5,2) and (9.7,3) .. (10,3.5) .. controls (9,3) and (8,3) .. (7.25,3);
\draw (7.25,3) .. controls (7,3) and (6.9,2) .. (7,1.5) .. controls (7.5,1.5) and (8.5,1.4) .. (9.5,1.1) .. controls (9.5,2) and (9.7,3) .. (10,3.5) .. controls (9,3) and (8,3) .. (7.25,3);

\draw[fill=gray!60,opacity=0.3,draw=none] (15,4) .. controls (15.5,3.5) and (16.5,3.5) .. (17.5,3.75) .. controls (18.5,4) and (19,4.25) .. (20,4) .. controls (18.9,3) and (18.8,2) .. (19,0.75) .. controls (18,0.5) and (17.5,0.5) .. (16.5,0.75) .. controls (16,0.9) and (15,1.1) .. (14.5,1) .. controls (14.5,2.5) and (14.7,3.5) .. (15,4);
\draw (15,4) .. controls (15.5,3.5) and (16.5,3.5) .. (17.5,3.75) .. controls (18.5,4) and (19,4.25) .. (20,4) .. controls (18.9,3) and (18.8,2) .. (19,0.75) .. controls (18,0.5) and (17.5,0.5) .. (16.5,0.75) .. controls (16,0.9) and (15,1.1) .. (14.5,1) .. controls (14.5,2.5) and (14.7,3.5) .. (15,4);

\draw[fill=gray!100,opacity=0.4,draw=none] (8.25,2.25) ellipse (1.5cm and 0.6cm);
\draw(8.25,2.25) ellipse (1.5cm and 0.6cm);
\draw[fill=gray!100,opacity=0.75,draw=none] (8.25,2.25) ellipse (1cm and 0.3cm);
\draw[line width=1.5pt] (8.25,2.25) ellipse (1cm and 0.3cm);

\draw[fill=gray!100,opacity=0.75,draw=none] (16.75,2.25) ellipse (1.5cm and 0.6cm);
\draw (16.75,2.25) ellipse (1.5cm and 0.6cm);
\draw[fill=white!100,draw=none] (16.75,2.25) ellipse (1cm and 0.3cm);
\draw[line width=1.5pt] (16.75,2.25) ellipse (1cm and 0.3cm);

\draw (8.25,2.25) node{\large$\bullet$};
\draw[gray!50] (16.75,2.25) node{\large$\bullet$};
\draw (2.75,2.75) node{\large$\bullet$};

\draw (3.75,8.75) node{\small$0$};
\draw (0.75,8.75) node{\small$(y,\rho,w)$};
\draw (0.65,2.75) node{\small$(y,\rho w)$};
\draw[->, line width=1pt] (0.75,8.25) -- (0.75,3.25);
\draw[->, line width=1pt] (0.75,8.25) -- (0.75,3.25);
\draw[line width=1pt] (0.65,8.25) -- (0.85,8.25);
\draw[->] (1.75,8.5) arc (180:360:1cm and 0.4cm);
\draw (2.75,7.9) node{\small$w$};
\draw[->] (1.75,8.75) -- (1.75,11.25);
\draw (1.5,10) node{\small$\rho$};
\draw (3.75,1.25) node{\small$\R^d$};
\draw (3,12) node{\small$\R^e\times\R\times\sph^{d-e-1}$};
\draw (1.8,5.5) node{\small$u\circ\widehat{\pi}\circ\Phi$};
\draw (1.8,4.9) node{\small\eqref{diagpol}};
\draw[->, line width=1pt] (4,9) -- (6.5,9);
\draw (5.25,9.25) node{\small$\Phi$};
\draw (5.25,9.75) node{\small\ref{bigstepa2}(ii)};
\draw (5.25,8.675) node{\tiny$\cong$ onto ${\rm im}(\Phi)$};
\draw[->, line width=1pt] (5.5,2.75) -- (4.5,2.75);
\draw (5,3) node{\small$u$};
\draw[->, line width=1pt] (8.5,5.5) -- (8.5,4);
\draw (8.25,4.75) node{\small$\widehat{\pi}$};
\draw[->, line width=1pt] (14,2.25) -- (11,2.25);
\draw (12.5,2.6) node{\small$\widehat{\pi}\circ g$};
\draw (12.5,2) node{\small Cor. \ref{altdb}};
\draw (10.8,3.75) node{\small$M$};
\draw (8.7,2.25) node{\small$N$};
\draw[gray!50] (17.2,2.25) node{\small$N$};
\draw(18,3) node{\small$V_1\setminus V$};
\draw (9.5,3) node{\small$U_i$};
\draw (19,3.75) node{\small$M\setminus V$};
\draw (13.5,10.75) node{\small$(x,v,t)$};
\draw (19.1,10.75) node{\small$(x,t\psi_2(x,v))$};
\draw[->, line width=1pt] (14.5,10.75) -- (17.5,10.75);
\draw[line width=1pt] (14.5,10.65) -- (14.5,10.85);
\draw (16,11.1) node{\small$\varphi_1\circ\widehat{\pi}\circ\varphi_2^{-1}$};
\draw (16,10.5) node{\small Lem. \ref{neighfeten}};
\draw[->, line width=1pt] (12,4.5) -- (17,7.5);
\draw (14.25,6.25) node{\small$\varphi_1$};
\draw[->, line width=1pt] (10,8.75) -- (12,8.75);
\draw (11,9) node{\small$\varphi_2$};
\draw[line width=1pt] (16.75,4) -- (13.8,5.4);
\draw[->, line width=1pt] (13.4,5.6) -- (11.5,6.5);
\draw (14,5) node{\small$g$};
\draw (16,5.25) node{\small$\cong$ onto $\widetilde{M}_+$};
\draw[->, line width=1pt] (17.5,4) -- (18,7.5);
\draw (18.75,5.75) node{\small$\varphi_1|_{V_1\setminus V}$};
\draw (20.25,7.75) node{\small$\mathscr{E}$};
\draw (19.75,8.75) node{\footnotesize$\mathscr{E}_\delta$};
\draw (19,8.75) node{\small$N$}; 
\draw (15,7.25) node{\small$\mathscr{F}$};
\draw (15.25,9) node{\small$\mathscr{F}_{\delta\circ\widehat{\pi}|_R}$};
\draw (8.5,8.25) node{\tiny$R:=\widehat{\pi}^{-1}(N)$};
\draw (13.5,8.15) node{\small$R\times\{0\}$};
\draw (8.5,6.5) node{\small$\widetilde{M}_-$};
\draw (8.5,10.15) node{\small$\widetilde{M}_+$};
\draw (6.5,8.15) node{\small$\widehat{M}$};


\end{tikzpicture}
\end{center}
\caption{Full picture of the drilling blow-up $\widetilde{M}_+$ of $M$ with center $N$.\label{fig4}}
\end{figure}
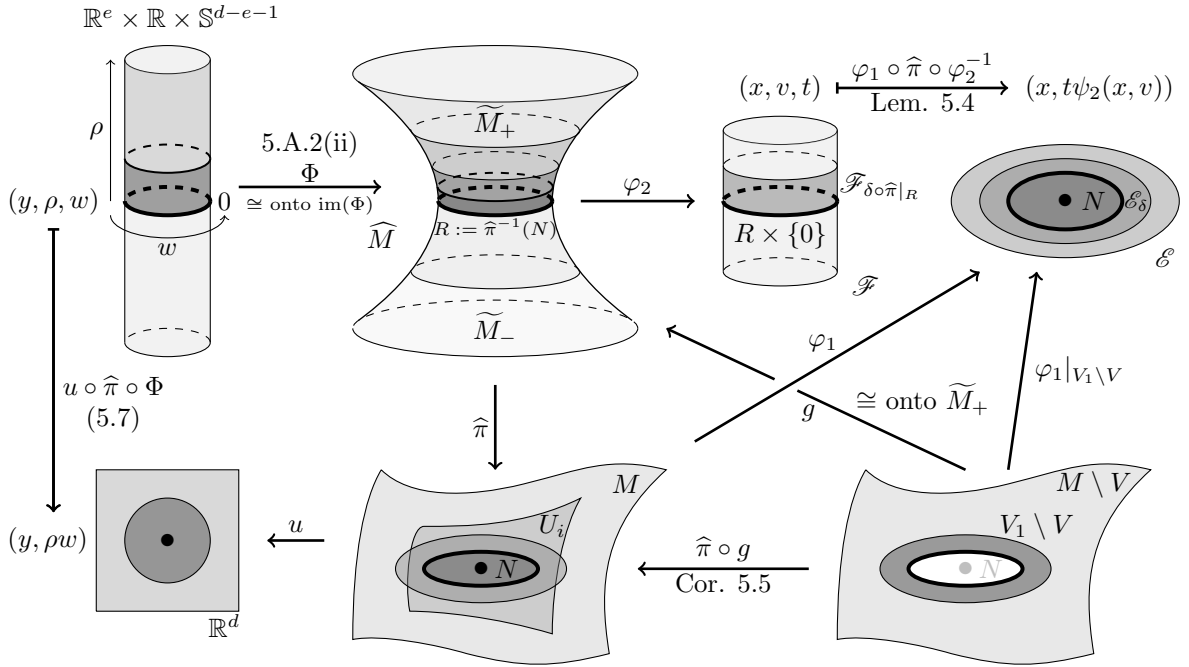

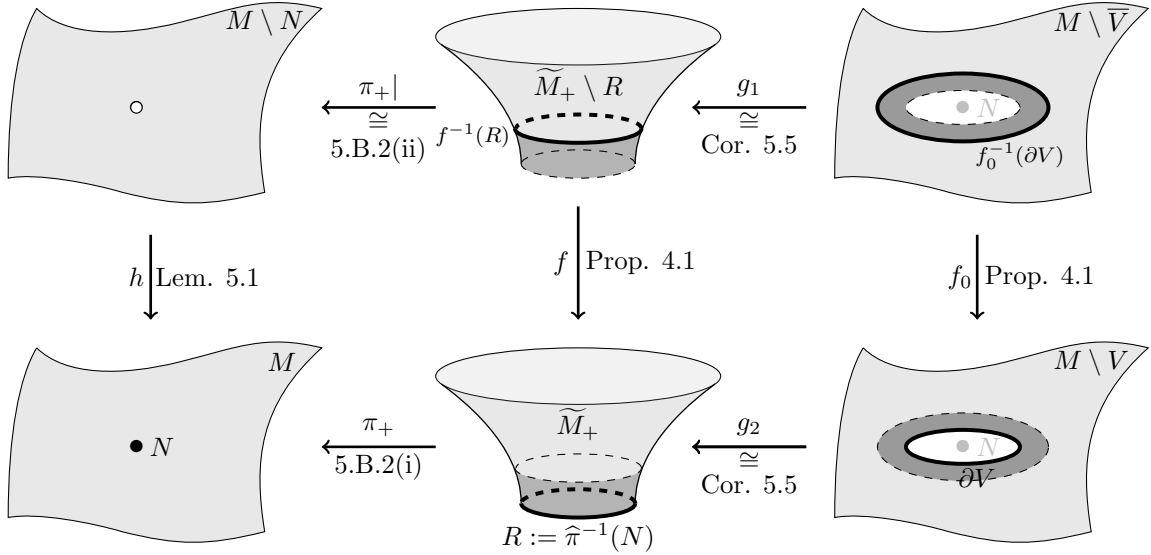
\begin{figure}[!ht]
\begin{center}
\begin{tikzpicture}[scale=0.75]

\draw[fill=gray!20,opacity=0.5,draw=none] (10,3.5) ellipse (2.5cm and 0.5cm);

\draw (10,3.5) ellipse (2.5cm and 0.5cm);

\draw[fill=gray!60,opacity=0.3,draw=none,rotate=-90] (-1.25,9) arc (270:450:0.25cm and 1cm) parabola bend (-1.25,11) (-3.36,12.4) -- (-3.5,12.5) arc (90:-90:0.5cm and 2.5cm) -- (-3.36,7.6) parabola bend (-1.25,9) (-1.25,9);

\draw[fill=black!60,opacity=0.4,draw=none,rotate=-90] (-1.25,9) arc (270:450:0.25cm and 1cm) parabola bend (-1.25,11) (-1.875,11.11) arc (90:-90:0.25cm and 1.11cm) parabola bend (-1.25,9) (-1.25,9);

\draw[dashed] (8.89,1.875) arc (180:360:1.11cm and 0.25cm);
\draw[dashed] (11.11,1.875) arc (0:180:1.1cm and 0.25cm);

\draw[line width=1.5pt] (9,1.25) arc (180:360:1cm and 0.25cm);
\draw[line width=1.5pt,dashed] (11,1.25) arc (0:180:1cm and 0.25cm);

\draw[line width=0.5pt,rotate=-90] (-1.25,9) parabola bend (-1.25,9) (-3.36,7.6);
\draw[line width=0.5pt,rotate=-90] (-1.25,11) parabola bend (-1.25,11) (-3.36,12.4);

\draw[fill=gray!60,opacity=0.3,draw=none] (0.5,4) .. controls (1,3.5) and (2,3.5) .. (3,3.75) .. controls (4,4) and (4.5,4.25) .. (5.5,4) .. controls (4.4,3) and (4.3,2) .. (4.5,0.75) .. controls (3.5,0.5) and (3,0.5) .. (2,0.75) .. controls (1.5,0.9) and (0.5,1.1) .. (0,1) .. controls (0,2.5) and (0.2,3.5) .. (0.5,4);
\draw (0.5,4) .. controls (1,3.5) and (2,3.5) .. (3,3.75) .. controls (4,4) and (4.5,4.25) .. (5.5,4) .. controls (4.4,3) and (4.3,2) .. (4.5,0.75) .. controls (3.5,0.5) and (3,0.5) .. (2,0.75) .. controls (1.5,0.9) and (0.5,1.1) .. (0,1) .. controls (0,2.5) and (0.2,3.5) .. (0.5,4);

\draw[fill=gray!60,opacity=0.3,draw=none] (15,4) .. controls (15.5,3.5) and (16.5,3.5) .. (17.5,3.75) .. controls (18.5,4) and (19,4.25) .. (20,4) .. controls (18.9,3) and (18.8,2) .. (19,0.75) .. controls (18,0.5) and (17.5,0.5) .. (16.5,0.75) .. controls (16,0.9) and (15,1.1) .. (14.5,1) .. controls (14.5,2.5) and (14.7,3.5) .. (15,4);
\draw (15,4) .. controls (15.5,3.5) and (16.5,3.5) .. (17.5,3.75) .. controls (18.5,4) and (19,4.25) .. (20,4) .. controls (18.9,3) and (18.8,2) .. (19,0.75) .. controls (18,0.5) and (17.5,0.5) .. (16.5,0.75) .. controls (16,0.9) and (15,1.1) .. (14.5,1) .. controls (14.5,2.5) and (14.7,3.5) .. (15,4);

\draw[fill=gray!100,opacity=0.75,draw=none] (16.75,2.25) ellipse (1.5cm and 0.6cm);
\draw[dashed] (16.75,2.25) ellipse (1.5cm and 0.6cm);
\draw[fill=white!100,draw=none] (16.75,2.25) ellipse (1cm and 0.3cm);
\draw[line width=1.5pt] (16.75,2.25) ellipse (1cm and 0.3cm);

\draw (2.25,2.25) node{\large$\bullet$};
\draw[gray!50] (16.75,2.25) node{\large$\bullet$};

\draw[fill=gray!20,opacity=0.5,draw=none] (10,9.5) ellipse (2.5cm and 0.5cm);

\draw (10,9.5) ellipse (2.5cm and 0.5cm);

\draw[fill=gray!60,opacity=0.3,draw=none,rotate=-90] (-7.25,9) arc (270:450:0.25cm and 1cm) parabola bend (-7.25,11) (-9.36,12.4) -- (-9.5,12.5) arc (90:-90:0.5cm and 2.5cm) -- (-9.36,7.6) parabola bend (-7.25,9) (-7.25,9);

\draw[fill=black!60,opacity=0.4,draw=none,rotate=-90] (-7.25,9) arc (270:450:0.25cm and 1cm) parabola bend (-7.25,11) (-7.875,11.11) arc (90:-90:0.25cm and 1.11cm) parabola bend (-7.25,9) (-7.25,9);

\draw[line width=1.5pt] (8.89,7.875) arc (180:360:1.11cm and 0.25cm);
\draw[line width=1.5pt,dashed] (11.11,7.875) arc (0:180:1.1cm and 0.25cm);

\draw[dashed] (9,7.25) arc (180:360:1cm and 0.25cm);
\draw[dashed] (11,7.25) arc (0:180:1cm and 0.25cm);

\draw[line width=0.5pt,rotate=-90] (-7.25,9) parabola bend (-7.25,9) (-9.36,7.6);
\draw[line width=0.5pt,rotate=-90] (-7.25,11) parabola bend (-7.25,11) (-9.36,12.4);

\draw[fill=gray!60,opacity=0.3,draw=none] (0.5,10) .. controls (1,9.5) and (2,9.5) .. (3,9.75) .. controls (4,10) and (4.5,10.25) .. (5.5,10) .. controls (4.4,9) and (4.3,8) .. (4.5,6.75) .. controls (3.5,6.5) and (3,6.5) .. (2,6.75) .. controls (1.5,6.9) and (0.5,7.1) .. (0,7) .. controls (0,8.5) and (0.2,9.5) .. (0.5,10);
\draw (0.5,10) .. controls (1,9.5) and (2,9.5) .. (3,9.75) .. controls (4,10) and (4.5,10.25) .. (5.5,10) .. controls (4.4,9) and (4.3,8) .. (4.5,6.75) .. controls (3.5,6.5) and (3,6.5) .. (2,6.75) .. controls (1.5,6.9) and (0.5,7.1) .. (0,7) .. controls (0,8.5) and (0.2,9.5) .. (0.5,10);

\draw[fill=gray!60,opacity=0.3,draw=none] (15,10) .. controls (15.5,9.5) and (16.5,9.5) .. (17.5,9.75) .. controls (18.5,10) and (19,10.25) .. (20,10) .. controls (18.9,9) and (18.8,8) .. (19,6.75) .. controls (18,6.5) and (17.5,6.5) .. (16.5,6.75) .. controls (16,6.9) and (15,7.1) .. (14.5,7) .. controls (14.5,8.5) and (14.7,9.5) .. (15,10);
\draw (15,10) .. controls (15.5,9.5) and (16.5,9.5) .. (17.5,9.75) .. controls (18.5,10) and (19,10.25) .. (20,10) .. controls (18.9,9) and (18.8,8) .. (19,6.75) .. controls (18,6.5) and (17.5,6.5) .. (16.5,6.75) .. controls (16,6.9) and (15,7.1) .. (14.5,7) .. controls (14.5,8.5) and (14.7,9.5) .. (15,10);

\draw[fill=gray!100,opacity=0.75,draw=none] (16.75,8.25) ellipse (1.5cm and 0.6cm);
\draw[line width=1.5pt] (16.75,8.25) ellipse (1.5cm and 0.6cm);
\draw[fill=white!100,draw=none] (16.75,8.25) ellipse (1cm and 0.3cm);
\draw[dashed] (16.75,8.25) ellipse (1cm and 0.3cm);

\draw (2.25,8.25) node{\large$\bullet$};
\draw[fill=white,draw] (2.25,8.25) circle (1mm);
\draw[gray!50] (16.75,8.25) node{\large$\bullet$};

\draw[->, line width=1pt] (14,2.25) -- (12,2.25);

\draw[->, line width=1pt] (2.5,6) -- (2.5,4.5);
\draw (2.25,5.25) node{\small$h$};
\draw (3.5,5.25) node{\small Lem. \ref{eraseresidual}};
\draw[->, line width=1pt] (10,6.5) -- (10,4.5);
\draw (9.7,5.5) node{\small$f$};
\draw (11.1,5.5) node{\small Prop. \ref{doubleivc}};
\draw[->, line width=1pt] (17,6) -- (17,4.5);
\draw (16.7,5.25) node{\small$f_0$};
\draw (18.1,5.25) node{\small Prop. \ref{doubleivc}};
\draw[->, line width=1pt] (14,2.25) -- (12,2.25);
\draw (13,2.6) node{\small$g_2$};
\draw (13,2) node{\small$\cong$};
\draw (13,1.6) node{\small Cor. \ref{altdb}};
\draw[->, line width=1pt] (7.5,2.25) -- (5.5,2.25);
\draw (6.5,2.6) node{\small$\pi_+$};
\draw (6.5,1.9) node{\small\ref{bigstepb2}(i)};
\draw[->, line width=1pt] (14,8.25) -- (12,8.25);
\draw (13,8.6) node{\small$g_1$};
\draw (13,8) node{\small$\cong$};
\draw (13,7.6) node{\small Cor. \ref{altdb}};
\draw[->, line width=1pt] (7.5,8.25) -- (5.5,8.25);
\draw (6.5,8.6) node{\small$\pi_+|$};
\draw (6.5,8) node{\small$\cong$};
\draw (6.5,7.5) node{\small\ref{bigstepb2}(ii)};

\draw (4.8,3.75) node{\small$M$};
\draw (4.5,9.75) node{\small$M\setminus N$};
\draw (2.7,2.25) node{\small$N$};
\draw[gray!50] (17.2,2.25) node{\small$N$};
\draw (19,3.75) node{\small$M\setminus V$};
\draw[gray!50] (17.2,8.25) node{\small$N$};
\draw (19,9.75) node{\small$M\setminus\overline{V}$};
\draw (17.75,7.4) node{\scriptsize$f_0^{-1}(\partial V)$};
\draw (17,1.7) node{\small$\partial V$};

\draw (10,0.65) node{\small$R:=\widehat{\pi}^{-1}(N)$};
\draw (10,2.65) node{\small$\widetilde{M}_+$};
\draw (10,8.65) node{\small$\widetilde{M}_+\setminus R$};
\draw (8.15,7.75) node{\scriptsize$f^{-1}(R)$};

\end{tikzpicture}
\end{center}
\caption{Geometry of the tool introduced in Lemma \ref{eraseresidual} to erase a closed Nash submanifold $N$ from the Nash manifold $M$.\label{fig5}}
\end{figure}

\subsection{Relationship between drilling blow-up and classical blow-up}\label{dbme}
We justify next the second part of the name of the drilling blow-up relating it with the classical blow-up. Let $M\subset\R^m$ be a Nash manifold of dimension $d$ and let $N\subset M$ be a closed Nash submanifold of dimension $e$. Let $f_1,\ldots,f_k\in{\mathcal N}(M)$ be a system of generators of $I(N)$. Define
$$
\Gamma':=\{(x,(f_1(x):\ldots:f_k(x)))\in M\times\R\PP^{k-1}:\ x\in M\setminus N\}.
$$
The closure $B(M,N)$ of $\Gamma'$ in $M\times\R\PP^{k-1}$ together with the restriction $\pi'$ to $B(M,N)$ of the projection $M\times\R\PP^{k-1}\to M$ is the classical \em blow-up of $M$ with center $N$\em. 

\begin{cor}\label{bu}
Let $(\widehat{M},\widehat{\pi})$ be the twisted Nash double of the drilling blow-up $(\widetilde{M}_+,\pi_+)$. Let $\sigma:\widehat{M}\to\widehat{M},\ (a,b)\mapsto(a,-b)$ be the involution of $\widehat{M}$ without fixed points. Consider the Nash map
$$
\Theta:M\times\sph^{k-1}\to M\times\R\PP^{k-1}, (p,q)\to (p,[q])
$$
and its restriction $\theta:\widehat{M}\to B(M,N)$. We have
\begin{itemize}
\item[(i)] $\theta(\widehat{M})=B(M,N)$, $\theta\circ\sigma=\theta$, $\pi'\circ\theta=\widehat{\pi}$ and $\theta^{-1}(a,[b])=\{(a,b),(a,-b)\}$ for each $(a,[b])\in B(M,N)$. 
\item[(ii)] $\theta$ is an unramified two to one Nash covering of $B(M,N)$.
\end{itemize}
\end{cor}
\begin{remark}
Many well-known properties of $(B(M,N),\pi')$ concerning: the fibers of $\pi'$, the local representations \eqref{lr} of $\pi'$ at the points of ${\pi'}^{-1}(N)$, finite coverings of $B(M,N)$ whose members are Nash diffeomorphic either to $\R^e\times\R\PP^{d-e}$ or to $\R^d$, the fact that $\pi'$ is proper and the restriction $\pi'|:B(M,N)\setminus {\pi'}^{-1}(N)\to M\setminus N$ is a Nash diffeomorphism, the fact that $B(M,N)$ does not depend on the generators of $I(N)$, etc. follow at once from \ref{bigstep}, \ref{bbu} and Corollary \ref{bu}.
\end{remark}

\section{Connected Nash manifolds with boundary as Nash images of Euclidean spaces}\label{s6}

In this section we prove Theorem \ref{mstone}. By Proposition \ref{doubleivc} every Nash manifold $H$ with boundary is the image under a Nash map of its interior $\Int(H)$. Consequently, we are reduced to prove the following.

\begin{thm}\label{mstone0}
Let $M\subset\R^m$ be a connected $d$-dimensional Nash manifold. Then $M$ is a Nash image of $\R^d$.
\end{thm}

The proof of Theorem \ref{mstone0} still requires some preliminary work that we develop next. We prove first that connected Nash manifolds with boundary are connected by Nash paths, so they are under the assumptions of Theorem \ref{main}.

\begin{lem}\label{mcnp}
Let $M\subset\R^m$ be a connected Nash manifold. Then $M$ is connected by Nash paths.
\end{lem}
\begin{proof}
By \cite[Thm.2.4.5 \& Prop.2.5.13]{bcr} $M$ is semialgebraically path connected. Let $x,y\in M$ and let $\alpha:[0,1]\to M$ be a continuous semialgebraic path that connects $x$ and $y$. Let $\veps>0$ and let $\widehat{\alpha}:(-\veps,1+\veps)\to M$ be any (continuous) semialgebraic extension of $\alpha$ to the interval $(-\veps,1+\veps)$. By \cite[Cor.II.5.7]{sh} there exists a Nash approximation $\beta:(-\veps,1+\veps)\to M$ of $\alpha$ such that $\beta(0)=\alpha(0)=x$ and $\beta(1)=\alpha(1)=y$. Thus, $M$ is connected by Nash paths.
\end{proof}

\begin{cor}\label{imnm}
Let $M\subset\R^m$ be a connected Nash manifold and let $f:M\to\R^n$ be a Nash map. Then $\Ss:=f(M)$ is pure dimensional and connected by Nash paths.
\end{cor}
\begin{proof}
Assume $\Ss$ is not pure dimensional. Let $B\subset\R^n$ be a small ball such that $\dim(B\cap\Ss)<\dim(\Ss)$. Let $Y$ be the Zariski closure of $B\cap\Ss$ and let $P\in\R[\x]$ be a polynomial equation of $Y$. As the Nash function $P\circ f$ vanish on the open set $f^{-1}(B)$ of the connected Nash manifold $M$, the composition $P\circ f$ is identically zero on $M$. Consequently, $\Ss\subset Y$, which is a contradiction. Thus, $\Ss$ is pure dimensional.

To prove that $\Ss$ is connected by Nash paths pick $x,y\in\Ss$ and let $a,b\in M$ be such that $f(a)=x$ and $f(b)=y$. As $M$ is connected by Nash paths, there exists a Nash path $\beta:[0,1]\to M$ connecting $a$ and $b$. Thus, $\alpha:=f\circ\beta$ is a Nash path connecting $x$ and $y$, as required.
\end{proof}
\begin{remark}
By Proposition \ref{doubleivc} and Corollary \ref{imnm} connected Nash manifolds with boundary are connected by Nash paths.
\end{remark}

\subsection{Managing semialgebraic triangulations}\setcounter{paragraph}{0}

The proof of Theorem \ref{mstone0} involves an inductive argument on the number of simplices of a suitable (semialgebraic) triangulation of the connected Nash manifold $M$. Let $\sigma\subset\R^n$ be a simplex of dimension $d$. The \em facets of $\sigma$ \em are the faces of $\sigma$ of dimension $d-1$. As usual, we denote the relative interior of $\sigma$ with $\sigma^0$ and we will say that $\sigma$ is a $d$-simplex. The first step of the inductive argument concerns the following statement: \em The interior of a simplex is a Nash image of an Euclidean space\em.

\begin{lem}\label{simnashim}
Let $\sigma\subset\R^n$ be a simplex of dimension $n$ and let $\sigma^0$ be its interior. Then $\sigma^0$ is Nash diffeomorphic to $\R^n$.
\end{lem}
\begin{proof}
It is enough to consider the simplex $\sigma:=\{\x_1\geq0,\ldots,\x_n\geq0,\x_1+\cdots+\x_n\leq1\}\subset\R^n$. Consider the open orthant $\Qq:=\{\y_1>0,\ldots,\y_n>0\}$ and the Nash diffeomorphism
$$
f:\sigma^0\to\Qq,\ (x_1,\ldots,x_n)\mapsto\Big(\frac{x_1}{1-\sum_{i=1}^nx_i},\ldots,\frac{x_n}{1-\sum_{i=1}^nx_i}\Big)
$$
whose inverse is the Nash map
$$
f^{-1}:\Qq\to\sigma^0,\ (y_1,\ldots,y_n)\mapsto\Big(\frac{y_1}{1+\sum_{i=1}^ny_i},\ldots,\frac{y_n}{1+\sum_{i=1}^ny_i}\Big).
$$
We are reduced to prove that $\Qq$ is Nash diffeomorphic to $\R^n$. To that end, we show that \em the open interval $(0,+\infty)$ is Nash diffeomorphic to $\R$\em. Consider
$$
h_1:(0,+\infty)\to\R,\ t\mapsto t-\frac{1}{t}\quad\text{and}\quad h_1^{-1}:\R\to(0,+\infty),\ t\mapsto\frac{t+\sqrt{t^2+4}}{2}.
$$
We are done.
\end{proof}

The following result is the clue to erase a simplex from a semialgebraic triangulation of a Nash manifold. 

\begin{lem}[Erasing simplices]\label{homeo}
Let $\sigma_1,\sigma_2\subset\R^n$ be two simplices of dimension $n$ that only share a facet $\tau$. Let $\Dd:=\sigma_1^0\cup(\sigma_2\setminus\partial\tau)=(\tau^0\cup\sigma_1^0)\cup(\sigma_2\setminus\tau)$ and $\ol{\Dd}:=\sigma_1\cup\sigma_2$. Then there exists a semialgebraic homeomorphism $\psi:\sigma_2\to\ol{\Dd}$ such that $\psi(\sigma_2\setminus\tau)=\Dd$ and $\psi|_{\partial\sigma_2\setminus\tau^0}=\id_{\partial\sigma_2\setminus\tau^0}$.
\end{lem}
\begin{proof}
The proof is conducted in several steps. Figure \ref{fig6} summarizes the followed strategy.

\paragraph{}\label{step1sim}Let $v_i$ be the vertex of $\sigma_i$ not contained in $\tau$. We may assume $\tau\subset\{\x_n=0\}$, its barycenter is the origin of $\R^n$, $v_2:=(v_{21},\ldots,v_{2n})\in\{\x_n>0\}$ and $v_1$ is $-{\tt e}_n:=(0,\ldots,0,-1)$. We claim: \em after a semialgebraic homeomorphism of $\R^n$ that keeps $\sigma_1$ invariant, we may assume $v_2={\tt e}_n:=(0,\ldots,0,1)$\em. 

Consider the semialgebraic homeomorphism 
$$
\gamma:\R^n\to\R^n,\ x:=(x_1,\ldots,x_n)\mapsto
\begin{cases} 
x-\frac{x_n}{v_{2n}}(v_{2,1},\ldots,v_{2,n-1},v_{2,n}-1)&\text{if $x_n>0$},\\
x&\text{if $x_n\leq0$},
\end{cases}
$$
whose inverse map is
$$
\gamma^{-1}:\R^n\to\R^n,\ y:=(y_1,\ldots,y_n)\mapsto
\begin{cases} 
y+y_n(v_{2,1},\ldots,v_{2,n-1},v_{2,n}-1)&\text{if $y_n>0$},\\
y&\text{if $y_n\leq0$}.
\end{cases}
$$
We have $\gamma(v_2)={\tt e}_n$ and $\gamma|_{\{\x_n\leq0\}}=\id_{\{\x_n\leq0\}}$. As $\tau\subset\{\x_n=0\}$, it holds $\gamma(\sigma_2)$ is the simplex whose vertices are those of $\tau$ and ${\tt e}_n$. 

\begin{figure}[!ht]
\begin{center}
\begin{tikzpicture}[scale=0.9]

\draw[fill=gray!60,opacity=0.75,draw=none] (0,2) -- (5,4.5) -- (3.5,2) -- (0,2);
\draw[fill=gray!30,opacity=0.75,draw=none] (0,2) -- (3.5,2) -- (1.75,0) -- (0,2);
\draw[line width=1.5pt] (0,2) -- (3.5,2);
\draw[line width=1.5pt] (0,2) -- (5,4.5);
\draw[line width=1.5pt] (3.5,2) -- (5,4.5);
\draw[line width=1.5pt,dashed] (0,2) -- (1.75,0);
\draw[line width=1.5pt,dashed] (1.75,0) -- (3.5,2);

\draw[fill=gray!90,opacity=0.75,draw=none] (7,2) -- (8.75,4) -- (10.5,2) -- (8.75,2.75) -- (7,2);
\draw[fill=gray!40,opacity=0.75,draw=none] (7,2) -- (8.75,0) -- (10.5,2) -- (8.75,2.75) -- (7,2);

\draw[line width=1.5pt,dashed] (7,2) -- (10.5,2);
\draw[line width=1.5pt] (8.75,2.75) -- (8.75,4);
\draw[line width=1.5pt] (8.75,2.75) -- (8.75,0);
\draw[line width=1.5pt] (8.75,2.75) -- (7,2);
\draw[line width=1.5pt] (8.75,2.75) -- (10.5,2);
\draw[line width=1.5pt] (8.75,4) -- (10.5,2);
\draw[line width=1.5pt] (7,2) -- (8.75,4);
\draw[line width=1.5pt] (8.75,4) -- (10.5,2);
\draw[line width=1.5pt,dashed] (10.5,2) -- (8.75,0); 
\draw[line width=1.5pt,dashed] (8.75,0) -- (7,2);

\draw[fill=gray!90,opacity=0.75,draw=none] (13.5,2) -- (15.25,4) -- (17,2) -- (15.25,2.75) -- (13.5,2);
\draw[fill=gray!40,opacity=0.75,draw=none] (13.5,2) -- (17,2) -- (15.25,2.75) -- (13.5,2);

\draw[line width=1.5pt,dashed] (13.5,2) -- (17,2);
\draw[line width=1.5pt] (13.5,2) -- (15.25,4);
\draw[line width=1.5pt] (15.25,4) -- (17,2);
\draw[line width=1.5pt] (15.25,2.75) -- (15.25,4);
\draw[line width=1.5pt] (15.25,2.75) -- (15.25,2);
\draw[line width=1.5pt] (15.25,2.75) -- (17,2);
\draw[line width=1.5pt] (15.25,2.75) -- (13.5,2);

\draw[fill=black] (5,4.5) circle (0.75mm);
\draw[fill=black] (8.75,4) circle (0.75mm);
\draw[fill=black] (15.25,4) circle (0.75mm);
\draw[fill=black] (15.25,2.75) circle (0.75mm);
\draw[fill=black] (8.75,2.75) circle (0.75mm);

\draw[fill=white,draw] (0,2) circle (0.75mm);
\draw[fill=white,draw] (3.5,2) circle (0.75mm);
\draw[fill=white,draw] (1.75,0) circle (0.75mm);
\draw[fill=white,draw] (7,2) circle (0.75mm);
\draw[fill=white,draw] (10.5,2) circle (0.75mm);
\draw[fill=white,draw] (8.75,0) circle (0.75mm);
\draw[fill=white,draw] (13.5,2) circle (0.75mm);
\draw[fill=white,draw] (15.25,2) circle (0.75mm);
\draw[fill=white,draw] (17,2) circle (0.75mm);

\draw (1.75,1.75) node{\small$\tau$};
\draw (15.25,1.75) node{\small$\tau$};
\draw (8.25,1.75) node{\small$\tau$};
\draw (1.75,1) node{\small$\sigma_1^0$};
\draw (1.5,3.5) node{\small$\Dd$};
\draw (7.5,0.75) node{\small$\sigma_1^0$};
\draw (3.25,3) node{\small$\sigma_2$};
\draw (6.75,3.25) node{\small$\gamma(\sigma_2)\leadsto\sigma_2$};
\draw (14.25,3.75) node{\small$\sigma_2\setminus\tau$};
\draw (5.25,4.75) node{\small$v_2$};
\draw (8.5,4.3) node{\small$\gamma(v_2)\leadsto v_2={\tt e}_n$};
\draw (15.25,4.25) node{\small${\tt e}_n$};
\draw (14.9,2.85) node{\small$w_0$};
\draw (2.75,0) node{\small$v_1=-{\tt e}_n$};
\draw (9.75,0) node{\small$v_1=-{\tt e}_n$};

\draw[->, line width=1pt] (4.5,2.25) -- (6.5,2.25);
\draw[->, line width=1pt] (6.5,1.75) -- (4.5,1.75);
\draw[<-, line width=1pt] (11,2.25) -- (13,2.25);
\draw[<-, line width=1pt] (13,1.75) -- (11,1.75);

\draw (5.5,2.5) node{\small$\gamma$};
\draw (5.5,1.5) node{\small$\gamma^{-1}$};
\draw (12,2.5) node{\small$\psi$};
\draw (12,1.5) node{\small$\psi^{-1}$};

\draw (9.25,3) node{\small$\eta_{2i}$};
\draw (9.25,1.5) node{\small$\epsilon_{1i}$};

\draw (15.75,3) node{\small$\eta_{2i}$};
\draw (15.75,2.25) node{\small$\eta_{1i}$};

\end{tikzpicture}
\end{center}
\caption{Construction of the homeomorphisms between $\Dd$ and $\sigma_2\setminus\tau$.\label{fig6}}
\end{figure}
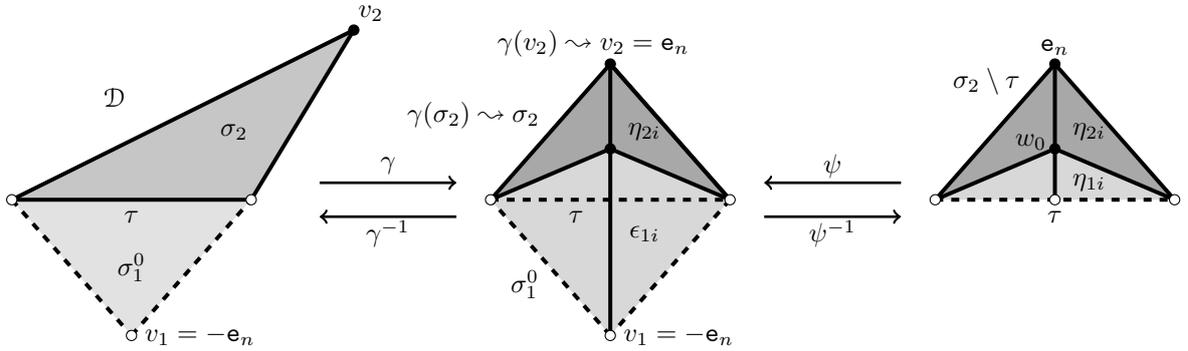

\paragraph{} Denote the simplex whose vertices are the affinely independent points $p_1,\ldots,p_k\in\R^n$ with $[p_1,\ldots,p_k]$. Let $w_1,\ldots,w_n$ be the vertices of the $(n-1)$-simplex $\tau$. The barycenter of $\tau$ is the origin, so the barycenter of $\sigma_2$ is $w_0:=\frac{1}{n+1}{\tt e}_n$. Consider the $n$-simplices
\begin{align*}
&\eta_{1i}:=[0,w_0,w_1,\ldots,w_{i-1},w_{i+1},\ldots,w_n]\quad i=1,\ldots,n,\\
&\eta_{2i}:=[{\tt e}_n,w_0,w_1,\ldots,w_{i-1},w_{i+1},\ldots,w_n]\quad i=1,\ldots,n.
\end{align*}
The family $\{\eta_{11},\ldots,\eta_{1n},\eta_{21},\ldots,\eta_{2n}\}$ provides a triangulation of $\sigma_2$. Consider also the $n$-simplices 
$$
\epsilon_{1i}:=[-{\tt e}_n,w_0,w_1,\ldots,w_{i-1},w_{i+1},\ldots,w_n]\quad i=1,\ldots,n.
$$
Our choices done in \ref{step1sim} assure that the collection $\{\epsilon_{11},\ldots,\epsilon_{1n},\eta_{21},\ldots,\eta_{2n}\}$ provides a triangulation of $\ol{\Dd}$. 

\paragraph{} Fix $i=1,\ldots,n$ and take barycentric coordinates in $\R^n$ with respect to the affine basis $\Bb_i:=\{0,w_0,w_1,\ldots,w_{i-1},w_{i+1},\ldots,w_n\}$ and consider the affine isomorphism of $\R^n$ given by
$$
\psi_i:\R^n\to\R^n,\ \Big(1-\sum_{k\neq i}\lambda_k\Big)0+\sum_{k\neq i}\lambda_kw_k\mapsto\Big(1-\sum_{k\neq i}\lambda_k\Big)(-{\tt e}_n)+\sum_{k\neq i}\lambda_kw_k. 
$$
Observe that $\psi_i(\eta_{1i})=\epsilon_{1i}$, $\psi_i|_{\eta_{1i}\cap\eta_{1j}}=\psi_j|_{\eta_{1i}\cap\eta_{1j}}$ and $\psi_i|_{\eta_{1i}\cap\eta_{2j}}=\id_{\eta_{1i}\cap\eta_{2j}}$ for $1\leq i,j\leq n$. Consequently, the semialgebraic map
$$
\psi:\sigma_2\to\ol{\Dd},\ x\mapsto
\begin{cases}
\psi_i(x)&\text{if $x\in\eta_{1i}$ for $i=1,\dots,n$,}\\
x&\text{if $x\in\eta_{2i}$ for $i=1,\dots,n$}
\end{cases}
$$
is a well-defined homeomorphism such that $\psi(\sigma_2\setminus\tau)=\Dd$ and $\psi|_{\partial\sigma_2\setminus\tau^0}=\id_{\partial\sigma_2\setminus\tau^0}$, as required.
\end{proof}

To take advantage of the full strength of Lemma \ref{homeo} we need the following result to subdivide simplices in the appropriate way, see Figure \ref{fig7}.

\begin{lem}\label{divide}
Let $\sigma\subset\R^n$ be a simplex of dimension $n$ and let $\tau_1,\ldots,\tau_k$ be facets of $\sigma$ for some $k=1,\ldots,n+1$. Let $\epsilon$ be either the intersection of the remaining facets $\tau_{k+1},\ldots,\tau_{n+1}$ if $k<n+1$ or $\sigma$ if $k=n+1$. Let $b$ be the barycenter of $\epsilon$ and let $\eta_i$ be the convex hull of $\tau_i\cup\{b\}$ for $i=1,\ldots,k$. Then the simplices $\eta_1,\ldots,\eta_k$ provide a triangulation of $\sigma$.
\end{lem}

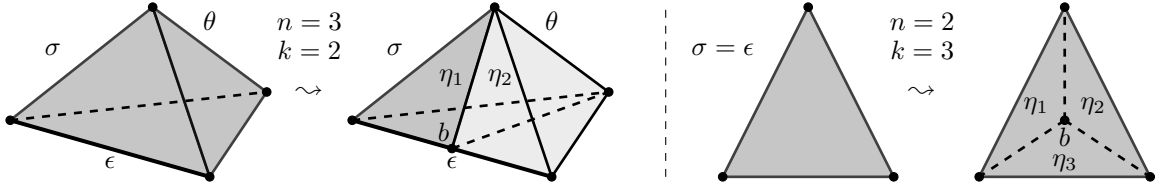
\begin{figure}[!ht]
\begin{center}
\begin{tikzpicture}[scale=0.75]

\draw[fill=gray!60,opacity=0.75,line width=1pt,draw] (0,1) -- (3.5,0) -- (2.5,3) -- (0,1);
\draw[fill=gray!60,opacity=0.75,line width=1pt,draw] (3.5,0) -- (2.5,3) -- (4.5,1.5) -- (3.5,0);
\draw[line width=1pt,dashed] (0,1) -- (4.5,1.5);
\draw[line width=1.5pt] (0,1) -- (3.5,0);

\draw[fill=black] (0,1) circle (0.75mm);
\draw[fill=black] (3.5,0) circle (0.75mm);
\draw[fill=black] (2.5,3) circle (0.75mm);
\draw[fill=black] (4.5,1.5) circle (0.75mm);
\draw (1.75,0.25) node{\small$\epsilon$};
\draw (0.75,2.25) node{\small$\sigma$};
\draw (3.5,2.75) node{\small$\theta$};

\draw[fill=gray!60,opacity=0.75,line width=1pt,draw] (6,1) -- (7.75,0.5) -- (8.5,3) -- (6,1);
\draw[fill=gray!30,opacity=0.5,line width=1pt,draw=none] (7.75,0.5) -- (8.5,3) -- (10.5,1.5) -- (9.5,0) -- (7.75,0.5);

\draw[line width=1pt,dashed] (6,1) -- (10.5,1.5);
\draw[line width=1.5pt] (6,1) -- (9.5,0);
\draw[line width=1pt] (7.75,0.5) -- (8.5,3);
\draw[line width=1pt] (9.5,0) -- (8.5,3) -- (10.5,1.5) -- (9.5,0);
\draw[line width=1pt,dashed] (7.75,0.5) -- (10.5,1.5);

\draw[fill=black] (6,1) circle (0.75mm);
\draw[fill=black] (9.5,0) circle (0.75mm);
\draw[fill=black] (8.5,3) circle (0.75mm);
\draw[fill=black] (7.75,0.5) circle (0.75mm);
\draw[fill=black] (10.5,1.5) circle (0.75mm);
\draw (7.75,0.25) node{\small$\epsilon$};
\draw (6.75,2.25) node{\small$\sigma$};
\draw (9.5,2.75) node{\small$\theta$};
\draw (5.25,1.5) node{\small$\leadsto$};
\draw (5.25,2.75) node{\small$n=3$};
\draw (5.25,2.25) node{\small$k=2$};
\draw (7.6,0.8) node{\small$b$};
\draw (7.75,1.75) node{\small$\eta_1$};
\draw (8.6,1.75) node{\small$\eta_2$};

\draw[dashed] (11.5,0) -- (11.5,3);

\draw[fill=gray!60,opacity=0.75,line width=1pt,draw] (12.5,0) -- (14,3) -- (15.5,0) -- (12.5,0);

\draw[fill=black] (12.5,0) circle (0.75mm);
\draw[fill=black] (14,3) circle (0.75mm);
\draw[fill=black] (15.5,0) circle (0.75mm);

\draw[fill=gray!60,opacity=0.75,line width=1pt,draw] (17,0) -- (18.5,3) -- (20,0) -- (17,0);

\draw[fill=black] (17,0) circle (0.75mm);
\draw[fill=black] (18.5,3) circle (0.75mm);
\draw[fill=black] (20,0) circle (0.75mm);
\draw[fill=black] (18.5,1) circle (0.75mm);
\draw[line width=1pt,dashed] (18.5,1) -- (17,0);
\draw[line width=1pt,dashed] (18.5,1) -- (18.5,3);
\draw[line width=1pt,dashed] (18.5,1) -- (20,0);

\draw (16,1.5) node{\small$\leadsto$};
\draw (16,2.75) node{\small$n=2$};
\draw (16,2.25) node{\small$k=3$};
\draw (12.5,2.25) node{\small$\sigma=\epsilon$};
\draw (18.5,0.7) node{\small$b$};
\draw (18,1.25) node{\small$\eta_1$};
\draw (19,1.25) node{\small$\eta_2$};
\draw (18.5,0.25) node{\small$\eta_3$};

\end{tikzpicture}
\end{center}
\caption{Triangulation of $\sigma$ induced by $\epsilon:=\tau_{k+1}\cap\cdots\cap\tau_{n+1}$.\label{fig7}}
\end{figure}

\begin{proof}
We have to prove: \em $\sigma=\bigcup_{i=1}^k\eta_i$ and if $\rho_1$ is a face of $\eta_i$ and $\rho_2$ is a face of $\eta_j$, then $\rho_1\cap\rho_2$ is either the empty-set or a common face of $\rho_1$ and $\rho_2$\em. 

\paragraph{} If $k=n+1$, then $\sigma=\epsilon$ and $b$ is the barycenter of $\sigma$, so $\eta_1,\ldots,\eta_{n+1}$ provide a triangulation of $\sigma$. We assume in the following $k<n+1$. Let $W$ be the affine subspace of $\R^n$ generated by $\epsilon$ and let $L$ be the affine subspace generated by $\theta:=\tau_1\cap\cdots\cap\tau_k$. We have $\dim(W)=k-1$, $\dim(L)=n-k$ and $W\cap L=\varnothing$. Let $E$ be the affine hyperplane of $\R^n$ that contains $L$ and is parallel to $W$. Consider the projection 
$$
\pi:\R^n\setminus E\to W,\ x\mapsto (\{x\}+L)\cap W,
$$
where $\{x\}+L$ denotes the affine subspace of $\R^n$ generated by $\{x\}\cup L$. We claim: $\pi(\sigma\setminus E)=\epsilon$. As $\pi|_W=\id_W$ and $\epsilon\subset W$, it is enough to check that $\pi(\sigma\setminus E)\subset\epsilon$. 

It holds: \em the vertices of $\sigma$ are either vertices of $\epsilon$ or $\theta$\em.

Both $\epsilon$ and $\theta$ are faces of $\sigma$. The $k$ vertices of $\epsilon$ (it is a simplex of dimension $k-1$) are vertices of $\sigma$ and the $n-k+1$ vertices of $\theta$ (it is a simplex of dimension $n-k$) are vertices of $\sigma$. We have $k+(n-k+1)=n+1$ vertices. As $\epsilon\cap\theta=\varnothing$, we have all the vertices of $\sigma$. 

Denote the vertices of $\epsilon$ with $v_1,\ldots,v_k$ and the vertices of $\theta$ with $v_{k+1},\ldots,v_{n+1}$. For each $x\in\sigma\setminus E$ there exist $\lambda_i\geq0$ such that $\sum_{i=1}^{n+1}\lambda_i=1$ and $x=\sum_{i=1}^{n+1}\lambda_iv_i$. As $\sigma\cap E=\theta$ and $x\in\sigma\setminus E$, we have $\mu:=\sum_{i=1}^k\lambda_i>0$. If $\mu=1$, then $x\in\epsilon$ and $\pi(x)=x\in\epsilon$. If $\mu<1$, consider the points
$$
p=\sum_{i=1}^k\tfrac{\lambda_i}{\mu}v_i\in\epsilon\subset W\quad\text{and}\quad q=\sum_{i=k+1}^{n+1}\tfrac{\lambda_i}{1-\mu}v_i\in\theta\subset L
$$ 
that satisfy $x=\mu p+(1-\mu)q$. Consequently, 
$$
p=\tfrac{1}{\mu}x+\tfrac{\mu-1}{\mu}q\in(\{x\}+L)\cap W\quad\text{and}\quad\pi(x)=p\in\epsilon.
$$

\paragraph{} We are ready to prove $\sigma=\bigcup_{i=1}^k\eta_i$. Let $x\in\sigma$. If $x\in\theta$, then $x\in\eta_1\cap\cdots\cap\eta_k$. So we assume $x\not\in\theta$. Consider the simplex $\rho$ of base $\theta$ and vertex $\pi(x)\in\epsilon$. Observe that $x\in\rho$ and we may write $x=\alpha\pi(x)+(1-\alpha)y$ for some $y\in\theta$ and $\alpha\in[0,1]$. The $k$ facets of $\epsilon$ are the intersections of $\epsilon$ with the facets $\tau_1,\ldots,\tau_k$ of $\sigma$. In addition, $\eta_i\cap\epsilon$ is the cone of base $\tau_i\cap\epsilon$ (a facet of $\epsilon$) and vertex $b$ (the barycenter of $\epsilon$). Consequently, $\epsilon=\bigcup_{i=1}^k(\eta_i\cap\epsilon)$ and we assume that $\pi(x)\in\eta_1\cap\epsilon$. As $\pi(x)\in\eta_1$ and $y\in\theta\subset\tau_1\subset\eta_1$, we conclude $x=\alpha\pi(x)+(1-\alpha)y\in\eta_1$.

\paragraph{}Let $\rho_1$ be a face of $\eta_i$ and $\rho_2$ a face of $\eta_j$. We have to prove that $\rho_1\cap\rho_2$ is either the empty-set or a common face of $\rho_1$ and $\rho_2$.

Observe that $\rho_i$ is either a face of $\sigma$, the vertex $b$ or the cone of basis a face of $\sigma$ and vertex $b$. 

(1) Suppose $\rho_1=\{b\}$. Then $\rho_1\cap\rho_2$ is either the empty-set or $\{b\}$.

(2) Suppose $\rho_1$ is a face of $\sigma$. Then $\rho_1\cap\rho_2=\rho_1\cap\rho_2\cap\sigma$ is the intersection of two faces of $\sigma$, so it is either the empty-set or a face of $\sigma$, which is a common face of $\rho_1$ and $\rho_2$.

(3) Suppose $\rho_1$ and $\rho_2$ are cones over faces $\delta_1$ and $\delta_2$ of $\sigma$ with vertex $\{b\}$. Consequently, $\rho_1\cap\rho_2$ is either $\{b\}$ or the cone of base $\delta_1\cap\delta_2$ and vertex $\{b\}$, which is a face of $\rho_1$ and $\rho_2$, as required.
\end{proof}

\subsection{Proof of Theorem \ref{mstone0}}\setcounter{paragraph}{0}
The proof is conducted in several steps. Figure \ref{fig8} summarizes the followed strategy.

\paragraph{}We may assume $M$ is bounded in $\R^m$. By \cite[Thm.9.2.1 \& Rmk.9.2.3]{bcr} there exist a finite simplicial complex $K$ and a semialgebraic homeomorphism $\Phi:|K|\to\cl(M)$ such that
\begin{itemize}
\item The semialgebraic sets $M$ and $\cl(M)\setminus M$ are finite unions of images $\Phi(\sigma^0)$ where $\sigma\in K$.
\item The restriction $\Phi|_{\sigma^0}:\sigma^0\to\cl(M)$ is a Nash embedding for each $\sigma\in K$.
\end{itemize}
Denote the simplices of dimension $d$ of $K$ with $\sigma_1,\ldots,\sigma_r$.

\paragraph{}\label{ordersim} Let $E$ be the union of the simplices of $K$ of dimension $\leq d-2$. Denote $\sigma_i':=(\Phi^{-1}(M)\setminus E)\cap\sigma_i$. Let us \em reorder the indices $1,\ldots,r$ in such a way that $\sigma_i'$ shares the interior of a face of dimension $d-1$ with $\bigcup_{j=1}^{i-1}\sigma_j'$ for $i=2,\ldots,r$\em. Consequently, \em $M\setminus\bigcup_{j=i+1}^r\Phi(\sigma_i)$ is connected for $i=1,\ldots,r-1$\em. 

Indeed, as $M$ is a connected Nash manifold, $\Phi^{-1}(M)\setminus E$ is connected. Fix $1\leq s<r$ and assume that we have ordered the simplices $\sigma_1,\ldots,\sigma_s$ in such a way that $\sigma_i'$ shares the interior of a face of dimension $d-1$ with $\bigcup_{j=1}^{i-1}\sigma_j'$ for $i=2,\ldots,s$. Let $C_1:=\bigcup_{j=1}^s\sigma_j'$ and $C_2:=\bigcup_{j=s+1}^r\sigma_j'$. As $\Phi^{-1}(M)\setminus E=C_1\cup C_2$ is connected and $C_1,C_2$ are closed in $\Phi^{-1}(M)\setminus E$, we may assume that $C_1\cap\sigma_{s+1}'\neq\varnothing$, so $\sigma_{s+1}'$ shares the interior of a face of dimension $d-1$ with $\bigcup_{j=1}^s\sigma_j'$. This is so because $K$ is a triangulation of $\cl(M)$ and $(\Phi^{-1}(M)\setminus E)\cap\eta=\varnothing$ for each face $\eta$ of dimension $\leq d-2$. Proceeding recursively \ref{ordersim} follows.

\paragraph{}We proceed by induction on $r$ to prove the statement. If $r=1$, then $M$ is Nash diffeomorphic to an open simplex, so by Lemma \ref{simnashim} $M$ is Nash diffeomorphic to $\R^d$. Assume the result true for $r-1$. As $M\setminus\Phi(\sigma_r)$ is connected, there exists a surjective Nash map $h_1:\R^d\to M\setminus\Phi(\sigma_r)$. Let us check that the statement is also true for $r$. To that end, it is enough to show that \em $M$ is the image under a surjective Nash map $h_0:M\setminus\Phi(\sigma_r)\to M$\em. Once this is done, $h:=h_0\circ h_1:\R^d\to M$ is the surjective Nash map we sought. 

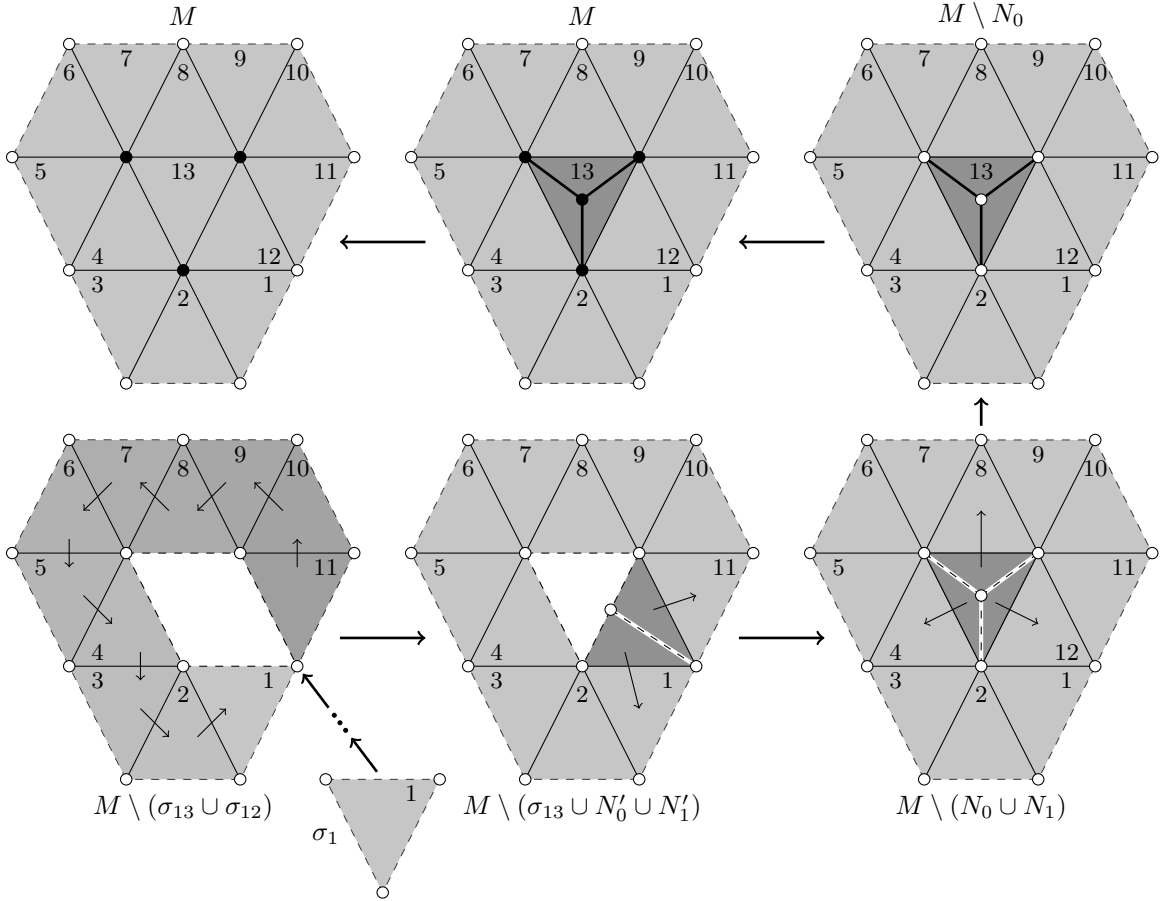
\begin{figure}[!ht]
\begin{center}
\begin{tikzpicture}[scale=0.75]


\draw[fill=gray!60,opacity=0.75,dashed] (2,9) -- (0,13) -- (1,15) -- (5,15) -- (6,13) -- (4,9) -- (2,9);
\draw (1,15) -- (4,9);
\draw (3,15) -- (5,11);
\draw (0,13) -- (6,13);
\draw (5,15) -- (2,9);
\draw (1,11) -- (5,11);
\draw (3,15) -- (1,11);

\draw[fill=white,draw] (2,9) circle (1mm);
\draw[fill=white,draw] (0,13) circle (1mm);
\draw[fill=white,draw] (1,15) circle (1mm);
\draw[fill=white,draw] (5,15) circle (1mm);
\draw[fill=white,draw] (6,13) circle (1mm);
\draw[fill=white,draw] (4,9) circle (1mm);
\draw[fill=white,draw] (3,15) circle (1mm);
\draw[fill=white,draw] (1,11) circle (1mm);
\draw[fill=white,draw] (5,11) circle (1mm);

\draw[fill=black,draw] (3,11) circle (1mm);
\draw[fill=black,draw] (2,13) circle (1mm);
\draw[fill=black,draw] (4,13) circle (1mm);

\draw (4.5,10.75) node{\footnotesize$1$};
\draw (3,10.5) node{\footnotesize$2$};
\draw (1.5,10.75) node{\footnotesize$3$};
\draw (1.5,11.25) node{\footnotesize$4$};
\draw (0.5,12.75) node{\footnotesize$5$};
\draw (1,14.5) node{\footnotesize$6$};
\draw (2,14.75) node{\footnotesize$7$};
\draw (3,14.5) node{\footnotesize$8$};
\draw (4,14.75) node{\footnotesize$9$};
\draw (5,14.5) node{\footnotesize$10$};
\draw (5.5,12.75) node{\footnotesize$11$};
\draw (4.5,11.25) node{\footnotesize$12$};
\draw (3,12.75) node{\footnotesize$13$};
\draw (3,15.5) node{\small$M$};


\draw[fill=gray!60,opacity=0.75,dashed] (9,9) -- (7,13) -- (8,15) -- (12,15) -- (13,13) -- (11,9) -- (9,9);
\draw (8,15) -- (11,9);
\draw (10,15) -- (12,11);
\draw (7,13) -- (13,13);
\draw (12,15) -- (9,9);
\draw (8,11) -- (12,11);
\draw (10,15) -- (8,11);

\draw[fill=white,draw] (9,9) circle (1mm);
\draw[fill=white,draw] (7,13) circle (1mm);
\draw[fill=white,draw] (8,15) circle (1mm);
\draw[fill=white,draw] (12,15) circle (1mm);
\draw[fill=white,draw] (13,13) circle (1mm);
\draw[fill=white,draw] (11,9) circle (1mm);
\draw[fill=white,draw] (10,15) circle (1mm);
\draw[fill=white,draw] (8,11) circle (1mm);
\draw[fill=white,draw] (12,11) circle (1mm);

\draw[fill=gray,opacity=0.75] (10,11) -- (9,13) -- (11,13) -- (10,11);

\draw[fill=black,draw] (10,11) circle (1mm);
\draw[fill=black,draw] (9,13) circle (1mm);
\draw[fill=black,draw] (11,13) circle (1mm);
\draw[fill=black,draw] (10,12.25) circle (1mm);

\draw[line width=1pt] (10,11) -- (10,12.25);
\draw[line width=1pt] (9,13) -- (10,12.25);
\draw[line width=1pt] (11,13) -- (10,12.25);

\draw (11.5,10.75) node{\footnotesize$1$};
\draw (10,10.5) node{\footnotesize$2$};
\draw (8.5,10.75) node{\footnotesize$3$};
\draw (8.5,11.25) node{\footnotesize$4$};
\draw (7.5,12.75) node{\footnotesize$5$};
\draw (8,14.5) node{\footnotesize$6$};
\draw (9,14.75) node{\footnotesize$7$};
\draw (10,14.5) node{\footnotesize$8$};
\draw (11,14.75) node{\footnotesize$9$};
\draw (12,14.5) node{\footnotesize$10$};
\draw (12.5,12.75) node{\footnotesize$11$};
\draw (11.5,11.25) node{\footnotesize$12$};
\draw (10,12.75) node{\footnotesize$13$};
\draw (10,15.5) node{\small$M$};


\draw[fill=gray!60,opacity=0.75,dashed] (16,9) -- (14,13) -- (15,15) -- (19,15) -- (20,13) -- (18,9) -- (16,9);
\draw (15,15) -- (18,9);
\draw (17,15) -- (19,11);
\draw (14,13) -- (20,13);
\draw (19,15) -- (16,9);
\draw (15,11) -- (19,11);
\draw (17,15) -- (15,11);

\draw[fill=white,draw] (16,9) circle (1mm);
\draw[fill=white,draw] (14,13) circle (1mm);
\draw[fill=white,draw] (15,15) circle (1mm);
\draw[fill=white,draw] (19,15) circle (1mm);
\draw[fill=white,draw] (20,13) circle (1mm);
\draw[fill=white,draw] (18,9) circle (1mm);
\draw[fill=white,draw] (17,15) circle (1mm);
\draw[fill=white,draw] (15,11) circle (1mm);
\draw[fill=white,draw] (19,11) circle (1mm);

\draw[fill=gray,opacity=0.75] (17,11) -- (16,13) -- (18,13) -- (17,11);

\draw[line width=1pt] (17,11) -- (17,12.25);
\draw[line width=1pt] (16,13) -- (17,12.25);
\draw[line width=1pt] (18,13) -- (17,12.25);

\draw[fill=white,draw] (17,11) circle (1mm);
\draw[fill=white,draw] (16,13) circle (1mm);
\draw[fill=white,draw] (18,13) circle (1mm);
\draw[fill=white,draw] (17,12.25) circle (1mm);

\draw (18.5,10.75) node{\footnotesize$1$};
\draw (17,10.5) node{\footnotesize$2$};
\draw (15.5,10.75) node{\footnotesize$3$};
\draw (15.5,11.25) node{\footnotesize$4$};
\draw (14.5,12.75) node{\footnotesize$5$};
\draw (15,14.5) node{\footnotesize$6$};
\draw (16,14.75) node{\footnotesize$7$};
\draw (17,14.5) node{\footnotesize$8$};
\draw (18,14.75) node{\footnotesize$9$};
\draw (19,14.5) node{\footnotesize$10$};
\draw (19.5,12.75) node{\footnotesize$11$};
\draw (18.5,11.25) node{\footnotesize$12$};
\draw (17,12.75) node{\footnotesize$13$};
\draw (17,15.5) node{\small$M\setminus N_0$};


\draw[dashed] (2,2) -- (0,6) -- (1,8) -- (5,8) -- (6,6) -- (4,2) -- (2,2);
\draw[fill=gray,opacity=0.75,draw=none] (5,4) -- (6,6) -- (4,6) -- (5,4);
\draw[fill=gray!96,opacity=0.75,draw=none] (6,6) -- (4,6) -- (5,8) -- (6,6);
\draw[fill=gray!92,opacity=0.75,draw=none] (4,6) -- (5,8) -- (3,8) -- (4,6);
\draw[fill=gray!88,opacity=0.75,draw=none] (3,8) -- (4,6) -- (2,6) -- (3,8);
\draw[fill=gray!84,opacity=0.75,draw=none] (2,6) -- (3,8) -- (1,8) -- (2,6);
\draw[fill=gray!80,opacity=0.75,draw=none] (1,8) -- (2,6) -- (0,6) -- (1,8);
\draw[fill=gray!76,opacity=0.75,draw=none] (2,6) -- (0,6) -- (1,4) -- (2,6);
\draw[fill=gray!72,opacity=0.75,draw=none] (1,4) -- (3,4) -- (2,6) -- (1,4);
\draw[fill=gray!68,opacity=0.75,draw=none] (1,4) -- (3,4) -- (2,2) -- (1,4);
\draw[fill=gray!64,opacity=0.75,draw=none] (2,2) -- (3,4) -- (4,2) -- (2,2);
\draw[fill=gray!60,opacity=0.75,draw=none] (3,4) -- (4,2) -- (5,4) -- (3,4);

\draw (1,8) -- (2,6);
\draw[dashed] (2,6) -- (3,4);
\draw (3,4) -- (4,2);
\draw (3,8) -- (4,6);
\draw[dashed] (4,6) -- (5,4);
\draw (0,6) -- (2,6);
\draw[dashed] (2,6) -- (4,6);
\draw (4,6) -- (6,6);
\draw (5,8) -- (4,6);
\draw (3,4) -- (2,2);
\draw (1,4) -- (3,4);
\draw[dashed] (3,4) -- (5,4);
\draw (3,8) -- (1,4);

\draw[fill=white,draw] (2,2) circle (1mm);
\draw[fill=white,draw] (0,6) circle (1mm);
\draw[fill=white,draw] (1,8) circle (1mm);
\draw[fill=white,draw] (5,8) circle (1mm);
\draw[fill=white,draw] (6,6) circle (1mm);
\draw[fill=white,draw] (4,2) circle (1mm);
\draw[fill=white,draw] (3,8) circle (1mm);
\draw[fill=white,draw] (1,4) circle (1mm);
\draw[fill=white,draw] (5,4) circle (1mm);

\draw[fill=white,draw] (3,4) circle (1mm);
\draw[fill=white,draw] (2,6) circle (1mm);
\draw[fill=white,draw] (4,6) circle (1mm);

\draw (4.5,3.75) node{\footnotesize$1$};
\draw (3,3.5) node{\footnotesize$2$};
\draw (1.5,3.75) node{\footnotesize$3$};
\draw (1.5,4.25) node{\footnotesize$4$};
\draw (0.5,5.75) node{\footnotesize$5$};
\draw (1,7.5) node{\footnotesize$6$};
\draw (2,7.75) node{\footnotesize$7$};
\draw (3,7.5) node{\footnotesize$8$};
\draw (4,7.75) node{\footnotesize$9$};
\draw (5,7.5) node{\footnotesize$10$};
\draw (5.5,5.75) node{\footnotesize$11$};
\draw (3,1.5) node{\small$M\setminus(\sigma_{13}\cup\sigma_{12})$};


\draw[->] (5,5.75) -- (5,6.25);
\draw[->] (4.75,6.75) -- (4.25,7.25);
\draw[->] (3.75,7.25) -- (3.25,6.75);
\draw[->] (2.75,6.75) -- (2.25,7.25);
\draw[->] (1.75,7.25) -- (1.25,6.75);
\draw[->] (1,6.25) -- (1,5.75);
\draw[->] (1.25,5.25) -- (1.75,4.75);
\draw[->] (2.25,4.25) -- (2.25,3.75);
\draw[->] (2.25,3.25) -- (2.75,2.75);
\draw[->] (3.25,2.75) -- (3.75,3.25);


\draw[fill=gray!60,opacity=0.75,dashed] (9,2) -- (7,6) -- (8,8) -- (12,8) -- (13,6) -- (11,2) -- (9,2);
\draw[fill=white,draw=none] (10,4) -- (9,6) -- (11,6) -- (10,4);
\draw[fill=gray,opacity=0.75,draw=none] (10,4) -- (12,4) -- (11,6) -- (10,4);

\draw (8,8) -- (9,6);
\draw[dashed] (9,6) -- (10,4);
\draw (10,4) -- (11,2);
\draw (10,8) -- (12,4);
\draw (7,6) -- (9,6);
\draw[dashed] (9,6) -- (11,6);
\draw (11,6) -- (13,6);
\draw (12,8) -- (11,6);
\draw[dashed] (11,6) -- (10,4);
\draw (10,4) -- (9,2);
\draw (8,4) -- (12,4);
\draw (10,8) -- (8,4);

\draw[line width=1.5pt,white] (10.5,5) -- (12,4);
\draw[dashed] (10.5,5) -- (12,4);

\draw[fill=white,draw] (9,2) circle (1mm);
\draw[fill=white,draw] (7,6) circle (1mm);
\draw[fill=white,draw] (8,8) circle (1mm);
\draw[fill=white,draw] (12,8) circle (1mm);
\draw[fill=white,draw] (13,6) circle (1mm);
\draw[fill=white,draw] (11,2) circle (1mm);
\draw[fill=white,draw] (10,8) circle (1mm);
\draw[fill=white,draw] (8,4) circle (1mm);
\draw[fill=white,draw] (12,4) circle (1mm);

\draw[fill=white,draw] (10,4) circle (1mm);
\draw[fill=white,draw] (9,6) circle (1mm);
\draw[fill=white,draw] (11,6) circle (1mm);

\draw[fill=white,draw] (10.5,5) circle (1mm);

\draw (11.5,3.75) node{\footnotesize$1$};
\draw (10,3.5) node{\footnotesize$2$};
\draw (8.5,3.75) node{\footnotesize$3$};
\draw (8.5,4.25) node{\footnotesize$4$};
\draw (7.5,5.75) node{\footnotesize$5$};
\draw (8,7.5) node{\footnotesize$6$};
\draw (9,7.75) node{\footnotesize$7$};
\draw (10,7.5) node{\footnotesize$8$};
\draw (11,7.75) node{\footnotesize$9$};
\draw (12,7.5) node{\footnotesize$10$};
\draw (12.5,5.75) node{\footnotesize$11$};
\draw (10,1.5) node{\small$M\setminus(\sigma_{13}\cup N_0'\cup N_1')$};

\draw[->] (11.25,5) -- (12,5.25);
\draw[->] (10.75,4.25) -- (11,3.25);


\draw[fill=gray!60,opacity=0.75,dashed] (16,2) -- (14,6) -- (15,8) -- (19,8) -- (20,6) -- (18,2) -- (16,2);
\draw (15,8) -- (18,2);
\draw (17,8) -- (19,4);
\draw (14,6) -- (20,6);
\draw (19,8) -- (16,2);
\draw (15,4) -- (19,4);
\draw (17,8) -- (15,4);

\draw[fill=white,draw] (16,2) circle (1mm);
\draw[fill=white,draw] (14,6) circle (1mm);
\draw[fill=white,draw] (15,8) circle (1mm);
\draw[fill=white,draw] (19,8) circle (1mm);
\draw[fill=white,draw] (20,6) circle (1mm);
\draw[fill=white,draw] (18,2) circle (1mm);
\draw[fill=white,draw] (17,8) circle (1mm);
\draw[fill=white,draw] (15,4) circle (1mm);
\draw[fill=white,draw] (19,4) circle (1mm);

\draw[fill=gray,opacity=0.75] (17,4) -- (16,6) -- (18,6) -- (17,4);

\draw[line width=1.5pt,white] (17,4) -- (17,5.25);
\draw[line width=1.5pt,white] (16,6) -- (17,5.25);
\draw[line width=1.5pt,white] (18,6) -- (17,5.25);

\draw[dashed] (17,4) -- (17,5.25);
\draw[dashed] (16,6) -- (17,5.25);
\draw[dashed] (18,6) -- (17,5.25);

\draw[fill=white,draw] (17,4) circle (1mm);
\draw[fill=white,draw] (16,6) circle (1mm);
\draw[fill=white,draw] (18,6) circle (1mm);
\draw[fill=white,draw] (17,5.25) circle (1mm);

\draw (18.5,3.75) node{\footnotesize$1$};
\draw (17,3.5) node{\footnotesize$2$};
\draw (15.5,3.75) node{\footnotesize$3$};
\draw (15.5,4.25) node{\footnotesize$4$};
\draw (14.5,5.75) node{\footnotesize$5$};
\draw (15,7.5) node{\footnotesize$6$};
\draw (16,7.75) node{\footnotesize$7$};
\draw (17,7.5) node{\footnotesize$8$};
\draw (18,7.75) node{\footnotesize$9$};
\draw (19,7.5) node{\footnotesize$10$};
\draw (19.5,5.75) node{\footnotesize$11$};
\draw (18.5,4.25) node{\footnotesize$12$};
\draw (17,1.5) node{\small$M\setminus(N_0\cup N_1)$};

\draw[->] (17,5.75) -- (17,6.75);
\draw[->] (17.25,5.125) -- (18,4.75);
\draw[->] (16.75,5.125) -- (16,4.75);


\draw[fill=gray!60,opacity=0.75,dashed] (5.5,2) -- (7.5,2) -- (6.5,0) -- (5.5,2);

\draw[fill=white,draw] (5.5,2) circle (1mm);
\draw[fill=white,draw] (7.5,2) circle (1mm);
\draw[fill=white,draw] (6.5,0) circle (1mm);
\draw (7,1.75) node{\footnotesize$1$};
\draw (5.5,1) node{\small$\sigma_1$};

\draw[->,line width=1pt] (5.75,4.5) -- (7.25,4.5);
\draw[->,line width=1pt] (12.75,4.5) -- (14.25,4.5);
\draw[->,line width=1pt] (17,8.25) -- (17,8.75);
\draw[->,line width=1pt] (14.25,11.5) -- (12.75,11.5);
\draw[->,line width=1pt] (7.25,11.5) -- (5.75,11.5);

\draw[->,line width=1pt] (6.4,2.1333333333) -- (5.95,2.7333333333);
\draw[->,line width=1pt] (5.55,3.2666666666) -- (5.1,3.8666666666);
\draw[fill=black,draw] (5.65,3.1333333333) circle (1pt);
\draw[fill=black,draw] (5.75,3) circle (1pt);
\draw[fill=black,draw] (5.85,2.8666666666) circle (1pt);

\end{tikzpicture}
\end{center}
\caption{Strategy to prove Theorem \ref{mstone0}.\label{fig8}}
\end{figure}

\paragraph{} Let $\tau_1,\ldots,\tau_s$ be the facets of $\sigma_r$ that are facets of another simplex $\sigma_k$ for some $k=1,\ldots,r-1$. Let $\epsilon$ be either the intersection of the remaining facets $\tau_{s+1},\ldots,\tau_{d+1}$ of $\sigma_r$ if $s<d+1$ or $\sigma_r$ if $s=d+1$. Thus, $\epsilon$ is a face of $\sigma_r$ of dimension $s-1$. The $s$ facets of $\epsilon$ are the intersections of $\epsilon$ with the facets $\tau_1,\ldots,\tau_s$ of $\sigma_r$. Let $b$ be the barycenter of $\epsilon$ and let $\eta_i$ be the convex hull of $\tau_i\cup\{b\}$ for $i=1,\ldots,s$. By Lemma \ref{divide} the simplices $\eta_1,\ldots,\eta_s$ provides a triangulation of $\sigma$.

Let ${\mathfrak H}_\ell$ be the collection of all the faces of $\eta_1,\ldots,\eta_s$ of dimension $\ell$ for $\ell=0,\ldots,d-2$ and let ${\mathfrak H}_{d-1}$ be the collection of all the facets of $\eta_1,\ldots,\eta_s$ different from $\tau_1,\ldots,\tau_s$. Define
$$
N_\ell:=\bigcup_{\sigma\in{\mathfrak H}_\ell}\Phi(\sigma^0)\cap M
$$
and observe that $N_\ell$ is a closed Nash submanifold of $M\setminus\bigsqcup_{j=0}^{\ell-1}N_j$, where $\bigsqcup_{j=0}^{\ell-1}N_j:=\varnothing$ if $\ell=0$. By Lemma \ref{eraseresidual} there exists a surjective Nash map $g_\ell:M\setminus\bigsqcup_{j=0}^{\ell}N_j\to M\setminus\bigsqcup_{j=0}^{\ell-1}N_j$. Thus,
$$
g:=g_{d-1}\circ\cdots\circ g_0:M\setminus\bigsqcup_{j=0}^{d-1}N_j\to M
$$
is a surjective Nash map.

\paragraph{}Observe that $M\setminus\bigsqcup_{j=0}^{d-1}N_j=(M\setminus\Phi(\sigma_r))\cup\bigcup_{j=1}^s\Phi(\tau_j^0\cup\eta_j^0)$. For each $j=1,\ldots,s$ let $\sigma_{i_j}$ be the simplex of $K$ such that $\sigma_{i_j}\cap\sigma_r=\tau_j$. By Lemma \ref{homeo} there exists a semialgebraic homeomorphism $f_j:\sigma_{i_j}\setminus\tau_j\to(\sigma_{i_j}\setminus\tau_j)\cup(\tau_j^0\cup\eta_j^0)$ such that $f_j|_{\partial\sigma_{i_j}\setminus\tau_j}=\id_{\partial\sigma_{i_j}\setminus\tau_j}$.

Consider the semialgebraic homeomorphism
$$
f:M\setminus\Phi(\sigma_r)\to(M\setminus\Phi(\sigma_r))\cup\bigcup_{j=1}^s\Phi(\tau_j^0\cup\eta_j^0),\ x\mapsto\begin{cases}
x&\text{if $x\not\in\bigcup_{j=1}^s\Phi(\sigma_{i_j})$,}\\
\Phi(f_j(\Phi^{-1}(x)))&\text{if $x\in\Phi(\sigma_{i_j})$.}
\end{cases}
$$
Let $f':M\setminus\Phi(\sigma_r)\to M\setminus\bigsqcup_{j=0}^{d-1}N_j$ be a close Nash approximation of $f$ (use \ref{shiota}). As $f$ is a semialgebraic homeomorphism, $f'$ is by Lemma \ref{surapprox} surjective. Consequently $h_0:=f'\circ g:M\setminus\Phi(\sigma_r)\to M$ is a surjective Nash map, as required.
\qed

\section{Main properties of well-welded semialgebraic sets}\label{s7}

In this section we describe the main properties of well-welded semialgebraic sets. Given a continuous semialgebraic path $\alpha:[0,1]\to\R^n$ we define $\eta(\alpha)$ as the image $\alpha(A)$ of the smallest (finite) subset $A\subset(0,1)$ such that $\alpha|_{(0,1)\setminus A}$ is a Nash map. Recall that a semialgebraic set $\Ss\subset\R^n$ is \em well-welded \em if $\Ss$ is pure dimensional and for each pair of points $x,y\in\Ss$ there exists a continuous semialgebraic path $\alpha:[0,1]\to\Ss$ such that $\alpha(0)=x$, $\alpha(1)=y$ and ${\eta}(\alpha)\subset\Reg(\Ss)$.

The following two results provide examples of well-welded semialgebraic sets. Once we prove Theorem \ref{main} we will conclude that there are no more well-welded semialgebraic sets.

\begin{lem}\label{npww}
Let $\Ss\subset\R^n$ be a semialgebraic set that is connected by Nash paths. Then $\Ss$ is well-welded.
\end{lem}
\begin{proof}
As $\Ss$ is connected by Nash paths, we only have to check that $\Ss$ is pure dimensional. Suppose by contradiction that $\Ss$ is not pure dimensional. Pick $x\in\Ss$ and an open semialgebraic neighborhood $U\subset\R^n$ of $x$ such that $\dim(\Ss\cap U)<\dim(\Ss)$. Let $Y$ be the Zariski closure of $\Ss\cap U$ in $\R^n$ and pick a point $y\in\Ss\setminus Y$. Let $\alpha:[0,1]\to\Ss$ be a Nash path such that $\alpha(0)=x$ and $\alpha(1)=y$. As $\alpha^{-1}(\Ss\cap Y)\subset[0,1]$ is a neighborhood of the origin and $Y$ is an algebraic set, we deduce by the identity principle that $y\in\im(\alpha)\subset Y$, which is a contradiction. Thus, $\Ss$ is pure dimensional, as required.
\end{proof}

\begin{lem}\label{imp1}
Let $\Ss\subset\R^n$ be a pure dimensional semialgebraic set. Assume that there exists a Nash path $\alpha:[0,1]\to\Ss$ whose image meets all the connected components of $\Reg(\Ss)$. Then $\Ss$ is well-welded.
\end{lem}
\begin{proof}
Let $M_1,\ldots,M_r$ be the connected components of $\Reg(\Ss)$. Let $x_i\in M_i\cap\im(\alpha)$ and let $t_i\in(0,1)$ be such that $\alpha(t_i)=x_i$. We may assume $t_1<\cdots<t_r$. As $\Ss$ is pure dimensional, $\Ss=\bigcup_{i=1}^r\cl(M_i)\cap\Ss$. Let $y_1,y_2\in\Ss$ and assume $y_1\in\cl(M_i)\cap\Ss$ and $y_2\in\cl(M_j)\cap\Ss$ with $i<j$. By the Nash curve selection lemma there exist Nash arcs $\alpha_k:(-1,1)\to\R^n$ such that $\alpha_1((-1,0))\subset M_i$, $\alpha_1((-1,0))\subset M_j$ and $\alpha_k(0)=y_k$ for $k=1,2$. As $M_i$ and $M_j$ are connected Nash manifolds, there exists Nash paths $\gamma_1:[0,1]\to M_i$ and $\gamma_2:[0,1]\to M_j$ such that
$$
\gamma_k(0)=z_k:=\alpha_k(-\tfrac{1}{2})\quad\text{and}\quad\gamma_k(1)=\begin{cases}
x_i&\text{if $k=1$},\\
x_j&\text{if $k=2$}.
\end{cases}
$$ 
Consider the continuous semialgebraic path
$$
\beta:=(\alpha_1|_{[-\frac{1}{2},0]})^{-1}*\gamma_1*\alpha|_{[t_i,t_j]}*\gamma_2^{-1}*\alpha_2|_{[-\frac{1}{2},0]}
$$ 
that connects the points $y_1,y_2$ and satisfies 
$$
\eta(\beta)\subset\{z_1,x_i,x_j,z_2\}\subset\Reg(\Ss),
$$
see Figure \ref{fig9}. Consequently, $\Ss$ is well-welded, as required.
\end{proof}

\begin{figure}[!ht]
\begin{center}
\begin{tikzpicture}[scale=0.75]

\draw[fill=gray!60,opacity=0.75,dashed] (0,0) -- (2,0) -- (2,2) -- (0,2) -- (0,0);
\draw[fill=gray!60,opacity=0.75,dashed] (2,2) -- (5,2) -- (5,4) -- (2,4) -- (2,2);
\draw[fill=gray!60,opacity=0.75,dashed] (5,2) -- (8,2) -- (8,0) -- (5,0) -- (5,2);
\draw[fill=gray!60,opacity=0.75,dashed] (8,2) -- (11,2) -- (11,4) -- (8,4) -- (8,2);
\draw[fill=gray!60,opacity=0.75,dashed] (11,2) -- (14,2) -- (14,0) -- (11,0) -- (11,2);
\draw[fill=gray!60,opacity=0.75,dashed] (14,2) -- (14,4) -- (17,4) -- (17,2) -- (14,2);
\draw[fill=gray!60,opacity=0.75,dashed] (17,2) -- (17,0) -- (19,0) -- (19,2) -- (17,2);

\draw[line width=1pt] (0.5,1) parabola bend (0.5,1) (2,2) parabola bend (3.5,3) (5,2) parabola bend (6.5,1) (8,2) parabola bend (9.5,3) (11,2) parabola bend (12.5,1) (14,2) parabola bend (15.5,3) (17,2) parabola bend (18.5,1) (17,2);

\draw[line width=1pt] (2,3.75) -- (3,3.75) -- (3.5,3);
\draw[line width=1pt] (15.5,3) -- (16,3.75) -- (17,3.75);

\draw[fill=black,draw] (2,2) circle (0.75mm);
\draw[fill=black,draw] (5,2) circle (0.75mm);
\draw[fill=black,draw] (8,2) circle (0.75mm);
\draw[fill=black,draw] (11,2) circle (0.75mm);
\draw[fill=black,draw] (14,2) circle (0.75mm);
\draw[fill=black,draw] (17,2) circle (0.75mm);

\draw[fill=black,draw] (0.5,1) circle (0.75mm);
\draw[fill=black,draw] (3.5,3) circle (0.75mm);
\draw[fill=black,draw] (6.5,1) circle (0.75mm);
\draw[fill=black,draw] (9.5,3) circle (0.75mm);
\draw[fill=black,draw] (12.5,1) circle (0.75mm);
\draw[fill=black,draw] (15.5,3) circle (0.75mm);
\draw[fill=black,draw] (18.5,1) circle (0.75mm);
\draw[fill=black,draw] (2,3.75) circle (0.75mm);
\draw[fill=black,draw] (17,3.75) circle (0.75mm);
\draw[fill=black,draw] (3,3.75) circle (0.75mm);
\draw[fill=black,draw] (16,3.75) circle (0.75mm);

\draw (4.7,3.75) node{\small$M_i$};
\draw (5.5,0.25) node{\small$M_{i+1}$};
\draw (13.5,0.25) node{\small$M_{j-1}$};
\draw (14.3,3.75) node{\small$M_j$};
\draw (3.5,2.75) node{\small$x_i$};
\draw (6.5,1.25) node{\small$x_{i+1}$};
\draw (12.5,1.25) node{\small$x_{j-1}$};
\draw (15.5,2.75) node{\small$x_j$};
\draw (1.7,3.75) node{\small$y_1$};
\draw (17.4,3.75) node{\small$y_2$};
\draw (9.5,2.5) node{\small${\rm Im}(\alpha)$};
\draw (3.4,3.75) node{\small$z_1$};
\draw (15.6,3.75) node{\small$z_2$};

\end{tikzpicture}
\end{center}
\caption{Sketch of proof of Lemma \ref{imp1}.\label{fig9}}
\end{figure}
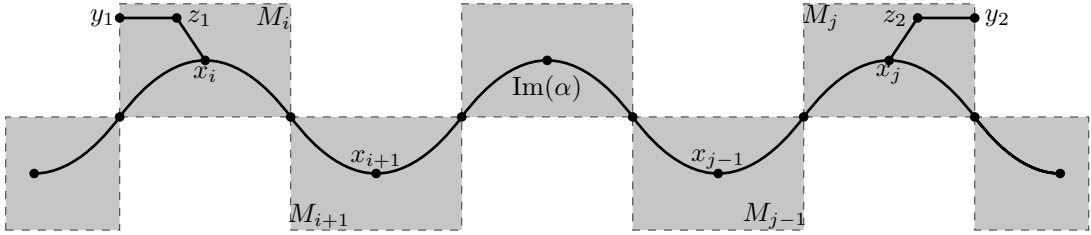

\subsection{Basic properties of well-welded semialgebraic sets I}
We show next that well-welded semialgebraic sets have irreducible Zariski closure and well-welded closure. We will prove later in \ref{bpii} that they are in fact irreducible and that the image of a well-welded semialgebraic set under a Nash map is again well-welded. 

\begin{lem}\label{zarirred}
Let $\Ss\subset\R^n$ be a well-welded semialgebraic set. Then $\ol{\Ss}^{\zar}$ is irreducible.
\end{lem}
\begin{proof}
Suppose by contradiction that $\ol{\Ss}^{\zar}$ is reducible. Let $X_1,\ldots,X_r$ be the irreducible components of $\ol{\Ss}^{\zar}$. We have 
$$
\Sing(\ol{\Ss}^{\zar})=\bigcup_{i=1}^r\Sing(X_i)\cup\bigcup_{i\neq j}(X_i\cap X_j).
$$
We claim: \em $\Ss\cap X_i\setminus\Sing(\ol{\Ss}^{\zar})\neq\varnothing$ for $i=1,\ldots,r$\em. 

Suppose that $\Ss\cap X_1\setminus\Sing(\ol{\Ss}^{\zar})=\varnothing$. Then $\Ss\subset\Sing(X_1)\cup\bigcup_{i=2}^r(\Ss\cap X_i)$, so 
$$
\bigcup_{i=1}^rX_i=\ol{\Ss}^{\zar}\subset\Sing(X_1)\cup\bigcup_{i=2}^rX_i, 
$$
which is a contradiction. 

Let $x_i\in\Ss\cap X_i\setminus\Sing(\ol{\Ss}^{\zar})$ for $i=1,2$ and let $f_1\in\R[\x]$ be a polynomial equation of $X_1$. Let $\alpha:[0,1]\to\Ss$ be a (continuous) semialgebraic map such that $\alpha(0)=x_1$, $\alpha(1)=x_2$ and $\eta(\alpha)\subset\Reg(\Ss)$. Let 
$$
t_0:=\inf\{t\in(0,1):\ \alpha(t)\not\in X_1\}.
$$
Note that $\alpha(t_0)\not\in\eta(\alpha)$ because $\alpha(t_0)\in X_1\cap X_i\subset\Sing(\ol{\Ss}^{\zar})$ for some $i\neq1$ and $\eta(\alpha)\subset\Reg(\Ss)$. Let $\veps>0$ be such that $\alpha|_{(t_0-\veps,t_0+\veps)}$ is a Nash map, $\alpha((t_0-\veps,t_0))\subset X_1$ and $\alpha((t_0,t_0+\veps))\cap X_1=\varnothing$. As $f_1\circ\alpha|_{(t_0-\veps,t_0+\veps)}$ is a Nash function such that $f_1\circ\alpha|_{(t_0-\veps,t_0)}\equiv0$, we have $f_1\circ\alpha|_{(t_0-\veps,t_0+\veps)}\equiv0$, which is a contradiction. Consequently, $\ol{\Ss}^{\zar}$ is irreducible, as required.
\end{proof}

\begin{lem}\label{chain}
Let $\Ss\subset\Tt\subset\R^n$ be semialgebraic sets such that $\Ss$ is well-welded and $\Tt\subset\cl(\Ss)$. Then $\Tt$ is well-welded.
\end{lem}
\begin{proof}
Let $x_1,x_2\in\Tt$. As $\Ss$ is pure dimensional, $\cl(\Reg(\Ss))=\cl(\Ss)=\cl(\Tt)$. By the Nash curve selection lemma there exist Nash arcs $\alpha_k:(-1,1)\to\R^n$ such that $\alpha_k((-1,0))\subset\Reg(\Ss)$, $\alpha_k(0)=x_k$ for $k=1,2$. Consider the points $y_k:=\alpha_k(-\frac{1}{2})$ for $k=1,2$. As $\Ss$ is well-welded, there exists a continuous semialgebraic path $\alpha_3:[0,1]\to\Ss$ such that $\alpha_3(0)=y_1$, $\alpha_3(1)=y_2$ and $\eta(\alpha_3)\subset\Reg(\Ss)$. As $\ol{\Ss}^{\zar}=\ol{\Tt}^{\zar}$, we have 
$$
\Reg(\Ss)=\Int_{\Reg(\ol{\Ss}^{\zar})}(\Ss\setminus\Sing(\ol{\Ss}^{\zar}))\subset\Int_{\Reg(\ol{\Ss}^{\zar})}(\Tt\setminus\Sing(\ol{\Ss}^{\zar}))=\Reg(\Tt).
$$
The continuous semialgebraic path $\alpha:=(\alpha_1|_{[-\frac{1}{2},0]})^{-1}*\alpha_3*\alpha_2|_{[-\frac{1}{2},0]}$ connects $x_1$ with $x_2$ and $\eta(\alpha)\subset\eta(\alpha_3)\cup\{y_1,y_2\}\subset\Reg(\Ss)\subset\Reg(\Tt)$. Consequently, $\Tt$ is well-welded.
\end{proof}

\subsection{Description of well-welding in terms of bridges}
It is difficult to handle the definition of well-welded semialgebraic set. Our purpose next is to provide a characterization of well-welding in terms of the existence of Nash arcs between the connected components of the set of regular points of a semialgebraic set.

\begin{define}
Let $M_1,M_2\subset\R^n$ be two Nash manifolds. A \em (Nash) bridge between $M_1$ and $M_2$ \em is the image $\Gamma$ of a Nash arc $\alpha:(-1,1)\to\R^n$ such that $\alpha((-1,0))\subset M_1$ and $\alpha((0,1))\subset M_2$.
\end{define}

To lighten the presentation we write bridge when referring to Nash bridge. As a straightforward consequence of Proposition \ref{charww0} below we have the following.

\begin{cor}\label{charww}
Let $\Ss\subset\R^n$ be a pure dimensional semialgebraic set and let $M_1,\ldots M_r$ be the connected components of $\Reg(\Ss)$. The following assertions are equivalent:
\begin{itemize}
\item[(i)] $\Ss$ is well-welded.
\item[(ii)] We can reorder the indices $i=1,\ldots,r$ in such a way that there exist bridges $\Gamma_i\subset\Ss$ between $M_i$ and $\bigsqcup_{j=1}^{i-1}M_j$ for $i=2,\ldots,r$.
\end{itemize}
\end{cor}

When dealing with well-welded semialgebraic sets, we need often that a continuous semialgebraic path connecting two points avoids certain algebraic set. In this direction we present the following three results.

\begin{lem}\label{inclusion}
Let $\alpha:[0,1]\to\R^n$ be a continuous semialgebraic path and let $Y\subset\R^n$ be an algebraic set. Assume that $\im(\alpha)\not\subset Y$ and $\eta(\alpha)\cap Y=\varnothing$. Then $\alpha^{-1}(Y)$ is a finite set.
\end{lem}
\begin{proof}
Suppose $\alpha^{-1}(Y)$ is not a finite set. Then $\alpha^{-1}(Y)$ is a closed semialgebraic subset of the interval $[0,1]$ of dimension $1$. Let $A\subset(0,1)$ be the smallest (finite) subset of $(0,1)$ such that $\alpha|_{(0,1)\setminus A}$ is a Nash map. Let $C$ be a connected component of $[0,1]\setminus(A\cup\{0,1\})$ such that $\dim(C\cap\alpha^{-1}(Y))=1$. As $\alpha|_{C}$ is a Nash map, $C\subset\alpha^{-1}(Y)$. As $\alpha^{-1}(Y)$ is closed, $\cl(C)\subset\alpha^{-1}(Y)$. Notice that $\cl(C)\setminus C\subset A\cup\{0,1\}$. If $\cl(C)\setminus C=\{0,1\}$, we have $\cl(C)=[0,1]$, so $\im(\alpha)\subset Y$, which is a contradiction. Consequently, $A\neq\varnothing$ and $(\cl(C)\setminus C)\cap A\neq\varnothing$, so $\eta(\alpha)\cap Y\neq\varnothing$, against the hypothesis. Thus, $\alpha^{-1}(Y)$ is a finite set, as required.
\end{proof}

\begin{prop}\label{charww0}
Let $\Ss\subset\R^n$ be a pure dimensional semialgebraic set and let $M_1,\ldots, M_r$ be the connected components of $\Reg(\Ss)$. Let $Z\subset\R^n$ be an algebraic set such that $M_i\not\subset Z$ for $i=1,\ldots,r$. The following assertions are equivalent:
\begin{itemize}
\item[(i)] For each pair of points $x_1,x_2\in\Ss$ there exists a continuous semialgebraic path $\alpha:[0,1]\to\Ss$ such that $\alpha(0)=x_1$, $\alpha(1)=x_2$ and ${\eta}(\alpha)\subset\Reg(\Ss)\setminus Z$.
\item[(ii)] The indices $i=1,\ldots,r$ can be reordered to have bridges $\Gamma_i\subset\Ss$ between $M_i\setminus Z$ and $\bigsqcup_{j=1}^{i-1}(M_j\setminus Z)$ for $i=2,\ldots,r$.
\end{itemize}
\end{prop}
\begin{proof}
(ii) $\Longrightarrow$ (i) Let $x_1,x_2\in\Ss$. As $\Ss$ is pure dimensional, $\Reg(\Ss)\setminus Z$ is dense in $\Ss$ and we may assume $x_1\in\cl(M_1\setminus Z)$ and $x_2\in\cl(M_\ell\setminus Z)$ for some $\ell=1,\ldots,r$. By the Nash curve selection lemma there exist Nash paths $\alpha_k:(-1,1)\to\R^n$ such that $\alpha_1((-1,0))\subset M_1\setminus Z$, $\alpha_2((-1,0))\subset M_\ell\setminus Z$ and $\alpha_k(0)=x_k$ for $k=1,2$. Consider the points $y_k:=\alpha_k(-\frac{1}{2})$ for $k=1,2$. We proceed by induction on $\ell$. 

If $\ell=1$, then $y_1,y_2\in M_1\setminus Z$. As $M_1$ is a connected Nash manifold, there exists a Nash path $\beta:[0,1]\to M_1$ such that $\beta(0)=y_1$ and $\beta(1)=y_2$. Consider the continuous semialgebraic path $\alpha:=(\alpha_1|_{[-\frac{1}{2},0]})^{-1}*\beta*\alpha_2|_{[-\frac{1}{2},0]}$ that connects the points $x_1$ and $x_2$ and satisfies $\eta(\alpha)\subset\{y_1,y_2\}\subset\Reg(\Ss)\setminus Z$ and $\im(\alpha)\subset M_1\cup\{x_1,x_2\}\subset\Ss$.

Assume that for each point $y\in M_i$ for $i=2,\ldots,\ell-1$ there exists a continuous semialgebraic path $\alpha:[0,1]\to\Ss$ such that $\alpha(0)=x_1$, $\alpha(1)=y$ and $\eta(\alpha)\subset\Reg(\Ss)\setminus Z$. Let $\gamma:(-1,1)\to\Ss$ be a Nash arc such that $\gamma((-1,1))=\Gamma_\ell$, $\gamma((-1,0))\subset\bigsqcup_{j=1}^{\ell-1}(M_j\setminus Z)$ and $\gamma((0,1))\subset M_\ell\setminus Z$. Let $u:=\gamma(-\frac{1}{2})\in\bigsqcup_{j=1}^{\ell-1}(M_j\setminus Z)$ and $v:=\gamma(\frac{1}{2})\in M_\ell\setminus Z$. By induction hypothesis there exists a continuous semialgebraic path $\rho_1:[0,1]\to\Ss$ such that $\rho_1(0)=x_1$, $\rho_1(1)=u$ and $\eta(\rho_1)\subset\Reg(\Ss)\setminus Z$. As $M_\ell$ is a connected Nash manifold, there exists a Nash path $\rho_2:[0,1]\to M_\ell$ such that $\rho_2(0)=v$ and $\rho_2(1)=y_2$. Consider the continuous semialgebraic path $\alpha:=\rho_1*\gamma|_{[-\frac{1}{2},\frac{1}{2}]}*\rho_2*\alpha_2|_{[-\frac{1}{2},0]}$ that connects the points $x_1$ and $x_2$ and satisfies $\eta(\alpha)\subset\eta(\rho_1)\cup\{u,v,y_2\}\subset\Reg(\Ss)\setminus Z$ and $\im(\alpha)\subset\Ss$, see Figure \ref{fig10}.

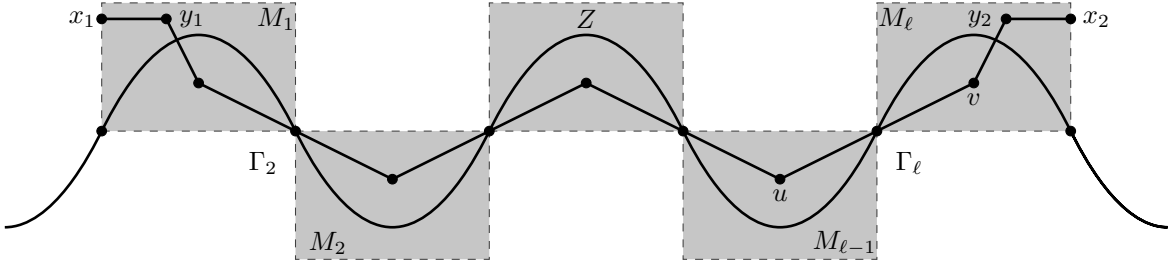
\begin{figure}[!ht]
\begin{center}
\begin{tikzpicture}[scale=0.85]

\draw[fill=gray!60,opacity=0.75,dashed] (2,2) -- (5,2) -- (5,4) -- (2,4) -- (2,2);
\draw[fill=gray!60,opacity=0.75,dashed] (5,2) -- (8,2) -- (8,0) -- (5,0) -- (5,2);
\draw[fill=gray!60,opacity=0.75,dashed] (8,2) -- (11,2) -- (11,4) -- (8,4) -- (8,2);
\draw[fill=gray!60,opacity=0.75,dashed] (11,2) -- (14,2) -- (14,0) -- (11,0) -- (11,2);
\draw[fill=gray!60,opacity=0.75,dashed] (14,2) -- (14,4) -- (17,4) -- (17,2) -- (14,2);

\draw[line width=1pt] (0.5,0.5) parabola bend (0.5,0.5) (2,2) parabola bend (3.5,3.5) (5,2) parabola bend (6.5,0.5) (8,2) parabola bend (9.5,3.5) (11,2) parabola bend (12.5,0.5) (14,2) parabola bend (15.5,3.5) (17,2) parabola bend (18.5,0.5) (17,2);

\draw[line width=1pt] (2,3.75) -- (3,3.75) -- (3.5,2.75);
\draw[line width=1pt] (15.5,2.75) -- (16,3.75) -- (17,3.75);
\draw[line width=1pt] (3.5,2.75) -- (6.5,1.25) -- (9.5,2.75) -- (12.5,1.25) -- (15.5,2.75);

\draw[fill=black,draw] (2,2) circle (0.75mm);
\draw[fill=black,draw] (5,2) circle (0.75mm);
\draw[fill=black,draw] (8,2) circle (0.75mm);
\draw[fill=black,draw] (11,2) circle (0.75mm);
\draw[fill=black,draw] (14,2) circle (0.75mm);
\draw[fill=black,draw] (17,2) circle (0.75mm);

\draw[fill=black,draw] (3.5,2.75) circle (0.75mm);
\draw[fill=black,draw] (6.5,1.25) circle (0.75mm);
\draw[fill=black,draw] (9.5,2.75) circle (0.75mm);
\draw[fill=black,draw] (12.5,1.25) circle (0.75mm);
\draw[fill=black,draw] (15.5,2.75) circle (0.75mm);
\draw[fill=black,draw] (2,3.75) circle (0.75mm);
\draw[fill=black,draw] (17,3.75) circle (0.75mm);
\draw[fill=black,draw] (3,3.75) circle (0.75mm);
\draw[fill=black,draw] (16,3.75) circle (0.75mm);

\draw (4.7,3.75) node{\small$M_1$};
\draw (5.5,0.25) node{\small$M_2$};
\draw (13.5,0.25) node{\small$M_{\ell-1}$};
\draw (14.3,3.75) node{\small$M_\ell$};
\draw (4.5,1.5) node{\small$\Gamma_2$};
\draw (14.5,1.5) node{\small$\Gamma_\ell$};
\draw (1.7,3.75) node{\small$x_1$};
\draw (17.4,3.75) node{\small$x_2$};
\draw (9.5,3.75) node{\small$Z$};
\draw (3.4,3.75) node{\small$y_1$};
\draw (15.6,3.75) node{\small$y_2$};
\draw (12.5,1) node{\small$u$};
\draw (15.5,2.5) node{\small$v$};

\end{tikzpicture}
\end{center}
\caption{Sketch of proof of implication (ii) $\Longrightarrow$ (i) in Proposition \ref{charww0}.\label{fig10}}
\end{figure}

(i) $\Longrightarrow$ (ii) If $\Reg(\Ss)$ is connected, there is nothing to prove. Assume that the result is true if the number of connected components of $\Reg(\Ss)$ is $<r$ and let us check that the result is also true if the number of connected components of $\Reg(\Ss)$ is $r$.

Let $x_i\in M_i\setminus Z$ for $i=1,\ldots,r$. By hypothesis one can construct a continuous semialgebraic path $\alpha:[0,1]\to\Ss$ such that $\alpha(0)=x_1$, $x_i\in\im(\alpha)$ for $i=2,\ldots,r$ and $\eta(\alpha)\subset\Reg(\Ss)\setminus Z$. By Lemma \ref{inclusion} $\im(\alpha)\cap Z$ is a finite set. 

We may reorder the connected components $M_i$ in such a way that if $i<j$, then 
$$
t_i:=\inf\{t\in(0,1):\ \alpha(t)\in M_i\}<\inf\{t\in(0,1):\ \alpha(t)\in M_j\}=:t_j.
$$ 
Let us check: \em $\Ss':=\cl(\Ss\setminus\cl(M_r))\cap\Ss$ is pure dimensional, satisfies the hypothesis of \em (i) \em and $\Reg(\Ss')=\bigsqcup_{i=1}^{r-1}M_i$\em. It is enough to show: 
\begin{itemize}
\item[(a)] \em $\Reg(\Ss')=\bigsqcup_{i=1}^{r-1}M_i$ is dense in $\Ss'$\em. Consequently, $\Ss'$ is pure dimensional.
\item[(b)] \em For each point $x\in\Ss'$ there exists a continuous semialgebraic path $\beta:[0,1]\to\Ss'$ such that $\beta(0)=x_1$, $\beta(1)=x$ and $\eta(\beta)\subset\Reg(\Ss')\setminus Z$\em.
\end{itemize}

Let us prove first (a). By Lemma \ref{zarirred} $\ol{\Ss}^{\zar}$ is irreducible because $\Ss$ is well-welded. As $\bigsqcup_{i=1}^{r-1}M_i$ is a non-empty open subset of $\Ss$, we conclude $\ol{\Ss'}^{\zar}=\ol{\Ss}^{\zar}$. Thus,
$$
\Reg(\Ss')=\Int_{\Reg(\ol{\Ss}^{\zar})}(\Ss'\setminus\Sing(\ol{\Ss}^{\zar}))\subset\Reg(\Ss)=\bigsqcup_{i=1}^rM_i.
$$
We deduce
\begin{multline*}
\bigsqcup_{i=1}^{r-1}M_i\subset\Reg(\Ss')\subset\Ss'\cap\Reg(\Ss)=\cl(\Reg(\Ss)\setminus\cl(M_r))\cap\Reg(\Ss)\\
=\cl\Big(\bigsqcup_{i=1}^rM_i\setminus\cl(M_r)\Big)\cap\bigsqcup_{i=1}^rM_i\subset\cl\Big(\bigsqcup_{i=1}^{r-1}M_i\Big)\cap\bigsqcup_{i=1}^rM_i=\bigsqcup_{i=1}^{r-1}M_i,
\end{multline*}
so $\Reg(\Ss')=\bigsqcup_{i=1}^{r-1}M_i$.

As $M_1,\ldots,M_r$ are the connected components of $\Reg(\Ss)$, we have $M_i\subset(\cl(M_i)\setminus\cl(M_r))\cap\Ss$, so $\cl((\cl(M_i)\setminus\cl(M_r))\cap\Ss)=\cl(M_i)$ for $i=1,\ldots,r-1$. As $\Ss=\bigcup_{i=1}^r\cl(M_i)\cap\Ss$,
\begin{equation*}
\begin{split}
\Ss'&=\cl\Big(\bigcup_{i=1}^r\cl(M_i)\cap\Ss\setminus\cl(M_r)\Big)\cap\Ss=\cl\Big(\bigcup_{i=1}^{r-1}\cl(M_i)\cap\Ss\setminus\cl(M_r)\Big)\cap\Ss\\
&=\bigcup_{i=1}^{r-1}\cl((\cl(M_i)\setminus\cl(M_r))\cap\Ss)\cap\Ss=\bigcup_{i=1}^{r-1}\cl(M_i)\cap\Ss=\cl(\Reg(\Ss'))\cap\Ss',
\end{split}
\end{equation*}
so $\Reg(\Ss')$ is dense in $\Ss'$. 

Let us show next (b). As $x\in\Ss'$, there exists an index $1\leq i\leq r-1$ such that $x\in\cl(M_i\setminus Z)$. By the Nash curve selection lemma there exists a Nash arc $\gamma_1:(-1,1)\to\R^n$ such that $\gamma_1((-1,0))\subset M_i\setminus Z$ and $\gamma_1(0)=x$. Let $s_i\in(0,t_{i+1})$ be such that $\alpha(s_i)\in M_i\setminus Z$. As $M_i$ is a connected Nash manifold, there exists a Nash path $\gamma_2:[0,1]\to M_i$ such that $\gamma_2(0)=\alpha(s_i)$ and $\gamma_2(1)=\gamma_1(-\frac{1}{2})$. The continuous semialgebraic path $\beta:=\alpha|_{[0,s_i]}*\gamma_2*\gamma_1|_{[-\frac{1}{2},0]}$ connects $x_1$ with $x$ and satisfies $\im(\beta)\subset\Ss'$ and
$$
\eta(\beta)\subset\eta(\alpha|_{[0,s_i]})\cup\{\alpha(s_i),\gamma_1(-\tfrac{1}{2})\}\subset\Reg(\Ss')\setminus Z.
$$
Consequently, $\Ss'$ satisfies the desired conditions.

By induction hypothesis we may assume that there exist bridges $\Gamma_i\subset\Ss$ between $M_i\setminus Z$ and $\bigsqcup_{j=1}^{i-1}(M_j\setminus Z)$ for $i=2,\ldots,r-1$. Recall that $\im(\alpha)\cap Z$ is a finite set. Let $\veps>0$ be such that the restriction $\alpha|_{(t_r-\veps,t_r+\veps)}$ is a Nash map, $\alpha((t_r,t_r+\veps))\subset M_r\setminus Z$ and $\alpha((t_r-\veps,t_r))\subset M_i\setminus Z$ for some $i=1,\ldots,r-1$. It holds (after reparameterizing) that $\Gamma_r:=\alpha((t_r-\veps,t_r+\veps))$ is a bridge between $\bigsqcup_{j=1}^{r-1}M_j\setminus Z$ and $M_r\setminus Z$, as required.
\end{proof}

\begin{cor}\label{clue}
Let $\Ss\subset\R^n$ be a well-welded semialgebraic set and let $Z\subset\R^n$ be an algebraic set that does not contain $\Ss$. Then for each pair of points $x,y\in\Ss$ there exists a continuous semialgebraic path $\alpha:[0,1]\to\Ss$ such that $\alpha(0)=x$, $\alpha(1)=y$ and $\eta(\alpha)\subset\Reg(\Ss)\setminus Z$.
\end{cor}
\begin{proof}
Let $M_1,\ldots,M_r$ be the connected components of $\Reg(\Ss)$. As $\Ss$ is well-welded, we can reorder the indices $i=1,\ldots,r$ in such a way that there exist bridges $\Gamma_i\subset\Ss$ between $M_i$ and $\bigsqcup_{j=1}^{i-1}M_j$ for $i=2,\ldots,r$ (see Corollary \ref{charww}). Denote $X=\ol{\Ss}^{\zar}$, which is an irreducible algebraic set not contained in $Z$. By Theorem \ref{hi1} there exist a non-singular algebraic set $X_1$ and a proper regular map $f:X_1\to X$ such that $\Ss\subset f(X_1)$ and
$$
f|_{X_1\setminus f^{-1}(\Sing(X))}:X_1\setminus f^{-1}(\Sing(X))\to X\setminus\Sing(X)
$$
is a diffeomorphism whose inverse map is also regular. 

Denote $\Lambda_i:=\cl(f^{-1}(\Gamma_i\setminus\Sing(X)))\cap f^{-1}(\Gamma_i)$ and $N_i:=f^{-1}(M_i)$ for $i=1,\ldots,r$. As $M_i\subset X\setminus\Sing(X)$, we have that $N_i\subset X_1\setminus f^{-1}(\Sing(X))$ is a Nash manifold. After shrinking $\Gamma_i$ if necessary, we may assume by Lemma \ref{irredgamma} that $\Lambda_i$ is a bridge between $N_i$ and $\bigsqcup_{j=1}^{i-1}N_j$ such that $\Lambda_i\setminus\bigsqcup_{j=1}^rN_j=\{q_i\}$ and $f(q_i)\in\Ss$ for some $q_i\in X_1$.

Consider the algebraic set $Z':=f^{-1}(Z\cap X)\subset X_1$. As $X$ is irreducible, also $X_1$ is irreducible. If $Z'$ contains $N_i$, then $Z'$ contains $X_1$, so $Z$ contains $X$, which is a contradiction. Consequently, $Z'$ contains no $N_i$. As $X_1$ is a Nash manifold, there exists by Lemma \ref{mod} a bridge $\Lambda_i'$ between $N_i\setminus Z'$ and $\bigsqcup_{j=1}^{i-1}(N_j\setminus Z')$ such that $\Lambda_i'\setminus\bigsqcup_{j=1}^rN_j=\{q_i\}$. Consequently, $\Gamma_i':=f(\Lambda_i')$ is a bridge between $M_i\setminus Z$ and $\bigsqcup_{j=1}^{i-1}(M_j\setminus Z)$ such that $\Gamma_i'\setminus\Reg(\Ss)=\{f(q_i)\}\subset\Ss$. By Proposition \ref{charww0} the result follows.
\end{proof}

\subsection{Basic properties of well-welded semialgebraic sets II}\label{bpii}
We prove next the remaining announced properties of well-welded semialgebraic sets.

\begin{cor}
Let $\Ss\subset\R^n$ be a well-welded semialgebraic set. Then $\Ss$ is irreducible.
\end{cor}
\begin{proof}
Denote $d:=\dim(\Ss)$. By \cite[Lem.3.6]{fg3} it is enough to prove that if $f$ is a Nash function on $\Ss$ whose zero-set has dimension $d$, then $f$ is identically zero.

Let $f$ be a Nash function on $\Ss$ whose zero-set has dimension $d$. Let $M_1,\ldots,M_r$ be the connected components of $\Reg(\Ss)$. We may assume that $f$ is identically zero only on $M_1,\ldots,M_k$. Observe that $k\geq1$ and assume by contradiction that $k<r$. Let $\Tt:=\{f=0\}\cap\bigsqcup_{j=k+1}^rM_j$, which is a semialgebraic set of dimension $<d$. Let $Y$ be the Zariski closure of $\Tt$, which has dimension $<d$. 

Pick points $x_1\in\bigsqcup_{i=1}^k(M_i\setminus Y)$ and $x_2\in\bigsqcup_{i=k+1}^r(M_i\setminus Y)$. By Corollary \ref{clue} there exists a continuous semialgebraic path $\alpha:[0,1]\to\Ss$ such that $\alpha(0)=x_1$, $\alpha(1)=x_2$ and $\eta(\alpha)\subset\Reg(\Ss)\setminus Y$. By Lemma \ref{inclusion} the inverse image $\alpha^{-1}(Y\cap\Ss)$ is a finite set, so there exist $\veps>0$ and $t_0\in(0,1)$ such that $\alpha|_{(t_0-\veps,t_0+\veps)}$ is a Nash map,
$$
\alpha((t_0-\veps,t_0))\subset\bigsqcup_{i=1}^k(M_i\setminus Y)\quad\text{and}\quad\alpha((t_0,t_0+\veps))\subset\bigsqcup_{i=k+1}^r(M_i\setminus Y).
$$
As $f\circ\alpha|_{(t_0-\veps,t_0)}$ is identically zero, $f\circ\alpha|_{(t_0-\veps,t_0+\veps)}$ is identically zero, so $\alpha((t_0,t_0+\veps))\subset\Tt\subset Y$, which is a contradiction. Consequently, $k=r$ and $f$ is identically zero, as required. 
\end{proof}

\begin{cor}\label{imageww}
Let $\Ss_1\subset\R^m$ and $\Ss_2\subset\R^n$ be semialgebraic sets. Assume that there exists a surjective Nash map $f:\Ss_1\to\Ss_2$ and that $\Ss_1$ is well-welded. Then $\Ss_2$ is well-welded. 
\end{cor}
\begin{proof}
Let $y_1,y_2\in\Ss_2$ and $x_1,x_2\in\Ss_1$ be such that $f(x_i)=y_i$. Let $Y$ be the Zariski closure of $f^{-1}(\Sing(\Ss_2))$, which has dimension $<\dim(\Ss_1)$ because $\Ss_1$ is irreducible. By Corollary \ref{clue} there exists a continuous semialgebraic path $\alpha:[0,1]\to\Ss_1$ such that $\alpha(0)=x_1$, $\alpha(1)=x_2$ and $\eta(\alpha)\subset\Reg(\Ss_1)\setminus Y$. The (continuous) semialgebraic map $\beta:=f\circ\alpha$ satisfies $\beta(0)=y_1$, $\beta(1)=y_2$ and $\eta(\beta)\subset f(f^{-1}(\Ss_2\setminus\Sing(\Ss_2)))=\Reg(\Ss_2)$. Consequently, $\Ss_2$ is well-welded.
\end{proof}

\begin{example}
There exist pure dimensional, irreducible semialgebraic sets that are not well-welded. Let 
$$
\Ss:=\{(4x^2-y^2)(4y^2-x^2)>0,y>0\}\cup\{(0,0)\}\subset\R^2,
$$ 
which is a pure dimensional irreducible semialgebraic set, see Figure \ref{fig11}. Let us check that it is not well-welded. 

\begin{figure}[!ht]
\begin{center}
\begin{tikzpicture}[scale=0.75]

\draw[fill=gray!60,opacity=0.75,dashed,draw] (4.5,1) -- (0.5,3) arc (153.43494882292201:116.56505117707798:4.47213595499958cm) -- (4.5,1) -- (8.5,3) arc (26.56505117707799:63.43494882292202:4.47213595499958cm) -- (4.5,1);

\draw[->] (4.5,0) -- (4.5,6);
\draw[->] (0,1) -- (9,1);

\draw[fill=black] (4.5,1) circle (0.75mm);
\draw[fill=black] (5.5,2) circle (0.75mm);
\draw[fill=black] (3.5,2) circle (0.75mm);

\draw (6.1,2.25) node{\footnotesize$(1,1)$};
\draw (2.9,2.25) node{\footnotesize$(-1,1)$};

\draw (7,4.25) node{\small$\Aa_1$};
\draw (2,4.25) node{\small$\Aa_2$};
\draw (5,5) node{\small$\Ss$};

\draw[fill=gray!60,opacity=0.75,dashed,draw] (15.5,1) -- (11.5,3) arc (153.43494882292201:116.56505117707798:4.47213595499958cm) -- (15.5,1) -- (19.5,3) arc (26.56505117707799:63.43494882292202:4.47213595499958cm) -- (15.5,1);

\draw[->] (15.5,0) -- (15.5,6);
\draw[->] (11,1) -- (20,1);

\draw[fill=white,draw] (15.5,1) circle (1mm);
\draw[fill=black] (16.5,2) circle (0.75mm);
\draw[fill=black] (14.5,2) circle (0.75mm);

\draw (17.1,2.25) node{\footnotesize$(1,1)$};
\draw (13.9,2.25) node{\footnotesize$(-1,1)$};

\draw (18,4.25) node{\small$\Aa_1$};
\draw (13,4.25) node{\small$\Aa_2$};
\draw (16.5,5) node{\small${\rm Reg}(\Ss)$};

\end{tikzpicture}
\end{center}
\caption{Semialgebraic set $\Ss:=\{(4x^2-y^2)(4y^2-x^2)>0,y>0\}\cup\{(0,0)\}\subset\R^2$ and its set of regular points.\label{fig11}}
\end{figure}
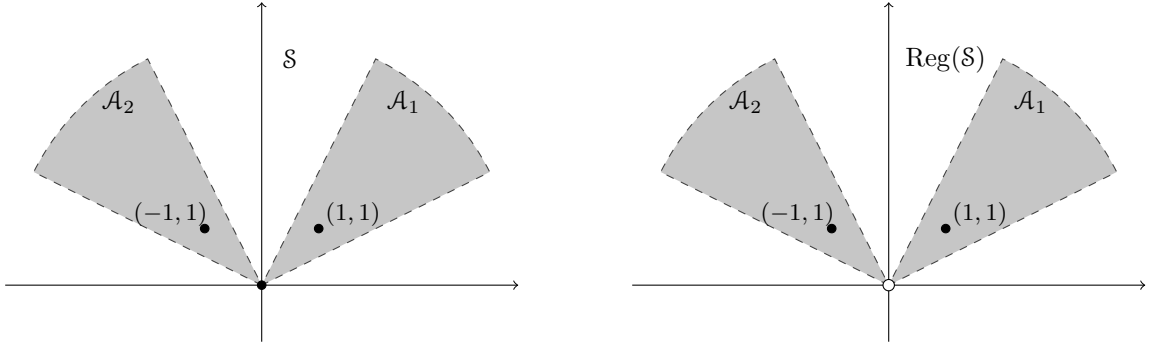

Observe first that $\Sing(\Ss)=\{(0,0)\}$. Pick the points $x:=(1,1),y:=(-1,1)\in\Ss$ and assume that $\Ss$ is well-welded. There exists a continuous semialgebraic path $\alpha:[0,1]\to\Ss$ such that $\alpha(0)=x$, $\alpha(1)=y$ and $\eta(\alpha)\subset\Ss\setminus\{(0,0)\}$. Consider the open semialgebraic sets $\Aa_1:=\Reg(\Ss)\cap\{x>0\}$ and $\Aa_2:=\Reg(\Ss)\cap\{x<0\}$, which satisfy $\Ss=\Aa_1\cup \Aa_2\cup\{(0,0)\}$. Let $t_0:=\inf(\alpha^{-1}(\Aa_2))>0$. Note that $\alpha(t_0)=(0,0)$ and there exists $\veps>0$ such that 
\begin{itemize}
\item[(i)] $\alpha((t_0-\veps,t_0))\subset \Aa_1$, $\alpha((t_0,t_0+\veps))\subset \Aa_2$,
\item[(ii)] $\alpha'(t)\neq0$ and $\alpha(t)\neq\alpha(t_0)$ for all $t\in(t_0-\veps,t_0+\veps)\setminus\{t_0\}$. 
\end{itemize}

The tangent line to $\im(\alpha|_{(t_0-\veps,t_0+\veps)})$ at $\alpha(t_0)$ is the line generated by the vector 
$$
w=\lim_{t\to t_0}\frac{\alpha(t)-\alpha(t_0)}{(t-t_0)^k}
$$
where $2k$ is the order of the series $\|\alpha\|^2$ at $t_0$. Note that
\begin{align*}
&w=\lim_{t\to t_0^{+}}\frac{\alpha(t)-\alpha(t_0)}{(t-t_0)^k}\in\cl(\Aa_1)\setminus\{(0,0)\},\\
&w=\lim_{t\to t_0^{-}}\frac{\alpha(t)-\alpha(t_0)}{(t-t_0)^k}\in\cl(\Aa_2\cup-\Aa_2)\setminus\{(0,0)\},
\end{align*}
which is a contradiction because $\cl(\Aa_1)\cap\cl(\Aa_2\cup-\Aa_2)=\{(0,0)\}$. Thus, $\Ss$ is not well-welded.\qed
\end{example}

\subsection{Alternative description of well-welded semialgebraic sets}
We describe next well-welded semialgebraic sets using piecewise analytic paths instead of continuous semialgebraic paths. We say that a continuous path $\alpha:[0,1]\to\R^n$ is \em piecewise analytic \em if there exists a finite set $A'\subset(0,1)$ such that $\alpha|_{(0,1)\setminus A'}$ is an analytic map. Let $A$ be the smallest set with the previous property and define $\eta(\alpha)=\alpha(A)$.

\begin{lem}
Let $\Ss\subset\R^n$ be a semialgebraic set. Then the following assertions are equivalent:
\begin{itemize}
\item[(i)] $\Ss$ is well-welded.
\item[(ii)] For each pair of points $x,y\in\Ss$ there exists a piecewise analytic path $\alpha:[0,1]\to\Ss$ such that $\alpha(0)=x$, $\alpha(1)=y$ and $\eta(\alpha)\subset\Reg(\Ss)$.
\end{itemize}
\end{lem}
\begin{proof}
The implication (i) $\Longrightarrow$ (ii) is immediate. For the converse, we proceed as follows. 

Let $M_1,\ldots,M_r$ be the connected components of $\Reg(\Ss)$. An \em analytic bridge \em between $M_i$ and $M_j$ is the image $\Lambda$ of an analytic arc $\alpha:(-1,1)\to\R^n$ such that $\alpha((-1,0))\subset M_i$ and $\alpha((0,1))\subset M_j$. Proceeding as in the proof of Proposition \ref{charww0} we can reorder the indices $i=1,\ldots,r$ in such a way that there exist analytic bridges $\Lambda_i\subset\Ss$ between $M_i$ and $\bigsqcup_{j=1}^{i-1}M_j$ for $i=2,\ldots,r$. By Lemma \ref{mod} we can substitute the analytic bridges $\Lambda_i\subset\Ss$ by (Nash) bridges $\Gamma_i\subset\Ss$ between $M_i$ and $\bigsqcup_{j=1}^{i-1}M_j$ for $i=2,\ldots,r$. By Corollary \ref{charww} $\Ss$ is well-welded.
\end{proof}

The following two results are the counterpart of Lemmas \ref{npww} and \ref{imp1} for analytic paths. As the proofs are pretty similar to those of Lemmas \ref{npww} and \ref{imp1}, we leave the details to the reader. 

\begin{lem}\label{apww}
Let $\Ss\subset\R^n$ be a semialgebraic set that is connected by analytic paths. Then $\Ss$ is well-welded.
\end{lem}

\begin{lem}\label{imp2}
Let $\Ss\subset\R^n$ be a pure dimensional semialgebraic set. Assume that there exists an analytic path $\alpha:[0,1]\to\Ss$ whose image meets all the connected components of $\Reg(\Ss)$. Then $\Ss$ is well-welded.
\end{lem}

\subsection{Strict transforms of well-welded semialgebraic sets}\setcounter{paragraph}{0}
We prove next that as it happens with irreducible arc-analytic sets \cite[Thm.2.6]{k} the strict transform of a well-welded semialgebraic set under a sequence of blow-ups is a well-welded semialgebraic set of its same dimension.

\begin{lem}\label{pww}
Let $X\subset\R^m$ and $Z\subset Y\subset\R^n$ be algebraic sets. Let $f:X\to Y$ be a proper regular map such that the restriction $f|_{X\setminus f^{-1}(Z)}:X\setminus f^{-1}(Z)\to Y\setminus Z$ is bijective. Let $\Ss\subset Y$ be a well-welded semialgebraic set of dimension $d$ such that $\Ss\not\subset Z$. Then the strict transform $\Ss_1:=\cl(f^{-1}(\Ss\setminus Z))\cap f^{-1}(\Ss)$ of $\Ss$ under $f$ is a well-welded semialgebraic set of dimension $d$. 
\end{lem}
\begin{proof}
The proof is conducted in several steps:

\paragraph{}\em We may assume $Y=\ol{\Ss}^{\zar}$ and $X=\ol{X\setminus f^{-1}(Z)}^{\zar}$.\em

Let $Y':=\ol{\Ss}^{\zar}$, $Z'=\ol{\Ss}^{\zar}\cap Z$ and $X'=\ol{f^{-1}(Y'\setminus Z')}^{\zar}$. Consider the proper regular map $f':=f|_{X'}:X'\to Y'$. The restriction $f'|_{X'\setminus {f'}^{-1}(Z')}:X'\setminus f'^{-1}(Z')\to Y'\setminus Z'$ is bijective because $X'\setminus f'^{-1}(Z')=f^{-1}(Y'\setminus Z')$. 

\paragraph{} We claim: \em $\Ss_1$ is pure dimensional\em. As the restriction $f|_{X\setminus Z}:X\setminus f^{-1}(Z)\to Y\setminus Z$ is proper and bijective, it is a semialgebraic homeomorphism. As $\Ss\setminus Z$ is pure dimensional of dimension $d$, also $f^{-1}(\Ss\setminus Z)$ is pure dimensional of dimension $d$. As $f^{-1}(\Ss\setminus Z)\subset\Ss_1\subset\cl(f^{-1}(\Ss\setminus Z))$, we conclude that $\Ss_1$ is pure dimensional of dimension $d$ as well.

\paragraph{} Observe that $\dim(X)=\dim(Y)=d$ and $\dim(Z)<d$. The algebraic set 
$$
Z_1:=\ol{f(\Sing(X)\cup\Sing(\Ss_1))\cup\Sing(Y)}^{\zar}\cup Z
$$ 
has by \cite[Thm.2.8.8]{bcr} dimension $<d$. We claim: \em $f^{-1}(Z_1)$ has dimension $<d$\em. 

As the restriction $f|_{X\setminus Z}:X\setminus f^{-1}(Z)\to Y\setminus Z$ is bijective, $\dim (f^{-1}(Z_1\setminus Z))=\dim(Z_1\setminus Z)<d$ by \cite[Thm.2.8.8]{bcr}. If $f^{-1}(Z)$ has dimension $d$, it contains an irreducible component of $X$, which is a contradiction because $X=\ol{X\setminus f^{-1}(Z)}^{\zar}$. Thus, $\dim(f^{-1}(Z))<d$, so $\dim(f^{-1}(Z_1))<d$. 

\paragraph{} The restriction $f|_{X\setminus f^{-1}(Z_1)}:X\setminus f^{-1}(Z_1)\to Y\setminus Z_1$ is a bijective proper regular map between the Nash manifolds $X\setminus f^{-1}(Z_1)$ and $Y\setminus Z_1$. Let $C$ be the set of critical points of $f|_{X\setminus f^{-1}(Z_1)}$, which is by \cite[Thm.9.6.2 \& Lem.9.6.3]{bcr} a semialgebraic set of dimension $<d$. Let $Z_2:=Z_1\cup\ol{f(C)}^{\zar}$, which is an algebraic set of dimension $<d$. Again $f^{-1}(Z_2)=f^{-1}(Z_2\setminus Z)\cup f^{-1}(Z)$ has dimension $<d$. Consequently, the restriction map $f|_{X\setminus f^{-1}(Z_2)}:X\setminus f^{-1}(Z_2)\to Y\setminus Z_2$ is a Nash diffeomorphism between the Nash manifolds $X\setminus f^{-1}(Z_2)$ and $Y\setminus Z_2$. 

\paragraph{} As $\Ss_1$ has dimension $d$, $\Sing(\Ss_1)$ has dimension $<d$ and $\ol{f(\Sing(\Ss_1))}^{\zar}$ has dimension $<d$. Define $Z_3:=\ol{f(\Sing(\Ss_1))}^{\zar}\cup Z_2$ and observe that $f^{-1}(Z_3)=f^{-1}(Z_3\setminus Z)\cup f^{-1}(Z)$ is an algebraic set of dimension $<d$.

\paragraph{}\label{coser0} Let $M_1,\ldots,M_r$ be the connected components of $\Reg(\Ss)$. By Proposition \ref{charww0} and Corollary \ref{clue}, we may reorder the indices $1,\ldots,r$ in such a way that for each $j=2,\ldots,r$ there exists a bridge $\Gamma_j\subset\Ss$ between $M_j\setminus Z_3$ and $\bigsqcup_{k=1}^{j-1}(M_k\setminus Z_3)$. We have $\Gamma_j\setminus\Reg(\Ss)=\{p_j\}$ for $j=2,\ldots,r$. Shrinking each bridge $\Gamma_j$, we may assume that $\Gamma_j\cap Z_3\subset\{p_j\}$ for $j=2,\ldots,r$. 

\paragraph{}\label{coser} Let $C_{i1},\ldots,C_{is_i}$ be the connected components of $M_i\setminus Z_3$. As $\Gamma_i\cap Z_3\subset\{p_i\}$, we may assume that $\Gamma_i$ is a bridge between $C_{i1}$ and $\bigsqcup_{k=1}^{i-1}\bigsqcup_{j=1}^{s_k}C_{kj}$. As $M_i$ is a connected Nash manifold, we may construct using Lemmas \ref{mod} and \ref{mcnp} bridges $\Gamma_{ij}\subset M_i$ between $C_{ij}$ and $\bigsqcup_{k=1}^{j-1}C_{ik}$.

\paragraph{} Denote the connected components of $\Reg(\Ss)\setminus Z_3$ with $N_1,\ldots,N_s$. By \ref{coser0} and \ref{coser} we may assume that there exist bridges $\Lambda_j\subset\Ss$ between $N_j$ and $\bigsqcup_{k=1}^{j-1}N_k$ such that the intersection $\Lambda_j\cap Z_3$ is either the empty-set or a singleton for $j=2,\ldots,s$.

Denote $\Lambda_j':=\cl(f^{-1}(\Lambda_j\setminus Z))\cap f^{-1}(\Lambda_j)\subset\Ss_1$. By Lemma \ref{irredgamma} $\Lambda_j'\cap f^{-1}(Z)$ is either the empty-set or a singleton $\{z_j\}$ and the curve germ $\Lambda'_{j,z_j}$ is irreducible. Denote $N_j':=f^{-1}(N_j)$. As $\Ss\setminus Z_3$ is Nash diffeomorphic to $f^{-1}(\Ss\setminus Z_3)$, it holds that $N_1',\ldots,N_s'$ are the connected components of the Nash manifold $f^{-1}(\Ss\setminus Z_3)$. Shrinking $\Lambda_j'$, we may assume that it is a bridge between $N_j'$ and $\bigsqcup_{i=1}^{j-1}N_i'$ for $j=2,\ldots,s$. 

\paragraph{} As $\Sing(\Ss_1)\subset f^{-1}(Z_3)$, we have $\Ss_1\setminus f^{-1}(Z_3)\subset\Reg(\Ss_1)$. Note that 
$$
\Ss_1\setminus f^{-1}(Z_3)=f^{-1}(\Ss\setminus Z_3)=\bigsqcup_{j=1}^sN_j'
$$ 
because $Z\subset Z_3$, $f^{-1}(\Ss\setminus Z)\subset\Ss_1\subset f^{-1}(\Ss)$ and
$$
f^{-1}(\Ss\setminus Z_3)=f^{-1}(\Ss\setminus Z)\setminus f^{-1}(Z_3)\subset\Ss_1\setminus f^{-1}(Z_3)\subset f^{-1}(\Ss\setminus Z_3).
$$

Let $M_1',\ldots,M_\ell'$ be the connected components of $\Reg(\Ss_1)$. As $\Sing(\Ss_1)\subset f^{-1}(Z_3)$,
$$
\bigsqcup_{j=1}^sN_j'=\Ss_1\setminus f^{-1}(Z_3)\subset\Reg(\Ss_1)=\bigsqcup_{k=1}^\ell M_k',
$$
so each $N_j'$ is contained in some $M_k'$. As $\dim(f^{-1}(Z_3))<d$, we have $\Reg(\Ss_1)\subset\cl(\bigsqcup_{j=1}^sN_j')$. Thus, for each $k=1,\ldots,\ell$ there exists $1\leq j\leq s$ such that $N_j'\subset M_k'$. Define 
$$
j(k):=\min\{j=1,\ldots,s:\ N_j'\subset M_k'\}
$$ 
and note that $j(k_1)\neq j(k_2)$ if $k_1\neq k_2$. We may assume $j(k_1)<j(k_2)$ if $k_1<k_2$. Observe that $\Lambda_{j(k)}'\subset\Ss_1$ is a bridge between $M_k'$ and $\bigsqcup_{i=1}^{k-1}M_i'$. We conclude by Corollary \ref{charww} that $\Ss_1$ is well-welded, as required.
\end{proof}

\begin{cor}\label{weldedbu}
Let $Y\subset X\subset\R^n$ be algebraic sets and let $\Ss\subset X$ be a well-welded semialgebraic set of dimension $d$ such that $\Ss\not\subset Y$. Let $(B(X,Y),\pi)$ be the blow-up of $X$ with center $Y$ and let $\Ss_1:=\cl(\pi^{-1}(\Ss\setminus Y))\cap\pi^{-1}(\Ss)$ be the strict transform of $\Ss$ under $\pi$. Then $\Ss_1$ is a well-welded semialgebraic set of dimension $d$.
\end{cor}

\section{Well-welded semialgebraic sets as Nash images of Euclidean spaces}\label{s8}

In this section we prove Theorem \ref{main}. The most difficult part, implication (vii) $\Longrightarrow$ (i), is approached in two steps. We prove first that each well-welded semialgebraic set is the image under a Nash map of a `checkerboard set' of its same dimension. Afterwards we show that a checkerboard set of dimension $d$ is the image under a Nash map of its set of regular points, which is a connected Nash manifold of dimension $d$ and consequently a Nash image of $\R^d$ by Theorem \ref{mstone0}. We define the boundary of a semialgebraic set $\Ss\subset\R^n$ as $\partial\Ss:=\cl(\Ss)\setminus\Reg(\Ss)$. It holds $\Sing(\Ss)\subset\partial\Ss$.

\subsection{Checkerboard sets}\label{chess}
A pure dimensional semialgebraic set $\Ss\subset\R^n$ is a \em checkerboard set \em (see Figure \ref{fig12}) if it satisfies the following properties:
\begin{itemize}
\item $\ol{\Ss}^{\zar}$ is a non-singular real algebraic set.
\item $\ol{\partial\Ss}^{\zar}$ is a normal-crossings divisor of $\ol{\Ss}^{\zar}$. 
\item $\Reg(\Ss)$ is connected.
\end{itemize} 

\begin{figure}[!ht]
\begin{center}
\begin{tikzpicture}[scale=0.75]

\draw[fill=gray!100,opacity=0.75,draw=none] (0,0) -- (0,5) -- (5,5) -- (5,0) -- (0,0);
\draw[fill=white,draw=none] (1,2.5) -- (1,4) -- (2.5,4) -- (2.5,2.5) -- (1,2.5);
\draw[fill=white,draw=none] (2.5,1) -- (4,1) -- (4,2.5) -- (2.5,2.5) -- (2.5,1);

\draw[line width=1pt] (0,0) -- (1.5,0);
\draw[line width=1pt,dashed] (0,0) -- (0,1.5);
\draw[line width=1pt] (0,1.5) -- (0,5);
\draw[line width=1pt] (3.5,5) -- (0,5);
\draw[line width=1pt,dashed] (3.5,5) -- (5,5);
\draw[line width=1pt] (5,0) -- (5,5);
\draw[line width=1pt,dashed] (1.5,0) -- (5,0);
\draw[line width=1pt] (1,2.5) -- (1,4) -- (2.5,4) -- (2.5,2.5) -- (1,2.5);
\draw[line width=1pt] (4,1) -- (2.5,1) -- (2.5,2.5) -- (4,2.5);
\draw[line width=1pt,dashed] (4,1) -- (4,2.5);

\draw[fill=white,draw] (0,0) circle (0.75mm);
\draw[fill=white,draw] (0,5) circle (0.75mm);
\draw[fill=black,draw] (5,5) circle (0.75mm);
\draw[fill=white,draw] (5,0) circle (0.75mm);
\draw[fill=black,draw] (3.5,5) circle (0.75mm);
\draw[fill=black,draw] (1.5,0) circle (0.75mm);
\draw[fill=white,draw] (0,1.5) circle (0.75mm);

\draw[fill=black,draw] (1,2.5) circle (0.75mm);
\draw[fill=black,draw] (1,4) circle (0.75mm);
\draw[fill=black,draw] (2.5,4) circle (0.75mm);
\draw[fill=black,draw] (2.5,2.5) circle (0.75mm);
\draw[fill=black,draw] (2.5,1) circle (0.75mm);
\draw[fill=white,draw] (4,1) circle (0.75mm);
\draw[fill=white,draw] (4,2.5) circle (0.75mm);

\draw[fill=gray!40,draw=none] (7,0) -- (7,5) arc (180:90:1cm)-- (13,6) arc (90:0:1cm)-- (14,0) arc (0:-90:1cm) -- (13,-1) -- (8,-1) arc (270:-180:1cm);
\draw[fill=gray!100,opacity=0.75,draw=none] (8,0) -- (8,5) -- (13,5) -- (13,0) -- (8,0);
\draw[fill=gray!40,draw=none] (9,2.5) -- (9,4) -- (10.5,4) -- (10.5,2.5) -- (9,2.5);
\draw[fill=gray!40,draw=none] (10.5,1) -- (12,1) -- (12,2.5) -- (10.5,2.5) -- (10.5,1);

\draw[line width=1pt] (8,0) -- (8,5) -- (13,5) -- (13,0) -- (8,0);
\draw[line width=1pt] (9,2.5) -- (9,4) -- (10.5,4) -- (10.5,2.5) -- (9,2.5);
\draw[line width=1pt] (10.5,1) -- (12,1) -- (12,2.5) -- (10.5,2.5) -- (10.5,1);

\draw (8,-1) -- (8,6);
\draw (9,-1) -- (9,6);
\draw (10.5,-1) -- (10.5,6);
\draw (12,-1) -- (12,6);
\draw (13,-1) -- (13,6);
\draw (7,0) -- (14,0);
\draw (7,1) -- (14,1);
\draw (7,2.5) -- (14,2.5);
\draw (7,4) -- (14,4);
\draw (7,5) -- (14,5);

\draw (3.75,4.25) node{\small$\Ss$};
\draw (3.75,3.75) node{\scriptsize checkerboard};
\draw (3.75,3.25) node{\scriptsize set};

\draw (6,2.5) node{\small$\leadsto$};

\draw (11.25,3.25) node{\small${\rm Cl}(\Ss)$};
\draw (9.75,4.25) node{\small$\partial\Ss$};
\draw (9.65,5.5) node{\small$\overline{\partial\Ss}^{\rm zar}$};
\draw (14,5.75) node{\small$\overline{\Ss}^{\rm zar}$};

\end{tikzpicture}
\end{center}
\caption{Checkerboard set $\Ss$, its closure $\cl(\Ss)$, its Zariski closure $\ol{\Ss}^{\zar}$, $\partial\Ss$ and the Zariski closure of $\partial\Ss$.\label{fig12}}
\end{figure}
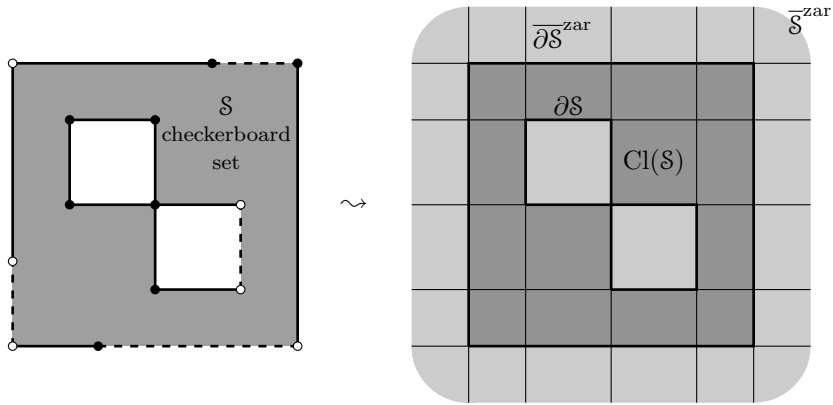

\begin{remark}\label{diff}
The difference $\Ss\setminus\ol{\partial\Ss}^{\zar}$ is a union of connected components of $\ol{\Ss}^{\zar}\setminus\ol{\partial\Ss}^{\zar}$.

Observe that $\Ss\setminus\ol{\partial\Ss}^{\zar}=\cl(\Ss)\setminus\ol{\partial\Ss}^{\zar}=\Reg(\Ss)\setminus\ol{\partial\Ss}^{\zar}$. Consequently, $\Ss\setminus\ol{\partial\Ss}^{\zar}$ is an open and closed subset of $\ol{\Ss}^{\zar}\setminus\ol{\partial\Ss}^{\zar}$, so it is a union of connected components of $\ol{\Ss}^{\zar}\setminus\ol{\partial\Ss}^{\zar}$.
\end{remark}

Each checkerboard set is a well-welded semialgebraic set by the following result.

\begin{lem}\label{ncww}
Let $X\subset\R^n$ be a non-singular algebraic set and let $Y\subset X$ be a normal-crossings divisor. Let $\Cc\subset X\setminus Y$ be a union of connected components of $X\setminus Y$ and let $\Ss\subset X$ be a semialgebraic set such that $\Cc\subset\Ss\subset\cl(\Cc)$. Then $\Ss$ is well-welded if and only if $\Ss$ is connected. 
\end{lem}
\begin{proof}
The `only if' implication is straightforward. We prove next the `if' implication.

Let $\Cc_1,\ldots,\Cc_r$ be the connected components of $\Cc$. We claim: \em there exist points $p_2,\ldots,p_r\in Y\cap\Ss$ such that, after reordering indices, $p_i\in\Ss\cap\cl(\Cc_i)\cap\bigcup_{j=1}^{i-1}\cl(\Cc_j)$ for $i=2,\ldots,r$\em.

We proceed by induction on $r$. If $r=1$, the claim is clear. Assume the result true if the number of connected components of $\Cc$ is $<r$ and let us see that it is also true when it is equal to $r$. As $\Ss$ is connected, there exist a continuous semialgebraic path $\alpha:[0,1]\to\Ss$ whose image meets all the connected components $\Cc_i$. We may reorder the indices $1,\ldots,r$ in such a way that if $i<j$, then 
$$
t_i:=\inf\{t\in(0,1):\ \alpha(t)\in\cl(\Cc_i)\cap\Ss\}\leq\inf\{t\in(0,1):\ \alpha(t)\in\cl(\Cc_j)\cap\Ss\}=:t_j.
$$ 
Consider the semialgebraic set $\Ss':=\cl(\Ss\setminus\cl(\Cc_r))\cap\Ss$. It holds: \em $\Ss'$ is connected\em. 

As $\Cc_1,\ldots,\Cc_r$ are the connected components of $\Cc$, we have $\cl(\Cc_i)=\cl(\cl(\Cc_i)\setminus\cl(\Cc_r))$ for $i\neq r$. Consequently,
\begin{equation}\label{intclear}
\Ss'=\cl(\Ss\setminus\cl(\Cc_r))\cap\Ss=\bigcup_{i=1}^{r-1}\cl(\cl(\Cc_i)\setminus\cl(\Cc_r))\cap\Ss=\bigcup_{i=1}^{r-1}\cl(\Cc_i)\cap\Ss
\end{equation}

Consider the connected semialgebraic set $\Tt:=\alpha([0,t_r])$. Observe that $\Tt\subset\Ss'$ and $\Tt\cap\cl(\Cc_i)\neq\varnothing$ for $i=1,\ldots,r-1$. Consequently, $\Ss'$ is connected. 

As $\Ss'\setminus Y=\bigsqcup_{i=1}^{r-1}\Cc_i$, there exist by induction hypothesis points $p_2,\ldots,p_{r-1}\in Y\cap\Ss'$ such that, after reordering indices, $p_i\in\Ss\cap\cl(\Cc_i)\cap\bigcup_{j=1}^{i-1}\cl(\Cc_j)$ for $i=2,\ldots,r-1$. As $\Ss=\Ss'\cup(\cl(\Cc_r)\cap\Ss)$ is connected and $\Ss'$ and $\cl(\Cc_r)\cap\Ss$ are closed in $\Ss$, there exists $p_r\in\Ss'\cap\cl(\Cc_r)$. By \eqref{intclear} there exists $1\leq i\leq r-1$ such that $p_r\in\cl(\Cc_i)\cap\cl(\Cc_r)\cap\Ss$. Notice that $\cl(\Cc_i)\cap\cl(\Cc_r)\subset Y$ because $\Cc$ is a union of connected components of $X\setminus Y$. Thus, the claim follows.

Next we prove: \em $\Ss_0:=\Cc\cup\{p_2,\ldots,p_r\}$ is a well-welded semialgebraic set\em. 

Fix $i=2,\ldots,r$ and let $1\leq j\leq i-1$ be such that $p_i\in\cl(\Cc_j)\cap\cl(\Cc_i)\subset Y$. As $Y$ is a normal-crossings divisor of $X$, there exists an open semialgebraic neighborhood $U\subset X$ of $p_i$ and a Nash diffeomorphism $\psi:U\to\R^d$ such that $\psi(0)=p_i$ and $\psi(U\cap Y)=\{x_1\cdots x_s=0\}$. We may assume 
$$
\{x_1>0,\ldots,x_s>0\}\subset\psi(\Cc_i\cap U)\quad\text{and}\quad\{\veps_1x_1>0,\ldots,\veps_sx_s>0\}\subset\psi(\Cc_j\cap U)
$$
where $\veps_k=\pm1$ for $k=1,\ldots,s$. Consider the Nash curve $\beta:=(\beta_1,\ldots,\beta_d):(-1,1)\to\R^d$ where
$$
\beta_k(t)=\begin{cases}
t&\text{if $\veps_k=-1$}\\
t^2&\text{if $\veps_k=1$}
\end{cases}
$$
for $k=1,\ldots,s$ and $\beta_k(t)=0$ for $k=s+1,\ldots,d$. Observe that 
$$
\beta((-1,0))\subset\{\veps_1x_1>0,\ldots,\veps_sx_s>0\}\quad\text{and}\quad
\beta((0,1))\subset\{x_1>0,\ldots,x_s>0\}.
$$
Consider the Nash curve $\gamma:=\psi^{-1}\circ\beta:(-1,1)\to\Cc_i\cup\Cc_j\cup\{p_i\}$ that satisfies $\gamma((-1,0))\subset\Cc_j$ and $\gamma((0,1))\subset\Cc_i$. Thus, $\Gamma_i:=\gamma((-1,1))\subset\Ss$ is a bridge between $\Cc_j$ and $\Cc_i$. By Corollary \ref{charww} $\Ss_0$ is well-welded.

As $\Ss_0\subset\Ss\subset\cl(\Ss_0)$, we conclude by Lemma \ref{chain} that $\Ss$ is well-welded, as required.
\end{proof}

The following result will allow us to lighten the presentation of the proof of implication (vii) $\Longrightarrow$ (i) of Theorem \ref{main}.

\begin{lem}\label{lighten}
Let $\Ss\subset\R^n$ be a pure dimensional semialgebraic set of dimension $d$. Suppose that ${\tt Sth}(\Ss)$ is connected, there exists a Nash manifold $M$ of dimension $d$ that contains $\Ss$ and the smallest Nash subset $Y$ of $M$ that contains $(\cl(\Ss)\cap M)\setminus{\tt Sth}(\Ss)$ is a Nash normal crossings divisor of $M$. Then there exists a Nash embedding $\varphi:M\hookrightarrow\R^m$ for some $m\geq1$ such that $\varphi(\Ss)$ is a checkerboard subset of $\R^m$, $\Reg(\varphi(\Ss))=\varphi({\tt Sth}(\Ss))$ and $\Sing(\varphi(\Ss))=\varphi({\tt NSth}(\Ss))$.
\end{lem}
\begin{proof}
As $\Ss$ is connected, we may assume that $M$ is connected. By Lemma \ref{retoque} there exist a Nash embedding $\varphi:M\hookrightarrow\R^m$ such that 
\begin{itemize}
\item[(i)] $\varphi(M)$ is a connected component of its Zariski closure $V$ in $\R^m$ that is a non-singular real algebraic set of dimension $d$. 
\item[(ii)] The Zariski closure $X$ of $\varphi(Y)$ in $\R^m$ is a normal crossing divisor of $V$ such that $\varphi(M)\cap X=\varphi(Y)$.
\end{itemize}
Observe that $\ol{\varphi(S)}^{\rm zar}=\ol{\varphi(M)}^{\rm zar}=V$. By Remark \ref{regsmooth} and Lemma \ref{sth} we deduce $\varphi({\tt Sth}(\Ss))={\tt Sth}(\varphi(\Ss))=\Reg(\varphi(\Ss))$ and $\partial\varphi(\Ss)=\cl(\varphi(\Ss))\setminus\Reg(\varphi(\Ss))=\varphi((\cl(\Ss)\cap M)\setminus{\tt Sth}(\Ss))$. Thus, $\varphi(Y)$ is the smallest Nash subset of $\varphi(M)$ that contains $\partial\varphi(\Ss)$, so $\ol{\partial\varphi(\Ss)}^{\rm zar}$ is $X$, which is a normal crossing divisor of $V=\ol{\varphi(S)}^{\rm zar}$. Consequently, $\varphi(\Ss)$ is a checkerboard subset of $\R^m$, $\Reg(\varphi(\Ss))=\varphi({\tt Sth}(\Ss))$ and $\Sing(\varphi(\Ss))=\varphi({\tt NSth}(\Ss))$, as required.
\end{proof}

\subsection{Well-welded semialgebraic sets as Nash images of checkerboard sets}\setcounter{paragraph}{0}
Our purpose next is to prove that a well-welded semialgebraic set is the image under a proper surjective regular map of a checkerboard set of its same dimension.

\begin{thm}\label{clave}
Let $\Ss\subset\R^n$ be a well-welded semialgebraic set of dimension $d$ and denote $X:=\ol{\Ss}^{\zar}$. Then there exist a checkerboard set $\Tt\subset\R^m$ of dimension $d$ and a proper regular map $f:Y:=\ol{\Tt}^{\zar}\to X$ such that $f(\Tt)=\Ss$.
\end{thm}
\begin{remark}
As $f$ is proper, if $\Ss$ is in addition bounded, then $\Tt$ is also bounded.
\end{remark}

The proof of Theorem \ref{clave} is quite involved and requires some preliminary work.

\begin{prop}\label{clave0}
Let $X\subset\R^n$ be a non-singular algebraic set of dimension $d$ and let $Z\subset X$ be a normal-crossings divisor. Let $\Ss\subset X$ be a connected semialgebraic set such that $\Cc:=\Ss\setminus Z$ is a union of connected components of $X\setminus Z$ and $\Cc\subset\Ss\subset\cl(\Cc)$. Then there exist a checkerboard set $\Tt\subset\R^m$ of dimension $d$ and a proper surjective regular map $f:Y:=\ol{\Tt}^{\zar}\to X$ such that the restriction $f|_{Y\setminus f^{-1}(Z)}:Y\setminus f^{-1}(Z)\to X\setminus Z$ is a regular diffeomorphism and $f(\Tt)=\Ss$.
\end{prop}
\begin{proof}[Proof of Proposition \em \ref{clave0} \em when $\Ss$ is closed]
Let $M_1,\ldots,M_s$ be the connected components of the Nash manifold $\Reg(\Ss)$ and define $\Rr:=\bigcup_{i\neq j}\cl(M_i)\cap\cl(M_j)$. Notice that
$$
\Rr=\{x\in\Ss:\ \Reg(\Ss)_x\not\subset M_{i,x}\ \forall i=1,\ldots,s\}\subset\{x\in\Ss:\ \Reg(\Ss)_x\ \text{is not connected}\}.
$$
The irreducible components of $Z$ are non-singular. Denote them with $Z_1,\ldots,Z_r$. 

\paragraph{}\label{zi}We claim: \em each irreducible component of $A:=\ol{\Rr}^{\zar}$ is an irreducible component of some intersection $Z_{i_1}\cap\cdots\cap Z_{i_\ell}$\em. 

As $\Rr\subset Z$, also $A\subset Z$. Let $A_1$ be an irreducible component of $A$ and let $\Rr_1:=\Rr\cap A_1$. Let $A_2,\ldots,A_r$ be the remaining irreducible components of $A$. Observe that 
$$
\Rr_1\setminus\bigcup_{j=2}^rA_j=\Rr\setminus\bigcup_{j=2}^rA_j\neq\varnothing\quad\text{and}\quad\dim(A_1)=\dim\Big(\Rr_1\setminus\bigcup_{j=2}^rA_j\Big)=\dim(\Rr_1)
$$
because $A$ is the Zariski closure of $\Rr$. In addition, we have $A_1\cap\Reg(A)=\Reg(A_1)\cap\Reg(A)$ and $\Rr_1\cap\Reg(\Rr)=A_1\cap\Reg(\Rr)=\Reg(\Rr_1)\cap\Reg(\Rr)$.

Assume that $A_1\subset Z_k$ exactly for $k=1,\ldots,\ell$. Note that $\ell\leq d$. Pick a point 
$$
x\in\Rr_1\cap\Reg(\Rr)\setminus\bigcup_{j=\ell+1}^rZ_j\subset\Reg(A). 
$$
As $Z$ is a normal-crossings divisor of $X$, the intersection $Z_1\cap\cdots\cap Z_\ell$ is a non-singular algebraic set and there exists an open semialgebraic neighborhood $U\subset X\setminus\bigcup_{j=\ell+1}^rZ_j$ of $x$ equipped with a Nash diffeomorphism $u:U\to\R^d$ such that $u(x)=0$ and $u(Z\cap U)=\{x_1\cdots x_\ell=0\}$. We may assume in addition: 
\begin{itemize}\enlargethispage{5mm}
\item $\Rr\cap U=A\cap U=A_1\cap U=\Rr_1\cap U$ is a connected closed submanifold of $U$. 
\item $u(Z_i\cap U)=\{x_i=0\}$ for $i=1,\ldots,\ell$.
\item $\Ss_x\cap M_{1,x}\neq\varnothing$ and $\Ss_x\cap M_{2,x}\neq\varnothing$.
\end{itemize}
As $\Cc$ is a union of connected components of $X\setminus Z$, we have that $u(\Cc\cap U)$ is a union of sets of the type $\{\veps_1x_1>0,\ldots,\veps_\ell x_\ell>0\}$ where $\veps_i=\pm1$. Consider the projection $\pi:\R^\ell\times\R^{d-\ell}\to\R^\ell,\ (x,y)\mapsto x$ and observe that $u(\Cc\cap U)=\pi(u(\Cc\cap U))\times\R^{d-\ell}$. Consequently, $\Ss=\cl(\Cc)$ satisfies $u(\Ss\cap U)=\pi(u(\Ss\cap U))\times\R^{d-\ell}$. As $\Reg(\Ss)_x$ is not connected, $\Reg(\pi(u(\Ss\cap U)))_0$ is not connected. Consequently, for each $y\in(Z_1\cap\cdots\cap Z_\ell)\cap U$ the germ $\Reg(\Ss)_y$ is not connected. As $\Ss_x\cap M_{1,x}\neq\varnothing$ and $\Ss_x\cap M_{2,x}\neq\varnothing$, we deduce $\Ss_y\cap M_{1,y}\neq\varnothing$ and $\Ss_y\cap M_{2,y}\neq\varnothing$ for each $y\in(Z_1\cap\cdots\cap Z_\ell)\cap U$. Thus,
$$
A_1\cap U\subset(Z_1\cap\cdots\cap Z_\ell)\cap U\subset\Rr\cap U=A_1\cap U.
$$
As $Z_1\cap\cdots\cap Z_\ell$ is pure dimensional, we conclude that $\dim(A_1)=\dim(Z_1\cap\cdots\cap Z_\ell)$. As $A_1$ is irreducible, it is an irreducible component of $Z_1\cap\cdots\cap Z_\ell$.

\paragraph{}Next, we prove: $\dim(\Rr)\leq\dim(X)-2$. Assume by contradiction that $\dim(\Rr)=\dim(X)-1$. Let $x\in\Reg(Z)\cap\Reg(\Rr)$. There exists an open semialgebraic neighborhood $U\subset X$ of $x$ such that $U\cap Z$ is Nash diffeomorphic to $\{x_1=0\}$, so $\Cc_x$ is either Nash equivalent to $\{x_1>0\}$ or to $\{x_1>0\}\cup\{x_1<0\}$. Consequently, $\Ss_x$ is either Nash equivalent to $\{x_1\geq0\}$ or to $\R^d$. But in both cases $\Reg(\Ss)_x$ is connected, a contradiction.

\paragraph{}\label{cluestep} Let $A_1$ be an irreducible component of $A$ and let $(Y_1,f_1)$ be the blow-up of $X$ with center $A_1$. Denote $\Tt_1:=\cl(f_1^{-1}(\Ss\setminus A_1))$ and observe that $f_1(\Tt_1)=\Ss$ because $f$ is proper and surjective and $\Ss\setminus A_1$ is dense in $\Ss$ because $\Ss$ is pure dimensional. Let us prove:\em 
\begin{itemize}
\item $f_1^{-1}(Z)$ is a normal-crossings divisor of $Y_1$.
\item $\Tt_1\setminus f_1^{-1}(Z)$ is a union of connected components of $Y_1\setminus f_1^{-1}(Z)$.
\item $\Tt_1$ is connected.
\item $\Reg(\Tt_1)$ has at most $s-1$ connected components.
\end{itemize}\em
Assume that $A_1\subset Z_k$ exactly for $k=1,\ldots,\ell$. As $A_1$ is an irreducible component of $Z_1\cap\cdots\cap Z_\ell$, the inverse image $f_1^{-1}(Z)$ is a normal-crossings divisor of $Y_1$. As $X\setminus A_1$ and $Y_1\setminus f_1^{-1}(A_1)$ are Nash diffeomorphic, also $X\setminus Z$ and $Y_1\setminus f_1^{-1}(Z)$ are Nash diffeomorphic. As $\Cc$ is a union of connected components of $X\setminus Z$, the inverse image $f_1^{-1}(\Cc)$ is a union of connected components of $Y_1\setminus f_1^{-1}(Z)$. Consequently, $f_1^{-1}(\Cc)$ is closed in $Y_1\setminus f_1^{-1}(Z)$ and $f_1^{-1}(\Cc)=\cl(f_1^{-1}(\Cc))\setminus f_1^{-1}(Z)$. Let us check: $\Tt_1=\cl(f_1^{-1}(\Cc))$.

We have
$$
\Tt_1=\cl(f_1^{-1}(\Ss\setminus A_1))=\cl(f_1^{-1}(\Cc))\cup\cl(f_1^{-1}((\Ss\setminus\Cc)\setminus A_1)).
$$ 
As $\cl(f_1^{-1}((\Ss\setminus\Cc)\setminus A_1))\subset f_1^{-1}(Z)$, it holds 
$$
\Tt_1\setminus f_1^{-1}(Z)=\cl(f_1^{-1}(\Cc))\setminus f_1^{-1}(Z)=f_1^{-1}(\Cc).
$$
As $f_1^{-1}(\Ss\setminus A_1)$ and $\Ss\setminus A_1$ are Nash diffeomorphic and $\Ss\setminus A_1$ is pure dimensional, also $f_1^{-1}(\Ss\setminus A_1)$ is pure dimensional, so $\Tt_1=\cl(f_1^{-1}(\Ss\setminus A_1))$ is pure dimensional. As $\dim(f_1^{-1}(Z))=d-1=\dim(\Tt_1)-1$, we deduce $\Tt_1=\cl(\Tt_1\setminus f_1^{-1}(Z))=\cl(f_1^{-1}(\Cc))$.

By Lemma \ref{ncww} $\Ss$ is well-welded, so $\Tt_1$ is by Lemma \ref{pww} well-welded and therefore connected. It only remains to check that $\Reg(\Tt_1)$ has at most $s-1$ connected components. 

As $\dim(A_1)\leq\dim(M_i)-2$, the differences $M_i\setminus A_1$ are connected Nash manifolds, so the same happens with the sets $f_1^{-1}(M_i\setminus A_1)$. Observe that
\begin{multline}\label{s11}
\bigcup_{i=1}^sf_1^{-1}(M_i\setminus A_1)\subset\Reg(\Tt_1)\subset\Tt_1=\cl(f_1^{-1}(\Cc))\\
=\cl(f_1^{-1}(\Reg(\Ss)\setminus A_1))=\cl\Big(\bigcup_{i=1}^sf_1^{-1}(M_i\setminus A_1)\Big)=\bigcup_{i=1}^s\cl(f_1^{-1}(M_i\setminus A_1)),
\end{multline}
so $\Reg(\Tt_1)$ has at most $s$ connected components. Let us check that in fact it has at most $s-1$. This follows from equality \eqref{s11} if we prove that $\cl(f_1^{-1}(M_1\cup M_2)\setminus A_1)$ is connected.

Recall that $A_1$ is an irreducible component of $Z_1\cap\cdots\cap Z_\ell$. Pick a point 
$$
x\in\Reg(\Rr)\cap A_1\setminus\bigcup_{j=\ell+1}^rZ_j\subset\Reg(A). 
$$
We may assume $x\in\cl(M_1)\cap\cl(M_2)$. Let $U\subset X\setminus\bigcup_{j=s+1}^rZ_j$ be an open semialgebraic neighborhood of $x$ such that $\Rr\cap U=A\cap U=A_1\cap U$ that is equipped with a Nash diffeomorphism $u:U\to\R^d$ such that $u(A_1\cap U)=\{x_1=0,\ldots,x_\ell=0\}$ and $u(Z\cap U)=\{x_1\cdots x_\ell=0\}\subset\R^d$. In particular, 
$$
\cl(M_1)\cap\cl(M_2)\cap U\subset\Rr\cap U=A_1\cap U.
$$

Observe that $\Cc\cap U$ is a union of sets of the type $\Qq_\epsilon:=\{\epsilon_1x_1>0,\ldots,\epsilon_\ell x_\ell>0\}$ where $\epsilon:=(\epsilon_1,\ldots,\epsilon_\ell)\in\{-1,+1\}^\ell$, that is, there exists ${\mathfrak F}\subset\{-1,+1\}^\ell$ such that 
$$
\Cc\cap U=\bigcup_{\epsilon\in{\mathfrak F}}\Qq_\epsilon. 
$$
Denote $\ol{\Qq}_\epsilon:=\{\epsilon_1x_1\geq0,\ldots,\epsilon_\ell x_\ell\geq0\}$. Observe that $\Reg(\Ss)\cap U$ is not connected because it has at least two connected components $E_1:=u(M_1\cap U)$ and $E_2:=u(M_2\cap U)$. Let $\Qq_\epsilon\subset E_1$ and $\Qq_{\epsilon'}\subset E_2$. We have
$$
\ol{\Qq}_\epsilon\cap\ol{\Qq}_{\epsilon'}\subset\cl(E_1)\cap\cl(E_2)\subset u(\Rr\cap U)=u(A_1\cap U).
$$
As $\dim(A_1\cap U)=d-\ell$, we deduce $\epsilon'=-\epsilon$. This means in addition that $u(M_1\cap U)=\Qq_\epsilon$ and $u(M_2\cap U)=\Qq_{-\epsilon}$ and we assume $\epsilon=(1,\ldots,1)$. In fact, $\Ss\cap U=(\cl(M_1)\cup\cl(M_2))\cap U$ and
$$
u(\Ss\cap U)=\{x_1\geq0,\ldots,x_\ell\geq0\}\cup\{-x_1\geq0,\ldots,-x_\ell\geq0\}.
$$
This is so because if $x\in\cl(M_j)$ for $j\neq1$, we have $u(M_j\cap U)=\Qq_{\epsilon''}$ for some $\epsilon''\in\{-1,1\}^s$. As we have seen $\epsilon''=-\epsilon$, so $M_j=M_2$.

Let $y\in f_1^{-1}(x)$. There exists an open semialgebraic neighborhood $V\subset Y_1$ of $y$ and a Nash diffeomorphism $v:V\to\R^d$ such that $v(y)=0$ and $u\circ f_1\circ v^{-1}:\R^d\to\R^d$ is given by
$$
(x_1,\ldots,x_d)\mapsto(x_1,x_1x_2,\ldots,x_1x_\ell,x_{s+1},\ldots,x_d).
$$
Consequently, 
\begin{equation*}
\begin{split}
(u\circ f_1\circ v^{-1})^{-1}(\Cc\cap U)&=\{x_1>0,x_1x_2>0,\ldots,x_1x_\ell>0\}\cup\{x_1<0,x_1x_2<0,\ldots,x_1x_\ell<0\}\\
&=\{x_1>0,x_2>0,\ldots,x_\ell>0\}\cup\{x_1<0,x_2>0,\ldots,x_\ell>0\}\\
&=\{x_1\neq0,x_2>0,\ldots,x_\ell>0\},
\end{split}
\end{equation*} 
see Figure \ref{fig13}. Therefore $\Tt_1$ contains $\cl(v(\{x_1\neq0,x_2>0,\ldots,x_\ell>0\}))=\cl(v(\{x_2>0,\ldots,x_\ell>0\}))$ and
\begin{align*}
f_1(v(\{x_1>0,x_2>0,\ldots,x_\ell>0\})=(M_1\setminus A_1)\cap U,\\
f_1(v(\{x_1<0,x_2>0,\ldots,x_\ell>0\})=(M_2\setminus A_1)\cap U.
\end{align*}
Thus, $\Reg(\cl(f_1^{-1}((M_1\cup M_2)\setminus A_1)))\subset\Reg(\Tt_1)$ is connected, so $\Reg(\Tt_1)$ has by \eqref{s11} at most $s-1$ connected components.

\paragraph{}We repeat recursively the previous process until we obtain $\Tt$ and $f$ satisfying the conditions in the statement, as required.
\end{proof}

\begin{figure}[!ht]
\begin{center}
\begin{tikzpicture}[scale=0.75]

\draw[line width=1pt,dashed] (9.5,0) -- (9.5,4);

\draw[fill=gray!20,opacity=0.75,draw=none](0,1) -- (3,1) -- (3,4) -- (0,4) -- (0,1);
\draw[fill=gray!20,opacity=0.75,draw=none](5,1) -- (8,1) -- (8,4) -- (5,4) -- (5,1);

\draw[fill=gray!20,opacity=0.75,draw=none](11,1) -- (14,1) -- (14,4) -- (11,4) -- (11,1);
\draw[fill=gray!20,opacity=0.75,draw=none](16,1) -- (19,1) -- (19,4) -- (16,4) -- (16,1);

\draw[fill=gray!70,opacity=0.75,draw=none](1.5,2.5) -- (3,2.5) -- (3,4) -- (1.5,4) -- (1.5,2.5);
\draw[fill=gray!70,opacity=0.75,draw=none](0,1) -- (1.5,1) -- (1.5,2.5) -- (0,2.5) -- (0,1);
\draw[fill=gray!70,opacity=0.75,draw=none](6.5,2.5) -- (8,2.5) -- (8,4) -- (6.5,4) -- (6.5,2.5);
\draw[fill=gray!70,opacity=0.75,draw=none](5,2.5) -- (6.5,2.5) -- (6.5,4) -- (5,4) -- (5,2.5);
\draw[fill=gray!70,opacity=0.75,draw=none](12.5,2.5) -- (14,2.5) -- (14,4) -- (12.5,4) -- (12.5,2.5);
\draw[fill=gray!70,opacity=0.75,draw=none](11,2.5) -- (12.5,2.5) -- (12.5,4) -- (11,4) -- (11,2.5);
\draw[fill=gray!70,opacity=0.75,draw=none](17.5,2.5) -- (19,2.5) -- (19,4) -- (17.5,4) -- (17.5,2.5);
\draw[fill=gray!70,opacity=0.75,draw=none](16,1) -- (17.5,1) -- (17.5,2.5) -- (16,2.5) -- (16,1);

\draw[dashed] (0,2.5) -- (3,2.5);
\draw[dashed] (1.5,1) -- (1.5,4);
\draw[dashed] (5,2.5) -- (8,2.5);
\draw[dashed] (6.5,1) -- (6.5,4);
\draw[dashed] (11,2.5) -- (14,2.5);
\draw[line width=1pt] (12.5,1) -- (12.5,4);
\draw[dashed] (16,2.5) -- (19,2.5);
\draw[dashed] (17.5,1) -- (17.5,4);

\draw (0,1) -- (3,4);
\draw (0.75,1) -- (2.25,4);
\draw (0,1.75) -- (3,3.25);

\draw (5,4) parabola bend (6.5,2.5) (8,4);
\draw (5,3.25) parabola bend (6.5,2.5) (8,3.25);
\draw (5.707106781186548,4) parabola bend (6.5,2.5) (7.29289321881345,4);

\draw (11,2.75) -- (14,2.75);
\draw (11,3) -- (14,3);
\draw (11,3.5) -- (14,3.5);

\draw (16,1) -- (19,4);
\draw (16.75,1) -- (18.25,4);
\draw (16,1.75) -- (19,3.25);

\draw[fill=black,draw] (1.5,2.5) circle (0.75mm);
\draw[fill=black,draw] (6.5,2.5) circle (0.75mm);
\draw[fill=black,draw] (12.5,2.5) circle (0.75mm);
\draw[fill=black,draw] (17.5,2.5) circle (0.75mm);

\draw[line width=1pt,->] (3.5,2.5) -- (4.5,2.5);
\draw[line width=1pt,->] (14.5,2.5) -- (15.5,2.5);
\draw[line width=1pt,->] (2.8,0.25) -- (4.4,0.25);
\draw[line width=1pt] (2.8,0.15) -- (2.8,0.35);
\draw[line width=1pt,->] (13.8,0.25) -- (15.4,0.25);
\draw[line width=1pt] (13.8,0.15) -- (13.8,0.35);

\draw (1.5,0.25) node{\small$(x_1,\ldots,x_d)$};
\draw (6.5,0.25) node{\small$(x_1,x_1x_2,\ldots,x_1x_d)$};
\draw (12.5,0.25) node{\small$(x_1,\ldots,x_d)$};
\draw (17.5,0.25) node{\small$(x_1,x_1x_2,\ldots,x_1x_d)$};

\end{tikzpicture}
\end{center}
\caption{Behavior of unions of two quadrants under blow-up.\label{fig13}}
\end{figure}
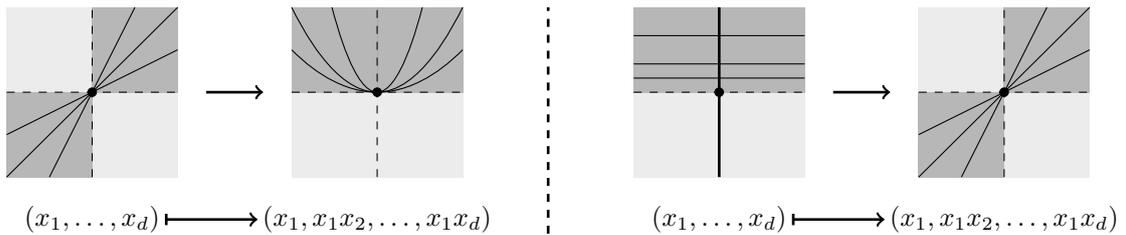

\begin{proof}[Proof of Proposition \em \ref{clave0} \em for the general case]
Let $\Ss_0:=\cl(\Ss)=\cl(\Cc)$. By Proposition \ref{clave0} for the closed case there exist a checkerboard set $\Tt_0$ and a proper surjective regular map $f_0:Y_0:=\ol{\Tt_0}^{\zar}\to X$ such that $Z_0:=f_0^{-1}(Z)$ is a normal-crossings divisor of $Y_0$, the restriction $f_0|_{Y_0\setminus Z_0}:Y_0\setminus Z_0\to X\setminus Z$ is a regular diffeomorphism, $\Tt_0=\cl(f_0^{-1}(\Cc))$ and $f_0(\Tt_0)=\Ss_0$. Let $\Tt_1:=f_0^{-1}(\Ss)\cap\Tt_0$ be the strict transform of $\Ss$ under $f_0$. As $\Ss$ is by Lemma \ref{ncww} well-welded, also $\Tt_1$ is by Lemma \ref{pww} well-welded, so $\Tt_1$ is connected. If $\Reg(\Tt_1)$ is connected, we are done, so we assume that $\Reg(\Tt_1)$ is not connected. Observe that $f_0^{-1}(\Cc)$ is dense in $\Tt_1$. Let $N_1,\ldots,N_s$ be the connected components of $\Reg(\Tt_1)$. As $\Tt_1$ is connected, we suppose $\cl(N_1)\cap\cl(N_2)\cap\Tt_1\neq\varnothing$. 

\paragraph{}\label{redcr} We may assume: \em there exist $q\in\cl(N_1)\cap\cl(N_2)\cap\Tt_1$ and an open semialgebraic neighborhood $U\subset Y_0$ of $q$ equipped with a Nash diffeomorphism $u:U\to\R^d$ such that $u(q)=0$,
$$
\{x_1>0,\ldots,x_\ell>0\}\subset u(N_1\cap U)\quad\text{and}\quad\{x_1<0,\ldots,x_\ell<0\}\subset u(N_2\cap U).
$$
\em 

Pick a point $p\in\cl(N_1)\cap\cl(N_2)\cap\Tt_1$ and let $e:=\dim(\cl(N_1)_p\cap\cl(N_2)_p)$. We distinguish two situations depending on the value of $e$:

\noindent{\em Case \em 1.} $e<d-1$. Let $A_1$ be an irreducible component of the Zariski closure $A$ of $\cl(N_1)\cap\cl(N_2)$ of maximal dimension passing through $p$. Let $Z_{01},\ldots,Z_{0\ell}$ be all the irreducible components of $Z_0$ that contain $A_1$. Proceeding similarly to the proof \ref{zi} one shows that $A_1$ is an irreducible component of $Z_{01}\cap\cdots\cap Z_{0\ell}$. Note that $e=d-\ell$. Consider the blow-up $(Y_0',f_1)$ of $Y_0$ with center $A_1$. Define $\Tt'_0:=\cl(f_1^{-1}(\Tt_0\setminus A_1))$. Proceeding similarly to the proof of \ref{cluestep} one shows: 
\begin{itemize}\em
\item[(i)] $f_1^{-1}(Z_0)$ is a normal-crossings divisor of $Y_0'$.
\item[(ii)] $\Tt'_0\setminus f_1^{-1}(Z_0)$ is a union of connected components of $Y_0'\setminus f_1^{-1}(Z_0)$.
\item[(iii)] $\Tt'_0$ is connected.\em
\end{itemize}
In addition, it holds:
\begin{itemize}
\item[(iv)] \em For each point $z\in f_1^{-1}(p)\cap\Tt'_0$, we have $\dim(\cl(f_1^{-1}(N_1))_z\cap\cl(f_1^{-1}(N_2))_z)=d-1$\em.
\end{itemize}
Denote the union of the irreducible components of $Z_0$ that do not contain $A_1$ with $Z_0'$. To prove (iv) pick 
$$
x\in\Reg(\cl(N_1)\cap\cl(N_2))\cap A_1\setminus Z_0'\subset\Reg(A)
$$ 
close to $p$. Let $U\subset Y_0$ be an open semialgebraic neighborhood of $x$ equipped with a Nash diffeomorphism $u:U\to\R^d$ such that $u(x)=0$, $u(A\cap U)=\{x_1=0,\ldots,x_\ell=0\}$ and $u(Z_0\cap U)=\{x_1\cdots x_\ell=0\}$. We may assume in addition: 
\begin{itemize}
\item $\cl(N_1)\cap\cl(N_2)\cap U=A\cap U$ is a connected closed submanifold of $U$. 
\item $u(Z_{0i}\cap U)=\{x_i=0\}$ for $i=1,\ldots,\ell$.
\end{itemize}
Assume $\Qq_1:=\{x_1>0,\ldots,x_\ell>0\}\subset u(N_1\cap U)$ and $\Qq_{\epsilon}:=\{\epsilon_1x_1>0,\ldots,\epsilon_\ell x_\ell>0\}\subset u(N_2\cap U)$ for some $\epsilon:=(\epsilon_1,\ldots,\epsilon_\ell)\in\{-1,+1\}^\ell$. As 
$$
\cl(\Qq_1)\cap\cl(\Qq_{\epsilon})\subset u(\cl(N_1)\cap\cl(N_2)\cap U)=u(A\cap U)=\{x_1=0,\ldots,x_\ell=0\}, 
$$
we have $\epsilon=(-1,\ldots,-1)$, $\Qq_1=u(N_1\cap U)$ and $\Qq_{\epsilon}=u(N_2\cap U)$.

Let $y\in Y_0'$ be a point close to $z$ such that $f_1(y)=x$. There exists an open semialgebraic neighborhood $V\subset Y_0'$ of $y$ and a Nash diffeomorphism $v:V\to\R^d$ such that $u\circ f_1\circ v^{-1}:\R^d\to\R^d$ is given by
$$
(x_1,\ldots,x_d)\mapsto(x_1,x_1x_2,\ldots,x_1x_\ell,x_{\ell+1},\ldots,x_d).
$$
We have
\begin{equation*}
\begin{split}
(u\circ f_1\circ v^{-1})^{-1}(\Qq_1\cup\Qq_{\epsilon})&=\{x_1>0,x_1x_2>0,\ldots,x_1x_\ell>0\}\\
&\cup\{x_1<0,x_1x_2<0,\ldots,x_1x_\ell<0\}\\
&=\{x_1>0,x_2>0,\ldots,x_\ell>0\}\cup\{x_1<0,x_2>0,\ldots,x_\ell>0\}\\
&=\{x_1\neq0,x_2>0,\ldots,x_\ell>0\}.
\end{split}
\end{equation*}
So $\cl(f_1^{-1}(N_1))\cap\cl(f_1^{-1}(N_2))\cap V$ contains $v^{-1}(\{x_1=0\})$, which has dimension $d-1$. Consequently, $\dim(\cl(f_1^{-1}(N_1))_z\cap\cl(f_1^{-1}(N_2))_z)=d-1$.

We assume in what follows $e=d-1$.

\noindent{\em Case \em 2.} $e=d-1$. Consider the blow-up $(Y_0',f_1)$ of $Y_0$ with center $\{p\}\subset\Tt_1$. Consider the strict transform $\Tt_1':=\cl(f_1^{-1}(\Tt_1\setminus\{p\}))\cap f_1^{-1}(\Tt_1)$. We have: 
\begin{itemize}\em
\item $f_1^{-1}(Z_0)$ is a normal-crossings divisor of $Y_0'$.
\item $\Tt_1'\setminus f_1^{-1}(Z_0)$ is a union of connected components of $Y_0'\setminus f_1^{-1}(Z_0)$.
\item $\Tt_1'$ is connected \em (use the argument described above involving Lemmas \ref{ncww} and \ref{pww}).
\end{itemize}
As $\dim(\cl(N_1)_p\cap\cl(N_2)_p)=d-1$, we assume there exists an open semialgebraic neighborhood $U$ of $p$ in $Y_0$ and a Nash diffeomorphism $u:U\to\R^d$ such that $u(p)=0$, $u(Z_0\cap U)=\{x_1\cdots x_\ell=0\}$ and 
\begin{align}
\Rr_1:=\{x_1>0,x_2>0\ldots,x_\ell>0\}\subset u(N_1\cap U)\label{Rr1},\\
\Rr_2:=\{x_1<0,x_2>0\ldots,x_\ell>0\}\subset u(N_2\cap U)\label{Rr2}.
\end{align}

Consider the Nash path germ $\alpha(t):=u^{-1}(t,t^2,\cdots,t^2)\subset N_1\cap U$ and observe that 
\begin{equation}\label{pick}
q:=\lim_{t\to0^+}f_1^{-1}(\alpha(t))\in f_1^{-1}(p)\cap\cl(f_1^{-1}(\Tt_1\setminus\{p\}))\subset\cl(f_1^{-1}(\Tt_1\setminus\{p\}))\cap f_1^{-1}(\Tt_1)=\Tt_1'.
\end{equation} 
Thus, there exists an open semialgebraic neighborhood $V\subset Y_0'$ of $q$ and a Nash diffeomorphism $v:V\to\R^d$ such that $v(q)=0$ and $u\circ f_1\circ v^{-1}:\R^d\to\R^d$ is given by
$$
(x_1,\ldots,x_d)\mapsto(x_1,x_1x_2,\ldots,x_1x_d).
$$
We have
\begin{equation*}
\begin{split}
(u\circ f_1&\circ v^{-1})^{-1}(\Rr_1\cup\Rr_2\cup\{0\})\\
&=\{x_1>0,x_1x_2>0,\ldots,x_1x_\ell>0\}\cup\{x_1<0,x_1x_2>0,\ldots,x_1x_\ell>0\}\cup\{x_1=0\}\\
&=\{x_1>0,\ldots,x_\ell>0\}\cup\{x_1<0,\ldots,x_\ell<0\}\cup\{x_1=0\},
\end{split}
\end{equation*}
see Figure \ref{fig13}. Consequently, \ref{redcr} holds true.

\paragraph{} Observe that $q\not\in N_i$ for $i=1,\ldots,s$. Consider the blow-up $(Y_0',f_1)$ of $Y_0$ with center $\{q\}\subset\Tt_1$ and the strict transform $\Tt_1':=\cl(f_1^{-1}(\Tt_1\setminus\{q\}))\cap f_1^{-1}(\Tt_1)$. We have: 
\begin{itemize}
\item $f_1^{-1}(Z_0)$ is a normal-crossings divisor of $Y_0'$.
\item $\Tt_1'\setminus f_1^{-1}(Z_0)$ is a union of connected components of $Y_0'\setminus f_1^{-1}(Z_0)$.
\item $\Tt_1'$ is connected by Lemmas \ref{pww} and \ref{ncww}.
\end{itemize}
Let us check that $\Reg(\Tt_1')$ has at most $s-1$ connected components. It holds
\begin{multline}\label{s12}
\bigcup_{i=1}^sf_1^{-1}(N_i)\subset\Reg(\Tt_1')\subset\Tt_1'=\cl(f_1^{-1}(\Reg(\Tt_1)\setminus\{q\}))\cap f_1^{-1}(\Tt_1)\\
=\cl\Big(\bigcup_{i=1}^sf_1^{-1}(N_i)\Big)\cap f_1^{-1}(\Tt_1)=\bigcup_{i=1}^s\cl(f_1^{-1}(N_i))\cap f_1^{-1}(\Tt_1),
\end{multline}
so $\Reg(\Tt_1')$ has at most $s$ connected components. Let us check that in fact it has at most $s-1$. 

Proceeding similarly to \eqref{pick} we find a point 
$$
y\in f_1^{-1}(q)\cap\cl(f_1^{-1}(\Tt_1\setminus\{q\}))\subset\cl(f_1^{-1}(\Tt_1\setminus\{p\}))\cap f_1^{-1}(\Tt_1)=\Tt_1'.
$$
Thus, there exist an open semialgebraic neighborhood $V\subset Y_0'$ of $y$ and a Nash diffeomorphism $v:V\to\R^d$ such that $v(y)=0$ and $u\circ f_1\circ v^{-1}:\R^d\to\R^d$ is given by
$$
(x_1,\ldots,x_d)\mapsto(x_1,x_1x_2,\ldots,x_1x_d).
$$
Consequently, if $\Rr_1$ is defined as in \eqref{Rr1}, then $-\Rr_1=\{x_1<0,x_2<0,\ldots,x_\ell<0\}$ and
\begin{equation*}
\begin{split}
(u\circ f_1\circ v^{-1})^{-1}((\Rr_1\cup-\Rr_1\cup\{0\})&=\{x_1>0,x_1x_2>0,\ldots,x_1x_\ell>0\}\\
&\cup\{x_1<0,x_1x_2<0\ldots,x_1x_\ell<0\}\cup\{x_1=0\}\\
&=\{x_2>0,\ldots,x_\ell>0\}\cup\{x_1=0\},
\end{split}
\end{equation*} 
see Figure \ref{fig13}. So $\Tt_1'$ contains $v(\{x_2>0,\ldots,x_\ell>0\})$ and
\begin{align*}
f_1( v(\{x_1>0,x_2>0\ldots,x_\ell>0\})\subset N_1\cap U,\\
f_1( v(\{x_1<0,x_2>0\ldots,x_\ell>0\})\subset N_2\cap U.
\end{align*}
Thus, $\Reg(\cl(f_1^{-1}(N_1\cup N_2))\cap f_1^{-1}(\Tt_1))\subset\Reg(\Tt_1')$ is connected, so $\Reg(\Tt_1')$ has by \eqref{s12} at most $s-1$ connected components.

\paragraph{} We repeat recursively the previous process until we obtain $\Tt$ and $f$ satisfying the conditions in the statement, as required.
\end{proof}

We are ready to prove Theorem \ref{clave}.

\begin{proof}[Proof of Theorem \em \ref{clave}]
Let $X$ be the Zariski closure of $\Ss$. By Theorem \ref{hi1} there exist a non-singular algebraic set $X'$ and a proper regular map $f:X'\to X$ such that the restriction $f|_{X'\setminus f^{-1}(\Sing(X))}:X'\setminus f^{-1}(\Sing(X))\to X\setminus\Sing(X)$ is a biregular diffeomorphism. The strict transform $\Ss':=\cl(f^{-1}(\Ss\setminus\Sing(X)))\cap f^{-1}(\Ss)$ is by Lemma \ref{pww} well-welded and $f(\Ss')=\Ss$. So we may assume from the beginning that $X$ is non-singular.

Let $Z$ be the Zariski closure of $\Rr:=\cl(\Ss)\setminus\Reg(\Ss)$. By Theorem \ref{hi2} there exist a non-singular algebraic set $X'$ and a proper surjective regular map $f:X'\to X$ such that $f^{-1}(Z)$ is a normal-crossings divisor of $X'$ and the restriction $f|_{X'\setminus f^{-1}(Z)}:X'\setminus f^{-1}(Z)\to X\setminus Z$ is a biregular diffeomorphism. The strict transform $\Ss':=f^{-1}(\Ss)\cap\cl(f^{-1}(\Ss\setminus Z))$ is by Lemma \ref{pww} well-welded and $f(\Ss')=\Ss$. Even more $\Ss\setminus Z=\cl(\Ss)\setminus Z=\Reg(\Ss)\setminus Z$ is a closed and open subset of $X\setminus Z$, so $\Cc:=f^{-1}(\Ss\setminus Z)$ is a closed and open subset of $X'\setminus f^{-1}(Z)$. Consequently, $\Cc$ is a union of connected components of $X'\setminus f^{-1}(Z)$. 

Thus, we may assume from the beginning:
\begin{itemize}
\item The Zariski closure of $\Ss$ is non-singular.
\item The Zariski closure of $\cl(\Ss)\setminus\Reg(\Ss)$ is contained in a normal-crossings divisor $Z$.
\item $\Ss\setminus Z$ is a union of connected components of $X\setminus Z$.
\item $\Ss$ is connected (and so well-welded by Lemma \ref{ncww}).
\end{itemize}

By Proposition \ref{clave0} there exist a checkerboard set $\Tt\subset\R^m$ and a proper surjective regular map $f:Y:=\ol{\Tt}^{\zar}\to X$ such that $f(\Tt)=\Ss$, as required.
\end{proof}

\subsection{Proof of Theorem \ref{main}}\setcounter{paragraph}{0}
The implications (i) $\Longrightarrow$ (ii) $\Longrightarrow$ (iii) $\Longrightarrow$ (iv) and (i) $\Longrightarrow$ (ii) $\Longrightarrow$ (v) $\Longrightarrow$ (vi) are straightforward. As quoted in the Introduction, only the proof of the non-completely trivial implication (ii) $\Longrightarrow$ `$\Ss$ is pure dimensional' requires a comment and it is shown in Corollary \ref{imnm}. The implication (iii) $\Longrightarrow$ (vii) is proved in Lemma \ref{npww} whereas (v) $\Longrightarrow$ (vii) follows from Lemma \ref{imp1}. In addition, (iv) $\Longrightarrow$ (vii) is shown in Lemma \ref{apww} and (vi) $\Longrightarrow$ (vii) in Lemma \ref{imp2}. To finish we prove (vii) $\Longrightarrow$ (i), that is, \em a well-welded semialgebraic set of dimension $d$ is a Nash image of $\R^d$\em. By Theorem \ref{clave} we may assume: \em $\Ss$ is a checkerboard set\em. Let $M_0:=\ol{\Ss}^{\zar}$, which is a Nash manifold, and let $Z$ be the smallest Nash subset of $M_0$ that contains $\partial\Ss$, which is a Nash normal crossings divisor of $M_0$ (because its irreducible components as a Nash set are connected components of the irreducible components of $\ol{\partial\Ss}^{\zar}$ as an algebraic set). By Remark \ref{diff} the difference $\Ss\setminus\ol{\partial\Ss}^{\zar}=\Reg(\Ss)\setminus\ol{\partial\Ss}^{\zar}$ is a union of connected components of $M_0\setminus\ol{\partial\Ss}^{\zar}$. As $\ol{\partial\Ss}^{\zar}$ is a normal crossings divisor of $M_0$, we conclude that $\partial\Ss=\cl(\Ss)\setminus\Reg(\Ss)$ is a pure dimensional semialgebraic set of dimension $d-1$. To get a general idea on how this proof works see Figure \ref{fig15}. The proof is conducted in several steps. 

\paragraph{}\label{pdet}
{\em Step $1$. Initial preparation.} 
Let $Z_1,\ldots,Z_s$ be the irreducible components of $Z$ as a Nash subset of $M_0$. Let $\Bb\subset\partial\Ss\subset Z$ be a semialgebraic set. For each $i=1,\ldots,s$ let $\Bb_i$ be the closure of the set of points of dimension $<\dim(\Bb)$ of the intersection $\Bb\cap Z_i$. Denote $\Bb^*:=\Bb\setminus\bigcup_{i=1}^s\Bb_i$ and observe that $\Bb^*$ is pure dimensional open semialgebraic subset of $\Bb$ and $\dim(\Bb\setminus{\tt Sth}(\Bb^*))<\dim(\Bb)$. We claim: \em if $\Aa$ is a connected component of ${\tt Sth}(\Bb^*)$, then for each $i=1,\ldots,s$ either $\Aa\cap Z_i=\varnothing$ or $\Aa\subset Z_i$\em.

Fix $i=1,\ldots,s$ such that $\Aa\cap Z_i\neq\varnothing$ and pick a point $p\in\Aa\cap Z_i$. Then $p\in(\Bb\cap Z_i)\setminus\Bb_i$, so $p$ is a point of dimension $\dim(\Bb)$ of $\Bb\cap Z_i$. As $\dim(\Bb\setminus{\tt Sth}(\Bb^*))<\dim(\Bb)$, we deduce $p\in{\tt Sth}(\Bb^*)$ is a point of dimension $\dim({\tt Sth}(\Bb^*))=\dim(\Bb)$ of ${\tt Sth}(\Bb^*)\cap Z_i$. As $\Aa$ is the connected component of ${\tt Sth}(\Bb^*)$ that contains $p$, we have $\dim(\Aa\cap Z_i)=\dim({\tt Sth}(\Bb^*)\cap Z_i)=\dim({\tt Sth}(\Bb^*))=\dim(\Aa)$. As $\Aa$ is a connected Nash manifold and $Z_i$ a Nash subset of $M_0$, we conclude by the identity principle $\Aa\subset Z_i$.

\paragraph{}\label{rktk} 
{\em Step $2$. Construction of the bad subset $\Tt$ of $\Ss$ and a suitable partition of $\Tt$ into Nash manifolds of different dimensions.}
As $\ol{\Ss}^{\rm zar}$ is a non-singular real algebraic set and $\Ss$ is pure dimensional, $\Reg(\Ss)={\tt Sth}(\Ss)$ and $\Sing(\Ss)={\tt NSth}(\Ss)$ by Remark \ref{regsmooth}. In particular, $\partial\Ss=\cl(\Ss)\setminus{\tt Sth}(\Ss)$. Let $\Gg$ be the set of points of $\Sing(\Ss)$ of local dimension $<d-1$. Define the \em bad subset \em of $\Ss$ as:
$$
\Tt:={\tt NSth}({\tt NSth}(\Ss))\cup(\cl(\Gg)\cap\Ss)\subset Z,
$$
which is a semialgebraic set of dimension $<d-1$. Define $\Rr_0:=\Tt$, $\Rr_k:=\Rr_{k-1}\setminus\Tt_k$ and $\Tt_k:={\tt Sth}(\Rr_{k-1}^*)\subset\Rr_{k-1}$ for $k\geq1$ (see \ref{pdet} for the definition of $\Rr_{k-1}^*$ from $\Rr_{k-1}$). Each semialgebraic set $\Tt_k$ is a Nash manifold (and an open subset of $\Rr_{k-1}$) and each semialgebraic set $\Rr_k$ is a closed subset of $\Tt$. In addition, if $1\leq k<j$, 
\begin{equation}\label{singsing}
\cl(\Tt_j)\cap\Tt_k=\cl(\Tt_j)\cap\Tt\cap\Tt_k\subset\Rr_{j-1}\cap\Tt_k\subset\Rr_k\cap\Tt_k=\varnothing. 
\end{equation}
Observe that $\dim(\Rr_{k+1})<\dim(\Rr_k)$ for $k\geq0$ (if $\Rr_k\neq\varnothing$), so $\Rr_{d-1}=\varnothing$. Consequently,
\begin{align*}
&\Tt=\Tt_1\sqcup\Rr_1=\Tt_1\sqcup\Tt_2\sqcup\Rr_2=\cdots=\bigsqcup_{k=1}^{d-1}\Tt_k,\\
&\Ss={\tt Sth}(\Ss)\sqcup({\tt NSth}(\Ss)\setminus\Tt)\sqcup\Tt.
\end{align*}

\paragraph{}\label{l0} 
\em Step $3$. The difference $\Ss\setminus\Tt$ is a Nash manifold with boundary the difference ${\tt NSth}(\Ss)\setminus\Tt$ and interior ${\tt Sth}(\Ss)$\em. 

It holds that ${\tt NSth}(\Ss)\setminus\Tt$ is either empty or a Nash manifold of dimension $d-1$. Assume ${\tt NSth}(\Ss)\setminus\Tt\neq\varnothing$ and pick a point $x\in{\tt NSth}(\Ss)\setminus\Tt$. As ${\tt NSth}(\Ss)\subset Z$ and $Z$ is a normal-crossings divisor, there exists an open semialgebraic neighborhood $U$ of $x$ in $M_0$ equipped with a Nash diffeomorphism $u:U\to\R^d$ such that $u(x)=0$ and $u(Z\cap U)=\{x_1\cdots x_r=0\}$ for some $1\leq r\leq d$. As ${\tt NSth}(\Ss)\setminus\Tt$ is a Nash manifold of dimension $d-1$, we may assume $u(({\tt NSth}(\Ss)\setminus\Tt)\cap U)=\{x_1=0\}$. As ${\tt Sth}(\Ss)\setminus Z$ is a union of connected components of $M_0\setminus Z$, the Nash manifold ${\tt Sth}(\Ss)$ is connected and $\Ss$ is pure dimensional, we may assume $u(\Ss\cap U)=\{x_1\geq0\}$. Consequently, $\Ss\setminus\Tt$ is a Nash manifold with boundary ${\tt NSth}(\Ss)\setminus\Tt$ and interior ${\tt Sth}(\Ss)$.

\paragraph{}\label{sate} {\em Step $4$.} Define ${\mathfrak G}:=\{1\leq k\leq d-1:\ \Tt_k\neq\varnothing\}$ and $\ell(\Ss):=\#{\mathfrak G}$. We prove next by induction on $\ell(\Ss)$ the following\em: There exist a connected Nash manifold $H$ with boundary $\partial H$ and a surjective Nash map $f:H\to\Ss$ such that $\Int(H)$ is Nash diffeomorphic to ${\tt Sth}(\Ss)$\em. Once this will be done, recall that $H$ is by Theorem \ref{mstone} a Nash image of $\R^d$, so also $\Ss$ will be a Nash image of $\R^d$ and the proof will be finished.

If $\ell(\Ss)=0$ (or equivalently $\Tt=\varnothing$), then $\Ss=\Ss\setminus\Tt$ is by \ref{l0} a Nash manifold with boundary and it is enough to take $H:=\Ss$ and $f:=\id_\Ss$. Assume statement \ref{sate} true for $\ell(\Ss)-1$ and let us check that it is also true for $\ell(\Ss)$. 

\paragraph{} Write $\ell:=\ell(\Ss)$ and observe that $\Tt_k=\varnothing$ if and only if $k\geq\ell+1$. As $\Tt_\ell$ is locally compact (because it is a Nash manifold), the semialgebraic set $\Cc:=\cl(\Tt_\ell)\setminus\Tt_\ell$ is closed. By \eqref{singsing} $\Cc$ does not meet $\Tt_j$ for $1\leq j\leq\ell$, so $\Cc\cap\Tt=\varnothing$. As $\Tt$ is a closed subset of $\Ss$ and $\Cc\subset\cl(\Tt)\setminus\Tt$, we have $\Ss\cap\Cc=\varnothing$. Let $M:=M_0\setminus\Cc$, which is a Nash manifold of dimension $d$ that contains $\Ss$. Observe that $\Tt_\ell$ is a closed Nash submanifold of $M$ of dimension $e<d$. Let $(\widetilde{M}_+,\pi_+)$ be the drilling blow-up of $M$ with center $\Tt_\ell$ and let $(\widehat{M},\widehat{\pi})$ be the twisted Nash double of $(\widetilde{M}_+,\pi_+)$. By \ref{bigstepa2g} $\pi_+^{-1}(\Tt_\ell)$ is a closed non-singular Nash hypersurface of the Nash manifold $\widehat{M}$. Denote $\Ee:=\cl(\pi_+^{-1}(Z\setminus\Tt_\ell))\cap\widetilde{M}_+$ and $\Kk:=\pi_+^{-1}(\Tt_\ell)\cap\Ee$, which are closed semialgebraic subsets of $\widehat{M}$. Observe that $\Ee=\pi_+^{-1}(Z\setminus\Tt_\ell)\sqcup\Kk$. Define 
$$
\Ss_1':=\pi_+^{-1}(\Ss)\cap\cl(\pi_+^{-1}(\Ss\setminus\Tt_\ell))\quad\text{and}\quad\Ss_1:=\Ss_1'\setminus\Kk.
$$
It holds $\Ss_1\setminus\pi_+^{-1}(Z)=\pi_+^{-1}(\Ss\setminus Z)$ and $(\Ss_1\setminus\Ee)\cap\pi_+^{-1}(\Tt_\ell)=\Ss_1\cap\pi_+^{-1}(\Tt_\ell)$. Denote $M':=\widehat{M}\setminus\Kk$, which is a Nash manifold, and $Z':=(\Ee\setminus\Kk)\sqcup(\pi_+^{-1}(\Tt_\ell)\cap\Ss_1)$. 

\paragraph{} We claim: \em $Z'$ is the smallest Nash subset of $M'$ that contains the semialgebraic set $\partial\Ss_1:=(\cl(\Ss_1)\cap M')\setminus{\rm NSth}(\Ss_1)$ and it is a Nash normal crossings divisor of $M'$\em. In addition, \em ${\tt Sth}(\Ss_1)$ is a connected Nash manifold of dimension $d$\em. Once this is proved, we deduce by Lemma \ref{lighten} that, up to a suitable Nash embedding of $M'$, the semialgebraic set $\Ss_1$ is a checkerboard set such that $\Reg(\Ss_1)={\tt Sth}(\Ss_1)$ and $\Sing(\Ss_1)={\tt NSth}(\Ss_1)$.

As $\pi_+|_{\widetilde{M}_+\setminus\pi_+^{-1}(\Tt_\ell)}:\widetilde{M}_+\setminus\pi_+^{-1}(\Tt_\ell)\to M\setminus\Tt_\ell$ is a Nash diffeomorphism, we have:
\begin{itemize}
\item $\Ee\setminus\Kk=\pi_+^{-1}(Z\setminus\Tt_\ell)$ is the smallest Nash subset of $\widetilde{M}_+\setminus\pi_+^{-1}(\Tt_\ell)$ that contains the semialgebraic set $\partial\Ss_1\setminus\pi_+^{-1}(\Tt_\ell)=\pi_+^{-1}(\partial\Ss\setminus\Tt_\ell)$, because $Z\setminus\Tt_\ell$ is the smallest Nash subset of $M\setminus\Tt_\ell$ that contains $\partial\Ss\setminus\Tt_\ell$. To prove this last fact recall that $\partial\Ss$ is pure dimensional of dimension $d-1$ and $\dim(\Tt_\ell)<d-1$.
\item $\Ee\setminus\Kk=\pi_+^{-1}(Z\setminus\Tt_\ell)$ is a Nash normal crossings divisor of $\widetilde{M}_+\setminus\pi_+^{-1}(\Tt_\ell)$ because $Z\setminus\Tt_\ell$ is a Nash normal crossings divisor of $M\setminus\Tt_\ell$. 
\end{itemize}

As $\widetilde{M}_+\setminus\pi_+^{-1}(\Tt_\ell)$ is an open and closed semialgebraic subset of $M'\setminus\pi_+^{-1}(\Tt_\ell)$, we deduce that $\Ee\setminus\Kk$ is the smallest Nash subset of $M'\setminus\pi_+^{-1}(\Tt_\ell)$ that contains the semialgebraic set $\partial\Ss_1\setminus\pi_+^{-1}(\Tt_\ell)$ and it is a Nash normal crossings divisor of $M'\setminus\pi_+^{-1}(\Tt_\ell)$. As $\Ee\setminus\Kk$ is a closed semialgebraic subset of $M'$, we conclude by \cite[Prop.II.5.3]{sh} that $\Ee\setminus\Kk$ is the smallest Nash subset of $M'$ that contains the semialgebraic set $\partial\Ss_1\setminus\pi_+^{-1}(\Tt_\ell)$ and a Nash normal crossings divisor of $M'$ .

\paragraph{} Let us check: \em $(\pi_+^{-1}(\Tt_\ell)\cap\Ss_1)$ is a closed non-singular Nash hypersurface of $M'$\em. As $(\pi_+^{-1}(\Tt_\ell)\cap\Ss_1)\cap(\Ee\setminus\Kk)=\varnothing$, this will show that $Z'$ is the smallest Nash subset of $M'$ that contains $\partial\Ss_1=\pi_+^{-1}(\partial\Ss\setminus\Tt_\ell)\sqcup(\pi_+^{-1}(\Tt_\ell)\cap\Ss_1)$ and it is a Nash normal crossings divisor of $M'$.

We have ${\tt Sth}(\Ss)\subset\Ss\setminus\Tt_\ell$, so ${\tt Sth}(\Ss)={\tt Sth}(\Ss\setminus\Tt_\ell)$. As $\pi_+^{-1}(\Tt_\ell)=\partial\widetilde{M}_+$ and $\dim(\Ss_1)=\dim(\widetilde{M}_+)$, we have $\Ss_1\cap\pi_+^{-1}(\Tt_\ell)\subset{\tt NSth}(\Ss_1)$ and ${\tt Sth}(\Ss_1)={\tt Sth}(\Ss_1\setminus\pi_+^{-1}(\Tt_\ell))$. As $\pi_+|_{\widetilde{M}_+\setminus\pi_+^{-1}(\Tt_\ell)}$ is a Nash diffeomorphism, 
$$
{\tt Sth}(\Ss_1)={\tt Sth}(\pi_+^{-1}(\Ss\setminus\Tt_\ell))=\pi_+^{-1}({\tt Sth}(\Ss\setminus\Tt_\ell))=\pi_+^{-1}({\tt Sth}(\Ss)),
$$
which is a connected Nash manifold of dimension $d$. Consequently, ${\tt NSth}(\Ss_1)=\pi_+^{-1}({\tt NSth}(\Ss)\setminus\Tt_\ell)\cup(\Ss_1\cap\pi_+^{-1}(\Tt_\ell))$ and $\partial\Ss_1=\pi_+^{-1}(\partial\Ss\setminus\Tt_\ell)\sqcup(\pi_+^{-1}(\Tt_\ell)\cap\Ss_1)$. As $\widetilde{M}_+$ is a Nash manifold with boundary $\pi_+^{-1}(\Tt_\ell)$, the difference $\widetilde{M}_+\setminus\Ee$ is a Nash manifold with boundary $\pi_+^{-1}(\Tt_\ell)\setminus\Ee=\pi_+^{-1}(\Tt_\ell)\setminus\Kk$.

Observe that $\Ss\setminus Z=\cl(\Ss)\setminus Z={\tt Sth}(\Ss)\setminus Z$ is a closed an open subset of $M_0\setminus Z=M\setminus Z$ (see Remark \ref{diff}). As $\pi_+^{-1}(\Tt_\ell)\cap{\tt Sth}(\Ss_1)=\varnothing$,
$$
\pi_+^{-1}(\Ss\setminus Z)=\pi_+^{-1}({\tt Sth}(\Ss)\setminus Z)={\tt Sth}(\Ss_1)\setminus\pi_+^{-1}(Z)={\tt Sth}(\Ss_1)\setminus\Ee
$$
is a union of connected components of $\pi_+^{-1}(M\setminus Z)=(\widetilde{M}_+\setminus\Ee)\setminus\pi_+^{-1}(\Tt_\ell)$, that is, ${\tt Sth}(\Ss_1)\setminus\Ee$ is a union of connected components of the interior $\Int(\widetilde{M}_+\setminus\Ee)$ of the Nash manifold with boundary $\widetilde{M}_+\setminus\Ee$. Thus, the closure $\Dd$ of ${\tt Sth}(\Ss_1)\setminus\Ee$ in $\widetilde{M}_+\setminus\Ee$ is a Nash manifold with boundary $\Dd\cap(\pi_+^{-1}(\Tt_\ell)\setminus\Kk)$. 

\paragraph{} We claim: \em $\Ss_1\setminus\Ee$ is the closure of ${\tt Sth}(\Ss_1)\setminus\Ee$ in $\widetilde{M}_+\setminus\Ee$\em.

As $\Ss\setminus Z=\cl(\Ss)\setminus Z$ and $\Tt_\ell\subset Z\cap\Ss$, we have $\cl(\Ss)\setminus(Z\setminus\Tt_\ell)=\Ss\setminus(Z\setminus\Tt_\ell)=(\Ss\setminus Z)\cup\Tt_\ell$. As $\pi_+^{-1}(\cl(\Ss)\setminus(Z\setminus\Tt_\ell))\setminus\Kk$ is a closed subset of $\widetilde{M}_+\setminus\Ee=(\pi_+^{-1}(M\setminus(Z\setminus\Tt_\ell))\setminus\Kk$ that contains $\pi_+^{-1}(\Ss\setminus Z)$, we deduce
\begin{multline*}
\Ss_1\setminus\Ee=\Ss_1'\setminus\Ee=(\pi_+^{-1}(\Ss\setminus(Z\setminus\Tt_\ell))\cap\cl(\pi_+^{-1}(\Ss\setminus\Tt_\ell)))\setminus\Kk\\
=(\pi_+^{-1}(\cl(\Ss)\setminus(Z\setminus\Tt_\ell))\cap\cl(\pi_+^{-1}(\Ss\setminus Z)))\setminus\Kk\subset(\widetilde{M}_+\setminus\Ee)\cap\cl(\pi_+^{-1}(\Ss\setminus Z))
\end{multline*}
is the closure of $\pi_+^{-1}(\Ss\setminus Z)$ in $\widetilde{M}_+\setminus\Ee$. 

\paragraph{} Thus, $\Ss_1\setminus\Ee$ is a Nash manifold with boundary $\Ss_1\cap\pi_+^{-1}(\Tt_\ell)$. As $\Ss_1\setminus\Ee$ is a closed subset of $\widetilde{M}_+\setminus\Ee$, the intersection $(\Ss_1\setminus\Ee)\cap\pi_+^{-1}(\Tt_\ell)=\Ss_1\cap\pi_+^{-1}(\Tt_\ell)$ is a closed subset of $(\widetilde{M}_+\setminus\Ee)\cap\pi_+^{-1}(\Tt_\ell)=\pi_+^{-1}(\Tt_\ell)\setminus\Kk$, which is itself a closed subset of $M'=\widehat{M}\setminus\Kk$. Consequently, $\Ss_1\cap\pi_+^{-1}(\Tt_\ell)$ is a closed non-singular Nash hypersurface of $M'$.

\paragraph{} We check next: $\ell(\Ss_1)=\ell(\Ss)-1$. 

As ${\tt NSth}(\Ss_1)=\pi_+^{-1}({\tt NSth}(\Ss)\setminus\Tt_\ell)\cup(\Ss_1\cap\pi_+^{-1}(\Tt_\ell))$, the restriction map $\pi_+|_{\widetilde{M}_+\setminus\pi_+^{-1}(\Tt_\ell)}:\widetilde{M}_+\setminus\pi_+^{-1}(\Tt_\ell)\to M\setminus\Tt_\ell$ is a Nash diffeomorphism and $\Ss_1\cap\pi_+^{-1}(\Tt_\ell)$ is a Nash manifold of dimension $d-1$, we deduce ${\tt NSth}({\tt NSth}(\Ss_1))=\pi_+^{-1}({\tt NSth}({\tt NSth}(\Ss))\setminus\Tt_\ell)$ and the set $\Gg'$ of points of ${\tt NSth}(\Ss_1)$ of dimension $<d-1$ is $\pi_+^{-1}(\Gg\setminus\Tt_\ell)$. We claim: \em the bad subset $\Tt':={\tt NSth}({\tt NSth}(\Ss_1))\cup(\cl(\Gg')\cap\Ss_1)$ of $\Ss_1$ equals $\pi_+^{-1}(\Tt\setminus\Tt_\ell)=\bigcup_{k=1}^{s-1}\pi_+^{-1}(\Tt_k)$\em. We have to prove: $\cl(\Gg')\cap\Ss_1=\pi_+^{-1}((\cl(\Gg)\cap\Ss)\setminus\Tt_\ell)$. 

As $\Gg'=\pi_+^{-1}(\Gg\setminus\Tt_\ell)\subset\cl(\pi_+^{-1}(Z\setminus\Tt_\ell))\cap\widetilde{M}_+=\Ee$, also $\cl(\Gg')\cap\widetilde{M}_+\subset\Ee$. Thus,
\begin{multline*}
\cl(\Gg')\cap\Ss_1=(\cl(\Gg')\cap\pi_+^{-1}(\Ss)\cap\cl(\pi_+^{-1}(\Ss\setminus\Tt_\ell)))\setminus(\pi_+^{-1}(\Tt_\ell)\cap\Ee)\\
=\cl(\Gg')\cap\Ee\cap(\pi_+^{-1}(\Ss)\setminus\pi_+^{-1}(\Tt_\ell))=\cl(\Gg')\cap\pi_+^{-1}(\Ss\setminus\Tt_\ell).
\end{multline*}
As $\pi_+|_{\widetilde{M}_+\setminus\pi_+^{-1}(\Tt_\ell)}$ is a Nash diffeomorphism and $\Gg'=\pi_+^{-1}(\Gg\setminus\Tt_\ell)$, it holds 
$$
\cl(\Gg')\cap\Ss_1=\cl(\Gg')\cap\pi_+^{-1}(\Ss\setminus\Tt_\ell)=\pi_+^{-1}(\cl(\Gg)\cap\Ss\setminus\Tt_\ell), 
$$
as claimed. Therefore, $\Tt'=\pi_+^{-1}(\Tt\setminus\Tt_\ell)$.

Decompose $\Tt'=\bigsqcup_{k=1}^{d-1}\Tt_k'$ following the algorithm proposed in \ref{rktk}. As $\pi_+|_{M\setminus\Tt_\ell}$ is a Nash diffeomorphism and $\Tt_k\cap\Tt_\ell=\varnothing$ if $k\neq\ell$, we deduce $\Tt_k'=\pi_+^{-1}(\Tt_k)$ for $k=1,\ldots,\ell-1$ and $\Tt_k'=\varnothing$ for $k=\ell,\ldots,d-1$. As $\Tt'_{\ell-1}\neq\varnothing$, we conclude
$$
\ell(\Ss_1)=\#\{1\leq k\leq d-1:\ \Tt_k'\neq\varnothing\}=\ell(\Ss)-1.
$$

\paragraph{} We claim: $\pi_+(\Ss_1)=\Ss$, see Figure \ref{fig14}. 

\begin{figure}[!ht]
\centering
\begin{tikzpicture}

\draw (0,4) -- (1.5,2.5) -- (3,2.5) -- (1.5,4) -- (0,4);
\draw[fill=black!20!white,opacity=0.75] (0,4) -- (1.5,2.5) -- (3,2.5) -- (1.5,4) -- (0,4);

\draw (6.5,2.5) -- (5,4) -- (3.5,4) -- (5,2.5) -- (6.5,2.5);
\draw[fill=black!20!white,opacity=0.75] (6.5,2.5) -- (5,4) -- (3.5,4) -- (5,2.5) -- (6.5,2.5);

\draw (0,4) -- (0,1.5) -- (1.5,0) -- (1.5,2.5) -- (0,4);
\draw[fill=black!20!white,opacity=0.75] (0,4) -- (0,1.5) -- (1.5,0) -- (1.5,2.5) -- (0,4);

\draw (1.5,2.5) -- (1.5,0) -- (6.5,0) -- (6.5,2.5) -- (5,2.5) arc (360:180:1cm) -- (1.5,2.5);
\draw[fill=black!20!white,opacity=0.75] (1.5,2.5) -- (1.5,0) -- (6.5,0) -- (6.5,2.5) -- (5,2.5) arc (360:180:1cm) -- (1.5,2.5);

\draw[dashed] (2,2.5) arc (180:360:2cm) -- (4.5,4) -- (3.5,4) -- (5,2.5) arc (360:180:1cm) -- (1.5,4) -- (0.5,4) -- (2,2.5);
\draw[fill=black!40!white,opacity=0.75,dashed] (2,2.5) arc (180:360:2cm) -- (4.5,4) -- (3.5,4) -- (5,2.5) arc (360:180:1cm) -- (1.5,4) -- (0.5,4) -- (2,2.5);

\draw (1.5,2.5) -- (1.5,0) -- (6.5,0) -- (6.5,2.5) -- (5,2.5) arc (360:180:1cm) -- (1.5,2.5);

\draw (9,1.5) -- (10.5,0) -- (10.5,2.5) -- (9,4) -- (9,1.5);
\draw[fill=black!20!white,opacity=0.75] (9,1.5) -- (10.5,0) -- (10.5,2.5) -- (9,4) -- (9,1.5);

\draw (10.5,0) -- (15.5,0) -- (15.5,2.5) -- (10.5,2.5) -- (10.5,0);
\draw[fill=black!20!white,opacity=0.75] (10.5,0) -- (15.5,0) -- (15.5,2.5) -- (10.5,2.5) -- (10.5,0);

\draw (10.5,2.5) -- (15.5,2.5) -- (14,4) -- (9,4) -- (10.5,2.5);
\draw[fill=black!20!white,opacity=0.75] (10.5,2.5) -- (15.5,2.5) -- (14,4) -- (9,4) -- (10.5,2.5);

\draw[dashed] (11,2.5) arc (180:360:2cm) -- (13.5,4) -- (9.5,4) -- (11,2.5);
\draw[fill=black!40!white,opacity=0.75,dashed] (11,2.5) arc (180:360:2cm) -- (13.5,4) -- (9.5,4) -- (11,2.5);
\draw[fill=black!60!white,opacity=0.75,dashed] (10.5,4) -- (12,2.5) arc (180:360:1cm) -- (14,2.5) -- (12.5,4) -- (10.5,4);

\draw (10.5,2.5) -- (15.5,2.5) -- (14,4) -- (9,4) -- (10.5,2.5);

\draw (3,2.5) arc (180:360:1cm);

\draw (2.5,3) arc (270:360:1cm);

\draw[fill=black!70!white,opacity=0.75] (2.5,3) -- (3,2.5) arc (180:360:1cm) -- (5,2.5) -- (3.5,4) arc (360:270:1cm);

\draw[line width=1.5pt,dashed] (2.5,4) -- (4,2.5);

\draw[white,line width=2pt] (1.5,4) -- (3,2.5);
\draw (1.45,4) -- (2.95,2.5);
\draw (1.55,4) -- (3.05,2.5);

\draw[white,line width=2pt] (3.5,4) -- (5,2.5);
\draw (3.45,4) -- (4.95,2.5);
\draw (3.55,4) -- (5.05,2.5);

\draw[white,line width=2pt] (4.5,1.625) -- (3,3.125);
\draw (4.45,1.625) -- (2.95,3.125);
\draw (4.55,1.625) -- (3.05,3.125);

\draw[line width=1.5pt] (11.5,4) -- (13,2.5);

\draw[->,line width=1pt] (6.5,3.5) -- (8.5,3.5);

\draw (7.5,3.8) node{$\pi_+|_{\Ss_1}$};

\draw[fill=white,draw] (4,2.5) circle (0.75mm);
\draw[fill=white,draw] (2.5,4) circle (0.75mm);
\draw[fill=white,draw] (1.5,4) circle (0.75mm);
\draw[fill=white,draw] (3.5,4) circle (0.75mm);
\draw[fill=white,draw] (3,2.5) circle (0.75mm);
\draw[fill=white,draw] (5,2.5) circle (0.75mm);
\draw[fill=white,draw] (13,2.5) circle (0.75mm);
\draw[fill=white,draw] (11.5,4) circle (0.75mm);
\draw[fill=white,draw] (4.5,1.625) circle (0.75mm);
\draw[fill=white,draw] (3,3.125) circle (0.75mm);

\draw (4,1.2) node{$\pi_+^{-1}(\Tt_\ell)\cap\Ss_1$};

\draw (2.5,0.25) node{$\Ss_1$};
\draw (11.5,0.25) node{$\Ss$};
\draw (12,4.35) node{$\Tt_\ell$};

\end{tikzpicture}
\caption{Behavior of the surjective Nash map $\pi_+|_{\Ss_1}:\Ss_1\to\Ss$.\label{fig14}}
\end{figure}
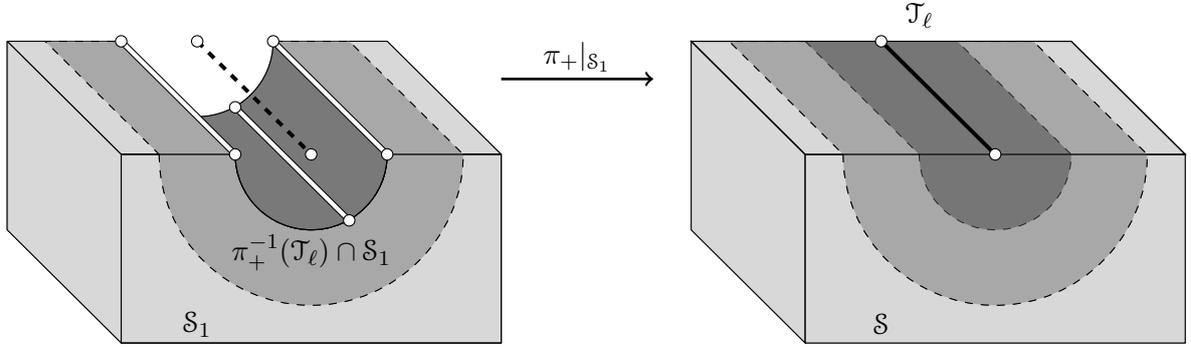

As $\Ss_1=(\pi_+^{-1}(\Ss)\cap\cl(\pi_+^{-1}(\Ss\setminus\Tt_\ell)))\setminus\Kk$ and $\Kk=\pi_+^{-1}(\Tt_\ell)\cap\Ee$, we have $\Ss\setminus\Tt_\ell\subset\pi_+(\Ss_1)\subset\Ss$. Thus, to prove $\pi_+(\Ss_1)=\Ss$, it is enough to show $\Tt_\ell\subset\pi_+(\Ss_1)$. Pick a point $a\in\Tt_\ell$. As $Z\setminus\Cc$ is a Nash normal crossing divisor of $M$, there exists an open semialgebraic neighborhood $U\subset M$ of $a$ equipped with a Nash diffeomorphism $u:U\to\R^d$ such that $u(a)=0$ and $u(U\cap Z)=\{x_t\cdots x_d=0\}$ for some $t\leq d$. By \ref{pdet} and the definition of $\Tt_\ell$ (see \ref{rktk}) the connected component $\Tt_\ell^a$ of $\Tt_\ell$ that contains $a$ is contained in the irreducible components of $Z$ that contain $a$. Thus, $u(\Tt_\ell^a\cap U)\subset\{x_t=0,\ldots,x_d=0\}$. Shrinking $U$ is necessary, we may assume $\Tt_\ell^a\cap U=\Tt_\ell\cap U$ is a closed subset of $U$ and by \ref{fnashnc2} we may modify $u$ in order to have $u(\Tt_\ell\cap U)=\{x_{e+1}=0,\ldots,x_d=0\}$ for some $1\leq e+1\leq t$. Consider coordinates $(x_{e+1},\ldots,x_d)$ in $\R^{d-e}$. By \ref{bigstepa5} there exists a Nash diffeomorphism $\Phi:\R^e\times[0,+\infty)\times\sph^{d-e-1}\to V:=\pi_+^{-1}(U)$ such that
$$
u\circ\pi_+\circ\Phi:\R^e\times[0,+\infty)\times\sph^{d-e-1}\to\R^d,\ (y,\rho,w)\mapsto(y,\rho w).
$$
Observe that $\Phi^{-1}(\Ee\cap V)=\R^e\times[0,+\infty)\times(\sph^{d-e-1}\cap\{x_t\cdots x_d=0\})$. We may assume
$$
\{x_t>0,\ldots,x_d>0\}\subset u((\Ss\setminus Z)\cap U). 
$$
Let $w_0\in\{x_t>0,\ldots,x_d>0\}\subset\R^{d-e}$ be a unitary vector. We have $(u\circ\pi_+\circ\Phi)(0,\rho,w_0)=\rho w_0\in u((\Ss\setminus Z)\cap U)$ for each $\rho\geq0$, so in particular $(u\circ\pi_+\circ\Phi)(0,0,w)=0$. Thus, 
$$
\Phi(0,0,w)\in(\pi^{-1}_+((\Ss\setminus(Z\setminus\Tt_\ell))\cap U)\cap\cl(\pi^{-1}_+((\Ss\setminus Z)\cap U)))\setminus\Ee\subset\Ss_1,
$$
so $a\in\pi_+(\Ss_1)$. Consequently, $\Tt_\ell\subset\pi_+(\Ss_1)$.

\paragraph{}
We have proved that $\Ss_1$ is (up to a suitable Nash embedding of $M'$) a checkerboard set such that $\ell(\Ss_1)=\ell(\Ss)-1$ and ${\tt Sth}(\Ss_1)$ is Nash diffeomorphic to ${\tt Sth}(\Ss)$ via $\pi_+$. By induction hypothesis there exist a connected Nash manifold $H$ with boundary and a surjective Nash map $f_1:H\to\Ss_1$ such that $\Int(H)$ is Nash diffeomorphic to ${\tt Sth}(\Ss_1)$, which is itself Nash diffeomorphic to ${\tt Sth}(\Ss)$ via $\pi_+$. Thus, $f:=\pi_+\circ f_1:H\to\Ss$ is a surjective Nash map and $\Int(H)$ is Nash diffeomorphic to ${\tt Sth}(\Ss)$, as required.
\qed

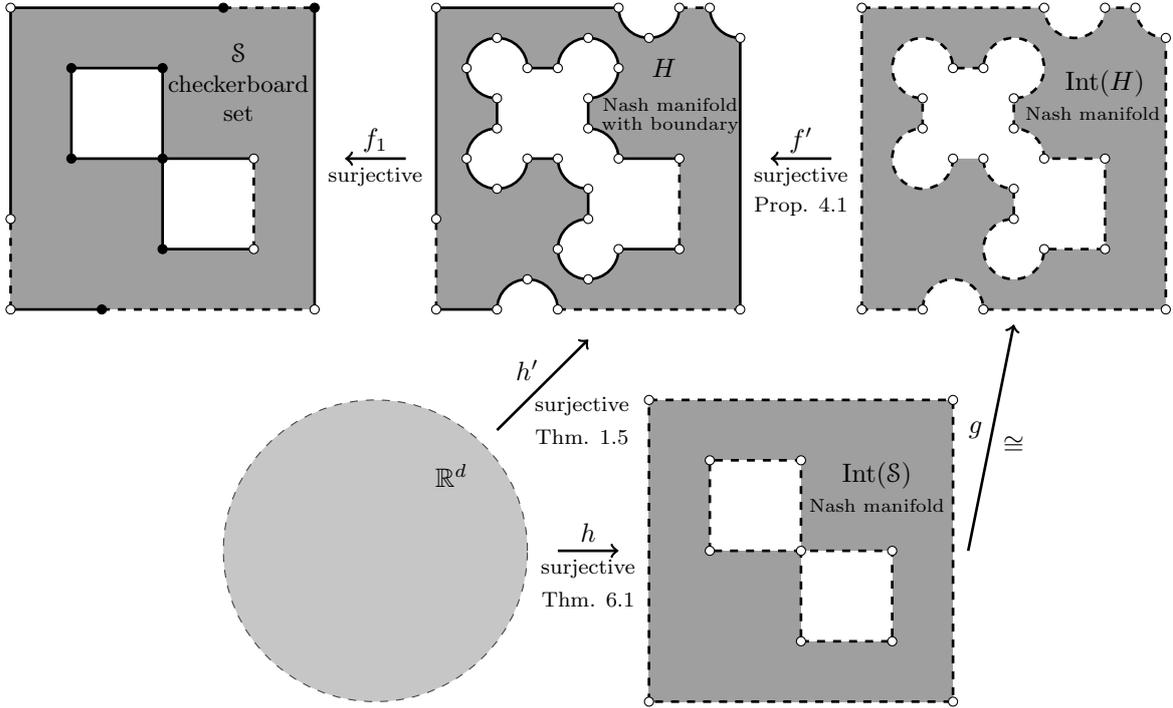
\begin{figure}[!ht]
\begin{center}
\begin{tikzpicture}[scale=0.8]


\draw[fill=gray!100,opacity=0.75,draw=none] (0,6.5) -- (0,11.5) -- (5,11.5) -- (5,6.5) -- (0,6.5);
\draw[fill=white,draw=none] (1,9) -- (1,10.5) -- (2.5,10.5) -- (2.5,9) -- (1,9);
\draw[fill=white,draw=none] (2.5,7.5) -- (4,7.5) -- (4,9) -- (2.5,9) -- (2.5,7.5);

\draw[line width=1pt] (0,6.5) -- (1.5,6.5);
\draw[line width=1pt,dashed] (0,6.5) -- (0,8);
\draw[line width=1pt] (0,8) -- (0,11.5);
\draw[line width=1pt] (3.5,11.5) -- (0,11.5);
\draw[line width=1pt,dashed] (3.5,11.5) -- (5,11.5);
\draw[line width=1pt] (5,6.5) -- (5,11.5);
\draw[line width=1pt,dashed] (1.5,6.5) -- (5,6.5);
\draw[line width=1pt] (1,9) -- (1,10.5) -- (2.5,10.5) -- (2.5,9) -- (1,9);
\draw[line width=1pt] (4,7.5) -- (2.5,7.5) -- (2.5,9) -- (4,9);
\draw[line width=1pt,dashed] (4,7.5) -- (4,9);

\draw[fill=white,draw] (0,6.5) circle (0.75mm);
\draw[fill=white,draw] (0,11.5) circle (0.75mm);
\draw[fill=black,draw] (5,11.5) circle (0.75mm);
\draw[fill=white,draw] (5,6.5) circle (0.75mm);
\draw[fill=black,draw] (3.5,11.5) circle (0.75mm);
\draw[fill=black,draw] (1.5,6.5) circle (0.75mm);
\draw[fill=white,draw] (0,8) circle (0.75mm);

\draw[fill=black,draw] (1,9) circle (0.75mm);
\draw[fill=black,draw] (1,10.5) circle (0.75mm);
\draw[fill=black,draw] (2.5,10.5) circle (0.75mm);
\draw[fill=black,draw] (2.5,9) circle (0.75mm);
\draw[fill=black,draw] (2.5,7.5) circle (0.75mm);
\draw[fill=white,draw] (4,7.5) circle (0.75mm);
\draw[fill=white,draw] (4,9) circle (0.75mm);

\draw (3.75,10.75) node{\small$\Ss$};
\draw (3.75,10.25) node{\footnotesize checkerboard};
\draw (3.75,9.75) node{\footnotesize set};


\draw[fill=gray!100,opacity=0.75,draw=none] (7,6.5) -- (7,11.5) -- (10,11.5) arc (180:360:0.5cm) -- (11.5,11.5) arc (180:270:0.5cm) -- (12,6.5) -- (9,6.5) arc (0:180:0.5cm) -- (7,6.5);
\draw[fill=white,draw=none] (8.5,9) arc (0:-270:0.5cm) -- (8,10) arc (270:0:0.5cm) -- (9,10.5) arc (180:-90:0.5cm) -- (9.5,9.5) arc (90:0:0.5cm) -- (11,9) -- (11,7.5) -- (10,7.5) arc (0:-270:0.5cm) -- (9.5,8) -- (9.5,8.5) arc (270:180:0.5cm);

\draw[line width=1pt] (8.5,9) arc (0:-270:0.5cm) -- (8,10) arc (270:0:0.5cm) -- (9,10.5) arc (180:-90:0.5cm) -- (9.5,9.5) arc (90:0:0.5cm) -- (11,9);

\draw[line width=1pt,dashed] (11,9) -- (11,7.5);

\draw[line width=1pt] (11,7.5) -- (10,7.5) arc (0:-270:0.5cm) -- (9.5,8) -- (9.5,8.5) arc (270:180:0.5cm) -- (8.5,9);

\draw[line width=1pt] (7,6.5) -- (8,6.5);
\draw[line width=1pt,dashed] (7,6.5) -- (7,8);
\draw[line width=1pt] (7,8) -- (7,11.5);
\draw[line width=1pt] (10,11.5) -- (7,11.5);
\draw[line width=1pt] (10,11.5) arc (180:360:0.5cm);
\draw[line width=1pt,dashed] (11,11.5) -- (11.5,11.5);
\draw[line width=1pt] (11.5,11.5) arc (180:270:0.5cm);
\draw[line width=1pt] (12,6.5) -- (12,11);
\draw[line width=1pt,dashed] (9,6.5) -- (12,6.5);
\draw[line width=1pt] (9,6.5) arc (0:180:0.5cm);

\draw[fill=white,draw] (7,6.5) circle (0.75mm);
\draw[fill=white,draw] (8.5,7) circle (0.75mm);
\draw[fill=white,draw] (7,11.5) circle (0.75mm);
\draw[fill=white,draw] (11,11.5) circle (0.75mm);
\draw[fill=white,draw] (10.5,11) circle (0.75mm);
\draw[fill=white,draw] (11.5,11.5) circle (0.75mm);
\draw[fill=white,draw] (12,11) circle (0.75mm);
\draw[fill=white,draw] (12,6.5) circle (0.75mm);
\draw[fill=white,draw] (10,11.5) circle (0.75mm);
\draw[fill=white,draw] (8,6.5) circle (0.75mm);
\draw[fill=white,draw] (9,6.5) circle (0.75mm);
\draw[fill=white,draw] (7,8) circle (0.75mm);

\draw[fill=white,draw] (8.5,9) circle (0.75mm);
\draw[fill=white,draw] (8,8.5) circle (0.75mm);
\draw[fill=white,draw] (7.5,9) circle (0.75mm);
\draw[fill=white,draw] (7.5,10.5) circle (0.75mm);
\draw[fill=white,draw] (8,11) circle (0.75mm);
\draw[fill=white,draw] (9.5,11) circle (0.75mm);
\draw[fill=white,draw] (10,10.5) circle (0.75mm);
\draw[fill=white,draw] (9,7.5) circle (0.75mm);
\draw[fill=white,draw] (9.5,7) circle (0.75mm);
\draw[fill=white,draw] (9,9) circle (0.75mm);
\draw[fill=white,draw] (8,9.5) circle (0.75mm);
\draw[fill=white,draw] (8,10) circle (0.75mm);
\draw[fill=white,draw] (8.5,10.5) circle (0.75mm);
\draw[fill=white,draw] (9,10.5) circle (0.75mm);
\draw[fill=white,draw] (9.5,9.5) circle (0.75mm);
\draw[fill=white,draw] (9.5,10) circle (0.75mm);
\draw[fill=white,draw] (10,9) circle (0.75mm);
\draw[fill=white,draw] (9.5,8.5) circle (0.75mm);
\draw[fill=white,draw] (9.5,8) circle (0.75mm);
\draw[fill=white,draw] (10,7.5) circle (0.75mm);
\draw[fill=white,draw] (11,7.5) circle (0.75mm);
\draw[fill=white,draw] (11,9) circle (0.75mm);

\draw (10.75,10.5) node{\small$H$};
\draw (10.85,9.9) node{\tiny Nash manifold};
\draw (10.85,9.55) node{\tiny with boundary};


\draw[fill=gray!100,opacity=0.75,draw=none] (14,6.5) -- (14,11.5) -- (17,11.5) arc (180:360:0.5cm) -- (18.5,11.5) arc (180:270:0.5cm) -- (19,6.5) -- (16,6.5) arc (0:180:0.5cm) -- (14,6.5);
\draw[fill=white,draw=none] (15.5,9) arc (0:-270:0.5cm) -- (15,10) arc (270:0:0.5cm) -- (16,10.5) arc (180:-90:0.5cm) -- (16.5,9.5) arc (90:0:0.5cm) -- (18,9) -- (18,7.5) -- (17,7.5) arc (0:-270:0.5cm) -- (16.5,8) -- (16.5,8.5) arc (270:180:0.5cm);

\draw[line width=1pt,dashed] (14,6.5) -- (14,11.5) -- (17,11.5) arc (180:360:0.5cm) -- (18.5,11.5) arc (180:270:0.5cm) -- (19,6.5) -- (16,6.5) arc (0:180:0.5cm) -- (14,6.5);
\draw[line width=1pt,dashed] (15.5,9) arc (0:-270:0.5cm) -- (15,10) arc (270:0:0.5cm) -- (16,10.5) arc (180:-90:0.5cm) -- (16.5,9.5) arc (90:0:0.5cm) -- (18,9) -- (18,7.5) -- (17,7.5) arc (0:-270:0.5cm) -- (16.5,8) -- (16.5,8.5) arc (270:180:0.5cm);

\draw[fill=white,draw] (14,6.5) circle (0.75mm);
\draw[fill=white,draw] (14,11.5) circle (0.75mm);
\draw[fill=white,draw] (18,11.5) circle (0.75mm);
\draw[fill=white,draw] (18.5,11.5) circle (0.75mm);
\draw[fill=white,draw] (19,11) circle (0.75mm);
\draw[fill=white,draw] (19,6.5) circle (0.75mm);
\draw[fill=white,draw] (17,11.5) circle (0.75mm);
\draw[fill=white,draw] (15,6.5) circle (0.75mm);
\draw[fill=white,draw] (16,6.5) circle (0.75mm);
\draw[fill=white,draw] (15.5,9) circle (0.75mm);
\draw[fill=white,draw] (16,9) circle (0.75mm);
\draw[fill=white,draw] (15,9.5) circle (0.75mm);
\draw[fill=white,draw] (15,10) circle (0.75mm);
\draw[fill=white,draw] (15.5,10.5) circle (0.75mm);
\draw[fill=white,draw] (16,10.5) circle (0.75mm);
\draw[fill=white,draw] (16.5,9.5) circle (0.75mm);
\draw[fill=white,draw] (16.5,10) circle (0.75mm);
\draw[fill=white,draw] (17,9) circle (0.75mm);
\draw[fill=white,draw] (16.5,8.5) circle (0.75mm);
\draw[fill=white,draw] (16.5,8) circle (0.75mm);
\draw[fill=white,draw] (17,7.5) circle (0.75mm);
\draw[fill=white,draw] (18,7.5) circle (0.75mm);
\draw[fill=white,draw] (18,9) circle (0.75mm);

\draw (18,10.25) node{\small${\rm Int}(H)$};
\draw (17.8,9.75) node{\tiny Nash manifold};


\draw[fill=gray!100,opacity=0.75,draw=none] (10.5,0) -- (10.5,5) -- (15.5,5) -- (15.5,0) -- (10.5,0);
\draw[fill=white,draw=none] (11.5,2.5) -- (11.5,4) -- (13,4) -- (13,2.5) -- (11.5,2.5);
\draw[fill=white,draw=none] (13,1) -- (14.5,1) -- (14.5,2.5) -- (13,2.5) -- (13,1);

\draw[line width=1pt,dashed] (10.5,0) -- (10.5,5) -- (15.5,5) -- (15.5,0) -- (10.5,0);
\draw[line width=1pt,dashed] (11.5,2.5) -- (11.5,4) -- (13,4) -- (13,2.5) -- (11.5,2.5);
\draw[line width=1pt,dashed] (13,1) -- (14.5,1) -- (14.5,2.5) -- (13,2.5) -- (13,1);

\draw[fill=white,draw] (10.5,0) circle (0.75mm);
\draw[fill=white,draw] (10.5,5) circle (0.75mm);
\draw[fill=white,draw] (15.5,5) circle (0.75mm);
\draw[fill=white,draw] (15.5,0) circle (0.75mm);

\draw[fill=white,draw] (11.5,2.5) circle (0.75mm);
\draw[fill=white,draw] (11.5,4) circle (0.75mm);
\draw[fill=white,draw] (13,4) circle (0.75mm);
\draw[fill=white,draw] (13,2.5) circle (0.75mm);
\draw[fill=white,draw] (13,1) circle (0.75mm);
\draw[fill=white,draw] (14.5,1) circle (0.75mm);
\draw[fill=white,draw] (14.5,2.5) circle (0.75mm);

\draw[fill=gray!60,opacity=0.75,draw,dashed] (6,2.5) circle (2.5cm);

\draw (14.25,3.75) node{\small${\rm Int}(\Ss)$};
\draw (14.25,3.25) node{\tiny Nash manifold};
\draw (7.25,3.75) node{\small$\R^d$};

\draw[line width=1pt,<-] (5.5,9) -- (6.5,9);
\draw(6,9.3) node{\small$f_1$};
\draw(6,8.7) node{\scriptsize surjective};
\draw[line width=1pt,<-] (12.5,9) -- (13.5,9);
\draw(13,9.3) node{\small$f'$};
\draw(13,8.7) node{\scriptsize surjective};
\draw(13,8.2) node{\scriptsize Prop. \ref{doubleivc}};
\draw[line width=1pt,->] (15.75,2.5) -- (16.5,6.25);
\draw(15.875,4.5) node{\small$g$};
\draw(16.5,4.25) node{\small$\cong$};
\draw[line width=1pt,->] (9,2.5) -- (10,2.5);
\draw(9.5,2.8) node{\small$h$};
\draw(9.5,2.2) node{\scriptsize surjective};
\draw(9.5,1.7) node{\scriptsize Thm. \ref{mstone0}};
\draw[line width=1pt,->] (8,4.5) -- (9.5,6);
\draw(8.5,5.5) node{\small$h'$};
\draw(9.4,4.9) node{\scriptsize surjective};
\draw(9.4,4.4) node{\scriptsize Thm. \ref{mstone}};

\end{tikzpicture}
\end{center}
\caption{Sketch of proof of the implication (vii) $\Longrightarrow$ (i) of Theorem \ref{main}.\label{fig15}}
\end{figure}

\section{Nash path-connected components of a semialgebraic set}\label{s9}

To take advantage of the full strength of Theorem \ref{main} applied to an arbitrary semialgebraic set $\Ss$ we introduce the Nash path-components of a semialgebraic set. Recall that by Theorem \ref{main} Nash path-connected and well-welded semialgebraic sets coincide.

\begin{define}\label{compww}
A semialgebraic set $\Ss\subset\R^n$ admits a decomposition into \em Nash path-connected components \em if there exist semialgebraic sets $\Ss_1,\ldots,\Ss_r\subset\Ss$ such that:
\begin{itemize}
\item[(1)] Each $\Ss_i$ is Nash path-connected.
\item[(2)] If $\Tt\subset\Ss$ is a Nash path-connected semialgebraic set that contains $\Ss_i$, then $\Ss_i=\Tt$.
\item[(3)] $\Ss_i\not\subset\bigcup_{j\neq i}\Ss_j$.
\item[(4)] $\Ss=\bigcup_{i=1}^r\Ss_i$.
\end{itemize}
\end{define}

\begin{thm}\label{decompww}
Let $\Ss\subset\R^n$ be a semialgebraic set. Then $\Ss$ admits a decomposition into Nash path-connected components and this decomposition is unique. In addition, the Nash path-connected components of a semialgebraic set are closed in $\Ss$.
\end{thm}

Before proving Theorem \ref{decompww} we need a preliminary result.

\begin{lem}\label{union}
Let $\Ss_1,\Ss_2\subset\R^n$ be two well-welded semialgebraic sets of dimension $d$ such that $\dim (\Ss_1\cap\Ss_2)=d$. Then $\Ss:=\Ss_1\cup\Ss_2$ is well-welded.
\end{lem}
\begin{proof}
Let $x_k\in\Ss_k$ for $k=1,2$ and let $y\in(\Ss_1\cap\Ss_2)\setminus\Sing(\Ss)$. As each $\Ss_k$ is well-welded, there exist by Corollary \ref{clue} continuous semialgebraic paths $\alpha_k:[0,1]\to\Ss_k$ such that $\alpha_k(0)=x_k,\alpha_k(1)=y$ and $\eta(\alpha_k)\subset\Reg(\Ss_k)\setminus\ol{\Sing(\Ss)}^{\zar}$. The continuous semialgebraic path $\alpha:=\alpha_1*\alpha_2^{-1}$ connects the points $x_1$ and $x_2$ and satisfies 
$$
\eta(\alpha)\subset\eta(\alpha_1)\cup\eta(\alpha_2)\cup\{y\}\subset\Reg(\Ss).
$$
Consequently, $\Ss$ is well-welded, as required.
\end{proof}

\begin{proof}[Proof of Theorem \em \ref{decompww}]
We divide the proof in two parts:

\noindent{\em Existence.} We proceed by induction on the dimension $d$ of $\Ss$. If $d=0$, the Nash path-connected components of $\Ss$ coincide with its connected components. Suppose the result true for dimension $\leq d-1$ and let us see that it is also true for dimension $d$.

Let $\Ss_0$ be the (semialgebraic) set of points of dimension $d$ of $\Ss$. Write $\Ss=\Ss_0\cup\Tt$ where $\Tt:=\Ss\cap\ol{\Ss\setminus\Ss_0}^{\zar}$. Observe that $\Ss_0$ and $\Tt$ are closed subsets of $\Ss$ and $\dim(\Tt)\leq d-1$. In addition, $\ol{\Tt}^{\zar}=\ol{\Ss\setminus\Ss_0}^{\zar}$. Let $M_1,\ldots,M_r$ be the connected components of $\Reg(\Ss_0)$. Observe that $\cl(M_i)\cap\Ss$ is Nash path-connected for $i=1,\ldots,r$. \renewcommand{\theparagraph}{\thesection.2.\arabic{paragraph}}\setcounter{paragraph}{0}

\paragraph{}\label{dimd} We claim: \em There exists a partition $I_1,\ldots,I_\ell$ of the set $\{1,\ldots,r\}$ such that 
$$
\Ss_k:=\bigcup_{i\in I_k}\cl(M_i)\cap\Ss
$$ 
is Nash path-connected for $k=1,\ldots,\ell$ and for each $J\subset\{1,\ldots,\ell\}$ of cardinal $\geq2$ the semialgebraic set $\bigcup_{j\in J}\Ss_j$ is not Nash path-connected\em.

Define 
$$
{\mathfrak F}_1:=\Big\{I\subset\{1,\ldots,r\}:\ 1\in I\quad\text{and}\quad \bigcup_{i\in I}\cl(M_i)\cap\Ss\quad\text{is Nash path-connected}\Big\}.
$$
Observe that ${\mathfrak F}_1\neq\varnothing$. If $J_1,J_2\subset{\mathfrak F}_1$, then by Lemma \ref{union} $J_1\cup J_2\in{\mathfrak F}_1$. Let $I_1$ be the maximum of ${\mathfrak F}_1$ ordered with respect to the inclusion. Denote $K:=\{1,\ldots,r\}\setminus I_1$. By induction there exists a partition $I_2,\ldots,I_\ell$ of $K$ such that $\Ss_k:=\bigcup_{i\in I_k}\cl(M_i)\cap\Ss$ is Nash path-connected for $k=2,\ldots,\ell$ and for each $J\subset\{2,\ldots,\ell\}$ of cardinal $\geq2$ the semialgebraic set $\bigcup_{j\in J}\Ss_j$ is not Nash path-connected. It can be checked that the sets $I_1,\ldots,I_\ell$ satisfy the required properties.

\paragraph{}Observe that $\Ss_0=\bigcup_{i=1}^\ell\Ss_i$ and by Lemma \ref{union} each intersection $\Ss_i\cap\Ss_j$ has dimension $\leq d-1$ if $i\neq j$. 

\paragraph{}By induction hypothesis there exist a family $\Tt_1,\ldots,\Tt_p$ of Nash path-connected components of $\Tt$. We may assume, after eliminating redundant $\Tt_j$, that $\Tt_1,\ldots,\Tt_m$ satisfy the following: $\Ss=\Ss_0\cup\bigcup_{j=1}^m\Tt_j$ and $\Tt_j\not\subset\Ss_0\cup\bigcup_{k\neq j}\Tt_k$ for $j=1,\ldots,m$. 

\paragraph{}Let us check that $\Ss_1,\ldots,\Ss_\ell,\Tt_1,\ldots,\Tt_m$ are a family of Nash path-connected components of $\Ss$. By construction they satisfy conditions (1), (3) and (4) of Definition \ref{compww}. Let us check that they also satisfy condition (2). In particular, by Lemma \ref{chain} a Nash path-connected component of $\Ss$ is a closed subset of $\Ss$. We distinguish two possibilities accordingly to the different dimensions of the semialgebraic sets $\Ss_i$ and $\Tt_j$.

\paragraph{}Let $\Rr\subset\Ss$ be a Nash path-connected semialgebraic set that contains $\Ss_1$. As $\Ss_1$ has dimension $d$, also $\Rr$ has dimension $d$. As $\Rr$ is pure dimensional, it is contained in $\Ss_0=\bigcup_{i=1}^\ell\Ss_i$. We may assume that $\dim(\Rr\cap\Ss_i)=d$ exactly for $1\leq i\leq s\leq\ell$. We claim: $s=1$.

Otherwise, each union $\Ss_j\cup\Rr$ is Nash path-connected by Lemma \ref{union} for $j=1,\ldots,s$. Consequently, by Lemma \ref{union} $\Rr':=\Rr\cup\bigcup_{j=1}^s\Ss_j$ is Nash path-connected. On the other hand, $\dim(\Rr\cap\Ss_j)<d$ for $j=s+1,\ldots,\ell$, so the semialgebraic set $\Cc':=\bigcup_{j=s+1}^\ell\Rr\cap\Ss_j$ has dimension $<d-1$ and satisfies 
$$
\Rr'=\Rr\cup\bigcup_{j=1}^s\Ss_j=\bigcup_{j=1}^s\Ss_j\cup\Cc'
$$
because $\Rr\subset\Ss_0=\bigcup_{i=1}^\ell\Ss_i$. As $\Rr'$ is pure dimensional,
$$
\cl(\Rr')\cap\Ss=\cl(\Rr'\setminus\Cc')\cap\Ss=\cl\Big(\bigcup_{j=1}^s\Ss_j\Big)\cap\Ss=\bigcup_{j=1}^s\Ss_j.
$$
Consequently, $\bigcup_{j=1}^s\Ss_j$ is Nash path-connected, which contradicts \ref{dimd}.

As $s=1$, it holds $\dim(\Rr\cap\Ss_j)<d$ for $2\leq j\leq \ell$. Thus, the semialgebraic set $\Cc:=\bigcup_{j=2}^\ell\Rr\cap\Ss_j$ has dimension $<d$ and satisfies $\Rr=\Ss_1\cup\Cc$. As $\Rr$ is pure dimensional, 
$$
\Ss_1\subset\Rr\subset\cl(\Rr)\cap\Ss=\cl(\Rr\setminus\Cc)\cap\Ss=\cl(\Ss_1\setminus\Cc)\cap\Ss=\Ss_1. 
$$
Consequently, $\Ss_1=\Rr$.

\paragraph{}Let $\Rr\subset\Ss$ be a Nash path-connected semialgebraic set that contains $\Tt_j$. We claim: $\Rr\subset\Tt$. Assume this proved for a while. As $\Tt_j$ is a Nash path-connected component of $\Tt$, we have $\Rr=\Tt_j$. Thus, the semialgebraic sets $\Ss_1,\ldots,\Ss_\ell$, $\Tt_1,\ldots,\Tt_m$ satisfy condition (2) of Definition \ref{compww}. Consequently, it is enough to prove: $\Rr\subset\Tt$.

As $\Rr$ is Nash path-connected, it is pure dimensional. As $\Rr\not\subset\Ss_0$ (because $\Tt_j\not\subset\Ss_0$) and $\Ss_0$ is a closed subset of $\Ss$,
$$
\dim(\Rr)=\dim(\Rr\setminus\Ss_0)=\dim(\Rr\cap(\Ss\setminus\Ss_0))\leq\dim(\Rr\cap\Tt)\leq\dim(\Rr).
$$
As $\Rr$ is Nash path-connected, it is irreducible. As $\dim(\Rr\cap\Tt)=\dim(\Rr)$, we have $\ol{\Rr}^{\zar}=\ol{\Rr\cap\Tt}^{\zar}\subset\ol{\Tt}^{\zar}=\ol{\Ss\setminus\Ss_0}^{\zar}$, so
$$
\Rr\subset\Ss\cap\ol{\Rr}^{\zar}\subset\Ss\cap\ol{\Tt}^{\zar}=\Ss\cap\ol{\Ss\setminus\Ss_0}^{\zar}=:\Tt. 
$$

\noindent{\em Uniqueness.} Let $\{\Ss_i\}_{i=1}^r$ and $\{\Rr_j\}_{j=1}^s$ be two families of semialgebraic sets satisfying the conditions of Definition \ref{compww}. Assume $s\leq r$.

Let $x\in\Rr_1\not\subset\bigcup_{j=2}^s\Rr_j$. As each $\Rr_j$ is a closed subset of $\Ss$, the difference $\Ss\setminus\bigcup_{j\neq 1}\Rr_j=\Rr_1\setminus \bigcup_{j\neq 1}\Rr_j$ is an open neighborhood of $x$ in $\Ss$, so $\Ss_x=\Rr_{1,x}$. In particular, $\Ss$ and $\Rr_1$ have the same dimension at $x$. As $\Ss_x=\Rr_{1,x}$ and $\dim(\Ss_x)=\max\{\dim(\Ss_{i,x}):\ i=1,\ldots,r\}$, we assume $\dim(\Ss_{1,x})=\dim(\Ss_x)=\dim(\Rr_{1,x})$. As $\Ss_{1,x}=\Ss_{1,x}\cap\Ss_x=\Ss_{1,x}\cap\Rr_{1,x}$, we have $\dim(\Ss_{1,x})=\dim(\Rr_{1,x})=\dim(\Ss_{1,x}\cap\Rr_{1,x})$. As $\Ss_1,\Rr_1$ are pure dimensional, $\dim(\Ss_1)=\dim(\Rr_1)=\dim(\Rr_1\cap\Ss_1)$. By Lemma \ref{union} $\Ss_1\cup\Rr_1$ is Nash path-connected, so by condition (2) $\Ss_1=\Ss_1\cup\Rr_1=\Rr_1$. Proceeding similarly with $\Rr_2,\ldots,\Rr_s$, we assume $\Rr_i=\Ss_i$ for $i=1,\ldots,s$. As $\Ss=\bigcup_{i=1}^s\Rr_i=\bigcup_{i=1}^s\Ss_i$, we deduce by condition (3) that $r=s$. Consequently, the families $\{\Ss_i\}_{i=1}^r$ and $\{\Rr_j\}_{j=1}^s$ coincide, as required.
\end{proof}

\begin{examples}
(i) Let $\Ss:=\{x^2z-y^2=0\}\subset\R^3$. The Nash path-connected components of $\Ss$ are $\Ss_1:=\Ss\cap\{z\geq 0\}$ and $\Ss_2:=\{x=0,y=0\}$.

\vspace{2mm}
(ii) Let $\Ss:=\Ss_1\cup\Ss_2\cup\Ss_3\subset\R^3$, where
\begin{multline*}
\Ss_1:=[-1,1]\times[-2,2]\times\{0\},\quad\Ss_2:=[-2,-1]\times\{-1,1\}\times[-1,1],\\
\&\quad\Ss_3:=[1,2]\times\{-1,1\}\times[-1,1].
\end{multline*}
The Nash path-connected components of $\Ss$ are $\Ss_1$, $\Ss_2\cap\{y=1\}$, $\Ss_2\cap\{y=-1\}$, $\Ss_3\cap\{y=1\}$ and $\Ss_3\cap\{y=-1\}$. In contrast, $X$ has three irreducible components \cite[Rmk.4.4(iii)]{fg3}, which are $\Ss\cap\{z=0\}$, $\Ss\cap\{y=1\}$ and $\Ss\cap\{y=-1\}$.
\end{examples}

\section{Two relevant consequences of Main Theorem \ref{main}}\label{s10}

In this section we prove Corollaries \ref{consq1} and \ref{consq2}.

\subsection{Proof of Corollary \ref{consq1}}\setcounter{paragraph}{0}
By Theorem \ref{main} it is enough to prove that $\Ss$ is well-welded. The proof is conducted in several steps.

\paragraph{}As $\Ss$ is an irreducible semialgebraic set, $\cl(\Ss)$ is an irreducible arc-symmetric semialgebraic set. Let $X$ be the Zariski closure of $\Ss$ and let $\pi:\widetilde{X}\to X$ be a resolution of the singularities of $X$ (Theorem \ref{hi1}). By \cite[Thm.2.6]{k} applied to the irreducible arc-symmetric set $\cl(\Ss)$ there exists a connected component $E$ of $\widetilde{X}$ such that $\pi(E)=\cl(\Reg(\cl(\Ss)))=\cl(\Ss)$ (recall that $\cl(\Ss)$ is pure dimensional). As $\pi|_{\widetilde{X}\setminus\pi^{-1}(\Sing(X))}:\widetilde{X}\setminus\pi^{-1}(\Sing(X))\to X\setminus\Sing(X)$ is a regular diffeomorphism and $\Ss\setminus\Sing(X)$ is dense in $\cl(\Ss)$, we have $\cl(\pi^{-1}(\Ss\setminus\Sing(X))=E$. Let $\widetilde{\Ss}:=\pi^{-1}(\Ss)\cap\cl(\pi^{-1}(\Ss\setminus\Sing(X)))=\pi^{-1}(\Ss)\cap E$ be the strict transform of $\Ss$ under $\pi$. We claim: \em $\widetilde{\Ss}$ is connected\em.

Otherwise, there exists a closed semialgebraic subset $\Cc\subset E$ such that $\widetilde{\Ss}\subset E\setminus\Cc$ and $E\setminus\Cc$ is not connected. As $\pi$ is proper, $\Cc':=\pi(\Cc)$ is a closed subset of $X$. Observe that $\Ss\cap\Cc'=\varnothing$. Let $Y$ be the smallest Nash subset of $\R^n\setminus\Cc'$ that contains $\Ss$. As $\Ss$ is irreducible, also $Y$ is irreducible. Observe that $Y$ is an irreducible component of the Nash subset $X\setminus\Cc'$ of $\R^n\setminus\Cc'$ and $(\widetilde{Y}:=\pi^{-1}(Y)\cap E,\pi|_{\widetilde{Y}})$ is a resolution of the singularities of $Y$. As $Y$ is irreducible, also $\widetilde{Y}$ is irreducible, so $\widetilde{Y}\subset E\setminus\Cc$ is connected. As $\widetilde{\Ss}$ is dense in $E$, we conclude that $\widetilde{Y}$ meets all the connected components of $E\setminus\Cc$, which is a contradiction because $\widetilde{Y}$ is connected but $E\setminus\Cc$ is not. Thus, $\widetilde{\Ss}$ is connected.

\paragraph{}We prove next: \em $\widetilde{\Ss}$ is well-welded\em. Once this is done, $\Ss=\pi(\widetilde{\Ss})$ is by Lemma \ref{imageww} well-welded. As ${\widetilde\Ss}$ is a connected dense semialgebraic subset of the Nash manifold $E$, it is enough to check: \em If $\Tt$ is a connected dense semialgebraic subset of a Nash manifold $M$, then $\Tt$ is well-welded\em.

Let $p\in\Tt\setminus\Int_M(\Tt)$, where $\Int_M(\Tt)$ denotes the interior of $\Tt$ in $M$. Let $\Cc_1,\ldots,\Cc_r$ be the connected components of $\Int_M(\Tt)$ whose closures contain $p$. Let $L\subset M$ be a compact semialgebraic neighborhood of $p$. By \cite[Thm.9.2.1 \& Rmk.9.2.3]{bcr} there exists a finite simplicial complex $K$ and a semialgebraic homeomorphism $\Phi:|K|\to L$ such that $\{p\}$ and each semialgebraic set $\Cc_k\cap L$ is a finite union of images $\Phi(\sigma^0)$ where $\sigma\in K$ and the restriction $\Phi|_{\sigma^0}:\sigma^0\to L\subset M$ is a Nash embedding for each $\sigma\in K$.

It holds that given two connected components $\Cc_i$ and $\Cc_j$ there exists finitely many connected components $\Cc_{i_1},\ldots,\Cc_{i_s}$ such that $\Cc_i=\Cc_{i_1}$, $\Cc_j=\Cc_{i_s}$ and $\cl(\Cc_{i_k})\cap\cl(\Cc_{i_{k+1}})\cap L$ contains the image under $\Phi$ of a simplex of dimension $d-1$ that contains $\{p\}$ as a vertex. 

\paragraph{}To prove that $\Tt$ is well-welded, it is enough by Corollary \ref{charww} to show the following: \em Let $\Cc_1$ and $\Cc_2$ be two connected components of $\Int_M(\Tt)$ that contain $p$ in their closure and $\cl(\Cc_1)\cap\cl(\Cc_2)\cap L$ contains the image under $\Phi$ of a simplex $\sigma$ of dimension $\dim(\Ss)-1$ that contains $\{p\}$ as a vertex. Then there exists an analytic path $\alpha:[-1,1]\to\Cc_1\cup\Cc_2\cup\{p\}$ such that $\alpha((-1,0))\subset\Cc_1$, $\alpha((0,1))\subset\Cc_2$ and $\alpha(0)=\{p\}$\em.

Denote $\Aa_0:=\Phi(\sigma^0)$. By the Nash curve selection lemma there exists a Nash path $\beta:(-1,1)\to M$ such that $\beta(0)=p$ and $\beta((0,1))\subset\Aa_0$. After a change of coordinates, we may assume by \cite[Cor.9.3.10]{bcr} that the projection $\rho:\R^d:=\R^{d-1}\times\R\to\R^{d-1}\times\{0\}$ induces a Nash diffeomorphism $\rho|_{\Aa}:\Aa\to\rho(\Aa)$ where $\Aa$ is an open semialgebraic subset of $\Aa_0$ that contains $p$ in its closure. Let $\phi:\rho(\Aa)\to\R$ be a Nash function such that $\Aa={\rm graph}(\phi)$. The Nash map 
$$
\varphi:\Omega:=\rho(\Aa)\times\R\to\rho(\Aa)\times\R,\ (x',x_d)\mapsto(x',x_d-\phi(x'))
$$ 
induces a Nash diffeomorphism that maps $\Aa$ onto $\rho(\Aa)\times\{0\}$. In addition, we may assume that $\varphi(\Cc_1\cap\Int(L)\cap\Omega)$ is an open semialgebraic subset of $\rho(\Aa)\times(-\infty,0)$ whose boundary contains $\rho(\Aa)\times\{0\}$ and $\varphi(\Cc_2\cap\Int(L)\cap\Omega)$ is an open semialgebraic subset of $\rho(\Aa)\times(0,+\infty)$ whose boundary contains $\rho(\Aa)\times\{0\}$. For each $\ell\geq1$ consider the Nash path 
$$
\gamma_\ell:(-1,1)\to\R^d,\ t\mapsto((\rho\circ\beta)(t^2),t^{2\ell+1}). 
$$
Observe that for $\ell\geq1$ large enough, we may assume $\gamma_\ell((-1,0))\subset\varphi(\Cc_1\cap\Int(L)\cap\Omega)$, $\gamma_\ell(0)=(\rho(p),0)$ and $\gamma_\ell((0,1))\subset\varphi(\Cc_2\cap\Int(L)\cap\Omega)$. Thus, for $\ell\geq1$ large enough the Nash path 
$$
\alpha_\ell:(-1,1)\to\R^d,\ t\mapsto\beta(t^2)+(0,\ldots,0,t^{2\ell+1})
$$
satisfies $\alpha_\ell((-1,0))\subset\Cc_1$, $\alpha_\ell(0)=p$ and $\alpha_\ell((0,1))\subset\Cc_2$ because $\alpha_\ell|_{(-1,1)\setminus\{0\}}=\varphi^{-1}\circ\gamma_\ell|_{(-1,1)\setminus\{0\}}$, as required.
\qed

\subsection{Proof of Corollary \ref{consq2}}
(i) By Theorem \ref{main} there exists a Nash map $f:\R^d\to\R^n$ whose image is $\Ss$. By Artin-Mazur's description \cite[Thm.8.4.4]{bcr} of Nash maps there exist $s\geq1$ and a non-singular irreducible algebraic set $Z\subset\R^{d+n+s}$ of dimension $d$, a connected component $M$ of $Z$ and a Nash diffeomorphism $g:\R^d\to M$ such that the following diagram is commutative.
$$
\xymatrix{
Z\ar@{^{(}->}[rr]&&\R^d\times\R^n\times\R^s\equiv\R^m\ar[dd]_{\pi_2}\ar[ddll]_{\pi_1}\\
M\ar@{^{(}->}[u]&&\\
\R^d\ar@{<->}[u]_{\cong}^g\ar[rr]^{f}\ar[rd]^{f}&&\R^n\\
&\Ss\ar@{^{(}->}[ru]&
}
$$
We denote the projection of $\R^d\times\R^n\times\R^s$ onto the first space $\R^d$ with $\pi_1$ and the projection of $\R^d\times\R^n\times\R^s$ onto the second space $\R^n$ with $\pi_2$. Write $m:=d+n+s$. By \cite{m} applied to $M$ and the union of the remaining connected components of $Z$ there exist finitely many polynomials $P_1,\ldots,P_\ell,Q_1,\ldots,Q_\ell\in\R[\x]:=\R[\x_1,\ldots,\x_m]$ such that each $Q_j$ is strictly positive on $\R^m$ and 
$$
M=Z\cap\Big\{\sum_{j=1}^\ell P_j\sqrt{Q_j}>0\Big\}.
$$
Observe that $M$ is the projection of the algebraic set
$$
Y:=\Big\{(x,y,t)\in Z\times\R^\ell\times\R:\ \Big(\sum_{j=1}^\ell P_jy_j^2\Big)t^2-1=0,\ y_j^4-Q_j=0\ \text{for $j=1,\ldots,\ell$}\Big\}
$$
under $\pi:\R^m\times\R^\ell\times\R\to\R^m,\ (x,y,t)\mapsto x$. 

Fix $\epsilon:=(\epsilon_1,\ldots,\epsilon_\ell,\epsilon_{\ell+1})\in\{-1,1\}^{\ell+1}$ and let $M_\epsilon:=Y\cap\{\epsilon_1y_1>0,\ldots,\epsilon_\ell y_\ell>0,\epsilon_{\ell+1}t>0\}$. Consider the Nash diffeomorphism 
$$
\varphi_\epsilon:M\to M_\epsilon,\ x\mapsto\Big(x,\epsilon_1\sqrt[4]{Q_1(x)},\ldots,\epsilon_\ell\sqrt[4]{Q_\ell(x)},\epsilon_{\ell+1}\frac{1}{\sqrt{\sum_{j=1}^\ell P_j(x)\sqrt{Q_j(x)}}}\Big)
$$
whose inverse map is the restriction to $M_\epsilon$ of the projection $\pi$.

Observe that $\{M_\epsilon\}_{\epsilon\in\{-1,1\}^{\ell+1}}$ is the collection of the connected components of $Y$. As $\pi(M_\epsilon)=M$ and using the diagram above, we deduce
$$
(\pi_2\circ\pi)(M_{\epsilon})=\pi_2(M)=(f\circ\pi_1)(M)=f(\R^d)=\Ss.
$$
In addition, each $M_\epsilon$ is Nash diffeomorphic to $\R^d$ and for $\veps\neq\veps'$ the polynomial map 
$$
\phi_{\epsilon,\epsilon'}:\R^m\times\R^\ell\times\R\to\R^m\times\R^\ell\times\R,\ (x,y,t)\mapsto(x,(\epsilon\cdot\epsilon')\cdot(y,t))
$$
induces an involution of $Y$ such that $\phi_{\epsilon,\epsilon'}(M_\epsilon)=M_{\epsilon'}$. We have denoted
$$
(\epsilon\cdot\epsilon')\cdot(y,t):=(\epsilon_1\epsilon_1'y_1,\ldots,\epsilon_\ell\epsilon_\ell'y_\ell,\epsilon_{\ell+1}\epsilon_{\ell+1}'t).
$$
As $Z$ is non-singular, also $Y$ is non-singular. Let $X$ be the irreducible component of $Y$ that contains $M_{(1,\ldots,1)}$. Then $k:=m+\ell+1-n$ and the non-singular algebraic set $X$ satisfy the requirements in the statement.

In addition, each connected component of $X$ is Nash diffeomorphic to $\R^d$ and it has finitely many. Thus, $X$ is Nash diffeomorphic to $\R^d\times\{1,\ldots,s\}$, where $s$ is the number of connected components of $X$.

(ii) Let $\Ss_1,\ldots,\Ss_r$ be the Nash path-components of $\Ss$, which satisfy $\Ss=\bigcup_{i=1}^r\Ss_i$. By (i) there exist $m\geq1$ and for each $i=1,\ldots,r$ a non-singular algebraic set $X_i\subset\R^m$ that is Nash diffeomorphic to a disjoint union of affine subspaces of $\R^{d+1}$ (all of them affinely equivalent to $\R^{d_i}$ where $d_i:=\dim(\Ss_i)\leq d=\dim(\Ss)$) and satisfies $\pi(X_i)=\Ss_i$, where $\pi:\R^n\times\R^{m-n}\to\R^n,\ (x,y)\mapsto x$ is the projection onto the first $n$ coordinates. Consider the pairwise disjoint union $X:=\bigsqcup_{i=1}^rX_i\times\{i\}\subset\R^{m+1}$ and the projection $\pi':\R^n\times\R^{m+1-n}\times\R\to\R^n,\ (x,y,t)\mapsto x$. Then $X$ is a non-singular algebraic set, which is Nash diffeomorphic to a pairwise disjoint union of affine subspaces of $\R^{d+1}$ and satisfies $\pi(X)=\Ss$, as required.
\qed

The following example shows that Corollary \ref{consq2} is somehow sharp.

\begin{example}\label{nint}
Let $X\subset\R^n$ be a real algebraic curve Nash diffeomorphic to $\R$. Let $\pi:\R^n\to\R$ be a linear projection. Then $\pi(X)$ is not a proper open interval of $\R$.

Notice that $Y:=\cl_{\R\PP^n}(X)=X\cup\{p_{\infty}\}$ where $p_{\infty}$ is a certain point of the hyperplane of infinity of $\R\PP^n$. Observe that $\pi$ is the restriction to $\R^n$ of a central projection $\Pi:\R\PP^n\dasharrow\R\PP^1$ with center a projective subspace $L$ of $H_{\infty}(\R)$ of dimension $n-2$. 

If $p_{\infty}\not\in L$, then $\Pi(Y)$ is a compact subset of $\R\PP^1$ and $\Pi(p_{\infty})$ is the point at infinity of $\R\PP^1$. Thus, $\pi(X)$ is a closed semialgebraic subset of $\R$.

If $p_{\infty}\in L$, we assume by contradiction that $\pi(X)$ is a proper open interval of $\R$. Then $Y$ has at least two different tangents at $p_{\infty}$. However, as $X$ is Nash diffeomorphic to $\R$, the analytic germ $Y_{p_{\infty}}$ has only one branch, which is a contradiction. Thus, $\pi(X)$ is not a proper open interval of $\R$.
\end{example}

\begin{remark}\label{nint2}
Let $\Ss:=(0,1)\subset\R$. By Corollary \ref{consq2} there exist $n>0$ and an algebraic set $X\subset\R^{n+1}$ whose connected components are Nash diffeomorphic to $\R$ and a projection $\pi:\R^{n+1}\to\R$ such that $\pi(X)=(0,1)$. By Example \ref{nint} we know that $X$ is not connected.
\end{remark}

\appendix

\section{Miscellanea of ${\mathcal C}^2$ semialgebraic diffeomorphisms between intervals}\label{A}

We present examples of ${\mathcal S}^2$ diffeomorphisms between intervals required in Section \ref{s4} and \ref{s5}.

\begin{examples}\label{s2diffeo}
(i) The function $f_1:[\frac{1}{4},1)\to[0,1)$ given by
$$
f_1(x):=\begin{cases}
\frac{5}{12}(4x-1)&\text{ if $\frac{1}{4}\leq x\leq\frac{1}{2}$,}\\
\frac{1}{3}(64x^4-160x^3+144x^2)-17x+\frac{9}{4}&\text{ if $\frac{1}{2}\leq x\leq\frac{3}{4}$,}\\
x&\text{ if $\frac{3}{4}\leq x\leq1$}
\end{cases}
$$
is an ${\mathcal S}^2$ diffeomorphism such that $f_1|_{[\frac{3}{4},1)}=\id_{[\frac{3}{4},1)}$, see Figure \ref{fig16}.

\begin{center}
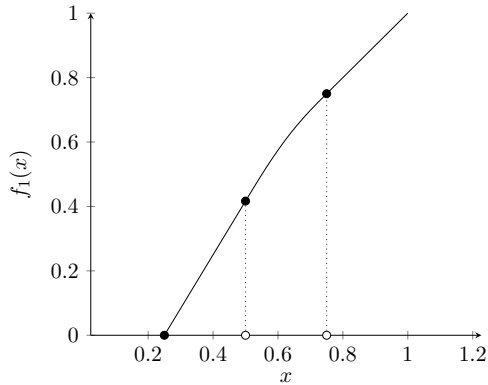

\begin{tikzpicture}[scale=0.75]
\begin{axis}[unit vector ratio=1 1 1,axis x line=bottom,
	axis y line=left,
	xlabel=$x$,ylabel=$f_1(x)$]
\addplot[domain=0.5:0.75,black] {(1/3)*(64*x^4-160*x^3+144*x^2)-17*x+9/4};
\addplot[domain=0.25:0.5,black] {5/12*(4*x-1)};
\addplot[domain=0.75:1,black] {x};
\draw[dotted] (axis cs:0.5,0) -- (axis cs:0.5,0.41666666);
\draw[dotted] (axis cs:0.75,0) -- (axis cs:0.75,0.75);
\addplot[holdot] coordinates{(0.5,0)(0.75,0)};
\addplot[soldot] coordinates{(0.25,0)(0.5,0.41666666)(0.75,0.75)};
\end{axis}
\end{tikzpicture}
\captionof{figure}{Graph of $f_1$.\label{fig16}}
\end{center}

(ii) The function $f_2:[\frac{1}{2},1)\to[0,1)$ given by 
$$
f_2(x):=\begin{cases}
\frac{11}{6}(2x-1)&\text{ if $\frac{1}{2}\leq x\leq\frac{5}{8}$,}\\
2048(\frac{x^4}{3}-\frac{11x^3}{12}+\frac{15x^2}{16})-863x+144&\text{ if $\frac{5}{8}\leq x\leq\frac{3}{4}$,}\\
x&\text{ if $\frac{3}{4}\leq x\leq1$}
\end{cases}
$$
is an ${\mathcal S}^2$ diffeomorphism such that $f_2|_{[\frac{3}{4},1)}=\id_{[\frac{3}{4},1)}$, see Figure \ref{fig17}.

\begin{center}
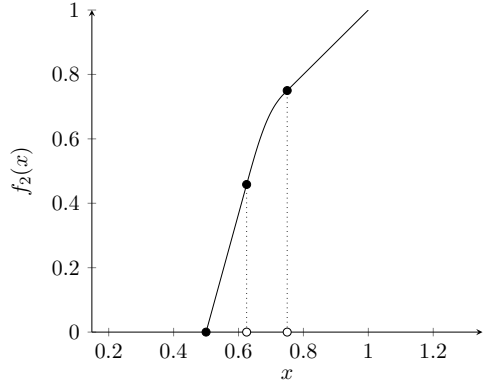

\begin{tikzpicture}[scale=0.75]
\begin{axis}[unit vector ratio=1 1 1,axis x line=bottom,
	axis y line=left,
	xlabel=$x$,ylabel=$f_2(x)$]
\addplot[domain=0.625:0.75,black] {682.66666*x^4-1877.33333*x^3+1920*x^2-863*x+144};
\addplot[domain=0.5:0.625,black] {11/6*(2*x-1)};
\addplot[domain=0.75:1,black] {x};
\draw[dotted] (axis cs:0.625,0) -- (axis cs:0.625,0.458333);
\draw[dotted] (axis cs:0.75,0) -- (axis cs:0.75,0.75);
\addplot[holdot] coordinates{(0.625,0)(0.75,0)};
\addplot[soldot] coordinates{(0.5,0)(0.625,0.458333)(0.75,0.75)};
\end{axis}
\end{tikzpicture}
\captionof{figure}{Graph of $f_2$.\label{fig17}}
\end{center}

(iii) The function $f_3:(-1,1)\to(-1,0)$ given by 
$$
f_3(x):=\begin{cases}
x&\text{ if $-1<x\leq-\frac{1}{2}$,}\\
\frac{1}{5}(16x^4+16x^3+x-1)&\text{ if $-\frac{1}{2}\leq x\leq0$,}\\
\frac{1}{5}(x-1)&\text{ if $0\leq x<1$}
\end{cases}
$$
is an ${\mathcal S}^2$ diffeomorphism such that $f_3|_{(-1,-\frac{1}{2}]}=\id_{(-1,-\frac{1}{2}]}$, see Figure \ref{fig18}.

\begin{center}
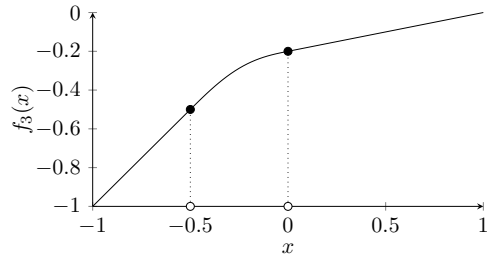

\begin{tikzpicture}[scale=0.75]
\begin{axis}[unit vector ratio*=1 1 1,axis x line=bottom,
	axis y line=left,
	xlabel=$x$,ylabel=$f_3(x)$]
\addplot[domain=-0.5:0,black] {-1/5+1/5*x+16/5*x^3+16/5*x^4};
\addplot[domain=-1:-0.5,black] {x};
\addplot[domain=0:1,black] {1/5*x-1/5};
\draw[dotted] (axis cs:-0.5,-1) -- (axis cs:-0.5,-0.5);
\draw[dotted] (axis cs:0,-1) -- (axis cs:0,-0.20);
\addplot[holdot] coordinates{(-0.5,-1)(0,-1)};
\addplot[soldot] coordinates{(0,-0.2)(-0.5,-0.5)};
\end{axis}
\end{tikzpicture}
\captionof{figure}{Graph of $f_3$.\label{fig18}}
\end{center}
\end{examples}

\section{Strict transform of Nash arcs under blow-up}\label{ap}

We recall here that the strict transform of an irreducible Nash curve germ under a sequence of blow-ups is again an irreducible Nash curve germ. We analyze first the images under Nash parameterization germs.

\begin{lem}\label{lac}
Let $g:=(g_1,\ldots,g_n):\R_0\to\R^n_0$ be an analytic map germ such that the germs $g(\{t>0\}_0)$ and $g(\{t<0\}_0)$ are different. Then $\im(g)$ is an irreducible analytic curve germ. In addition, if $g$ is Nash, then $\im(g)$ is an irreducible Nash curve germ.
\end{lem}
\begin{proof}
Let $G:\C_0\to\C^n_0$ be the complex analytic map germ induced by $g$. As $G$ is not identically zero, $G^{-1}(0)=\{0\}$. By \cite[Thm.3.4.24]{jp} $G$ is a finite analytic map germ. Using a finite representative of $G:\C_0\to\C^n_0$ and Remmert's Theorem \cite[Thm.VII.2.2]{na} we deduce that $Z_0:=G(\C_0)$ is a $1$-dimensional analytic germ. Consequently, $X_0:=Z_0\cap\R^n_0$ is an irreducible analytic curve germ. Let $\theta:\C_0\to Z_0$ be a normalization of $Z_0$ that is invariant under conjugation. As $X_0$ is coherent, $\theta(\R_0)=X_0$. There exists an analytic map germ $F:\C_0\to\C_0$ such that $G=\theta\circ F$. As $G$ and $\theta$ are invariant under conjugation, also $F$ is invariant under conjugation. Then $F$ is a univariate analytic function germ of order $k\geq1$ with real coefficients. As $g(\{t>0\}_0)$ and $g(\{t<0\}_0)$ are different, $k$ is odd and $g(\R_0)=X_0$ is an irreducible analytic curve germ. The Nash case is now straightforward.
\end{proof}

\begin{lem}\label{irredgamma}
Let $\Gamma_0\subset\R^n_0$ be an irreducible Nash curve germ and let $\Gamma$ be a representative. Let $f:M\to N$ be a Nash map between Nash manifolds and let $Z\subset M$ and $Y\subset N\subset\R^n$ be Nash subsets such that $0\in Y$. Assume that $\Gamma_0\subset N$, $\Gamma_0\not\subset Y$, the restriction $f|_{f^{-1}(\Gamma)\cap Z}:f^{-1}(\Gamma)\cap Z\to\Gamma$ is proper and the restriction $f|_{f^{-1}(\Gamma\setminus Y)\cap Z}:f^{-1}(\Gamma\setminus Y)\cap Z\to\Gamma\setminus Y$ is bijective. Denote $\Lambda:=\cl(f^{-1}(\Gamma\setminus Y)\cap Z)\cap f^{-1}(\Gamma)\cap Z$. Then there exists a point $p\in f^{-1}(0)\cap Z$ such that $\Lambda_p$ is an irreducible Nash curve germ, $\Lambda_p\cap f^{-1}(Y)=\{p\}$ and $f(\Lambda_p)=\Gamma_0$.
\end{lem}
\begin{proof}
We may assume $\Gamma\cap Y=\{0\}$ and $\Gamma=\im(\gamma)$ where $\gamma:(-1,1)\to N$ is a Nash arc such that $\gamma(0)=0$ and $\gamma$ is a homeomorphism onto its image. We claim: \em $\Lambda\cap f^{-1}(Y)\cap Z$ is a singleton $\{p\}$ and $\Lambda_p$ is an irreducible Nash curve germ\em.

The map $g:=f|_{\Lambda}:\Lambda\to\Gamma$ is proper. Thus, the restriction $g|_{\Lambda\setminus g^{-1}(0)}:\Lambda\setminus g^{-1}(0)\to\Gamma\setminus\{0\}$ is proper and bijective, so it is a semialgebraic homeomorphism. As $g$ is proper, $g^{-1}(0)\neq\varnothing$. Let $p\in g^{-1}(0)$ and observe that $p\in\cl(g^{-1}(\gamma((0,1))))\cup\cl(g^{-1}(\gamma((-1,0))))$.

Assume $p\in\Tt:=\cl(g^{-1}(\gamma((0,1))))$. Let $X_{1,p}$ be the smallest Nash germ that contains $\Tt_p$. Observe that $X_{1,p}\cap f^{-1}(Y)_p=\{p\}$ and $\gamma((0,1))_0\subset f(X_{1,p})\subset\Gamma_0$. As $f|_{f^{-1}(\Gamma\setminus Y)}:f^{-1}(\Gamma\setminus Y)\to\Gamma\setminus Y$ is bijective, $f(X_{1,p})$ is by Lemma \ref{lac} a Nash germ. Thus, $f(X_{1,p})=\Gamma_0$ and $p\in\cl(g^{-1}(\gamma((-1,0))))$. As $\Gamma$ is semialgebraically homeomorphic to $(-1,1)$, we conclude $\Lambda\setminus g^{-1}(0)$ is semialgebraically homeomorphic to $(-1,1)\setminus\{0\}$. Consequently, $\Lambda\cap f^{-1}(Y)=\{p\}$ is a singleton and $\Lambda_p=X_{1,p}$ is an irreducible Nash curve germ, as required.
\end{proof}

\section{Algebraic structure of a Nash normal crossing divisor}\label{app}

As an application of classical algebraization Artin--Mazur's result for Nash functions \cite[\S8.4]{bcr}, we show that if $Y$ is a Nash normal crossing divisor of a Nash manifold $M\subset\R^n$, there exist a Nash immersion of $M$ in some affine space $\R^m$ such that $M$ is a union of some connected components of its Zariski closure $V$ in $\R^m$ (a non-singular real algebraic subset of $\R^m$), the Zariski closure $X$ of $Y$ in $\R^m$ is a normal crossings divisor of $V$ and $M\cap X=Y$.

\renewcommand\theparagraph{\thethm.\arabic{paragraph}}

\begin{lem}\label{retoque}
Let $M\subset\R^n$ be a Nash manifold of dimension $d$ and let $Y\subset M$ be a Nash normal crossings divisor of $M$. Then up to a suitable Nash embedding of $M$ in some affine space $\R^m$ we may assume: 
\begin{itemize}
\item[(i)] $M$ is a union of connected components of its Zariski closure $V$ in $\R^m$, which is in addition a non-singular real algebraic subset of $\R^m$ of pure dimension $d$. 
\item[(ii)] The Zariski closure $X$ of $Y$ in $\R^m$ is a normal crossings divisor of $V$ and $M\cap X=Y$. 
\end{itemize}
\end{lem}
\begin{proof}
Let $Y_1,\ldots,Y_r$ be the irreducible components of $Y$ as a Nash subset of $M$, which are non-singular Nash hypersurfaces of $M$ in general position. As ${\mathcal N}(M)$ is a noetherian ring, there exist finitely many Nash functions $g_{ij}\in {\mathcal N}(M)$ such that $I(Y_i)=(g_{i1},\ldots,g_{i\ell}){\mathcal N}(M)$. Let $y\in Y$ and assume (after reordering the indices $i$ if necessary) that $y\in Y_i$ exactly for $i=1,\ldots,e$. As $Y$ is a Nash normal crossings divisor of $M$, we may assume (after reordering the indices $j$ if necessary) that the linear forms $d_yg_{11},\ldots,d_yg_{e1}:T_yM\to\R$ are linearly independent and $Y$ coincides with $\{g_{11}\cdots g_{e1}=0\}$ in some neighborhood of $y$ in $M$, see \cite[Prop.5.1]{fgr}.

By Artin--Mazur's description of Nash functions \cite[\S8.4]{bcr} $M$ is up to a Nash diffeomorphism an open and closed subset of a non-singular real algebraic set $V_0\subset\R^m$ of pure dimension $d$, which is in addition the Zariski closure of $M$, and the Nash functions $g_{ik}$ are restrictions to $M$ of polynomial functions $G_{ik}$ on $V_0$. Consider the real algebraic sets
$$
X_{0i}:=\{x\in V_0:\ G_{i1}(x)=0,\ldots,G_{i\ell}(x)=0\}\quad\text{and}\quad X_0:=X_{01}\cup\cdots\cup X_{0r}.
$$
The obstructions for $X_0$ to be a Nash normal crossings divisor of $V_0$ concentrate outside $M$, because $X_{0i}\cap M=Y_i$ for $i=1,\ldots,r$ and $Y$ is a Nash normal crossings divisor. In addition, $X_0\cap M=Y_1\cup\cdots\cup Y_r=Y$.

\paragraph{}\label{fin}
We claim: \em There exists an algebraic subset $Z$ of $V_0$ satisfying the following: 
\begin{itemize}
\item[(1)] $M\cap Z=\varnothing$. 
\item[(2)] Pick $x\in X_0\setminus Z$ and suppose (after reordering the indices $i$ if necessary) that $x\in X_{0i}$ exactly for $i=1,\ldots,e$. Then we may assume (after reordering the indices $j$ if necessary) that the linear forms $d_xG_{11},\ldots,d_xG_{e1}:T_yV_0\to\R$ are linearly independent and there exists an open Zariski neighborhood $U$ of $x$ in $V_0\setminus Z$ such that $X_0\cap U=\{G_{11}\cdots G_{e1}=0\}\cap U$.
\end{itemize}\em

Assume statement \ref{fin} proved for a while and let $h\in\R[\x]:=\R[\x_1,\ldots,\x_m]$ be a polynomial equation of $Z$ in $\R^m$. Consider the algebraic sets
$$
V:=\{(x,y)\in V_0\times\R:\ yh(x)=1\}\quad\text{and}\quad X:=\{(x,y)\in X_0\times\R:\ yh(x)=1\},
$$
which are biregularly equivalent to the constructible sets $V_0\setminus Z$ and $X_0\setminus Z$ via the projection onto the first $m$ coordinates. Using \cite[Cor.4.3.18]{jp} the reader can check readily that $V$ is a non-singular algebraic subset of $\R^{m+1}$ that, up to a Nash diffeomorphism, contains $M$ as an open and closed subset and $X$ is a normal crossing divisor of $V$ such that $X\cap M=Y$. 

\paragraph{}
Let $Z_0$ be the union of the singular sets $\Sing(\ol{Y_i}^{\rm zar})$ for $i=1,\ldots,r$ and all the irreducible components of $X_0$ different from $\ol{Y_1}^{\rm zar},\ldots,\ol{Y_r}^{\rm zar}$ (which are irreducible algebraic sets because $Y_1,\ldots,Y_r$ are irreducible Nash sets). The irreducible components of $X_0$ different from $\ol{Y_1}^{\rm zar},\ldots,\ol{Y_r}^{\rm zar}$ do not meet $M$ because: $X_0\cap M=Y$ and for each $y\in Y$ we assume $y\in Y_i$ exactly for $i=1,\ldots,e$ (after reordering the indices $i$ if necessary), the linear forms $d_yg_{11},\ldots,d_yg_{e1}:T_yM\to\R$ are linearly independent (after reordering the indices $j$ if necessary) and $Y$ coincides with $\{g_{11}\cdots g_{e1}=0\}$ in some neighborhood of $y$ in $M$ (recall that $g_{ij}=G_{ij}|_M$ for each pair of indices $i,j$). We have 
$$
Y_i\subset\ol{Y_i}^{\rm zar}\cap M\subset X_{0i}\cap M=Y_i\quad\text{and}\quad\Sing(\ol{Y_i}^{\rm zar})\cap M\subset\Sing(X_{0i})\cap M=\varnothing.
$$ 
Consequently, $Z_0\cap M=\varnothing$.

For each $i=1,\ldots,r$ let $G_{i1},\ldots,G_{i\ell},G_{i,\ell+1}\ldots,G_{is}$ be a system of generators of the ideal $I_\R(Y_i)=I_\R(\ol{Y}^{\rm zar})$ of polynomials in $\R[\x]$ vanishing identically on $Y_i$. Let $P_1,\ldots,P_s$ be a system of generators of the ideal $I_\R(M)=I_\R(V_0)$ of polynomials in $\R[\x]$ vanishing identically on $M$. As $V_0$ is non-singular and has dimension $d$, the rank
$$
\rk\{d_xP_1,\ldots,d_xP_s\}=m-d
$$
for each $x\in V_0$. Denote $I_r:=\{1,\ldots,r\}$ and for each non-empty subset $I\subset I_r$ define
$$
Z_I:=\Big\{x\in\bigcap_{i\in I}\ol{Y_i}^{\rm zar}:\ \rk\{d_xG_{ij},d_xP_j:T_x\R^m\to\R:\ i\in I,\ j=1,\ldots,s\}<m-d+\#I\Big\}\subset X_0,
$$
where $\#I$ denotes the cardinal of $I$, and observe that $Y\cap Z_I=\varnothing$. Thus, $Z:=Z_0\cup\bigcup_{\varnothing\neq I\subset I_r}Z_I$ is an algebraic subset of $X_0\subset V_0$. If $x\in X_0\setminus Z$, we define $I_x:=\{i\in I_r:\ x\in\ol{Y_i}^{\rm zar}\}$. Let us check: \em $Z\cap M=\varnothing$ and for each $x\in X_0\setminus Z$ the rank 
$$
\rk\{d_xG_{ij},d_xP_j(x):T_x\R^m\to\R:\ i\in I_x,\ j=1,\ldots,s\}=m-d+\#I_x
$$ 
and there exists an open Zarisiki neighborhood $U$ of $x$ in $V_0$ such that $X_0\cap U=\{G_{1j_1}\cdots G_{rj_r}=0\}\cap U$ for some indices $1\leq j_1,\ldots,j_r\leq s$ (depending on $x$)\em.

Observe first that
$$
M\cap Z=M\cap\bigcup_{\varnothing\neq I\subset I_r}Z_I=M\cap X_0\cap \bigcup_{\varnothing\neq I\subset I_r}Z_I=Y\cap \bigcup_{\varnothing\neq I\subset I_r}Z_I=\varnothing.
$$
Fix a point $x\in X_0\setminus Z$. As $x\not\in Z_{I_x}$, we have
$$
\rk\{d_xG_{ij},d_xP_j:T_x\R^m\to\R:\ i\in I_x,\ j=1,\ldots,s\}=(m-d)+\#I_x.
$$
We may assume (after reordering the indices $i,j$ if necessary) that $I_x:=\{1,\ldots,e\}$ and
\begin{equation}\label{gs}
\rk\{d_xG_{11},\ldots,d_xG_{e1},d_xP_1,\ldots,d_xP_{m-d}:T_x\R^m\to\R\}=(m-d)+e.
\end{equation}
Let $Z_x$ be the union of the irreducible components of the real algebraic set $\{G_{11}\cdots G_{e1}=0\}$ different from $\ol{Y_1}^{\rm zar},\ldots,\ol{Y_e}^{\rm zar}$ and observe that $\{G_{11}\cdots G_{e1}=0\}\setminus Z_x=(\bigcup_{i=1}^e\ol{Y_i}^{\rm zar})\setminus Z_x$. Condition \eqref{gs} guarantees that $x\not\in Z_x$. Consider the open Zariski neighborhood 
\begin{multline*}
U:=\{y\in V_0:\ \rk\{d_yG_{11},\ldots,d_yG_{e1},d_yP_1,\ldots,d_yP_{m-d}:T_y\R^m\to\R\}=(m-d)+e\}\\
\setminus\Big(Z_x\cup Z\cup\bigcup_{i=e+1}^r\ol{Y_i}^{\rm zar}\Big)
\end{multline*}
of $x$ in $V_0\setminus Z$. Observe that $X_0\cap U=(\bigcup_{i=1}^e\ol{Y_i}^{\rm zar})\cap U=\{G_{11}\cdots G_{e1}=0\}\cap U$, as required.
\end{proof}

\bibliographystyle{amsalpha}

\end{document}